\documentclass[a4paper,oneside,11pt]{article}

\usepackage{amsmath,amsfonts,amscd,amssymb}
\usepackage{longtable,geometry}
\usepackage[english]{babel}
\usepackage[utf8]{inputenc}
\usepackage[active]{srcltx}
\usepackage[T1]{fontenc}
\usepackage{graphicx}
\usepackage{pstricks}
\usepackage{bbm}
\usepackage{mdframed}
\usepackage{MnSymbol}
\usepackage{stmaryrd}
\usepackage{nicefrac}
\usepackage{tabularx}
\usepackage{multirow}
 \usepackage{rotating}
\usepackage{enumitem}
\usepackage{xcolor}
\usepackage{framed}

\colorlet{shadecolor}{blue!15}

\newtheorem{theorem}{Theorem}[section]
\newtheorem{conj}{Conjecture}

\newtheorem{corollary}[theorem]{Corollary}
\newtheorem{lemma}[theorem]{Lemma}
\newtheorem{proposition}[theorem]{Proposition}
\newtheorem{definition}[theorem]{Definition}

\newtheorem{exercise}{Exercise}

\newtheorem{remark}[theorem]{Remark}
\newtheorem{conjecture}[conj]{Conjecture}

\newcommand{\be}[1]{\begin{equation}\label{#1}}
\newcommand{\ee}{\end{equation}}
\numberwithin{equation}{section}

\newcommand{\ba}[1]{\begin{align}\label{#1}}
\newcommand{\ea}{\end{align}}
\numberwithin{equation}{section}

\newcommand{\ben}{\begin{equation*}}
\newcommand{\een}{\end{equation*}}
\numberwithin{equation}{section}

\newenvironment{proof}[1][\relax]
  {\paragraph{Proof\ifx#1\relax\else~of #1\fi}}%
  {~\hfill$\square$\par\bigskip}

\newcommand{\calA}{\mathcal{A}}
\newcommand{\calB}{\mathcal{B}}
\newcommand{\calC}{\mathcal{C}}
\newcommand{\calD}{\mathcal{D}}
\newcommand{\calE}{\mathcal{E}}
\newcommand{\calF}{\mathcal{F}}
\newcommand{\calG}{\mathcal{G}}
\newcommand{\calH}{\mathcal{H}}

\newcommand{\calK}{\mathcal{K}}

\newcommand{\calM}{\mathcal{M}}
\newcommand{\calN}{\mathcal{N}}

\newcommand{\calT}{\mathcal{T}}

\newcommand{\calV}{\mathcal{V}}

\newcommand{\bbC}{\mathbb{C}}

\newcommand{\bbE}{\mathbb{E}}

\newcommand{\bbG}{\mathbb{G}}
\newcommand{\bbH}{\mathbb{H}}

\newcommand{\bbL}{\mathbb{L}}

\newcommand{\bbN}{\mathbb{N}}

\newcommand{\bbP}{\mathbb{P}}

\newcommand{\bbR}{\mathbb{R}}
\newcommand{\bbS}{\mathbb{S}}
\newcommand{\bbT}{\mathbb{T}}
\newcommand{\bbU}{\mathbb{U}}
\newcommand{\bbV}{\mathbb{V}}

\newcommand{\bbZ}{\mathbb{Z}}

\newcommand{\U}{{\mathbf U}}
\newcommand{\V}{{\mathbf V}}

\newcommand{\X}{{\mathbf X}}
\newcommand{\Y}{{\mathbf Y}}

\newcommand{\g}{\gamma}
\newcommand{\ep}{\varepsilon}
\newcommand{\n}{{\mathbf n}}
\newcommand{\m}{{\mathbf m}}

\newcommand{\ol}{{\includegraphics[scale=0.2]{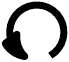}}}
\renewcommand{\Im}{\textrm{Im}}
\renewcommand{\Re}{\textrm{Re}}

\newcommand{\rk}[1]{\bgroup\color{red}%
  \par\medskip\hrule\smallskip%
  \noindent\textbf{#1}%
  \par\smallskip\hrule\medskip\egroup}

\newcommand{\bexo}{\begin{mdframed}[backgroundcolor=lightgray!20]
\scriptsize}
\newcommand{\eexo}{\end{mdframed}}
\title{Lectures on the Ising and Potts models on the hypercubic lattice}

\author{Hugo Duminil-Copin\thanks{
\texttt{duminil@ihes.fr} Institut des Hautes \'Etudes Scientifiques and Universit\'e de Gen\`eve
 \newline
This research was funded by a IDEX Chair from Paris Saclay and by the NCCR SwissMap from the Swiss NSF. These lecture notes describe the content of a class given at the PIMS-CRM probability summer school on the behavior of lattice spin models near their critical point. The author would like to thank warmly the organizers for offering him the opportunity to give this course. Also, special thanks to people who sent me comments, especially Timo Hirscher and Franco Severo. }}
\date{\today}

\begin{document}
\maketitle

\begin{abstract}
 Phase transitions are a central theme of statistical mechanics, and of probability more generally. Lattice spin models represent a general paradigm for phase transitions in finite dimensions, describing ferromagnets and even some fluids (lattice gases).
It has been understood since the 1980s that random geometric representations, such as the random walk and random current representations, are powerful tools to understand spin models. In addition to techniques intrinsic to spin models, such representations provide access to rich ideas from percolation theory. In recent years, for two-dimensional spin models, these ideas have been further combined with ideas from discrete complex analysis. Spectacular results obtained through these connections include the proofs that interfaces of the two-dimensional Ising model have conformally invariant scaling limits given by SLE curves, that the connective constant of the self-avoiding walk on the hexagonal lattice is given by $\sqrt{2+\sqrt 2}$. In higher dimensions, the understanding also progresses with the proof that the phase transition of Potts models is sharp, and that the magnetization of the three-dimensional Ising model vanishes at the critical point. These notes are largely inspired by {\cite{Dum11a,Dum13,Dum15}}.

\end{abstract}
\bigbreak
\begin{center}
\includegraphics[width=0.45\textwidth]{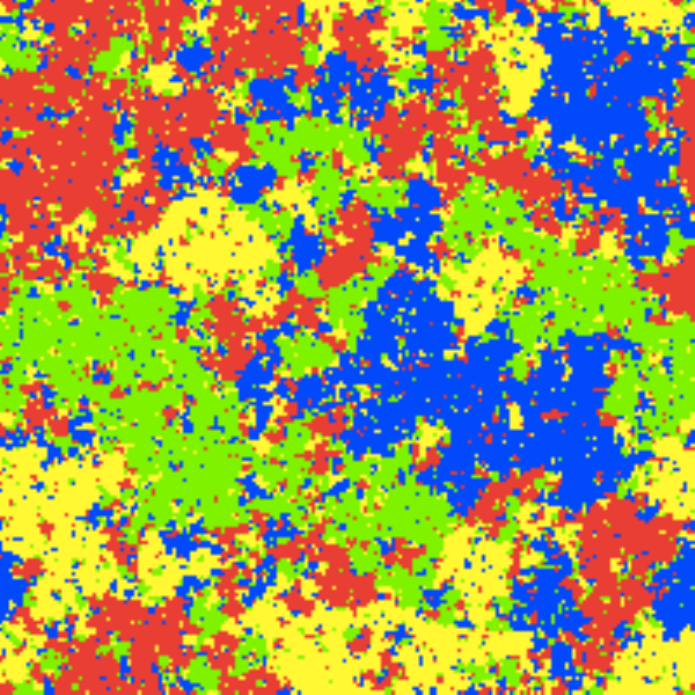}\\
{\footnotesize A simulation of the 4-state Potts model due to V. Beffara.}
\end{center}
\newpage

\tableofcontents

\section{Graphical representation of the Potts model}

\subsection{Lattice spin models}

Lattice models have been introduced as discrete models for real life experiments and were later on found useful to model a large variety of phenomena and systems ranging from ferroelectric materials to lattice gas. They also provide discretizations of Euclidean and Quantum Field Theories and are as such important from the point of view of theoretical physics. While the original motivation came from physics, they appeared as extremely complex and rich mathematical objects, whose study required the developments of important new tools that found applications in many other domains of mathematics. 

The zoo of lattice models is very diverse: it includes models of spin-glasses, quantum chains, random surfaces, spin systems, percolation models. Here, we focus on a smaller class of lattice models called spin systems. These systems are random collections of spin variables assigned to the vertices of a lattice. The archetypical example of such a model is provided by the Ising model, for which spins take value $\pm 1$.

\subsubsection{Definition of ferromagnetic lattice spin models}

In these notes, $\|\cdot\|$ denotes the Euclidean norm on $\bbR^d$. A graph $G=(V,E)$ is  given by a vertex-set $V$ and an edge set $E$ which is a subset of pairs $\{x,y\}\subset V$. We will denote an (unoriented) edge with endpoints $x$ and $y$ by $xy$.
 While lattice models could be defined on very general lattices, we focus on the special case of the lattice given by the vertex-set $\bbV:=\bbZ^d$ and the edge-set $\bbE$ composed of edges $xy$ with endpoints $x$ and $y$ (in $\bbZ^d$) satisfying $\|x-y\|=1$.  Below, we use the notation $\bbZ^d$ to refer both to the lattice and its vertex-set. For a subgraph $G=(V,E)$ of $\bbZ^d$, we introduce the boundary of $G$ defined by
 $$\partial G:=\{x\in V:\exists y\in \bbZ^d\text{ such that }xy\in\bbE\setminus E\}.$$ 

For a finite subgraph $G=(V,E)$ of $\bbZ^d$, attribute a {\em spin} variable $\sigma_x$ belonging to a certain set $\Sigma\subset\bbR^r$ to each vertex $x\in V$. A
 {\em spin configuration} $\sigma=(\sigma_x:x\in V)\in\Sigma^{V}$ is given by the collection of all the spins.
  Introduce the Hamiltonian of $\sigma$ defined by
$$H_G^{\rm f}(\sigma):=-\sum_{xy\in E}\,\sigma_x\cdot\sigma_y,$$
where $a\cdot b$ denotes the scalar product between $a$ and $b$ in $\bbR^d$. The above Hamiltonian corresponds to a ferromagnetic nearest-neighbor interaction. We will restrict ourselves to this case in these lectures, and refer to the corresponding papers for details on the possible generalizations to arbitrary interactions.

The {\em Gibbs measure on $G$ at inverse temperature $\beta\ge0$ with free boundary conditions} is defined by the formula
\begin{equation}\label{eq:Gibbs}\mu_{G,\beta}^{\rm f}[f]:=\frac{\displaystyle\int_{\Sigma^{V}}f(\sigma)\exp\big[-\beta H_{G}^{\rm f}(\sigma)\big]d\sigma}{\displaystyle\int_{\Sigma^{V}}\exp\big[-\beta H_{G}^{\rm f}(\sigma)\big]d\sigma}\end{equation}
for every $f:\Sigma^{V}\rightarrow \bbR$, where $d\sigma=\bigotimes_{x\in V}d\sigma_x$
is a product measure whose marginals $d\sigma_x$ are identical copies of a reference finite measure $d\sigma_0$ on $\Sigma$. Note that if $\beta=0$, then spins are chosen independently according to the probability measure $d\sigma_0/\int_{\Sigma} d\sigma_0$.

Similarly, for ${\rm b}\in\Sigma$, introduce the {\em Gibbs measure $\mu_{G,\beta}^{\rm b}$ on $G$ at inverse temperature $\beta$ with boundary conditions $\rm b$} defined as $\mu_{G,\beta}^{\rm f}[\,\cdot\,|\sigma_x={\rm b},\forall x\in \partial G]$.

A priori, $\Sigma$ and $d\sigma_0$ can be chosen arbitrarily, thus leading to different examples of lattice spin models. The following (far from exhaustive) list of spin models already illustrates the vast variety of possibilities that such a formalism offers.
\paragraph{Ising model.} $\Sigma=\{-1,1\}$ and $d\sigma_0$ is the counting measure on $\Sigma$.  This model was introduced by Lenz in 1920 \cite{Len20} to model the temperature, called Curie's temperature, above which a magnet looses its ferromagnetic properties. 
 It was studied in his PhD thesis by Ising \cite{Isi25}.

\paragraph{Potts model.} $\Sigma=\bbT_q$ ($q\ge2$ is an integer), where $\bbT_q$ is a simplex in $\bbR^{q-1}$ (see Fig.~\ref{fig:tetrahedron}) containing $1:=(1,0,\dots,0)$ such that for any $a,b\in\bbT_q$,
$$a\cdot b=\begin{cases}\ \ 1&\text{ if $a=b$,}\\
\ -\frac1{q-1}&\text{ otherwise.}\end{cases}$$
and $d\sigma_0$ is the counting measure on $\Sigma$. This model was introduced as a generalization of the Ising model to more than two possible spins by Potts in 1952 \cite{Pot52} following a suggestion of his adviser Domb. While the model received little attention early on, it became an object of great interest in the last fourty years. Since then, mathematicians and physicists have been studying it intensively, and a lot is known on its rich behavior.
\begin{figure}
\begin{center}
\includegraphics[width=0.70\textwidth]{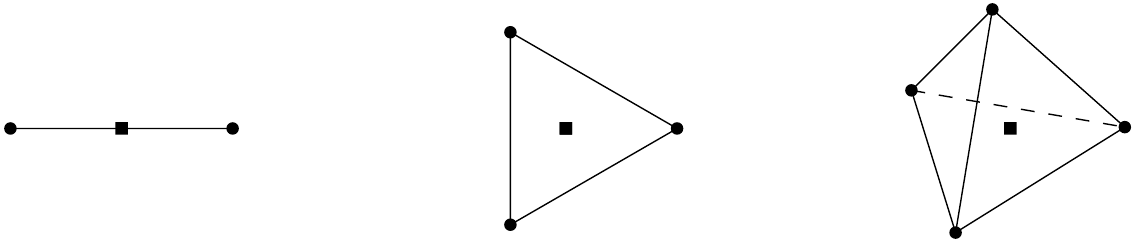}
\caption{From left to right, $\bbT_2$, $\bbT_3$ and $\bbT_4$.}\label{fig:tetrahedron}
\end{center}
\end{figure}

\paragraph{Spin $O(n)$ model.} $\Sigma$ is the unit sphere in dimension $n$ and $d\sigma_0$ is the surface measure. This model was introduced by Stanley in 1968 \cite{Sta68}. This is yet another generalization of the Ising model (the case $n=1$ corresponds to the Ising model) to continuous spins. The $n=2$ and $n=3$ models were introduced slightly before the general case and are called the $XY$ and (classical) Heisenberg models respectively.
\paragraph{Discrete Gaussian Free Field (GFF).} $\Sigma=\bbR$ and $d\sigma_0=\exp(-\sigma_0^2/2)d\lambda(\sigma_0),$ where $d\lambda$ is the Lebesgue measure on $\bbR$. The discrete GFF is a natural model for random surfaces fluctuations. We refer to Biskup's lecture notes for details.

\paragraph{The $\phi^4_d$ lattice model on $\bbZ^d$.} $\Sigma=\bbR$ and 
$d\sigma_0=\exp(-a\sigma_0^2-b\sigma_0^4)d\lambda(\sigma_0),$ where $a\in\bbR$ and $b\ge0$.
This model interpolates between the GFF corresponding to $a=1/2$ and $b=0$, and the Ising model corresponding to the limit as $b=-a/2$ tends to $+\infty$.  

\begin{mdframed}[backgroundcolor=green!00]
{\bf Notation.}
The family of lattice models is so vast that it would be hopeless to discuss them in full generality. For this reason, we chose already (in the definition above) to focus on nearest-neighbor ferromagnetic interactions. Also, we will mostly discuss two generalizations of the Ising model, namely the Potts and $O(n)$ models.\eexo

\subsubsection{Phase transition in Ising, Potts and $O(n)$ models}

We wish to illustrate that the theory of lattice spin models is both very challenging and very rich.  For this, we wish to screen quickly through the possible behaviors of spin models. An important disclaimer: this section is not rigorous and most of the claims will not be justified before much later in the lectures. It is therefore not surprising if some of the claims of this section sound slightly bold at this time.

Assume that the measures introduced above can be extended to infinite volume by taking weak limits of measures $\mu_{G,\beta}^{\rm f}$ and $\mu_{G,\beta}^{\rm b}$ as $G$ tends to $\bbZ^d$ (sometimes called taking the thermodynamical limit), and denote the associated limiting measures by $\mu^{\rm f}_\beta$ and $\mu^{\rm b}_\beta$.

The behavior of the model in infinite volume can differ greatly depending on $\beta$. In order to describe the possible behaviors, introduce the following properties: 
\begin{itemize}\item The model exhibits {\em spontaneous magnetization} at $\beta$ if 
\begin{equation}
\tag{\rm MAG$_\beta$}\mu_{\beta}^{\rm b}[
\sigma_0\cdot{\rm b}]>0.\end{equation}
\item The model exhibits {\em long-range ordering} at $\beta$ if 
\begin{equation}
\tag{\rm LRO$_\beta$}\lim_{\|x\|\rightarrow\infty}\mu_{\beta}^{\rm f}[
\sigma_0\cdot\sigma_x]>0.
\end{equation}
\item The model exhibits {\em exponential decay of correlations} at $\beta$ if 
\begin{equation}
\tag{\rm EXP$_\beta$}\exists c_\beta>0\text{ such that }\mu_{\beta}^{\rm f}[
\sigma_0\cdot\sigma_x]\le e^{-c_\beta \|x\|}\text{ for all }x\in\bbZ^d.
\end{equation}\end{itemize}
(Note that for the symmetries of $\Sigma$ implies that $\mu_{\beta}^{\rm b}[
\sigma_0\cdot{\rm b}]$ does not depend on the choice of ${\rm b}$.) These three properties lead to three critical parameters separating phases in which they occur or not:
\begin{align*}
\beta_c^{\rm mag}&:=\inf\{\beta>0:\text{(MAG$_\beta$)}\},\\
\beta_c^{\rm lro}&:=\inf\{\beta>0:\text{(LRO$_\beta$)}\},\\
\beta_c^{\rm exp}&:=\sup\{\beta>0:\text{(EXP$_\beta$)}\}.
\end{align*}
The first parameter $\beta_c^{\rm mag}$ is usually called the {\em critical inverse temperature} and is simply denoted $\beta_c$. In the cases we will study, $\beta_c^{\rm lro}=\beta_c$ (see Section~\ref{sec:uniqueness}) and we therefore do not discuss when they are distinct in details. 

Models with $\Sigma$ discrete for $d\ge2$, or arbitrary $\Sigma$ for $d\ge3$, are expected to have spontaneous magnetization for $\beta\gg1$ (thus proving that $\beta_c<\infty$). We will also see later that when $\beta_c<\infty$, one can often prove\footnote{One may also have $\beta_c^{\rm exp}<\beta_c<\infty$, as shown in \cite{FroSpe81} for the planar Clock model with $q\gg1$ states, but this situation is less common.} that $\beta_c^{\rm exp}=\beta_c$. In such case, we say that the model undergoes a {\em sharp order/disorder phase transition}. If the model satisfies (MAG$_{\beta_c}$), the phase transition is said to be {\em discontinuous}; otherwise, it is {\em continuous}.

On the contrary, the Mermin-Wagner theorem \cite{IofShlVel02,MerWag66} states that a model on $\bbZ^2$ for which $\Sigma$ is a compact continuous connected Lie group satisfies $\beta_c=+\infty$. Then, two cases are possible:
\medbreak\noindent
$\bullet$ $\beta^{\rm exp}=\infty$: the model does not undergo any phase transition. Polyakov \cite{Pol75} predicted this behavior for planar $O(n)$-models with $n\ge3$. We refer to \cite{DumPelSam14} and references therein for a more precise discussion.
\medbreak\noindent
$\bullet$ $\beta^{\rm exp}<\infty$: the model undergoes a {\em {Berezinsky-Kosterlitz-Thouless}} (BKT) phase transition. This type of phase transition is named after Berezinsky and Kosterlitz-Thouless\footnote{Kosterlitz and Thouless were awarded a Nobel prize in 2016 for their work on topological phase transitions.}, who introduced it (non rigorously) for the planar $XY$-model in two independent papers \cite{Ber72,KosTho73}. Note that in such case, there is no spontaneous magnetization at any $\beta$.
\begin{figure}
\includegraphics[width=0.27\textwidth,angle=90]{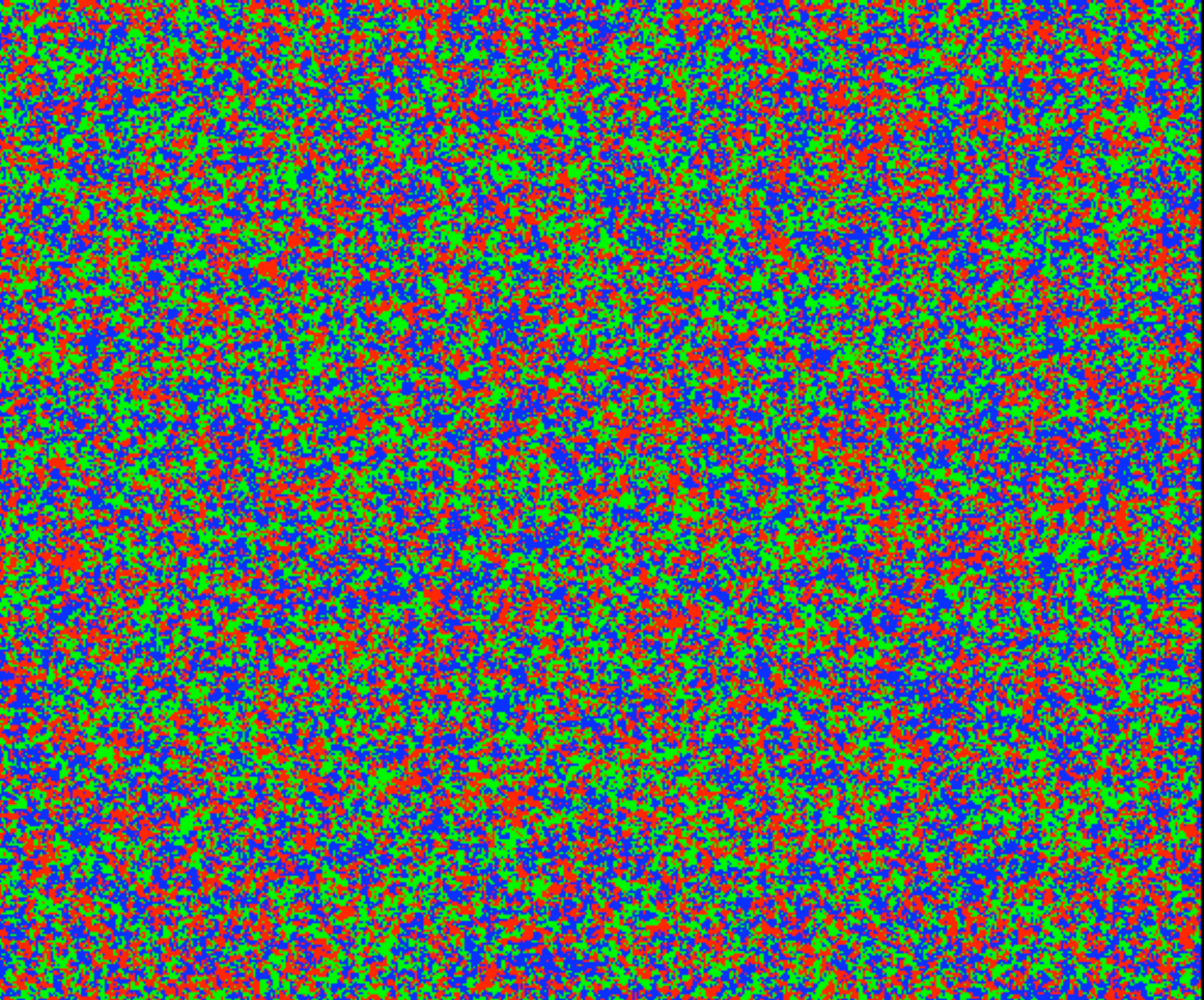}
\includegraphics[width=0.27\textwidth,angle=90]{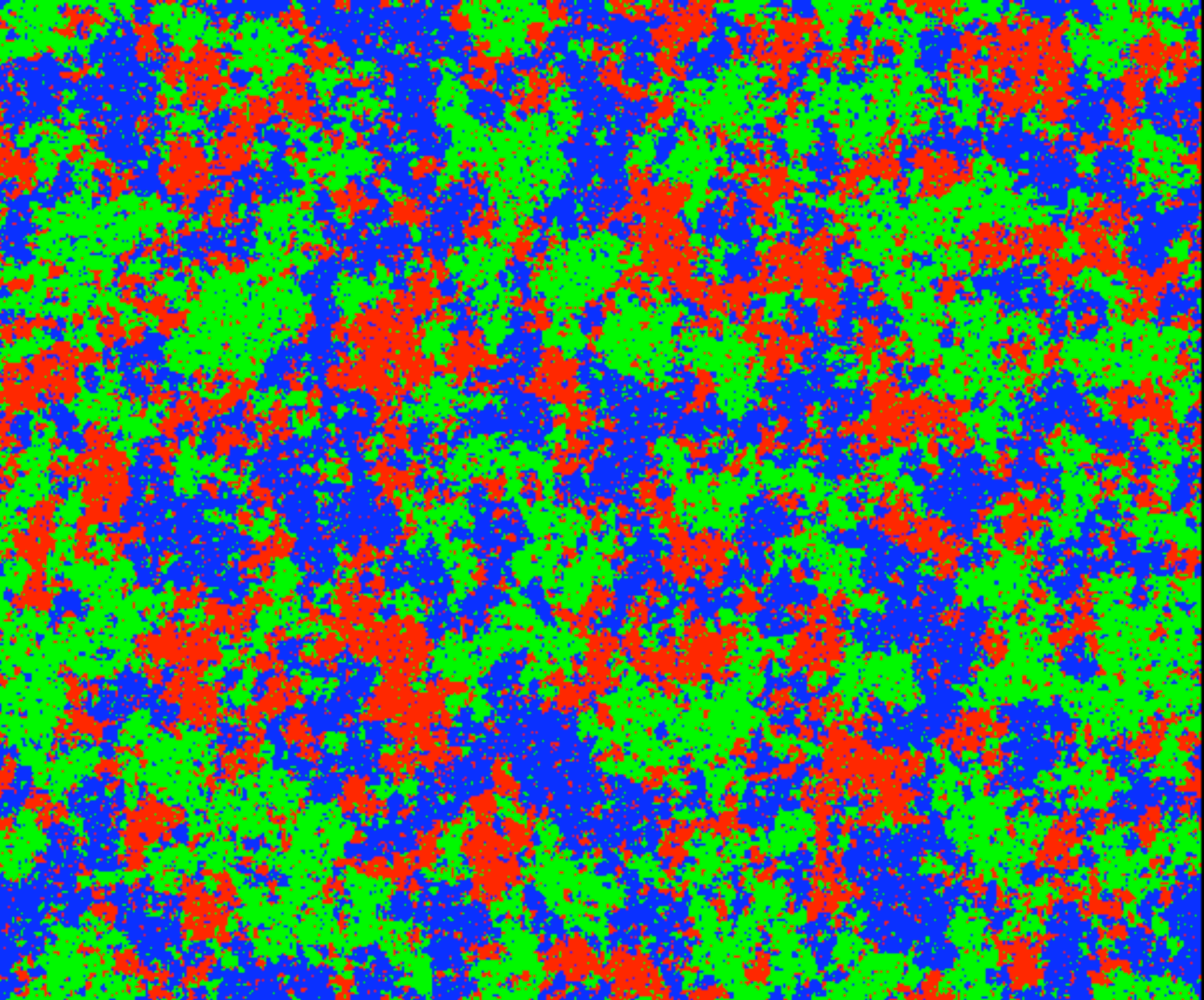}
\includegraphics[width=0.27\textwidth,angle=90]{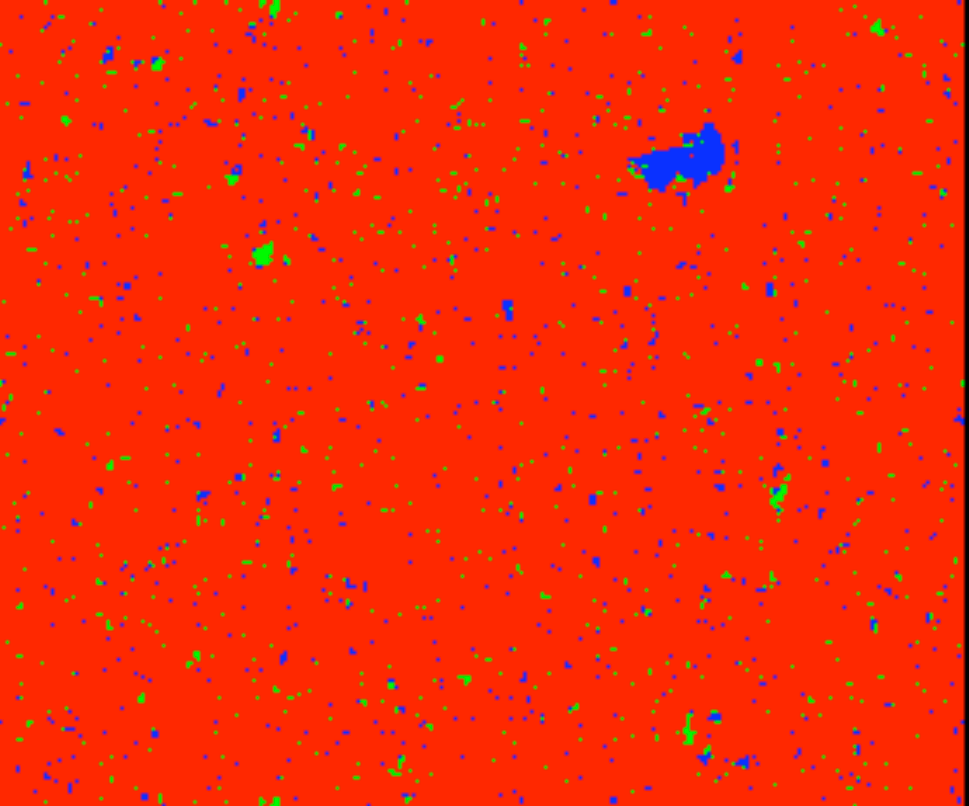}
\caption{Simulations of three-state planar Potts model at subcritical, critical and supercritical temperatures.}
\end{figure}

\bexo
\begin{exercise}
Prove that for the Ising model on $\bbZ$, $\beta_c=\beta_c^{\rm exp}=\beta_c^{\rm lro}=+\infty$. Prove the same result for the Potts model with $q\ge3$. What can be said for the spin $O(n)$ models?
\end{exercise}
\eexo
To conclude this section, let us draw a panorama of questions. The table below gathers the behaviors that are expected for the Ising, Potts and spin $O(n)$ models. \bigbreak
\renewcommand{\arraystretch}{1.5}
\quad\begin{tabularx}{\textwidth}{@{}c@{}c@{}|c@{}|c@{}} 

\cline{3-4}
    & & \multicolumn{1}{c}{$d=2$} & \multicolumn{1}{|c|}{$d\ge 3$}   \\
\cline{1-4}
  \multicolumn{2}{|c|}{\multirow{1}{*}{Ising}}      & \multicolumn{2}{c|}{Continuous sharp order-disorder PT} \\
    \cline{1-2}\cline{4-4}
   \multicolumn{1}{|c}{\multirow{2}{*}{Potts}} &  \multicolumn{1}{|c|}{\multirow{1}{*}{$q\in\{3,4\}$}} &  
  & \multicolumn{1}{c|}{}  \\
    \cline{2-2}\cline{3-3}
 \multicolumn{1}{|c}{}& \multicolumn{1}{|c|}{\multirow{1}{*}{$q\ge 5$}} & \multicolumn{2}{c|}{Discontinuous sharp order-disorder PT}  \\
    
 \cline{1-2}\cline{3-4}
 \multicolumn{1}{|c}{\multirow{2}{*}{$O(n)$}} & \multicolumn{1}{|c|}{\multirow{1}{*}{$n=2$}}   &  \multicolumn{1}{c|}{BKT PT} & \multicolumn{1}{c|}{}  \\    
    \cline{2-2}\cline{3-3}
 \multicolumn{1}{|c}{}& \multicolumn{1}{|c|}{\multirow{1}{*}{$n\ge3$}} & \multicolumn{1}{c|}{{Absence of PT}}  & \multicolumn{1}{l|}{{Continuous sharp} order-disorder PT}  \\
    \cline{1-4}
\end{tabularx}
\medskip
\medbreak
The claims about Ising and Potts models will all be proved, except the discontinuity of the phase transition for $q\ge3$ and $d\ge 3$, which is known  only for $q\ge q_c(d)\gg1$ \cite{KotShl82} or $d\ge d_c(q)\gg1$ \cite{BisCha03}. We will not deal with continuous spins, but we  mention that the understanding is more restricted there. In two dimensions, it is known that models with continuous spin symmetry cannot have an order-disorder phase transition \cite{MerWag66}.  The proof that the $O(1)$ model undergoes a BKT phase transition is due to Fr\"ohlich and Spencer \cite{FroSpe81}, while the existence of a phase transition in dimension $d\ge3$ goes back to Fr\"ohlich, Simon and Spencer \cite{FroSimSpe76}. The fact that the phase transition is continuous and sharp in dimension $d\ge3$ is still open. Proving Polyakov's conjecture, i.e.~that spin $O(n)$ models do not undergo any phase transition in dimension 2, is one of the biggest problem in mathematical physics.

%
\begin{figure}
\includegraphics[width=0.31\textwidth]{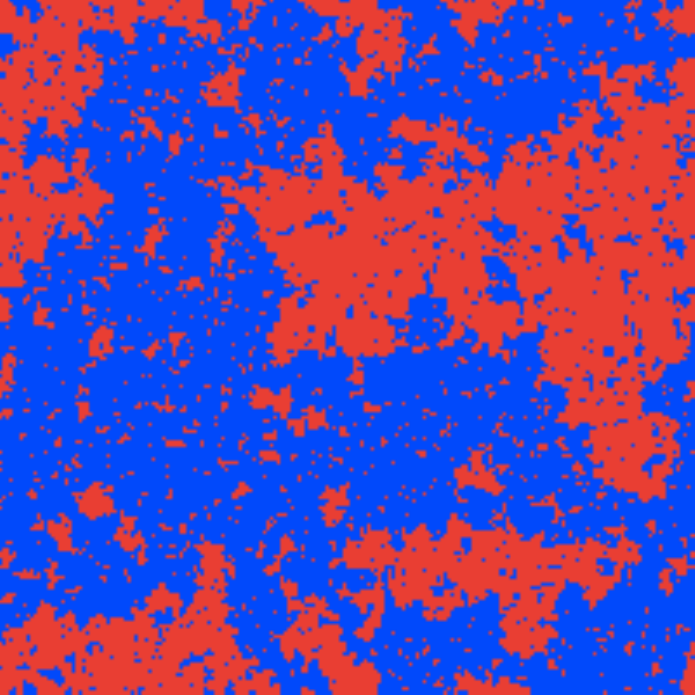}\quad\includegraphics[width=0.31\textwidth]{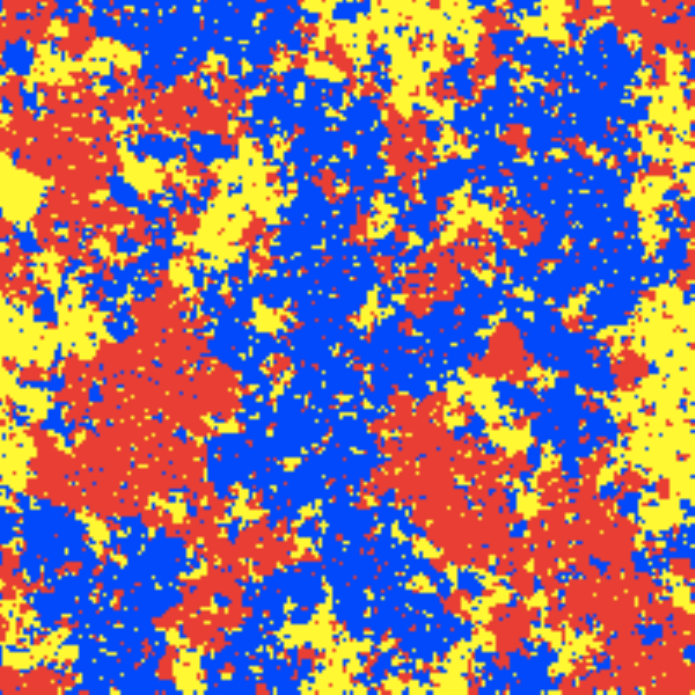}\quad\includegraphics[width=0.31\textwidth]{Potts_4}

\includegraphics[width=0.31\textwidth]{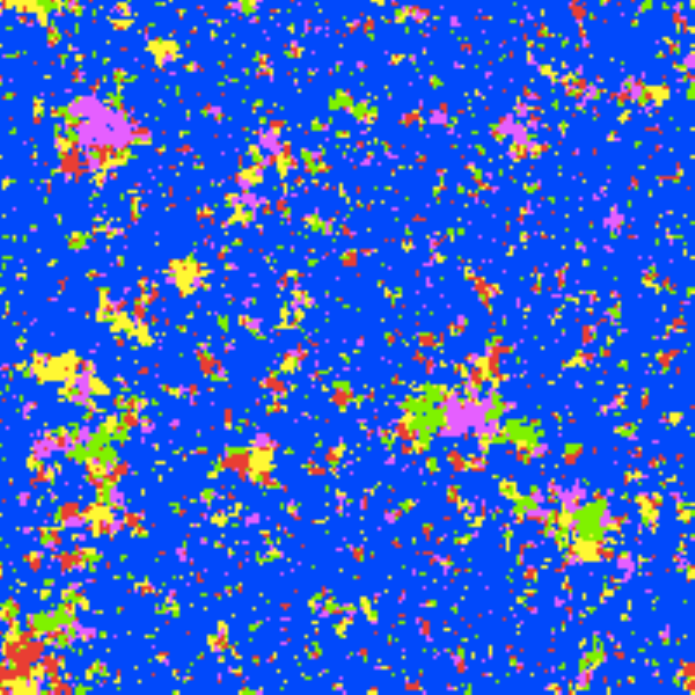}\quad\includegraphics[width=0.31\textwidth]{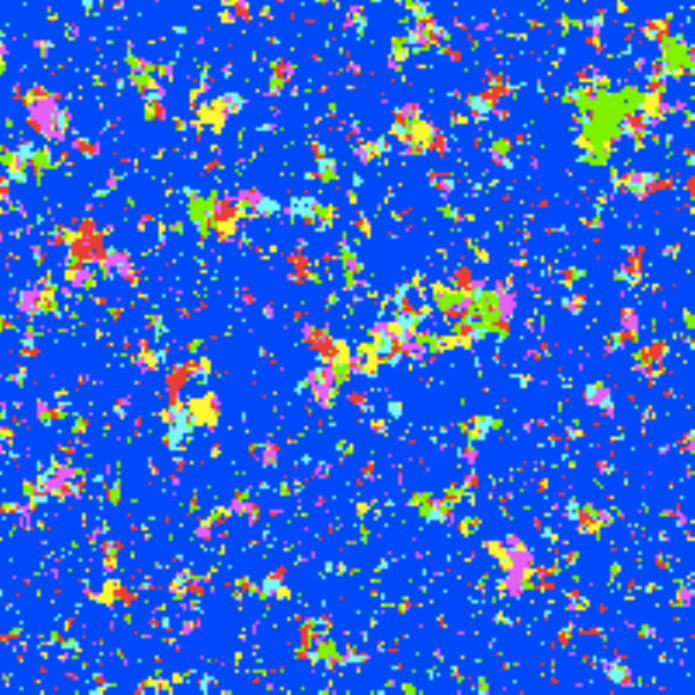}\quad\includegraphics[width=0.31\textwidth]{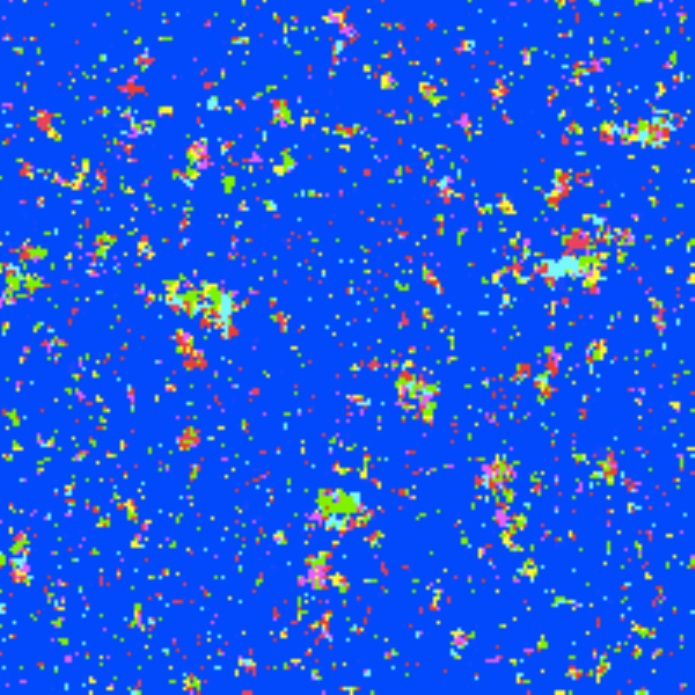}
\caption{Simulations of the critical planar Potts model with $q$ equal to $2$, $3$, $4$, $5$, $6$ and $9$ respectively. The behavior for $q\le 4$ is clearly different from the behavior for $q>4$. In the first three pictures, each color (corresponding to each element of $\bbT_q$) seems to play the same role, while in the last three, one color wins over the other ones.}
\end{figure}
%

\begin{mdframed}[backgroundcolor=green!00]
{\bf Notation.}
The behavior of lattice models with a space of spins which is continuous is quite different from the one with discrete spins. For this reason, we choose to focus on typical examples of the second kind. From now on, we work with the Ising and Potts model only. We denote the measure for the $q$-state Potts model $\mu^\#_{G,\beta,q}$. In order to lighten the notation, the measure for the Ising model is denoted by $\mu^\#_{G,\beta}$ rather than $\mu^\#_{G,\beta,2}$. Also, we will use $+$ and $-$ instead of $+1$ and $-1$.\eexo

\subsection{Graphical representation of Potts models}

We would like to have a more geometric grasp at correlations between spins of lattice models. In order to do so, we introduce another type of models, called percolation models. 

A {\em percolation configuration} $\omega=(\omega_e:e\in E)$ on $G=(V,E)$ is an element of $\{0,1\}^{E}$. If $\omega_e=1$, the edge $e$ is said to be {\em open}, otherwise $e$ is said to be {\em closed}.
A configuration $\omega$ can be seen as a subgraph of $G$ with vertex-set $V$ and  edge-set $\{e\in E:\omega_e=1\}$. A {\em percolation model} is given by a distribution on percolation configurations on $G$. 

In order to study the connectivity properties of the (random) graph $\omega$, we introduce some notation. A {\em cluster} is a maximal connected component of the graph $\omega$ (it may be an isolated vertex).
Two vertices $x$ and
$y$ are {\em connected in $\omega$} if they are in the same cluster. We denote this event
by $x\longleftrightarrow y$. For $A,B\subset\bbZ^d$, set
$A\longleftrightarrow B$ if there exists a vertex of $A$ connected to
a vertex of $B$. We also allow ourselves to consider $B=\infty$, in which case we mean that a vertex in $A$ is in an infinite cluster.

 The simplest example of percolation model is provided by {\em Bernoulli percolation}: each edge is open with probability $p$, and closed with probability $1-p$, independently of the states of other edges. Below, the measure is denoted by $\bbP_p$ (its expectation is denoted by $\bbE_p$). This model was introduced by Broadbent and Hammersley in 1957 \cite{BroHam57} and has been one of the most studied probabilistic model. We refer to \cite{Gri99} for a book on the subject.

Here, we will be interested in a slightly more complicated percolation model, named the {\em random-cluster model}, which is a percolation model in which the states open or closed of edges depend on each others. 
This model was introduced by Fortuin and Kasteleyn in 1972 \cite{ForKas72} and is sometimes referred to as the {\em Fortuin-Kasteleyn percolation}.

\subsubsection{Definition of the random-cluster model}

Let $G$ be a finite subgraph of $\bbZ^d$. Let 
$o(\omega)$ and $c(\omega)$ denote the number of open and
closed edges of $\omega$. Define
\emph{boundary conditions} $\xi$ to be
a partition $P_1\sqcup \dots\sqcup P_k$ of 
$\partial G$. For boundary conditions $\xi$, define the graph $\omega^\xi$ obtained from $\omega$ by contracting, for each $1\le i\le k$, all the vertices of $P_i$ into one vertex. Also, let $k(\omega^\xi)$ be the number of clusters in the graph $\omega^\xi$.

As an example, the {\em free boundary conditions} (denoted $0$) correspond
to the partition composed of singletons only: $\omega^0=\omega$ and we prefer the lighter notation $k(\omega)$ to $k(\omega^0)$. The {\em wired boundary conditions} (denoted $1$) correspond to the partition $\{\partial G\}$: $k(\omega^1)$ is the number of clusters obtained if all clusters touching the boundary are counted as 1. In general, a subgraph $\xi$ of $\bbZ^d$ induces boundary conditions as follows: two vertices of $\partial G$ are in the same $P_i$ if they are in the same cluster of $\xi$. In this case, boundary conditions will often be identified with the graph $\xi$.
 
 \bexo
\begin{exercise}
Construct the random-cluster on the torus as the random-cluster model on a finite box with a proper choice of boundary conditions.\end{exercise}
\eexo

\begin{definition}The probability measure 
$\phi^{\xi}_{G,p,q}$ of the random-cluster model on $G$ with {\em edge-weight} 
$p\in[0,1]$, {\em cluster-weight} $q>0$ and boundary conditions $\xi$ is defined by
\begin{equation}
  \label{probconf}
  \phi_{G,p,q}^{\xi} [\omega] :=
  \frac {p^{o(\omega)}(1-p)^{c(\omega)}q^{k(\omega^\xi)}}
  {Z_{G,p,q}^{\xi}}
\end{equation}
for every configuration $\omega\in\{0,1\}^{E}$. The constant $Z_{G,p,q}^{\xi}$ is a 
normalizing constant, referred to as the \emph{partition function}, defined in such a way that the sum over all configurations equals 1. 
\end{definition}

Fortuin and Kasteleyn introduced the random-cluster model as a unification of different models of statistical physics satisfying series/parallel laws when modifying the underlying graph:
\begin{itemize}
\item For $q=1$, the random-cluster model corresponds to Bernoulli percolation. In this case, and to distinguish with the case $q\ne 1$, we prefer the notation $\bbP_p$ instead of the random-cluster notation. 
\item For integers $q\ge2$, the model is related to Potts models; see Section~\ref{sec:ES}.
\item For $p\rightarrow 0$ and $q/p\rightarrow0$, the model is connected to electrical networks via Uniform Spanning Trees; see Exercise~\ref{exo:UST}.\end{itemize} 
\bexo
\begin{exercise}\label{exo:UST}
Consider a finite graph $G=(V,E)$. Prove that the limit of $\phi_{G,p,q}^0$ with $p\rightarrow 0$ and $q/p\rightarrow 0$ is the Uniform Spanning Tree on $G$, i.e.~the uniform measure on connected subgraphs of the form $H=(V,F)$, with $F$ not containing any cycle. \end{exercise}
\eexo

Let us mention two important properties of random-cluster models.
For boundary conditions $\xi=P_1\sqcup\cdots\sqcup P_k$ and $\psi\in\{0,1\}^{E\setminus\{e\}}$, where $e=xy$, one may easily check that
\begin{equation}\label{eq:arg}\phi_{G,p,q}^\xi[\omega_e=1|\omega_{|E\setminus\{e\}}=\psi]=\phi^{\psi^\xi}_{\{e\},p,q}[\omega_e=1]=\begin{cases}\ \ \ \ \ \ \  p&\text{ if $x\longleftrightarrow y$ in $\psi^\xi$,}\\\displaystyle\frac{p}{p+q(1-p)}&\text{ otherwise.}\end{cases}\end{equation}
Note that in particular the model satisfies the {\em finite energy property}, meaning that there exists $c_{\rm FE}>0$ such that for any $e$ and $\psi$
\begin{equation}\phi_{G,p,q}^\xi[\omega_e=1|\omega_{|E\setminus\{e\}}=\psi]\in [c_{\rm FE},1-c_{\rm FE}].\label{eq:finite energy}\tag{FE}\end{equation}
Also, \eqref{eq:arg} can be extended by induction to any subgraph $G'=(V',E')$ of $G$, in the sense that for any boundary conditions $\xi$ and any $\psi\in\{0,1\}^{E\setminus E'}$ and $\psi'\in\{0,1\}^{E'}$,
\begin{equation}\phi_{G,p,q}^\xi[\omega_{|E'}=\psi'|\omega_{|E\setminus E'}=\psi]=\phi^{\psi^\xi}_{G',p,q}(\psi').\label{eq:domain Markov}\tag{DMP}\end{equation}
(Recall the definition of the graph $\psi^\xi$ from above.) This last property is called the {\em domain Markov property}.

\bexo
\begin{exercise}
Prove carefully the finite energy property \eqref{eq:finite energy} and the domain Markov property \eqref{eq:domain Markov}.
\end{exercise}
\eexo

\subsubsection{The coupling between the random-cluster and Potts models}\label{sec:ES}

The random-cluster model enables us to rephrase correlations in Potts models in terms of random subgraphs of $\bbZ^d$. This is the object of this section.

Consider an integer $q\ge2$ and let $G$ be a finite graph. Assume that a configuration $\omega\in\{0,1\}^{E}$ is given. One can deduce a spin configuration $\sigma\in\bbT_q^{V}$ by assigning uniformly and independently to each cluster a spin. More precisely, consider a iid family of uniform random variables $\sigma_\calC$ on $\bbT_q$ indexed by clusters $\calC$ in $\omega$. We then define $\sigma_x$ to be equal to $\sigma_\calC$ for every $x\in\calC$. Note that all the vertices in the same cluster automatically receive the same spin.
\begin{proposition}[Coupling for free boundary conditions]\label{prop:coupling wired}
Fix an integer $q\ge2$, $p\in(0,1)$ and $G$ finite. If $\omega$ is distributed according to $\phi_{G,p,q}^0$ then $\sigma$ constructed above is distributed according to the $q$-state Potts measure $\mu_{G,\beta,q}^{\rm f}$, where \begin{equation}\label{eq:beta p}\beta:=-\tfrac{q-1}{q}\ln (1-p).\end{equation}
\end{proposition}
\begin{proof}  Consider the law
  ${\bf P}$ of the pair $(\omega,\sigma)$, where $\omega$ is a percolation
  configuration with free boundary conditions and $\sigma$ is the
  corresponding spin configuration constructed as explained
  above. By definition, the first marginal of the distribution is sampled according to $\phi^0_{G,p,q}$. 
We wish to compute the law of the second marginal.  

Say that the configurations $\sigma\in\bbT_q^{V}$ and $\omega\in  \{0,1\}^{E}$ are {\em compatible}  if $$\forall xy\in E~:~\omega_{xy}=1\Longrightarrow\sigma_x=\sigma_y.$$Then, if $\omega$ and $\sigma$ are not compatible, ${\bf P}[(\omega,\sigma)]=0$, and if they are,
  \begin{align*}
    {\bf P}[(\omega,\sigma)]~&=~\tfrac{1}{Z^0_{G,p,q}}\ p^{o(\omega)}
    (1-p)^{c(\omega)}q^{k(\omega)}\cdot
    q^{-k(\omega)}=~\tfrac{1}{Z^0_{G,p,q}}\ 
    p^{o(\omega)}(1-p)^{c(\omega)}.
  \end{align*}
For $\sigma\in\bbT_q^{V}$, introduce $E_\sigma:=\{xy\in E:\sigma_x\ne \sigma_y\}$ and note that $\omega$ compatible with $\sigma$ must satisfy $\omega_{xy}=0$ for edges $xy\in E_\sigma$, and that there is no restriction on $\omega_{xy}$ for edges $xy\notin E_\sigma$. Summing ${\bf P}[(\omega,\sigma)]$ over configurations $\omega$ compatible with $\sigma$, we find
    \begin{align*}
    {\bf P}[\sigma]
    ~&=~\frac{1}{Z^0_{G,p,q}}\ (1-p)^{|E_\sigma|}
    \underbrace{\sum_{\omega'\in\{0,1\}^{E\setminus E_\sigma}}
    p^{o(\omega')}(1-p)^{c(\omega')}}_{=1}=~\underbrace{\frac{{\rm e}^{-\beta |E|}}{Z^0_{G,p,q}}}_{C} \exp[-\beta H^{\rm f}_G(\sigma)].\end{align*}
In the second equality, we used that $1-p=\exp(-\tfrac q{q-1}\beta)$ and 
$$H^{\rm f}_G[\sigma]=\tfrac1{q-1}|E_\sigma|-|E\setminus E_\sigma|=\tfrac{q}{q-1}|E_\sigma|-|E|.$$
The proof follows readily since $C$ does not depend on $\sigma$, hence is equal to $1/Z^{\rm f}_{G,\beta,q}$.\end{proof}

\bexo

\begin{exercise}[reverse procedure] \label{eq:reverse ES} In the coupling above, what is the procedure to obtain the configuration $\omega$ from a configuration $\sigma$?
\end{exercise}
\eexo

The same coloring procedure as above, except for the clusters $\calC$ intersecting the boundary $\partial G$ for which $\sigma_\calC$ is automatically set to be equal to ${\rm b}$, provides us with another coupling.

\begin{proposition}[Coupling for monochromatic boundary conditions]\label{prop:coupling wired}
Fix an integer $q\ge2$, $p\in(0,1)$ and $G$ finite. If $\omega$ is distributed according to $\phi_{G,p,q}^1$, then $\sigma$ constructed above is distributed according to the $q$-state Potts measure $\mu_{G,\beta,q}^{\rm b}$, where $\beta=-\tfrac{q-1}{q}\ln (1-p)$.
\end{proposition}

\bexo
\begin{exercise}
Write carefully the proof of Proposition~\ref{prop:coupling wired}.
\end{exercise}
\eexo

This coupling provides us with a dictionary between the properties of the random-cluster model and the Potts model. In order to illustrate this fact, let us mention two consequences.

\begin{corollary}\label{cor:aa}Fix $d,q\ge2$. Let $G$ be a finite subgraph of $\bbZ^d$. Let $\beta>0$ and $p\in[0,1]$ be connected by \eqref{eq:beta p}. For any $x\in V$,
\begin{align}\mu_{G,\beta,q}^{\rm f}[\sigma_x\cdot \sigma_y]&=\phi_{G,p,q}^0[x\longleftrightarrow y],\label{eq:coupling 2 point}\\
\mu_{G,\beta,q}^{\rm b}[\sigma_x\cdot{\rm b}]&=\phi_{G,p,q}^1[x\longleftrightarrow\partial G].\label{eq:coupling magnetization}\end{align}
\end{corollary}
\begin{proof}We do the proof for $\mu_{G,\beta,q}^{\rm f}[\sigma_x\cdot \sigma_y]$. Consider the coupling ${\bf P}$ between $\omega$ and $\sigma$ and denote its expectation by ${\bf E}$. If $x\leftrightarrow y$ denotes the event that $x$ and $y$ are connected in $\omega$, we find that 
\begin{align*}\mu_{G,\beta,q}^{\rm f}[\sigma_x\cdot\sigma_y]&={\bf E}[\sigma_x\cdot\sigma_y\mathbbm 1_{x\longleftrightarrow y}]+{\bf E}[\sigma_x\cdot\sigma_y\mathbbm 1_{x{\not\longleftrightarrow}y}]=\phi_{G,p,q}^0[x\longleftrightarrow y],\end{align*}
where we used that $\sigma_x=\sigma_y$ if $x$ is connected to $y$, and $\sigma_x$ and $\sigma_y$ are independent otherwise. The same reasoning holds for $\mu_{G,\beta,q}^{\rm b}[\sigma_x\cdot{\rm b}]$. \end{proof}


As a side remark, note that we just proved that $\mu_{G,\beta,q}^{\rm f}[\sigma_x\cdot\sigma_y]$ and $\mu_{G,\beta,q}^{\rm b}[\sigma_x\cdot{\rm b}]$ are non-negative. In the case of the Ising model, one can extend the previous relation to the following: for any $A\subset V$, 
\begin{equation}\label{eq:gri}\mu_{G,\beta}^{\rm f}[\sigma_A]=\phi_{G,p,2}^0[\mathcal F_A],\end{equation}
where $\sigma_A:=\prod_{x\in A}\sigma_x$ and $\mathcal F_A$ is the event that every cluster of $\omega$ intersects $A$ an even number of times. In particular, we deduce the first Griffiths inequality
 $\mu_{G,\beta}^{\rm f}[\sigma_A]\ge 0$.

\bexo

\begin{exercise}\label{exo:Griffiths 1}
Prove \eqref{eq:gri}.\end{exercise}
\eexo

\subsection{The percolation phase transition for the random-cluster model}

\subsubsection{Positive association and monotonicity}

Up to now, we considered as granted the fact that spin-spin correlations of the Potts model were increasing in $\beta$, but this is not clear at all. One of the advantages of percolation configurations compared to spin configurations is that $\{0,1\}^{E}$ is naturally ordered (simply say that $\omega\le \omega'$ if $\omega_e\le \omega'_e$ for any $e\in E$) so that we may define the notion  of {\em increasing} event:
\begin{equation}\text{ $\calA$ is increasing}\qquad\Longleftrightarrow\qquad \text{[($\omega\in \calA$ and $\omega\le\omega'$)}\Longrightarrow \omega'\in \calA].\end{equation} The random-cluster model with cluster-weight $q\ge1$ enjoys some monotonicity properties regarding increasing events, and this special feature makes it more convenient  to work with than Potts models.\medbreak
\begin{mdframed}[backgroundcolor=green!00]
From now on, we always assume that the cluster-weight is larger or equal to 1, so that we will have the proper monotonicity properties (listed below).
\eexo
\medbreak
 We say that $\mu$ is {\em stochastically dominated} by $\nu$ if for any increasing event $\calA$,
$\mu[\calA]\le \nu[\calA]$. 
Note that there is a natural way of checking that $\mu$ is stochastically dominated by $\nu$. Assume that there exists a probability measure ${\bf P}$ on pairs $(\omega,\tilde\omega)\in\{0,1\}^E\times\{0,1\}^E$ such that 
\begin{itemize}[noitemsep,nolistsep]
\item the law of $\omega$ is $\mu$,
\item the law of $\tilde\omega$ is $\nu$,
\item ${\bf P}[\omega\le \tilde\omega]=1$.
\end{itemize}
Then, $\mu$ is automatically stochastically dominated by $\nu$, since for any increasing event $\calA$,
$$\mu[\calA]=\mathbf P[\omega\in \calA]=\mathbf P[\omega\in \calA\text{ and }\omega\le\tilde\omega]\le \mathbf P[\tilde\omega\in \calA]=\nu[\calA].$$
When $\mu$ and $\nu$ are equal to two Bernoulli percolation measures $\bbP_p$ and $\bbP_{p'}$ with $p\le p'$, it is quite simple to construct ${\bf P}$.
 Indeed, consider a collections of independent uniform $[0,1]$ random variables ${\bf U}_e$ indexed by edges in $E$. Then, define $\omega$ and $\tilde\omega$ as follows
$$\omega_e=\begin{cases} 1 &\text{if }\mathbf U_{e}\ge 1-p,\\
0 &\text{otherwise}\end{cases}\qquad\text{and}\qquad\tilde\omega_e=\begin{cases} 1 &\text{if }\mathbf U_{e}\ge 1-p',\\
0 &\text{otherwise}.\end{cases}$$
By construction, $\omega$ and $\tilde\omega$ are respectively sampled according to $\bbP_p$ and $\bbP_{p'}$ (the states of different edges are independent, and the probability that an edge is open is respectively $p$ and $p'$) and $\omega\le\tilde\omega$.

In general, it is more complicated to construct ${\bf P}$.  The next lemma provides us with a convenient criteria to prove the existence of such a coupling. We say that a measure $\mu$ on $\{0,1\}^E$ is {\em strictly positive} if $\mu(\omega)>0$ for any $\omega\in\{0,1\}^E$. 
\begin{lemma}\label{lem:dynamic}
Consider two strictly positive measures $\mu$ and $\nu$ on $\{0,1\}^E$ such that for any $e\in E$ and $\psi,\psi'\in\{0,1\}^{E\setminus\{e\}}$ satisfying $\psi\le \psi'$, one has
\begin{equation}\label{eq:condition}
\mu[\omega_e=1|\omega_{|E\setminus\{e\}}=\psi]\le \nu[\omega_e=1|\omega_{|E\setminus\{e\}}=\psi'].
\end{equation}
Then, there exists a measure ${\bf P}$ on pairs $(\omega,\tilde\omega)$ with ${\bf P}[\omega\le \tilde\omega]=1$ such that $\omega$ and $\tilde\omega$ have laws $\mu$ and $\nu$. In particular, $\mu$ is stochastically dominated by $\nu$.\end{lemma}
\begin{proof}
In order to construct {\bf P}, we use a continuous time Markov chain $(\omega^t,\tilde\omega^t)$ constructed as follows. Associate independently to each edge $e\in E$ an exponential clock and  a collection of independent uniform $[0,1]$ random variables $\mathbf U_{e,k}$.

At each time an exponential clock rings -- say we are at time $t$ and it is the $k$-th time the edge $e$ rings -- set (below $\omega^{t^-}$ and $\tilde\omega^{t^-}$ denote the configurations just before time $t$)
\begin{align*}\omega^t_e&=\begin{cases} 1 &\text{ if }\mathbf U_{e,k}\ge \mu[\omega_e=0|\omega_{|E\setminus \{e\}}=\omega^{t^-}_{|E\setminus\{e\}}]\\
0&\text{ otherwise,}
\end{cases}\\
\tilde\omega^t_e&=\begin{cases} 1 &\text{ if }\mathbf U_{e,k}\ge \nu[\omega_e=0|\omega_{|E\setminus \{e\}}=\tilde \omega^{t^-}_{|E\setminus\{e\}}]\\
0&\text{ otherwise.}\end{cases}\end{align*}
By definition, $(\omega^t)$ is an irreducible (because of strict positivity, one can go from any state to the state with all edges open, and back to any other configuration) continuous time Markov chain. The jump probabilities are such\footnote{The probability that $\omega^t_e=1$ is exactly the probability that $\omega_e=1$ knowing the state of all the other edges.} that $\mu$ is its (unique) stationary measure. As a consequence, the law of $\omega^t$ converges to $\mu$. 
Similarly, the law of $\tilde\omega^t$ converges to $\nu$. 

Finally, if the starting configurations $\omega^0$ and $\tilde\omega^0$ are respectively the configurations with all edges closed, and all edges open, then $\omega^0\le \tilde\omega^0$ and the condition \eqref{eq:condition} implies that for all $t\ge0$, $\omega^t\le \tilde\omega^t$. Letting $t$ tend to infinity provides us with a coupling of $\mu$ and $\nu$ with ${\bf P}[\omega\le\tilde\omega]=1$.
\end{proof}

\begin{theorem}[Positive association]Fix $q\ge1$, $p\in[0,1]$, $\xi$ some boundary conditions and $G$ finite. Then
\begin{itemize}
\item {\em (Comparison between boundary conditions)} For any increasing event $\calA$ and $\xi'\ge\xi$ (meaning that the partition $\xi'$ is coarser than the partition $\xi$),
\begin{equation}\label{eq:comparison}\tag{CBC}\phi_{G,p,q}^{\xi'}[\calA]~\ge~\phi_{G,p,q}^{\xi}[\calA] .
\end{equation}
\item {\em (Monotonicity)} For any increasing event $A$ and any $p'\ge p$,
\begin{equation}\label{eq:monotonicity}\tag{MON}\phi_{G,p',q}^{\xi}[\calA]~\ge~\phi_{G,p,q}^{\xi}[\calA] .
\end{equation}
\item {\em (Fortuin-Kasteleyn-Ginibre inequality)} For any increasing events $\calA$ and $\calB$,
\begin{equation}\label{eq:FKG}\tag{FKG}
\phi_{G,p,q}^{\xi}[\calA\cap \calB]~\ge~\phi_{G,p,q}^{\xi}[\calA]\phi_{G,p,q}^{\xi}[\calB].
\end{equation}
\end{itemize}
\end{theorem}
The assumption $q\ge1$ is not simply technical: the different properties above fail when $q<1$. For instance, a short computation on a small graph shows that the random-cluster model with $q<1$ does not satisfy the FKG inequality. Also recall that as $p\rightarrow0$ and $q/p\rightarrow0$, one may obtain the Uniform Spanning Tree, which is known to be edge negatively correlated. It is natural to expect some form of negative correlation for random-cluster models with $q<1$, but no general result is known as for today. 

One important feature of the comparison between boundary conditions is that the free and wired boundary conditions are extremal in the following sense: for any increasing event $\calA$ and any boundary conditions $\xi$,
\begin{equation}\label{extremality}
\phi^0_{G,p,q}[\calA]~\le~ \phi^{\xi}_{G,p,q}[\calA]~\le~\phi^1_{G,p,q}[\calA].
\end{equation}
For more applications of (CBC), we refer to Exercises~\ref{exo:0} and \ref{exo:unique Gibbs}.


\begin{proof}
We wish to apply the previous lemma. Consider an edge $e=xy\in E$ and $\psi\le \psi'$ two configurations in $\{0,1\}^{E\setminus\{e\}}$. Recall that \eqref{eq:arg} is stating that
$$\phi_{G,p,q}^\xi[\omega_e=1|\omega_{|E\setminus\{e\}}=\psi]=\begin{cases} \ \ \ \ \ \ \ p&\text{ if $x$ and $y$ are connected in $\psi^\xi$,}\\
\displaystyle\frac{p}{p+q(1-p)}&\text{ otherwise.}\end{cases}$$
Observe that if $x$ and $y$ are connected in $\psi^\xi$, they also are in $(\psi')^\xi$ (and a fortiori in $(\psi')^{\xi'}$), and that 
$p\ge\frac{p}{p+q(1-p)}$ (since $q\ge1$). 
With the previous observations, \eqref{eq:monotonicity} and \eqref{eq:comparison} follow readily from the previous lemma. 

For \eqref{eq:FKG}, we need to be slightly more careful. Without loss of generality, we may assume that $\calB$ has positive probability. Define the measures $\mu=\phi_{G,p,q}^{\xi}$ and $\nu=\mu[\cdot|\calB]$. One may easily check that \eqref{eq:condition} is satisfied. The measure $\nu$ is not strictly positive, but this played a role only in proving that the Markov chains had unique invariant measures. The fact that $\tilde\omega^0$ is in $\calB$ (since $\calB$ is non empty and increasing, and all the edges are open in $\tilde\omega^0$) implies that the stationary measure of $(\tilde\omega^t)$ is $\nu$, so that the conclusions of the previous lemma are still valid and $\nu$ stochastically dominates $\mu$. As a consequence, 
$$\phi_{G,p,q}^{\xi}[\calA]=\mu[\calA]\le \nu[\calA]=\frac{\phi_{G,p,q}^{\xi}[\calA\cap \calB]}{\phi_{G,p,q}^{\xi}[\calB]},$$
which proves \eqref{eq:FKG}.\end{proof}

The coupling between random-cluster and Potts models implies the  following nice consequence of monotonicity.\begin{corollary} Fix $G$ finite and $q\ge2$ an integer. The functions  $\beta\mapsto\mu_{G,\beta,q}^{\rm f}[\sigma_x\cdot \sigma_y]$ and $
\beta\mapsto\mu_{G,\beta,q}^{\rm b}[\sigma_x\cdot{\rm b}]$ are non-decreasing. \end{corollary}
\bexo

\begin{exercise}[Second Griffiths inequality]\label{exo:Griffiths 2}
Using the coupling with the random-cluster model, prove the {\em second Griffiths inequality} for the Ising model: for any set of vertices $A$ and $B$,
\begin{equation}\mu_{G,\beta}^{\rm f}[\sigma_A\sigma_B]\ge \mu_{G,\beta}^{\rm f}[\sigma_A]\mu_{G,\beta}^{\rm f}[\sigma_B].\tag{2nd Griffiths}\end{equation}
\end{exercise}

\begin{exercise}[Comparison with boundary conditions 1]\label{exo:0}
Fix $p\in[0,1]$, $q\ge1$, a finite graph $G=(V,E)$ and $\xi$ some boundary conditions. Let $F$ be a subset of $E$ and $H$ be the graph with edge-set $F$ and vertex-set given by the endpoints of edges in $F$. Then, for any increasing events $\calA$ and $\calB$ depending only on edges in $F$ and $E\setminus F$ respectively, show that
$$\phi^0_{H,p,q}[\calA]\le\phi^\xi_{G,p,q}[\calA|\calB]\le \phi^1_{H,p,q}[\calA].$$

\end{exercise}
\begin{exercise}[Comparison with boundary conditions 2]\label{exo:unique Gibbs}
Consider a graph $G=(V,E)$ and $F\subset E$. Let $G'=(W,F)$ be the graph with edge-set $F$ and vertex-set given by the endpoints of the edges in $F$. Let $\calA$ be an increasing event depending on edges in $F$ only. Let $\xi$ be some boundary conditions on $\partial G$.
\medbreak\noindent
1. Define the set $\mathsf S=\mathsf S(\omega)$ of vertices in $V$ not connected in $\omega$ to a vertex in $\partial G$. Show that for any $S\subset V$, the event $\{\mathsf S=S\}$ is measurable in terms of edges with at least one endpoint outside $\mathsf S$.
\medbreak\noindent
2. Fix $S\subset V$. Consider the graph $H$ with vertex-set $S$ and edge-set composed of edges in $E$ with both endpoints in $S$. Use the previous observation to prove that 
\begin{equation*}\phi_{G,p,q}^\xi[\,\calA,\mathsf S=S\,|\,\partial G'\not\longleftrightarrow\partial G]\le\phi_{H,p,q}^0[\calA]\phi_{G,p,q}^\xi[\mathsf S=S\,|\,\partial G'\not\longleftrightarrow\partial G].\end{equation*}
\noindent
3. Prove that
$\phi_{G,p,q}^\xi[\,\calA\,|\,\partial G'\not\longleftrightarrow\partial G]\le\phi_{G,p,q}^0[\calA].$
\medbreak\noindent
4. We now restrict ourself to two dimensions. A circuit is a path starting and ending at the same vertex. Let $\calB$ be the event that there exists an open circuit in $E\setminus F$ disconnecting $W$ from $\partial G$. Prove that 
$\phi_{G,p,q}^\xi[\,\calA\,|\,\calB]\ge\phi_{G,p,q}^1[\calA].$
\end{exercise}

\begin{exercise}[Holley and FKG lattice conditions]
1. Show that for strictly positive measures, \eqref{eq:condition} is equivalent to the Holley criterion: for any $\omega$ and $\omega'$,
\begin{equation}
\nu[\omega\vee \omega']\mu[\omega\wedge\omega']\ge \nu[\omega]\mu[\omega'],\tag{Holley}
\end{equation}
where $\vee$ and $\wedge$ are the min and max of two configurations.\medbreak\noindent
2. Show that for a strictly positive measure $\mu$, $\mathrm{(FKG)}$ holds if the FKG lattice condition holds: for any $\omega$ and any edges $e$ and $f$,
\begin{equation}
\mu[\omega^{ef}]\mu[\omega_{ef}]\ge \mu[\omega^e_f]\mu[\omega^f_e],\tag{FKG lattice condition}
\end{equation}
where $\omega^{ef}$, $\omega_{ef}$, $\omega^e_f$ and $\omega^f_e$ denote the configurations $\omega'$ coinciding with $\omega$ except at $e$ and $f$, where $(\omega'_e,\omega'_f)$ are equal respectively to $(1,1)$, $(0,0)$, $(1,0)$ and $(0,1)$.
\end{exercise}

\begin{exercise}
Is there a monotonicity in $q$ at fixed $p$?
\end{exercise}
\eexo

\subsubsection{Phase transition in the random-cluster and Potts models}

When discussing phase transitions, we implicitly considered infinite-volume Potts measures to define $\beta_c$. Their definition is not a priori clear since the Hamiltonian would then be an infinite sum of terms equal to 1 or $-1/(q-1)$. One can always consider sub-sequential limits of measures $\mu_{G,\beta,q}^{\rm f}$, but one can in fact do much better using the random-cluster model: monotonicity properties of the previous section enable us to prove convergence of certain sequences of measure.

Below and in the rest of this document, set for every $n\ge0$,
$\Lambda_n:=[-n,n]^d\cap\bbZ^d.$ Also, $E_n$ will denote the set of edges between two vertices of $\Lambda_n$.
\begin{proposition}\label{wired infinite}
 Fix $q\ge1$. There exist two (possibly equal) measures $\phi_{p,q}^0$ and
  $\phi_{p,q}^1$ on $\{0,1\}^{\bbE}$, called the infinite-volume random-cluster 
  measures with free and wired boundary conditions respectively, such
  that for any event $\calA$ depending on a finite number of edges,
  \begin{align*}
    &\lim_{n\rightarrow\infty}\phi_{\Lambda_n,p,q}^1[\calA]=
    \phi_{p,q}^1[\calA]\quad\text{~and~}\quad\lim_{n\rightarrow \infty}\phi_{\Lambda_n,p,q}^0[\calA]=
    \phi_{p,q}^0[\calA].
  \end{align*}
\end{proposition} 

One warning: while boundary conditions cannot be defined as a partition of the boundary in infinite volume, one still needs to keep track of the dependency on boundary conditions for finite-volume measures when constructing the measure. Therefore, the measures $\phi_{p,q}^1$ and $\phi_{p,q}^0$ have no reason to be the same and we will see examples of values of $p$ and $q$  for which they are in fact different. In addition to this, one may imagine other infinite-volume measures obtained via limits of measures on finite graphs with arbitrary (and possibly random) boundary conditions. 

\begin{proof}
We deal with the case of free boundary conditions. Wired boundary conditions are treated similarly. Fix an increasing event $\calA$ depending on edges in $\Lambda_N$ only. We find that for any $n\ge N$,
$$\phi_{\Lambda_{n+1},p,q}^0[\calA]\stackrel{\eqref{eq:domain Markov}}=\phi_{\Lambda_{n+1},p,q}^0[\phi_{\Lambda_{n},p,q}^\xi[\calA]]\stackrel{\eqref{eq:comparison}}\ge \phi_{\Lambda_n,p,q}^0[\calA],$$
where $\xi$ is the random boundary conditions induced by the configuration $\omega_{|E_{n+1}\setminus E_n}$.
We deduce that $(\phi_{\Lambda_n,p,q}^0[\calA])_{n\ge 0}$ is increasing, and therefore converges to a certain value $P[\calA]$ as $n$ tends to infinity. 

Since the probability of an event $\calB$ depending on finitely many edges can be written by inclusion-exclusion (see Exercise~\ref{exo:increasing}) as a combination of the probability of increasing events, taking the same combination defines a natural value $P(\calB)$ for which
$\phi_{\Lambda_n,p,q}^0[\calB]$ converges to $P(\calB)$.

The fact that $(\phi_{\Lambda_n,p,q}^0)_{n\ge0}$ are probability measures implies that the function $P$ (which is a priori defined on the set of events depending on finitely many edges) can be extended into a probability measure on $\mathcal F_{\mathbb E}$. We denote this measure by $\phi_{p,q}^0$.
\end{proof}
\bexo
\begin{exercise}\label{exo:increasing}
For $\psi\in\{0,1\}^E$, write $\{\omega\in\{0,1\}^\bbE:\omega_e=\psi_e,\forall e\in E\}$ as $A\setminus B$ with $B\subset A$ two increasing events. Deduce that any event depending on finitely many edges can be written by inclusion-exclusion using increasing events.\end{exercise}
\eexo
The properties of finite-volume measures (FKG inequality, monotonicity, ordering between boundary conditions) extend to infinite volume in a straightforward fashion. 
In particular, one may define a {\em critical parameter} $p_c\in[0,1]$ such that 
\begin{align*}p_c=p_c(q,d)&:=\inf\{p>0:\phi_{p,q}^1[0\leftrightarrow\infty]>0\}=\sup\{p>0:\phi_{p,q}^1[0\leftrightarrow\infty]=0\}.
\end{align*}
Let us conclude by explaining what this implies for Potts models. One can extend the coupling between random-cluster models and Potts models in order to construct $q+1$ measures $\mu^{\rm f}_{\beta,q}$ and $\mu^{\rm b}_{\beta,q}$ with ${\rm b}\in\bbT_q$ on $\bbZ^d$ by doing the same couplings as in finite-volume, except that clusters intersecting the boundary are replaced by infinite clusters. 

\begin{corollary}
The measures $\mu^{\rm f}_{\beta,q}$ and $\mu^{\rm b}_{\beta,q}$ with ${\rm b}\in\bbT_q$ are the limits of the measures $\mu^{\rm f}_{\Lambda_n,\beta,q}$ and $\mu^{\rm b}_{\Lambda_n,\beta,q}$. Furthermore, if $\beta$ and $p$ satisfy \eqref{eq:beta p}, then
$$m^*(\beta,q):=\mu^{\rm b}_{\beta,q}[\sigma_0\cdot{\rm b}]=\phi_{p,q}^1[0\longleftrightarrow\infty].$$
\end{corollary}

The proof of the corollary is immediate from the convergence of the random-cluster measures and the coupling. Note that this enables us to define rigorously\begin{align*}\beta_c=\beta_c(q,d)&:=\inf\{\beta>0:m^*(\beta,q)>0\}=\sup\{\beta>0:m^*(\beta,q)=0\},\end{align*}
which is related to $p_c(q,d)$ by the formula 
\begin{equation}\beta_c(q,d):=-\tfrac{q-1}{q}\log\big[1-p_c(q,d)\big].\end{equation}

\bexo
\begin{exercise}
Is there some ordering between the measures $\phi^1_{p,q}$ in $q\ge1$ at fixed $p$? Deduce from the study of Bernoulli percolation that $p_c(q,d)>0$.
\end{exercise}

\begin{exercise}\label{exo:12}
1. Prove that for $p<p_c$, $\phi^1_{p,q}=\phi^0_{p,q}$.
\medbreak\noindent
2. (To do after reading Section~\ref{sec:2.3}) Prove that on $\bbZ^2$, $\phi^1_{p,q}=\phi^0_{p,q}$ for $p>p_c$.
\end{exercise}

\begin{exercise} \label{exo:12}A probability measure $\phi$ on $\{0,1\}^\bbE$ is called an {\em infinite-volume random-cluster} measure with parameters $p$ and $q$ if for every finite graph $G=(V,E)$, 
\begin{equation*}\phi[\omega_{|E}=\eta\,|\calF_E]=\phi^\xi_{G,p,q}[\eta]\quad,\quad\forall \eta\in\{0,1\}^{E},\end{equation*}
where $\xi$ are the boundary conditions induced by the configuration outside $G$ and $\calF_E$ is the $\sigma$-algebra induced by $(\omega_e:e\notin E)$.
Prove that $\phi^0_{p,q}\le \phi\le \phi^1_{p,q}$ for any infinite-volume measure with parameters $p$ and $q\ge 1$. Deduce that there exists a unique infinite-volume measure if and only if $\phi^0_{p,q}=\phi^1_{p,q}$.
\end{exercise}

\begin{exercise}
Using \eqref{eq:finite energy}, prove that $p_c(q,d)>0$.
\end{exercise}
\eexo

\subsubsection{Long-range ordering and spontaneous magnetization}\label{sec:uniqueness}

Let us now focus on the following question:
is (LRO$_\beta$) equivalent to (MAG$_\beta$)?
In terms of random-cluster model, this gets rephrased as follows: is $\phi^1_{p,q}[0\leftrightarrow\infty]=0$ equivalent to $\phi^0_{p,q}[0\leftrightarrow x]$ tends to 0 as $\|x\|$ tends to infinity?
Two things could prevent this from happening. First, $\phi^1_{p,q}$ and $\phi^0_{p,q}$ could be different. Second, it may be that, when an infinite cluster exists, then automatically infinitely many of them do, so that the probability that two vertices are connected tends to zero.
\bigbreak
Let us first turn to the second problem and prove that the infinite cluster, when it exists, is unique.

\begin{theorem}\label{thm:uniqueness}
Fix $p\in[0,1]$ and $q\ge1$. For $\#$ equal to 0 or 1, either $\phi^\#_{p,q}[0\leftrightarrow \infty]=0$ or
 $\phi^\#_{p,q}[\exists\text{ a unique infinite cluster}]=1$.
\end{theorem} 
This result was first proved in \cite{AizKesNew87} for Bernoulli percolation. It was later obtained via different types of arguments. The beautiful argument presented here is due to {Burton and Keane} \cite{BurKea89}. 

We begin by studying ergodic properties of $\phi^1_{p,q}$ and $\phi^0_{p,q}$. Let $\tau_x$ be a translation of the lattice by $x\in\bbZ^d$. This translation induces a shift on the space of configurations $\{0,1\}^\bbE$.
Define
$\tau_x\calA:=\{\omega\in\{0,1\}^{\mathbb E}:\tau_x^{-1}\omega\in \calA\}$. An event $\calA$ is {\em invariant under translations} if for any $x\in \mathbb Z^d$, $\tau_x\calA=\calA$. A measure $\mu$ is {\em invariant under translations} if $\mu[\tau_x\calA]=\mu[\calA]$ for any event $\calA$ and any $x\in\bbZ^d$. The measure is said to be {\em ergodic} if any event invariant under translation has probability 0 or 1.
\begin{lemma}\label{ergodicity}
The measures $\phi^1_{p,q}$ and $\phi^0_{p,q}$ are invariant under translations and ergodic.
\end{lemma}

\begin{proof}
Let us treat the case of $\phi^1_{p,q}$, the case of $\phi^0_{p,q}$ is left to the reader (Exercise~\ref{exo:ergodicity}). Let $\calA$ be an increasing event depending on finitely many edges, and $x\in\bbZ^d$. Choose $k$ such that $x\in\Lambda_k$. Since
$\Lambda_{n-k}\subset\tau_x\Lambda_n\subset \Lambda_{n+k}$, the comparison between boundary conditions \eqref{eq:comparison} gives
$$\phi^1_{\Lambda_{n+k},p,q}[\tau_x\calA]\le \phi^1_{\tau_x\Lambda_n,p,q}[\tau_x\calA]\le \phi^1_{\Lambda_{n-k},p,q}[\tau_x\calA].$$ 
We deduce that
$$\phi^1_{p,q}[\calA]=\lim_{n\rightarrow\infty}\phi^1_{\Lambda_n,p,q}[\calA]=\lim_{n\rightarrow\infty}\phi^1_{\tau_x\Lambda_n,p,q}[\tau_x\calA]=\phi^1_{p,q}[\tau_x\calA].$$
Since the increasing events depending on finitely many edges span the $\sigma$-algebra of measurable events, we obtain that $\phi^1_{p,q}$ is invariant under translations.

Any event can be approximated by events depending on finitely many edges, hence the ergodicity follows from mixing (see Exercise~\ref{exo:mixing w}), i.e.~from the property 
that for any events $\calA$ and $\calB$ depending on finitely many edges,
\begin{equation}\lim_{\|x\|\rightarrow \infty}\phi^1_{p,q}[\calA\cap\tau_x \calB]=\phi^1_{p,q}[\calA]\phi^1_{p,q}[\calB].\label{eq:mixing}\tag{Mixing}\end{equation}
Observe that by inclusion-exclusion, it is sufficient to prove the equivalent result for $\calA$ and $\calB$ increasing and depending on finitely many edges. 
Let us give ourselves these two increasing events $\calA$ and $\calB$ depending on edges in $\Lambda_k$ only, and $x\in \bbZ^d$. The FKG inequality and the invariance under translations of $\phi^1_{p,q}$ imply that
$$\phi^1_{p,q}[\calA\cap\tau_{x} \calB]\ge \phi^1_{p,q}[\calA]\phi^1_{p,q}[\tau_{x}\calB]=\phi^1_{p,q}[\calA]\phi^1_{p,q}[\calB].$$
In the other direction, for any $n\ge2k$, if $x$ is far enough from the origin, then $\Lambda_n$ and $\tau_x\Lambda_n$ do not intersect. Thus, the comparison between boundary conditions (more precisely Exercise~\ref{exo:0} for $H=\Lambda_N$ with $N\ge n+k$, and then a limit as $N$ tends to infinity) gives
\begin{align*}\phi^1_{p,q}[\calA\cap\tau_{x} \calB]\le \phi^1_{\Lambda_n,p,q}[\calA]\phi^1_{\tau_{x}\Lambda_n,p,q}[\tau_x\calB]=\phi^1_{\Lambda_n,p,q}[\calA]\phi^1_{\Lambda_n,p,q}[\calB].\end{align*}
The result follows by taking $x$ to infinity. 
\end{proof}

\bexo
\begin{exercise}\label{exo:mixing w}
Prove that the mixing property \eqref{eq:mixing} implies ergodicity. {\em Hint.} Consider an event $\calA$ which is invariant by translation and approximate it by an event $\calB$ depending on finitely many edges. Then, use that the probability that $\calB\cap\tau_x \calB$ tends to the square of the probability of $\calB$ together with the fact that $\calA=\calA\cap\tau_x \calA$. 
\end{exercise}
\begin{exercise}\label{exo:ergodicity}
Prove that $\phi_{p,q}^0$ is invariant under translations and ergodic.
\end{exercise}
\eexo

\begin{figure}
\begin{center}
\includegraphics[width=0.80\textwidth]{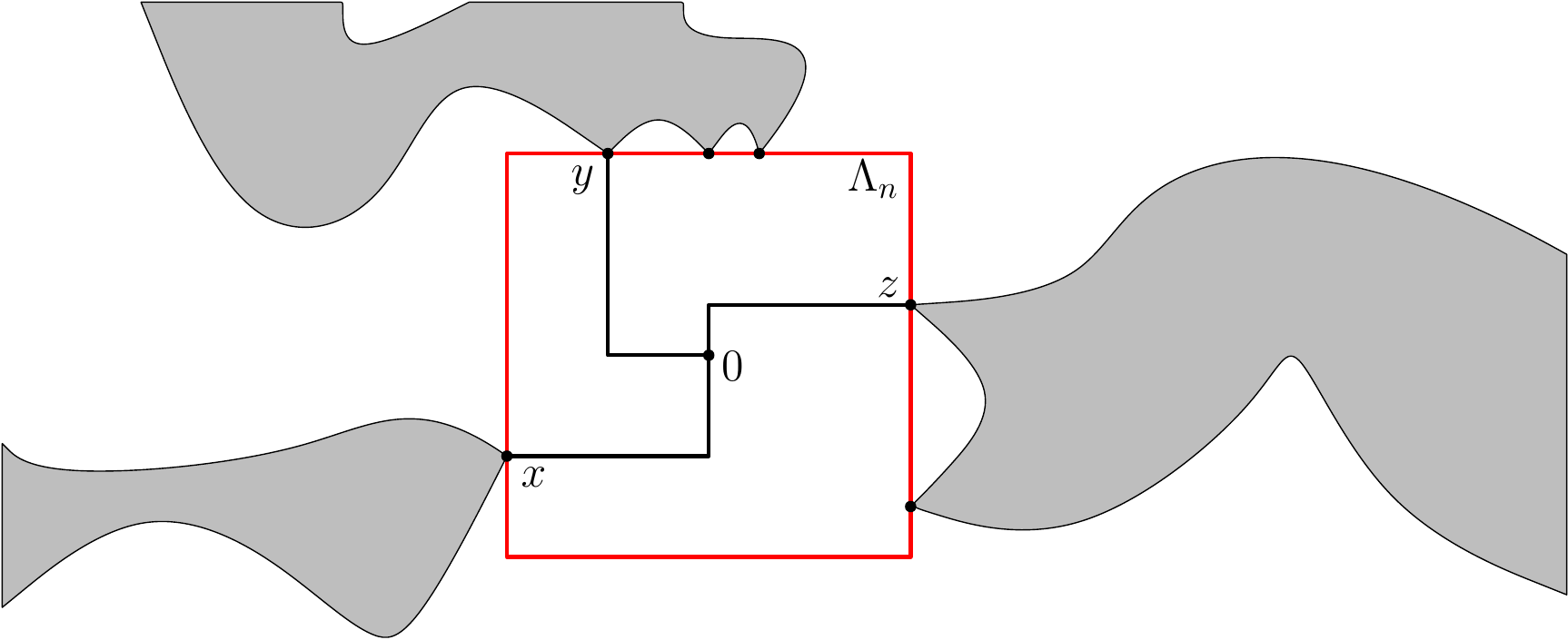}
\caption{Construction of a trifurcation at the origin starting from three disjoint infinite clusters (in gray) intersecting $\Lambda_n$. The three paths inside $\Lambda_n$ are vertex-disjoint, except at the origin.}\label{fig:trifurcation}\end{center}
\end{figure}

\begin{proof}[Theorem~\ref{thm:uniqueness}]We present the proof in the case of wired boundary conditions and for $p\in(0,1)$ (the result is obvious for $p$ equal to 0 or 1). Let $\calE_{\le 1}$, $\calE_{<\infty}$ and $\calE_\infty$ be the events that there is no more than one, finitely many and infinitely many infinite clusters respectively. Since having no infinite cluster is an event which is invariant under translations, it has probability 0 or 1 by ergodicity, and it is therefore sufficient to prove that $\phi^1_{p,q}[\calE_{\le 1}]=1$.
\medbreak
Let us start by showing that $\phi^1_{p,q}[\calE_{<\infty}\setminus\calE_{\le 1}]=0$. By ergodicity, $\calE_{<\infty}$ and $\calE_{\le 1}$ both have probability equal to 0 or 1. Since $\calE_{\le 1}\subset \calE_{<\infty}$, we only need to prove that $\phi^1_{p,q}[\calE_{<\infty}]>0$ implies $\phi^1_{p,q}[\calE_{\le 1}]>0$.
Let $\calF$ be the event that all (there may be none) the infinite clusters intersect $\Lambda_n$. 
Since $\calF$ is independent of $E_n$, \eqref{eq:domain Markov} together with \eqref{eq:finite energy} imply that
$$\phi^1_{p,q}[\calF\cap\{\omega_e=1,\forall e\in E_n\}]\ge \phi^1_{p,q}[\calF]\,c_{\rm FE}^{\ |E_n|}.$$
Now, assume that $\phi^1_{p,q}[\calE_{<\infty}]>0$. Since any configuration in the event on the left contains zero or one infinite cluster (all the vertices in $\Lambda_n$ are connected), choosing $n$ large enough that $\phi^1_{p,q}[\calF]\!\ge\!\tfrac12\phi^1_{p,q}[\calE_{<\infty}]>0$ implies that
 $\phi^1_{p,q}[\calE_{\le 1}]>0$.
\bigbreak We now exclude the possibility of an infinite number of infinite clusters. Consider $n>0$ large enough that 
\begin{equation}\label{eq:trif}\phi^1_{p,q}[K\text{ infinite clusters intersect the box $\Lambda_n$}]\ge\tfrac12\phi^1_{p,q}[\calE_\infty],\end{equation}
where $K=K(d)$ is large enough that three vertices $x,y,z$ of $\partial\Lambda_n$ at distance at least three of each others that are connected to infinity in $\omega_{|\bbE\setminus E_n}$. Using these three vertices, one may
modify\footnote{Note that one may wish to pick $K=3$ in \eqref{eq:trif} instead of a (a priori) larger $K$, but that this choice would make the construction of the trifurcations described below more difficult due to the fact that the three clusters may arrive very close to each others on the corner of $\Lambda_n$, and therefore prevent us from ``rewiring them'' to construct a trifurcation at the origin.
} the configuration in $E_n$ as follows: 
\begin{enumerate}[noitemsep,nolistsep]\item Choose three paths in $\Lambda_n$ intersecting each others only at the origin, and intersecting $\partial\Lambda_n$ only at one point, which is respectively $x$, $y$ and $z$. 
\item Open all the edges on these paths, and close all the other edges in $E_n$. 
\end{enumerate}We deduce from this construction that  
\begin{equation}\label{eq:pp}\phi^1_{p,q}[\calT_0]\ge c_{\rm FE}^{\ |E_n|}\cdot\tfrac12 \phi^1_{p,q}[\calE_\infty],\end{equation} where $\calT_0$ is the following event: $\bbZ^d\setminus\{0\}$ contains three distinct infinite clusters which are connected to 0 by an open edge. 
A vertex $x\in \bbZ^d$ is called a {\em trifurcation} if $\tau_x\calT_0=:\calT_x$ occurs. 

Fix $n\!\ge\!1$ and denote the number of trifurcations in $\Lambda_n$ by $\mathsf T$. By invariance under translation, $\phi^1_{p,q}[\calT_x]=\phi^1_{p,q}[\calT_0]$ and therefore
\begin{equation}\phi^1_{p,q}[\mathsf T]=\phi^1_{p,q}[\calT_0]\times |\Lambda_n|.\label{eq:aair}\end{equation}
 Let us now bound deterministically $\mathsf T$. In order to do this, first perform the following two ``peelings'' of the set $F_0:=\{e_1,\dots,e_r\}$ of edges in $E_n$ that are open in $\omega$.
 \begin{itemize}
 \item 
 For each $1\le i\le r$,  if $e_i$ is on a cycle formed by edges in $F_{i-1}$, set $F_i=F_{i-1}\setminus\{e_i\}$, otherwise, set $F_i=F_{i-1}$. At the end, the set $\tilde F_0:=F_r=\{f_1,\dots,f_s\}$ is a forest.
 \item  For each $1\le j\le s$, if $\tilde F_{j-1}\setminus \{f_j\}$ contains a cluster not intersecting $\partial\Lambda_n$, then set $\tilde F_j$ to be $\tilde F_{j-1}\setminus\{f_j\}$ and the cluster in question. Otherwise, set $\tilde F_j=\tilde F_{j-1}$. At the end, $\tilde F_s$ is a forest whose leafs belong to $\partial\Lambda_n$.
 \end{itemize} 
 Since the trifurcations are vertices of degree at least three in this forest, we deduce that $\mathsf T$ is smaller than the number of leafs in the forest, {i.e.}~$\mathsf T\le|\partial\Lambda_n|$.
This gives
$$\phi^1_{p,q}[\calT_0]\stackrel{\eqref{eq:aair}}= \frac{\phi^1_{p,q}[\mathsf T] }{|\Lambda_n|}\le \frac{|\partial\Lambda_n|}{|\Lambda_n|}\longrightarrow 0\quad\text{as $n\rightarrow \infty$. }$$
Combined with \eqref{eq:pp}, this implies that $\phi^1_{p,q}[\calE_\infty]=0$. The claim follows.\end{proof}

\bexo
\begin{exercise}\label{exo:uniqueness}
We say that an (countable) infinite locally finite transitive graph $\bbG$ is amenable if 
$$\inf_{G\subset \bbG} \frac{|\partial G|}{|G|}=0.$$
Show that Theorem~\ref{thm:uniqueness} still holds in this context. What about graphs which are not amenable, do we always have uniqueness of the infinite cluster?
\end{exercise}
\eexo

We now turn to the first problem and prove the following.

\begin{theorem}\label{uniqueness Dq}For $q\geq 1$, the set of edge-weights $p$ for which $\phi^1_{p,q}\ne \phi^0_{p,q}$ is at most countable.
\end{theorem}
The theorem implies that for any $p>p_c$, there exists $p'\in (p_c,p)$ such that $\phi^1_{p',q}=\phi^0_{p',q}$. As a consequence, 
\begin{equation}\label{eq:oiu}\phi^0_{p,q}[0\longleftrightarrow \infty]\ge \phi^0_{p',q}[0\longleftrightarrow \infty]=\phi^1_{p',q}[0\longleftrightarrow \infty]>0.\end{equation}
In other words, $\phi^0_{p,q}[0\leftrightarrow \infty]>0$ for any $p>p_c$ and we could have defined the critical point using the free boundary conditions instead of the wired ones. We will use this fact quite often.

The proof of Theorem~\ref{uniqueness Dq} goes back to Lebowitz and Martin-L\"of \cite{LebMar72} in the case of the Ising model. The very elegant argument harvests the convexity of the free energy (see Exercise~\ref{exo:free energy}). Here, we present a slightly rephrased version of this argument, which relies on the fact that the probability for an edge to be open is increasing.

\begin{proof}
Before diving into the proof, let us remark that 
$$\phi^0_{p,q}=\phi^1_{p,q}\qquad\Longleftrightarrow\qquad\phi_{p,q}^0[\omega_e]=\phi_{p,q}^1[\omega_e],\ \forall e\in \bbE.$$ The direct implication being obvious, we assume the assertion on the right  and try to prove the one on the left. Consider an increasing event $\calA$ depending on a finite set $E$ of edges, then if ${\bf P}_n$ denotes the increasing coupling between $\omega\sim\phi_{\Lambda_n,p,q}^0$ and $\tilde\omega\in \phi^1_{\Lambda_n,p,q}$ constructed in the proof of Lemma~\ref{lem:dynamic}, we find that 
\begin{align*}0\le \phi_{\Lambda_n,p,q}^1[\calA]-\phi_{\Lambda_n,p,q}^0[\calA]&={\bf P}_n[\tilde\omega\in \calA,\omega\notin\calA]\\
&\le \sum_{e\in E}{\bf P}_n[\tilde\omega_e=1,\omega_e=0]=\sum_{e\in E}\phi_{\Lambda_n,p,q}^1[\omega_e]-\phi_{\Lambda_n,p,q}^0[\omega_e].\end{align*}
Letting $n$ go to infinity implies that $\phi^1_{p,q}[\calA]=\phi^0_{p,q}[\calA]$. Since increasing events depending on finitely many edges generate the $\sigma$-algebra, this gives that $\phi^1_{p,q}=\phi^0_{p,q}$.
\medbreak
Our goal is to prove that $ \phi_{p,q}^1[\omega_0]=\phi_{p,q}^0[\omega_e]$ at any point of continuity of $p\mapsto \phi_{p,q}^1[\omega_e]$. Since this function is increasing, it  has at most countably many points of discontinuity and the theorem will follow. 
Below, we fix such a point of continuity $p$. 
We also consider $p'<p$ and set $a:=\phi_{p,q}^0[\omega_e]$ and $b:=\phi_{p',q}^1[\omega_e]$. 

Consider $\ep=\min\{1-a,b\}>0$ and $n\ge1$. The comparison between boundary conditions gives that
\begin{align}\label{eq:lp}\phi^0_{\Lambda_n,p,q}[o(\omega)]\le a|E_n|\quad&\text{so that}\quad\phi^0_{\Lambda_n,p,q}[o(\omega)\le (a+\ep)|E_n|]\ge\ep,\\
\label{eq:lpp}\phi^1_{\Lambda_n,p',q}[o(\omega)]\ge b|E_n|\quad&\text{so that}\quad\phi^1_{\Lambda_n,p',q}[o(\omega)\ge (b-\ep)|E_n|]\ge\ep.\end{align}
(For the inequalities on the right, we also used that $0\le o(\omega)\le |E_n|$.) Now, using that $k(\omega^1)\le k(\omega)\le k(\omega^1)+|\partial\Lambda_n|$ and setting $\lambda:=\frac{p'(1-p)}{(1-p')p}<1$, we find\footnote{ We use that for a random variable $X$, that $$\phi^\xi_{G,p',q}[X]=\frac{\phi^\xi_{G,p,q}[X\lambda^{o(\omega)}]}{\phi^\xi_{G,p,q}[\lambda^{o(\omega)}]}$$ since
\begin{align*}\sum_{\omega\in\{0,1\}^E}X(\omega)p'^{\,o(\omega)}(1-p')^{c(\omega)}q^{k(\omega)}&= (1-p')^{|E|} \sum_{\omega\in\{0,1\}^E}X(\omega)\big(\tfrac {p'}{1-p'}\big)^{o(\omega)}q^{k_\xi(\omega)}\\
&= (1-p')^{|E|} \sum_{\omega\in\{0,1\}^E}\lambda^{o(\omega)}X(\omega)\big(\tfrac {p}{1-p}\big)^{o(\omega)}q^{k(\omega)}\\
&= (\tfrac{1-p'}{1-p})^{|E|} \sum_{\omega\in\{0,1\}^E}\lambda^{o(\omega)}X(\omega)p^{o(\omega)}(1-p)^{c(\omega)}q^{k(\omega)}.
\end{align*}} that
\begin{align*}
\ep\stackrel{\eqref{eq:lpp}}\le \phi^1_{\Lambda_n,p',q}[o(\omega)>(b-\ep)|E_n|]&\stackrel{\phantom{\eqref{eq:lp}}}\le q^{|\partial\Lambda_n|}\,\phi^0_{\Lambda_n,p',q}[o(\omega)>(b-\ep)|E_n|]\\
&\stackrel{\phantom{\eqref{eq:lp}}}\le q^{|\partial\Lambda_n|}\,\frac{\phi^0_{\Lambda_n,p,q}[\lambda^{o(\omega)}\mathbbm{1}_{o(\omega)>(b-\ep)|E_n|}]}{\phi^0_{\Lambda_n,p,q}[\lambda^{o(\omega)}\mathbbm{1}_{o(\omega)\le(a+\ep)|E_n|}]}\\
&\stackrel{\eqref{eq:lp}}\le \frac{q^{|\partial\Lambda_n|}\lambda^{(b-a-2\ep) |E_n|}}{\ep}.
\end{align*}
The fact that $|E_n|/|\partial\Lambda_n|$ tends to infinity as $n$ tends to infinity implies that $b\le a+2\ep$. Since this is true for any $\ep>0$, we deduce $b\le a$. Letting $p'$ tend to $p$ and using the continuity of $p'\mapsto \phi_{p',q}^1[\omega_e]$ at $p$ gives that $\phi_{p,q}^1[\omega_e]\le\phi_{p,q}^0[\omega_e]$. Since we already have $\phi_{p,q}^1[\omega_0]\ge\phi_{p,q}^0[\omega_e]$, this concludes the proof.\end{proof}

\bexo
\begin{exercise}\label{exo:free energy}
1. Show that $\displaystyle Z^1_{\Lambda_{2n},p,q}\ge\big(Z^1_{\Lambda_{n},p,q}\big)^{2^d}$. 
\medbreak\noindent
2. Deduce that $\displaystyle f_n^1(p,q):=\tfrac1{|E_{2^n}|}\log (Z^1_{\Lambda_{2^n},p,q})$ converges to a quantity $f(p,q)$ (called the {\em free energy}).
 \medbreak\noindent
3. Show that $\displaystyle f_n^0(p,q):=\tfrac1{|E_{2^n}|}\log (Z^0_{\Lambda_{2^n},p,q})$ converges to $f(p,q)$ as well.
\medbreak\noindent
4. Show that the right and left derivatives of 
$$t\mapsto f\Big(\frac{e^t}{1+e^{t}},q\Big)+\log(1+e^{t})$$ are respectively $\phi^1_{p,q}[\omega_e]$ and $\phi^0_{p,q}[\omega_e]$.
\medbreak\noindent
5. Show that $p\mapsto f(p,q)$ is convex and therefore not differentiable in at most countably many points. Conclude.\end{exercise}
%

\eexo

Let us conclude this section by stating the following corollary for the Potts model.
\begin{corollary}
Consider the Potts model on $\bbZ^d$. For any $\beta>\beta_c$, {\em (LRO$_\beta$)} holds true, while for any $\beta<\beta_c$, {\em (LRO$_\beta$)} does not hold.
\end{corollary}
Note that we do not claim that the property is equivalent to (MAG$_\beta$) since at $\beta_c$, one may have (MAG$_{\beta_c}$) but not (LRO$_{\beta_c}$).
\begin{proof}
By the coupling with the random-cluster model, we need to prove that $\phi^0_{p,q}[0\leftrightarrow x]$ tends to 0 when $p<p_c$, which is obvious, and that $\phi^0_{p,q}[0\leftrightarrow x]$ does not tend to 0 when $p>p_c$, which follows from 
$$\phi^0_{p,q}[0\longleftrightarrow x]\ge \phi^0_{p,q}[0\longleftrightarrow\infty,x\longleftrightarrow \infty]\ge \phi^0_{p,q}[0\longleftrightarrow\infty]^2\stackrel{\eqref{eq:oiu}}>0,$$
where the first inequality is due to the uniqueness of the infinite cluster, and the second to the FKG inequality and the invariance under translations.
\end{proof}

\section{Computation of critical points and sharp phase transitions}

We would now like to discuss how  the critical point of a planar percolation model can sometimes be computed, and how fast correlations decay when $p<p_c$. We start by studying Bernoulli percolation, and then focus on the random-cluster model. 

\subsection{Kesten's theorem}\label{sec:2.1}

In this section, we focus on the case $d=2$. We begin by discussing the duality relation for Bernoulli percolation. Consider the dual lattice $(\bbZ^2)^*:=(\tfrac12,\tfrac12)+\bbZ^2$ of the lattice $\bbZ^2$ defined by putting a vertex in the middle of each face, and edges between nearest neighbors. Each edge $e\in\bbE$ is in direct correspondence with an edge $e^*$ of the dual lattice crossing it in its middle. For a finite graph $G=(V,E)$, let $G^*$ be the graph with edge-set $E^*=\{e^*,e\in E\}$ and vertex-set given by the endpoints of the edges in $E^*$. 

A configuration $\omega$ is naturally associated to a dual configuration $\omega^*$: every edge $e$ which is closed (resp.~open) in $\omega$ corresponds to a open (resp.~closed) edge $e^*$ in $\omega^*$. More formally, $$\omega^*_{e^*}:=1-\omega_e\qquad\forall e\in E.$$
Note that if $\omega$ is sampled according to $\bbP_p$, then $\omega^*$ is sampled according to $\bbP_{1-p}$.
This duality relation suggests that the critical point of Bernoulli percolation on $\bbZ^2$ is equal to 1/2. We discuss different levels of heuristic leading to this prediction.  

\paragraph{Heuristic level 0} The simplest non-rigorous justification of the fact that $p_c=1/2$ invokes the uniqueness of the phase transition, i.e.~the observation that the model should undergo a single change of macroscopic behavior as $p$ varies. This implies that $p_c$ must be equal to $1-p_c$, since otherwise the model will change at $p_c$ (with the appearance of an infinite cluster in $\omega$), and at $1-p_c$ (with the disappearance of an infinite cluster in $\omega^*$). Of course, it seems difficult to justify why there should be a unique phase transition. This encourages us to try to improve our heuristic argument.

\paragraph{Heuristic level 1} One may invoke a slightly more subtle argument. On the one hand, assume for a moment that $p_c<1/2$. In such case, for any $p\in(p_c,1-p_c)$, there (almost surely) exist infinite clusters in both $\omega$ and $\omega^*$. Since the infinite cluster is unique almost surely, this seems to be difficult to have coexistence of an infinite cluster in $\omega$ and an infinite cluster in $\omega^*$, and it therefore leads us to believe that $p_c\ge 1/2$. On the other hand, assume that $p_c>1/2$. In such case, for any $p\in(p_c,1-p_c)$, there (almost surely) exist no infinite cluster in both $\omega$ and $\omega^*$. This seems to contradict the intuition that if clusters are all finite in $\omega$, then $\omega^*$ should contain an infinite cluster. This reasoning is wrong in general (there may be no infinite cluster in both $\omega$ and $\omega^*$), but it seems still believable that this should not occur for a whole range of values of $p$. Again, the argument is fairly weak here and we should improve it.

\begin{figure}
\begin{center}
\includegraphics[width=0.50\textwidth]{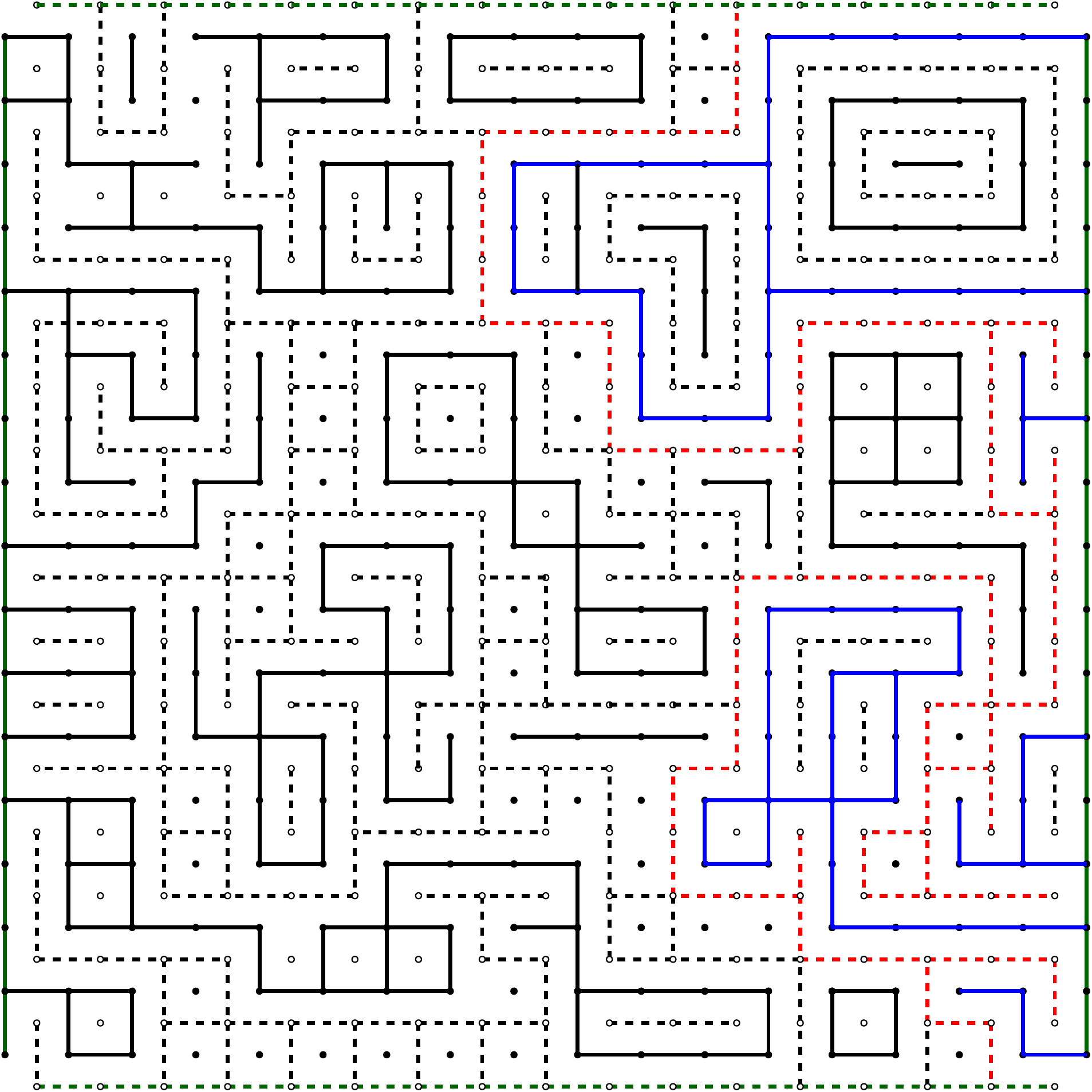}
\caption{The rectangle $R_n$ together with its dual $R_n^*$ (the green edges on the boundary are irrelevant for the crossing, so that we may consider only the black edges, for which the dual graph is isomorphic to the graph itself (by rotating it). The dual edges (in red) of the edge-boundary of the cluster of the right boundary in $\omega$ (in blue) is a cluster in $\omega^*$ crossing from top to bottom in $R_n^*$.}\label{fig:carre}\end{center}
\end{figure}

\paragraph{Heuristic level 2} Consider the event, called $\calH_n$, corresponding to the existence of a path of open edges of $\omega$ in $R_n:=[0,n]\times[0,n-1]$ going from the left to the right side of $R_n$. Observe that the complement of the event $\calH_n$ is the event that there exists a path of open edges in $\omega^*$ going from top to bottom in the graph $R_n^*$; see Fig.~\ref{fig:carre}. Using the rotation by $\pi/2$, one sees that at $p=1/2$, these two events have the same probability,  so that
\begin{equation}\label{eq:1/2}\bbP_{1/2}[\calH_n]=\tfrac12\quad\forall n\ge1.\end{equation}
Now, one may believe that for $p<p_c$, the clusters are so small that the probability that one of them contains a path crossing $R_n$ from left to right  tends to 0, which would imply that the probability of $\calH_n$ would tend to 0, and therefore that $p_c\le 1/2$. On the other hand, one may believe that for $p>p_c$, the infinite cluster is so omnipresent that it contains with very high probability a path crossing $R_n$ from left to right, thus implying that the probability of $\calH_n$ would tend to 1. This would give $p_c\ge1/2$. Unfortunately, the first of these two claims is difficult to justify. Nevertheless, the second one can be proved as follows.
\begin{proposition}\label{prop:lower}
Assume that $\bbP_p[0\leftrightarrow \infty]>0$, then $\displaystyle\lim_{n\rightarrow \infty}\bbP_p[\calH_n]=1.$ 
\end{proposition}
\begin{proof}
Fix $n\ge k\ge 1$. Since a path from $\Lambda_k$ to $\Lambda_n$ ends up either on the top, bottom, left or right side of $\Lambda_n$, the square root trick using the FKG inequality (See Exercise~\ref{exo:square root trick}) implies that 
$$\bbP_p[\Lambda_k\text{ is connected in $\Lambda_n$ to the left of }\Lambda_n]\ge 1-\bbP_p[\Lambda_k\not\longleftrightarrow \infty]^{1/4}.$$
Set $n'=\lfloor (n-1)/2\rfloor$. Consider the event $\calA_{n}$ that $(n',n')+\Lambda_k$ is connected in $R_n$ to the left of $R_{n}$, and $(n'+2,n')+\Lambda_k$ is connected in $R_{n}$ to the right of $R_{n}$. We deduce that $$\bbP_p[\calA_{n}]\ge 1-2\bbP_p[\Lambda_k\not\longleftrightarrow \infty]^{1/4}.$$
The uniqueness of the infinite cluster implies\footnote{The event $\calA_n\setminus\calH_n$ is included in the event that there are two distinct clusters in $\Lambda_n$ going from $\Lambda_k$ to $\partial\Lambda_n$. The intersection of the latter events for $n\ge1$ is included in the event that there are two distinct infinite clusters, which has zero probability. Thus, the probability of $\calA_n\setminus \calH_n$ goes to 0 as $n$ tends to infinity.} that 
$$\liminf_{n\rightarrow\infty}\bbP_p[\calH_{n}]=\liminf_{n\rightarrow\infty}\bbP_p[\calA_{n}]\ge 1-2\bbP_p[\Lambda_k\not\longleftrightarrow \infty]^{1/4}.$$
Letting $k$ tend to infinity and using that the infinite cluster exists almost surely, we deduce that $\bbP_p[\calH_{n}]$ tends to 1.  \end{proof}

\begin{figure}
\begin{center}
\includegraphics[width=0.50\textwidth]{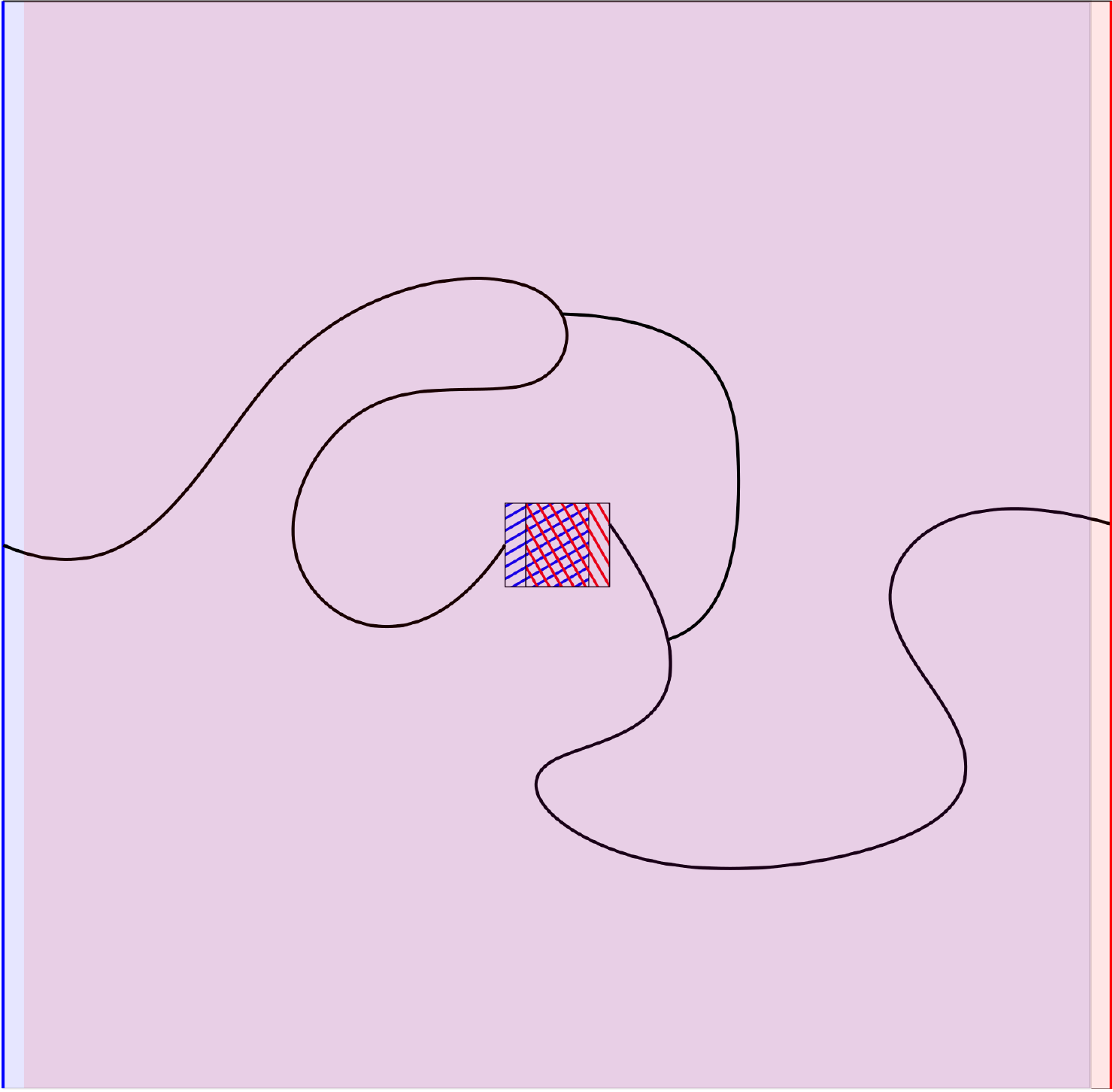}
\caption{Construction in the proof of Proposition~\ref{prop:lower}. One path connects the left side of $R_n$ (in blue) to the blue hatched area. The other one from the right side of $R_n$ (in red) to the red hatched area. The two paths must be in the same cluster (of $R_n$) by uniqueness, which therefore must contain a path from left to right.}\label{fig:carre}\end{center}
\end{figure}

\bexo
\begin{exercise}[Square root trick]\label{exo:square root trick}
Prove, using {\rm (FKG)}, that for any increasing events $\calA_1,\dots ,\calA_r$,
$$\max\{\bbP_p[\calA_i]:1\le i\le r\}\ge 1-\Big(1-\bbP_p\big[\bigcup_{i=1}^r\calA_i\big]\Big)^{1/r}.$$
\end{exercise}

\begin{exercise}[Zhang argument]
1. Show that 
$$\displaystyle\bbP_{1/2}[\text{top of }\Lambda_n\text{ is connected to infinity outside }\Lambda_n]\ge 1-\bbP_{1/2}[\Lambda_n\not\longleftrightarrow \infty]^{1/4}.$$
\medbreak\noindent
2. Deduce that the probability of the event $\calB_n$ that there exist infinite paths in $\omega$ from the top and bottom of $\Lambda_n$ to infinity in $\bbZ^2\setminus \Lambda_n$, and infinite paths in $\omega^*$ from the left and right sides to infinity satisfies
$$\bbP_{1/2}[\calB_n]\ge 1-4\bbP_{1/2}[\Lambda_n\not\longleftrightarrow \infty]^{1/4}.$$
3. Using \eqref{eq:finite energy} and the uniqueness of the infinite cluster, prove that $\bbP_{1/2}[\Lambda_n\not\leftrightarrow \infty]$ cannot tend to 0. 
\end{exercise}

\eexo

This proposition together with \eqref{eq:1/2} implies the following corollary
\begin{corollary}\label{cor:lower bound}
There is no infinite cluster at $p=1/2$. In particular, $p_c\ge1/2$.
\end{corollary}

As mentioned above, the last thing to justify rigorously is the fact that for $p<p_c$, $\bbP_p[\calH_n]$ tends to 0. There are alternative ways of getting the result, in particular by proving that the function $p\mapsto \bbP_p[\calH_n]$ undergoes a sharp threshold\footnote{A sequence $(f_n)$ of continuous heomomorphisms from $[0,1]$ onto itself satisfies a sharp threshold if for any $\ep>0$, $\Delta_n(\ep):=f_n^{-1}(1-\ep)-f_n^{-1}(\ep)$ tends to 0.} near $1/2$. This sharp threshold could be proved by hand (as done in \cite{Kes80}), or using abstract theorems coming from the theory of Boolean functions (as done in \cite{BolRio06c}). Overall, one obtains the following result, which goes back to the early eighties.
\begin{theorem}[Kesten \cite{Kes80}]
For Bernoulli percolation on $\bbZ^2$, $p_c$ is equal to 1/2. Furthermore, there is no infinite cluster at $p_c$.
\end{theorem}
In these lectures, we choose a different road to prove that $\bbP_p[\calH_n]$ tends to 0. Assume for a moment that for any $p<p_c$, there exists $c_p>0$ such that for all $n\ge1$,
$$\bbP_p[0\longleftrightarrow \partial\Lambda_n]\le \exp(-c_pn).$$
Then, $\bbP_p[\calH_n]$ tends to 0 as $n$ tends to infinity since
\begin{align*}\bbP_p[\calH_n]&\le \sum_{k=0}^{n-1}\bbP_p[(k,0)\text{ is connected to the right of $R_n$}]\\
&\le n\,\bbP_p[0\longleftrightarrow \partial\Lambda_n]\le n\exp(-c_pn).\end{align*}
Overall, Kesten's theorem thus follows from the following result.
\begin{theorem}\label{thm:perco} Consider Bernoulli percolation on $\bbZ^d$,
\begin{enumerate}
\item\label{item:1}For $p<p_c$,  there exists $c_p>0$ such that for all $n\ge1$,
$\bbP_p[0\leftrightarrow \partial\Lambda_n]\le \exp(-c_pn)$.
\item\label{item:2} There exists $c>0$ such that for $p>p_c$, $\bbP_p[0\leftrightarrow \infty]\ge
 c (p-p_c)$.
\end{enumerate}
\end{theorem}
Note that the second item, called the {\em mean-field lower bound} is not relevant for the proof of Kesten's Theorem. Also note that Theorem~\ref{thm:perco} is a priori way too strong compared to what is needed since it holds in arbitrary dimension. 
\bexo
\begin{exercise}[$p_c(\bbG)+p_c(\bbG^*)=1$] In this exercise, we use the notation $A\stackrel{B}\longleftrightarrow C$ the event that $A$ and $C$ are connected by a path using vertices in $B$ only. Consider Bernoulli percolation on a planar lattice $\bbG$ embedded in such a way that $\bbZ^2$ acts transitively on $\bbG$. We do not assume any symmetry of the lattice. We call the left, right, top and bottom parts of a rectangle $\mathsf{Left}$, $\mathsf{Right}$, $\mathsf{Top}$ and $\mathsf{Bottom}$. Also, $\calH(n,k)$ and $\calV(n,k)$ are the events that $[0,n]\times[0,k]$ is crossed horizontally and vertically by paths of open edges.  
\medbreak\noindent
1. Use the Borel-Cantelli lemma and Theorem~\ref{thm:perco} (one may admit the fact that the theorem extends to this context) to prove that for $p<p_c(\bbG)$, there exists finitely many open circuits surrounding a given vertex of $\bbG^*$. Deduce that
$p_c(\bbG)+p_c(\bbG^*)\le 1.$
\medbreak\noindent
We want to prove the converse inequality by contradiction. From now on, we assume that both $p>p_c(\bbG)$ and $p^*>p_c(\bbG^*)$. 
\medbreak\noindent
2. For $s>0$ and $x\in\bbZ^2$, define $S_x=x+[0,s]^2$. Prove that for any rectangle $R$, there exists $x=x(R)\in R\cap\bbZ^2$ such that there exists $x'$ and $x''$ neighbors of $x$ in $\bbZ^2$ satisfying
\begin{eqnarray}
&\bbP_p[S_x\stackrel{R}\longleftrightarrow \mathsf{Bottom}]\ge \bbP_p[S_x\stackrel{R}\longleftrightarrow \mathsf{Top}]& \bbP_p[S_x\stackrel{R}\longleftrightarrow \mathsf{Left}]\ge \bbP_p[S_x\stackrel{R}\longleftrightarrow \mathsf{Right}],\\
&\ \,\bbP_p[S_{x'}\stackrel{R}\longleftrightarrow \mathsf{Top}]\ge\bbP_p[S_{x'}\stackrel{R}\longleftrightarrow \mathsf{Bottom}]&\bbP_p[S_{x''}\stackrel{R}\longleftrightarrow \mathsf{Right}]\ge \bbP_p[S_{x''}\stackrel{R}\longleftrightarrow \mathsf{Left}].
\end{eqnarray}
3.  Set $\bbH:=\bbR_+\times\bbR$, $\ell_+:=\{0\}\times\bbR_+$, $\ell_-:=\{0\}\times\bbR_-$ and $\ell=\ell_-\cup\ell_+$.
Prove that there exists $x=x(m)$ with first coordinate equal to $m$ satisfying
\begin{align*}\bbP_p[S_x\stackrel{\bbH}\longleftrightarrow \ell_-]\ge\bbP_p[S_x\stackrel{\bbH}\longleftrightarrow \ell_+]\quad\text{ and }\quad\bbP_p[S_{x+(0,1)}\stackrel{\bbH}\longleftrightarrow \ell_-]\le\bbP_p[S_{x+(0,1)}\stackrel{\bbH}\longleftrightarrow \ell_+].
\end{align*}
\medbreak\noindent
4. Using the square root trick, deduce that 
$$\bbP_p[S_x\stackrel{\bbH}{\longleftrightarrow} \ell_-]\ge1-\sqrt{\bbP_p[S_x{\not\longleftrightarrow}\, \ell]}\quad\text{ and }\quad\bbP_p[S_{x+(0,1)}\stackrel{\bbH}{\longleftrightarrow} \ell_+]\ge1-\sqrt{\bbP_p[S_{x+(0,1)}{\not\longleftrightarrow} \,\ell]}.$$
\medbreak\noindent
5. Using the fact that there exists a unique infinite cluster in $\omega$ almost surely, prove that the probability that $\{0\}\times[0,1]$ is connected in $\omega^*\cap \bbH$ to infinity is tending to 0.  
\medbreak\noindent
6. Prove that the distance between $x(R)$ and the boundary of $R$ is necessarily tending to infinity as $\min\{n,k\}$ tends to infinity. 
\medbreak\noindent
7. Using $x(R)$, prove that $\max\{\bbP_p[\calV(n,k)],\bbP_p[\calH(n,k+1)]\}$ tends to 1 and $\min\{\bbP_p[\calV(n,k)],\bbP_p[\calH(n,k)]\}$ tends to 0 as $\min\{k,n\}$ tends to infinity. {\em Hint.} Use the square root trick and the uniqueness criterion like in the previous questions.
\medbreak\noindent
8. By considering the largest integer $k$ such that $\bbP_p[\calV(n,k)]\ge\bbP_p[\calH(n,k)]$, reach a contradiction. Deduce that $p_c(\bbG)+p_c(\bbG^*)\ge 1$.
\medbreak\noindent
9. (to do after Section~\ref{sec:2.3}) How does this argument extend to random-cluster models with $q\ge1$?
\end{exercise}
\eexo

\subsection{Two proofs of sharpness for Bernoulli percolation}\label{sec:2.2}

 Theorem~\ref{thm:perco} was first  proved by Aizenman,
  Barsky~\cite{AizBar87} and Menshikov~\cite{Men86} (these two proofs are presented in \cite{Gri99a}). Here, we choose to present two new arguments from \cite{DumTas15,DumTas15a} and \cite{DumRaoTas17}. 
    
  Before diving into the proofs, note that for any function $X:\{0,1\}^E\longrightarrow \bbR$ where $E$ is finite,
  \begin{equation}\label{eq:derivative formula}\tag{DF}\frac{{\rm d }\bbP_p[X]}{{\rm d}p}=\tfrac{1}{p(1-p)}\sum_{e\in E}{\rm Cov}_p[X,\omega_e],\end{equation}
(where ${\rm Cov}_p$ is the covariance for $\bbP_p$) which is obtained readily by differentiating the quantity $\bbE_p[X]=\displaystyle\sum_{\omega\in \{0,1\}^E}X(\omega)p^{\sum_{e}\omega_e}(1-p)^{\sum_e1-\omega_e}$. We insist on the fact that we are considering functions $X$ depending on finitely many edges only (in particular it is clear that $\bbE_p[X]$ is analytic).

\subsubsection{Proof using the $\varphi_p(S)$ quantity}

  Let $S$
be a finite set of vertices containing the origin. We say that $0\stackrel{S}{\longleftrightarrow} x$ if $0$ is connected to $x$ using only edges between vertices of $S$.
We denote the edge-boundary of $S$ by 
$$\Delta S=\big\{xy\subset \bbE:x\in S,y\notin S\big\}.$$
For $p\in[0,1]$ and $0\in
S\subset \bbZ^d$, define
\begin{equation}
  \label{eq:16}
  \varphi_p(S):=p\sum_{xy\in \Delta S} \bbP_p[0\stackrel{S}{\longleftrightarrow} x].
\end{equation}
Set 
\begin{equation}\label{eq:22.5}\tilde p_c:=\sup\big\{p\in[0,1]:\exists S\ni0\text{ finite with }\varphi_p(S)<1\big\}.\end{equation}

\paragraph{Step 1: for $p<\tilde p_c$, (EXP$_p$) holds true.} By definition, one can fix a finite set $S$
containing the origin, such that $
  \varphi_p(S)<1$.
Choose $L>0$ such that $S\subset \Lambda_{L-1}$. 
Consider $k\ge 1$ and assume
that the event $0\leftrightarrow\partial\Lambda_{kL}$ holds. Introduce the random variable $\mathsf C:=\{x\in S:x\stackrel{S}{\longleftrightarrow}0\}$ corresponding to the cluster of $0$ in $S$. Since $S\cap\partial
\Lambda_{kL}=\emptyset$, one can find an open edge $xy\in\Delta S$ such
that  $0\stackrel{S}{\longleftrightarrow}x$ and $y\stackrel{\mathsf C^c}{\longleftrightarrow}\partial \Lambda_{kL}$.
Using the union bound, and then a decomposition on the possible realizations of $\mathsf C$, we find
\begin{align*}
  \bbP_p[0\longleftrightarrow\partial\Lambda_{kL}]&\le \sum_{xy\in
    \Delta S}\sum_{C\subset S}\bbP_p\big[\{0\stackrel{S}{\longleftrightarrow}x\} \cap \{\mathsf C=C\}\cap \{\omega_{xy}=1\}\cap\{y\stackrel{C^c}{\longleftrightarrow}\partial
  \Lambda_{kL}\}\big]\\
&\le \sum_{xy\in
    \Delta S}\sum_{C\subset S}\bbP_p\big[\{0\stackrel{S}{\longleftrightarrow}x\}\cap\{\mathsf C=C\}\big]\cdot p\cdot \bbP_p\big[y\stackrel{C^c}{\longleftrightarrow}\partial
  \Lambda_{kL}\big]\\
  &\le p\Big(\sum_{xy\in
    \Delta S}\sum_{C\subset S}\bbP_p\big[\{0\stackrel{S}{\longleftrightarrow}x\}\cap\{\mathsf C=C\}\big]\Big)\bbP_p\big[0\longleftrightarrow\partial
  \Lambda_{(k-1)L}\big]\\
  &\le p\Big(\sum_{xy\in
    \Delta S}\bbP_p\big[0\stackrel{S}{\longleftrightarrow}x\big]\Big)\bbP_p\big[0\longleftrightarrow\partial
  \Lambda_{(k-1)L}\big]\\
  &=\varphi_p(S)\bbP_p\big[0\longleftrightarrow\partial
  \Lambda_{(k-1)L}\big].
\end{align*}
In the second line, we used that $\{y\stackrel{C^c}{\longleftrightarrow}\partial\Lambda_{kL}\}$, $\{\omega_{xy}=1\}$ and $\{0\stackrel{S}{\longleftrightarrow}x\}\cap \{\mathsf C=C\}$ are independent. Indeed, these events depend on disjoint sets of edges: the first one on edges with both endpoints outside of $C$, the second one on $xy$ only, and the third one on edges between vertices of $S$ with at least one endpoint in $C$. In the third line, we used $y\in
\Lambda_L$ implies
$$\bbP_p[y\stackrel{C^c}{\longleftrightarrow}\partial
\Lambda_{kL}]\le\bbP_p[0\longleftrightarrow\partial  \Lambda_{(k-1)L}].$$
In the fourth line, we used that the events $\{0\stackrel{S}{\longleftrightarrow}x\}\cap\{\mathsf C=C\}$ partition the event $0\stackrel{S}\longleftrightarrow x$. An induction on $k$ gives
$
  \bbP_p[0\leftrightarrow\partial\Lambda_{kL}]\le\varphi_p(S)^{k},
$
thus proving the claim.

\paragraph{Step 2: For $p>\tilde p_c$, $\bbP_p[0\leftrightarrow \infty]\ge  \frac{p-\tilde p_c}{p(1-\tilde p_c)}$.}
Let us start by the following lemma providing a differential
inequality valid for every $p$. Define $\theta_n(p):=\bbP_p[0\leftrightarrow \partial\Lambda_n]$.

 \begin{lemma}\label{lem:meanField}
Let $p\in(0,1)$ and $n\ge 1$,
  \begin{equation}
    \label{eq:4a}
    \theta_n'(p)\ge \tfrac1{p(1-p)}\cdot
    \inf_{\substack{S\subset\Lambda_n\\0\in S}}\varphi_p(S) \cdot\big(1-\theta_n(p)\big).
  \end{equation}
  \end{lemma}
  
  Let us first see how the second step follows from
  Lemma~\ref{lem:meanField}. Above $\tilde p_c$,  \eqref{eq:4a} becomes $\theta_n'\ge \frac1{p(1-p)}(1-\theta_n)$ which can be rewritten as 
  $$\big[\log\big(\tfrac{1}{1-\theta_n}\big)\big]'\ge \big[\log\big(\tfrac p{1-p}\big)\big]'.$$
Integrating between $\tilde p_c$ and $p$  implies that for every $n\ge1$,
  $$\theta_n(p)\ge \frac{p-\tilde p_c}{p(1-\tilde p_c)}.$$ By letting $n$
  tend to infinity, we obtain the desired lower bound on $\bbP_p[0\leftrightarrow\infty]$.

\begin{proof}[Lemma~\ref{lem:meanField}]Apply \eqref{eq:derivative formula} to $X:=-\mathbbm 1_{0\not\longleftrightarrow \partial\Lambda_n}$ to get
  \begin{align}\label{eq:pol}
        \theta_n'(p)&=\tfrac1{p(1-p)}\sum_{e\in E_n}
        \bbE_p\big[\mathbbm 1_{0\not\longleftrightarrow\partial\Lambda_n}(p-\omega_e)\big].
  \end{align}
  Fix an edge $e$ and consider the event $A$ that $\omega_{|E_n\setminus\{e\}}$ satisfies the following three properties
\begin{itemize}[noitemsep,nolistsep] 
\item[P1] one of the endpoints of $e$ is connected to $0$, 
\item[P2] the other one is connected to $\partial\Lambda_n$, 
\item[P3] 0 is not connected to $\partial\Lambda_n$. 
\end{itemize}
(This event corresponds in the standard terminology to the fact that the edge $e$ is {\em pivotal} for $0\not\longleftrightarrow\partial\Lambda_n$ but this is irrelevant here.) By definition, $\omega_e$ is independent of $\{0\not\longleftrightarrow \partial\Lambda_n\}\cap A^c$. Since $\omega_e$ is a Bernoulli random variable of parameter $p$, we deduce that 
$$\bbE_p[\mathbbm1_{A^c}\mathbbm 1_{0\not\longleftrightarrow\partial\Lambda_n}(p-\omega_e)]=0.$$ 
Also, for $\omega\in A$, $0$ is not connected to $\partial\Lambda_n$ if and only if the edge $e$ is closed, and in this case $\omega$ itself (not only its restriction to $E_n\setminus\{e\}$) satisfies P1, P2 and P3. Therefore, we can write
  $$\bbE_p[\mathbbm1_A\mathbbm 1_{0\not\longleftrightarrow\partial\Lambda_n}(p-\omega_e)]=p\,\bbP_p[\omega\text{ satisfies P1, P2 and P3}].$$
%
%
Overall, the previous discussion implies that \eqref{eq:pol} can be rewritten as
\begin{align}
        \theta_n'(p)&=\tfrac1{p(1-p)}\sum_{\substack{x,y\in \Lambda_n\\ xy\in E_n}}
        p\,\bbP_p\big[0\longleftrightarrow x,y\longleftrightarrow \partial\Lambda_n,0\not\longleftrightarrow \partial\Lambda_n\big].\label{eq:op}
  \end{align}
Introduce $\mathsf S:=\{z\in\Lambda_n:z\not\longleftrightarrow
\partial\Lambda_n\}$ and a fixed set $S$. The intersection of $\{\mathsf S=S\}$ with the event on the right-hand side of \eqref{eq:op} can be rewritten nicely. The fact that $0\not\longleftrightarrow \partial\Lambda_n$ becomes the condition that $S$ contains $0$. Furthermore, the conditions $0\longleftrightarrow x$ and $y\longleftrightarrow \partial\Lambda_n$ get rephrased as $xy\in\Delta S$ and $0$ is connected to $x$ in $S$. Thus, partitioning the event on the right of \eqref{eq:op} into the possible values of $\mathsf S$ gives\begin{align*}
        \theta_n'(p)&        =\tfrac1{p(1-p)}\sum_{0\in S\subset \Lambda_n}\sum_{xy\in \Delta S}
        p\,\bbP_p\big[0\stackrel{S}\longleftrightarrow x,\mathcal S=S\big]\\
        &=\tfrac1{p(1-p)}\sum_{0\in S\subset \Lambda_n}\sum_{xy\in \Delta S}
        p\,\bbP_p\big[0\stackrel{S}\longleftrightarrow x]\bbP_p[\mathcal S=S\big]\\
        &\ge \tfrac1{p(1-p)}\cdot\inf_{0\in S\subset\Lambda_n}\varphi_p(S)\cdot(1-\theta_n(p)),
  \end{align*}
where in the second line we used that $0\stackrel{S}\longleftrightarrow x$ is measurable in terms of edges with both endpoints in $S$, and $\mathsf S=S$ is measurable in terms of the other edges. In the last line, we used that the family of events $\{\mathsf S=S\}$ with $S\ni 0$ partition the event that 0 is not connected to $\partial\Lambda_n$.\end{proof}

Steps 1 and 2 conclude the proof since $\tilde p_c$ must be equal to $p_c$, and therefore the proof of the theorem. 
\bexo
\begin{exercise}[Percolation with long-range interactions]\label{exo:long range}
Consider a family $(J_{x,y})_{x,y\in \bbZ^d}$ of non negative coupling constants which is invariant under translations, meaning that $J_{x,y}=J(x-y)$ for some function $J$. Let ${\bf P}_\beta$ be the bond percolation measure on $\bbZ^d$ defined as follows: for
$x,y\in \bbZ^d$, $\{x,y\}$ is {\em open} with probability $1-\exp(-\beta J_{x,y})$, and
{\em closed} with probability $\exp(-\beta J_{x,y})$. 
\medbreak\noindent
1. Define the analogues $\tilde\beta_c$ and $\varphi_\beta(S)$ of $\tilde p_c$ and $\varphi_p(S)$ in this context.
\medbreak\noindent
2. Show that there exists $c>0$ such that for any $\beta\ge\tilde\beta_c$, ${\bf P}_\beta[0\longleftrightarrow \infty]\ge c(\beta-\tilde\beta_c)$.
\medbreak\noindent
3. Show that if the interaction is finite range (i.e.~that there exists $R>0$ such that $J(x)=0$ for $\|x\|\ge R$), then for any $\beta<\tilde\beta_c$, there exists $c_\beta>0$ such that ${\bf P}_\beta[0\longleftrightarrow \partial\Lambda_n]\le \exp(-c_\beta n)$ for all $n$.
\medbreak\noindent
4. In the general case, show that for any $\beta<\tilde\beta_c$, 
$\displaystyle\sum_{x\in\bbZ^d}{\bf P}_\beta[0\longleftrightarrow x]<\infty.$\\
{\em Hint.} Consider $S$ such that $\varphi_\beta(S)<1$ and show that for $n\ge1$ and $x\in\Lambda_n$,
$\displaystyle\sum_{y\in\Lambda_n}{\bf P}_\beta[x\stackrel{\Lambda_n}\longleftrightarrow y]\le \frac{|S|}{1-\varphi_\beta(S)}.$
\end{exercise}
\eexo
\begin{remark} Since $\varphi_p(\{0\})=2dp$, we find $p_c(d)\ge 1/{2d}$. Also, $p_c(d)\le p_c(2)=\tfrac12$.
  \end{remark}
  \begin{remark}
The set of parameters $p$
  such that there exists a finite set $0\in S \subset \bbZ^d$ with
  $\varphi_p(S)<1$ is an open subset of $[0,1]$. Since this set is coinciding with $[0,p_c)$, we deduce that $\varphi_{p_c}(\Lambda_n)\ge1$ for any $n\ge1$.  As a consequence, the expected size of the
  cluster of the origin satisfies at $p_c$, $$\sum_{x\in\bbZ^d}\bbP_{p_c}[0\longleftrightarrow x]\ge\tfrac1{dp_c}\sum_{n\ge 0} \varphi_{p_c}(\Lambda_n) =+\infty.$$
In particular, $\bbP_{p_c}[0\leftrightarrow x]$ cannot decay faster than algebraically (see Exercise~\ref{eq:definition correlation length} for more details).
\end{remark}

\bexo

\begin{exercise}[Definition of the correlation length] \label{eq:definition correlation length}Fix $d\ge2$ and set $e_1=(1,0,\dots,0)$.
\medbreak\noindent
1. Prove that, for any $p \in [0,1]$ and $n,m\ge0$,
$\mathbb{P}_p[x_0 \longleftrightarrow (m+n)e_1] \geq \mathbb{P}_p[x_0 \longleftrightarrow me_1] \cdot \mathbb{P}_p[x_0 \longleftrightarrow ne_1]$
\medbreak\noindent
2. Deduce that 
$\xi(p) = \left(\displaystyle \lim_{n\rightarrow \infty} -\tfrac{1}{n} \log \mathbb{P}_p[0 \longleftrightarrow ne_1]\right)^{-1} $
and that
$\mathbb{P}_p[0 \longleftrightarrow ne_1] \leq \exp(-n/\xi(p)).$
\medbreak\noindent
3. Prove that $\xi(p)$ tends to infinity as $p$ tends to $p_c$.
\medbreak\noindent
4. Prove that for any $x\ni\partial\Lambda_n$,
$$\bbP_{p_c}[0\longleftrightarrow 2n e_1]\ge \bbP_{p_c}[0\longleftrightarrow x]^2.$$
5. Using that $\varphi_{p_c}(\Lambda_n)\ge1$ for every $n$, prove that there exists $c>0$ such that for any $x\in\bbZ^d$, $\bbP_{p_c}[0\leftrightarrow x]\ge \frac c{\|x\|^{2d(d-1)}}$.
\end{exercise}
\eexo

\subsubsection{Proof using randomized algorithms}\label{sec:OSSS}

The second proof  uses the notion of random decision tree (or equivalently randomized algorithm).
In theoretical science, determining the computational complexity of tasks is a difficult problem (think of $P$ against $NP$). To simplified the problem, computer scientists came up with computational problems involving so-called {\em decision trees}. Informally speaking, a decision tree associated to a Boolean function $f$ takes $\omega\in\{0,1\}^n$ as an input, and reveals algorithmically the value of $\omega$ at different coordinates one by one. At each step, which coordinate will be revealed next depends on the values of $\omega$ revealed so far. The algorithm stops as soon as the value of $f$ is the same no matter the values of $\omega$ on the remaining coordinates. The question is then to determine how many bits of information must be revealed before the algorithm stops.

Formally, a decision tree is defined as follows. 
Consider a finite set $E$ of cardinality $n$. For a $n$-tuple $x=(x_1,\dots,x_n)$ and $t\le n$, write $x_{[t]}=(x_1,\dots,x_t)$ and $\omega_{x_{[t]}}=(\omega_{x_1},\dots,\omega_{x_t})$. A {\em decision tree} $T=(e_1,\psi_t,t<n)$ takes $\omega\in\{0,1\}^E$ as an input and gives back an ordered sequence $e=(e_1,\dots,e_n)$ constructed inductively as follows: for any $2\le t\le n$,
$$e_t=\psi_t(e_{[t-1]},\omega_{e_{[t-1]}})\in E\setminus \{e_1,\dots,e_{t-1}\},$$
where $\psi_t$ is a function interpreted as the decision rule at time $t$ ($\psi_t$ takes the location and the value of the bits for the first $t-1$ steps of the induction, and decides of the next bit to query).
For $f:\{0,1\}^E\rightarrow \bbR$, define 
\begin{equation*}\label{eq:ddd}\tau(\omega)=\tau_{f,T}(\omega):=\min\big\{t\ge1:\forall \omega'\in\{0,1\}^E,\quad\omega'_{e_{[t]}}=\omega_{e_{[t]}}\Longrightarrow f(\omega)=f(\omega')\big\}.\end{equation*}
\begin{remark}In computer science, a decision tree is usually associated directly to a boolean function $f$ and defined as a rooted directed tree in which each internal nodes are labeled by elements of $E$, leaves by possible outputs, and edges are in correspondence with the possible values of the bits at vertices (see \cite{OSSS} for a formal definition). In particular, the decision trees are usually defined up to $\tau$, and not later on. \end{remark}

The OSSS inequality, originally introduced in \cite{OSSS} as a step toward a conjecture of Yao \cite{yao1977probabilistic}, relates the variance of a Boolean function to the influence of the variables and the computational complexity of a random decision tree for this function. 
\begin{theorem}[OSSS for Bernoulli percolation]
\label{thm:OSSS}
Consider $p\in[0,1]$ and a finite set of edges $E$. Fix an increasing function $f:\{0,1\}^E\longrightarrow [0,1]$ and an algorithm $T$. We have
\begin{equation}
    \label{eq:OSSS}
 \mathrm{Var}_p(f)~\le~   2 \sum_{e\in E}  \delta_e(f,T) \, \mathrm{Cov}_p [f,\omega_e] ,
  \end{equation}
  where 
$
\delta_e(f,T):=\bbP_p\big[\exists t\le \tau(\omega)\::\:e_t=e\big]
$
is the revealment (of $f$) for the decision tree $T$.
\end{theorem}
The general inequality does not require $f$ to be increasing, but we will only use it in this context.
\newcommand{\e}{\mathbf e}
\begin{proof}
Our goal is to apply a Linderberg-type argument. Consider two independent sequences $\omega$ and $\tilde\omega$ of iid Bernoulli random variables of parameter $p$. Write $\mathbb P$ for the coupling between these variables (and $\mathbb E$ for its expectation). 
Construct $\e$ by setting $\e_1=e_1$ and for $t\ge1$,  
$\e_{t+1}:=\psi_t(\e_{[t]},\omega_{\e_{[t]}})$. Similarly, define $$\tau:=\min\big\{t\ge1:\forall x\in\{0,1\}^E,x_{\e_{[t]}}=\omega_{\e_{[t]}}\Rightarrow f(x)=f(\omega)\big\}.$$ Finally, for $0\le t\le n$, define $$\omega^t:=(\tilde\omega_{\e_1},\dots,\tilde\omega_{\e_t},\omega_{\e_{t+1}},\dots,\omega_{\e_{\tau-1}},\tilde\omega_{\e_\tau},\tilde\omega_{\e_{\tau+1}},\dots,\tilde\omega_{\e_n}),$$where it is understood that the $n$-tuple under parentheses is equal to $\tilde\omega$ if $t\ge \tau$. (We used a slight abuse of notation, the order here is shuffled to match the order in which the edges are revealed by the algorithm.)

Since $\omega^0$ and $\omega$ coincide on $\e_t$ for any $t\le \tau$, we deduce that $f(\omega^{0})=f(\omega)$. Also, $f(\omega^{n})=f(\tilde\omega)$ since $\omega^n=\tilde\omega$. As a consequence,  
conditioning on $\omega$ gives
\begin{equation*}\label{eq:a}{\rm Var}_p(f)\le \bbE_p\big[|f-\bbE_p[f]|\big]=\bbE\big[\big|\,\bbE[f(\omega^{0})|\omega]-\bbE[f(\omega^{n})|\omega]\,\big|\big]\le\mathbb E\big[|f(\omega^{0})-f(\omega^{n})|\big].\end{equation*}
The triangular inequality and the observation that $\omega^t=\omega^{t-1}$ for any $t>\tau$ gives that
\begin{align*}{\rm Var}_p(f)\le \sum_{t=1}^{n}\mathbb E\big[|f(\omega^t)-f(\omega^{t-1})|]=\sum_{t=1}^{n}\mathbb E\big[|f(\omega^t)-f(\omega^{t-1})|\mathbbm1_{t\le\tau}\big].\end{align*}
Let us now decomposed into the possible values for $\e_t$. Note that $\e_t$ is measurable in terms of $\omega_{[t-1]}$, and that $\tau$ is a stopping time, so that $\{t\le\tau\}=\{\tau\le t-1\}^c$ is also measurable in terms of $\omega_{[t-1]}$. Overall, we get that 
\begin{align*}{\rm Var}_p(f)&\le\sum_{e\in E} \sum_{t=1}^{n}\mathbb E\big[\mathbb E\big[|f(\omega^t)-f(\omega^{t-1})| ~\big|~\omega_{[t-1]}\big]\,\mathbbm1_{t\le \tau,\e_t=e}\big].
\end{align*}
Let $f^1(\omega)$ and $f^0(\omega)$ denote the function $f$ applied to the configuration equal to $\omega$ except at $e$ where it is equal to 1 or to 0 respectively. Note that since $f$ is increasing, we find that $f^1\ge f^0$.  Now, conditionally on $\omega_{[t-1]}$ and $\{t\le \tau,\e_t=e\}$, both $\omega^t$ and $\omega^{t-1}$ are sequences of iid Bernoulli random variables of parameter $p$, differing (potentially) exactly at $e$ (since $\omega^t_e=\tilde\omega_e$ and $\omega^{t-1}_e=\omega_e$). We deduce that 
\begin{equation*}\mathbb E\big[\,|f(\omega^t)-f(\omega^{t-1})| ~\big|~\omega_{[t-1]}\big]~=2p(1-p)\mathbb E_p[f^1(\omega)-f^0(\omega)]=2{\rm Cov}_p[f,\omega_e].\end{equation*}
Recalling that $\sum_{t=1}^{n}\mathbb P[t\le \tau,\e_t=e]=\delta_e(f,T)$ concludes the proof.
\end{proof}

\begin{figure}
\begin{center}\includegraphics[width=0.65\textwidth]{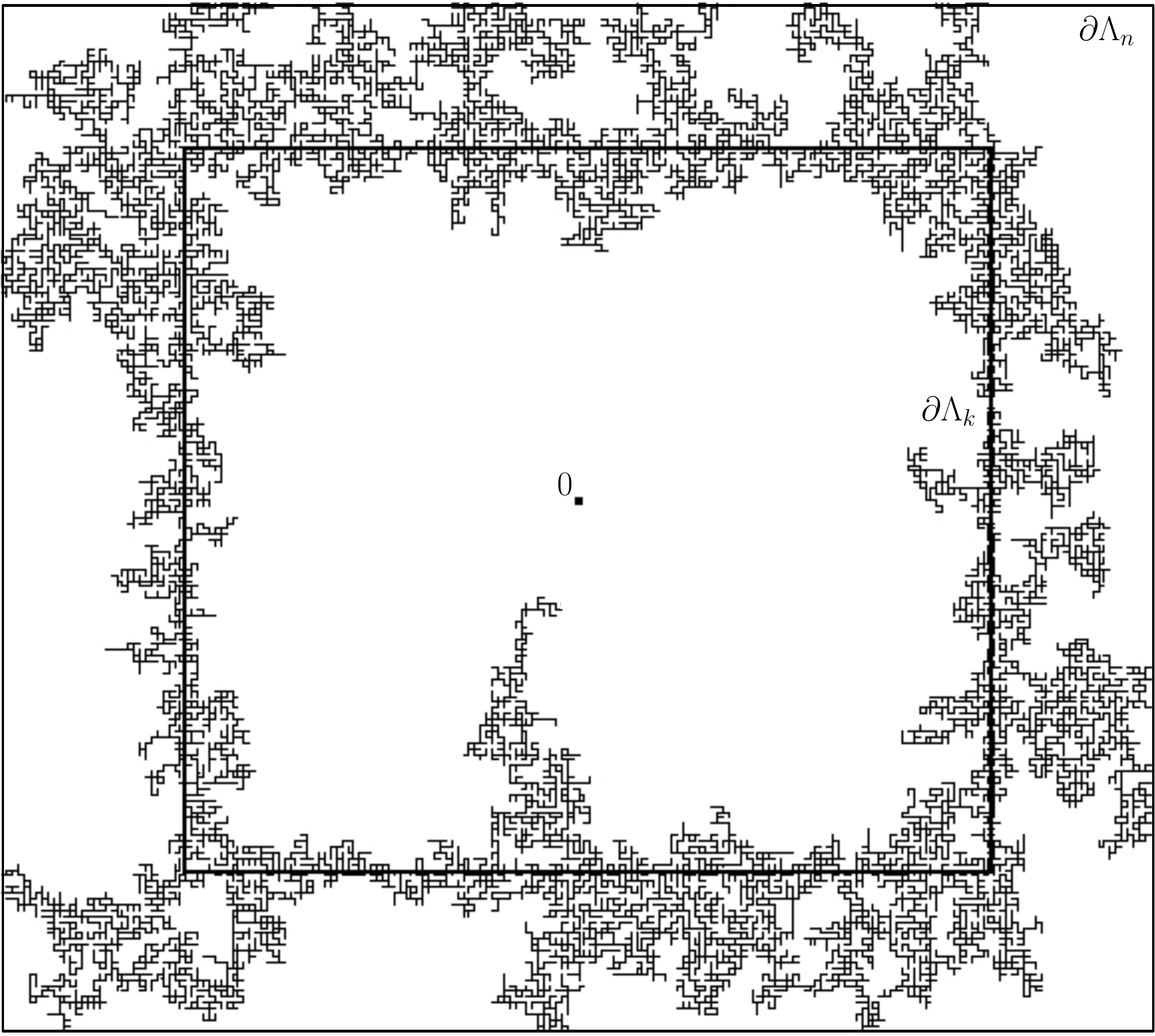}
\caption{\label{fig:algo} A realization of the clusters intersecting $\partial\Lambda_k$. Every edge having one endpoint in this set has been revealed by the decision tree. Furthermore in this specific case, we know that $0$ is not connected to the boundary of $\Lambda_n$. }\end{center}
\end{figure}

Let us start the proof by a general lemma.   
  \begin{lemma}\label{lem:technical}
Consider a converging sequence of increasing differentiable functions $f_n:[0,x_0]\longrightarrow [0,M]$ satisfying
 \begin{equation}\label{eq:mlem}f_n'\ge \frac{n}{\Sigma_{n}}f_n\end{equation}
 for all $n\ge1$, where $\Sigma_n=\sum_{k=0}^{n-1}f_k$. Then, there exists $x_1\in[0,x_0]$ such that  
 \begin{itemize}[noitemsep]
 \item[{\bf P1}] For any $x<x_1$, there exists $c_x>0$ such that for any $n$ large enough,
 $f_n(x)\le \exp(-c_x n).$
  \item[{\bf P2}] For any $x>x_1$, $\displaystyle f=\lim_{n\rightarrow \infty}f_n$ satisfies $f(x)\ge x-x_1.$
 \end{itemize}
  \end{lemma}
  
  \begin{proof}
Define 
$$ x_1 := \inf \Big\{ x \, : \, \limsup_{n \rightarrow \infty} \frac{\log \Sigma_n(x)}{\log n} \geq 1 \Big\}. $$

\paragraph{Assume $x<x_1$.} Fix $\delta>0$ and set $x'=x-\delta$ and $x''=x-2\delta$. We will prove that there is exponential decay at $x''$ in two steps. 

First, there exists an integer $N$ and $\alpha>0$  such that $\Sigma_n(x) \leq n ^ {1-\alpha}$ for all $n \geq N$. For such an integer $n$, integrating $f_n' \geq n^{\alpha} f_n$  between $x'$ and $x$ -- this differential inequality follows from \eqref{eq:mlem}, the monotonicity of the functions $f_n$ (and therefore $\Sigma_n$) and the previous bound on $\Sigma_n(x)$ -- implies that
$$ f_n(x') \leq M\exp(-\delta \,n^\alpha),\quad\forall n\ge N.$$

Second, this implies that there exists $\Sigma < \infty$ such that $\Sigma_n(x') \leq \Sigma$ for all $n$. Integrating $f_n' \geq \tfrac{n}{\Sigma} f_n$ for all $n$ between $x''$ and $x'$ -- this differential inequality is again due to \eqref{eq:mlem}, the monotonicity of $\Sigma_n$, and the bound on $\Sigma_n(x')$ -- leads to
\begin{equation*}f_n (x'') \leq M\exp(-\frac{\delta}{\Sigma} \,n),\quad\forall n\ge0.\label{eq:bb}\end{equation*}

\paragraph{Assume $x>x_1$.} For $n\ge 1$, define the function $T_n := \frac{1}{\log n} \sum_{i=1}^{n} \frac{f_i}{i}$. Differentiating $T_n$ and using \eqref{eq:mlem}, we obtain
$$T_n'~=~ \frac{1}{\log n}\, \sum_{i=1}^{n} \frac{f_i'}{i}  ~\stackrel{\eqref{eq:mlem}}{\geq}~ \frac{1}{\log n}\, \sum_{i=1}^{n} \frac{f_i}{\Sigma_{i}} ~\geq~ \frac{\log \Sigma_{n+1}-\log \Sigma_1}{\log n},$$
where in the last inequality we used that for every $i\ge1$,
$$\frac{f_i}{\Sigma_{i}} \geq \int_{\Sigma_{i}}^{\Sigma_{i+1}} \frac{dt}{t}=\log \Sigma_{i+1}-\log \Sigma_{i}.$$ 
For $x'\in(x_1,x)$, using that $\Sigma_{n+1}\ge\Sigma_n$ is increasing and integrating the previous differential inequality between $x'$ and $x$ gives
$$T_n (x) - T_n(x') \geq (x - x')\,  \frac{\log \Sigma_{n} (x')-\log M}{\log n}.$$
Hence, the fact that $T_n(x)$ converges to $f(x)$ as $n$ tends to infinity implies 
\begin{equation*} f (x) -f(x') 
~\geq~  (x - x')  \, \Big[\limsup_{n\rightarrow \infty} \frac{\log \Sigma_n (x')}{\log n}\Big]~\geq~ x- x'.
\end{equation*}
Letting $x'$ tend to $x_1$ from above, we obtain
$
f(x)\ge x-x_1.
$
\end{proof}

We now present the proof of Theorem~\ref{thm:perco}. We keep the notation introduced in the previous section
$$\theta_n(p)=\bbP_p[0\longleftrightarrow \partial\Lambda_n]\quad\text{and}\quad S_n:=\sum_{k=0}^{n-1}\theta_k.$$
\begin{lemma}\label{cor:OSSS}
For any $n\ge1$, one has
$$\sum_{xy\in E_n} {\rm Cov}_p[\mathbbm 1_{0\leftrightarrow\partial\Lambda_n},\omega_e]\ge \frac{n}{\displaystyle 8S_n}\cdot \theta_n(1-\theta_n).$$
  \end{lemma}

The proof is based on  Theorem~\ref{thm:OSSS} applied to a well chosen decision tree determining $\mathbbm 1_{0\leftrightarrow\partial\Lambda_n}$. One may simply choose the trivial decision tree checking every edge of the box $\Lambda_n$. Unfortunately, the revealment of the decision tree being 1 for every edge, the OSSS inequality will not bring us much information. A slightly better decision tree would be provided by the decision tree discovering the cluster of the origin ``from inside''. Edges far from the origin would then be revealed by the decision tree if (and only if) one of their endpoints is connected to the origin. This provides a good bound for the revealment of edges far from the origin, but edges close to the origin are still revealed with large probability. In order to avoid this last fact, we will rather choose a family of decision trees discovering the clusters of $\partial\Lambda_k$ for $1\le k\le n$ and observe that the average of their revealment for a fixed edge will always be small.

\begin{proof} 
For any $k\in \llbracket 1,n\rrbracket$, we wish to construct a decision tree $T$ determining $\mathbbm 1 _{0\leftrightarrow\partial\Lambda_n}$ such that for each $e=uv$,
\begin{equation}\delta_e(T)\le \bbP_p[u\longleftrightarrow \partial\Lambda_k]+\bbP_p[v\longleftrightarrow \partial\Lambda_k].\label{eq:zzz}\end{equation}
Note that this would conclude the proof since we obtain the target inequality by applying Theorem~\ref{thm:OSSS} for each $k$ and then summing on $k$. As a key, we use that for $u\in \Lambda_n$,
\begin{align*}\sum_{k=1}^n\bbP_p[u\longleftrightarrow \partial\Lambda_k]
&~\le~ \sum_{k=1}^n  \bbP_p[u\longleftrightarrow \partial\Lambda_{|k-d(u,0)|}(u)]~\le~ 2S_n.
\end{align*}

We describe the decision tree $T$, which corresponds first to an exploration of the clusters in $\Lambda_n$ intersecting $\partial\Lambda_k$ that does not reveal any edge with both endpoints outside these clusters, and then to a simple exploration of the remaining edges.
 
More formally, 
we define $\e$ (instead of the collection of decision rules $\phi_t$) using two  growing sequences  $\partial\Lambda_k=V_0\subset V_1\subset \cdots\subset V$ and $\emptyset=F_0\subset F_1\subset\cdots\subset F$ (where $F$ is the set of edges between two vertices within distance $n$ of the origin) that should be understood as follows: at step $t$, $V_t$ represents the set of vertices that the decision tree found to be connected to $\partial \Lambda_k$, and $F_t$ is the set of explored edges discovered by the decision tree until time $t$.

Fix an ordering of the edges in $F$. Set $V_0=\partial\Lambda_k$ and $F_0=\emptyset$. Now, assume that $V_t\subset V$ and $F_t\subset F$ have been constructed and distinguish between two cases:\begin{itemize}[noitemsep,nolistsep]
\item If there exists an edge $e=xy\in F\setminus F_t$ with $x\in V_t$ and $y\notin V_t$ (if more than one exists, pick the smallest one for the ordering), then set $\e_{t+1}=e$, $F_{t+1}=F_t\cup\{e\}$ and set $$V_{t+1}:=\begin{cases}V_t\cup\{x\}&\text{ if }\omega_e=1\\ V_t&\text{ otherwise}.\end{cases}$$
\item If $e$ does not exist, set $\e_{t+1}$ to be the smallest $e\in F\setminus F_t$ (for the ordering) and set $V_{t+1}=V_t$ and $F_{t+1}=F_t\cup\{e\}$.\end{itemize}
As long as we are in the first case, we are still discovering the clusters of $\partial\Lambda_k$. Also, as soon as we are in the second case, we remain in it. The fact that $\tau$ is not greater than the last time we are in the first case gives us \eqref{eq:zzz}.

Note that $\tau$ may a priori be strictly smaller than the last time we are in first case (since the decision tree may discover a path of open edges from 0 to $\partial \Lambda_n$ or a family of closed edges disconnecting the origin from $\partial\Lambda_n$ before discovering the whole clusters of $\partial\Lambda_k$).
\end{proof}

We are now in a position to provide our alternative proof of exponential decay. Fix $n\ge1$. Lemma~\ref{cor:OSSS} together with the different formula gives
\begin{equation*}\theta_n'=\tfrac1{p(1-p)}\sum_{e\in E_n}{\rm Cov}(\mathbbm 1_{0\leftrightarrow \partial\Lambda_n},\omega_e) ~\ge~\frac{n}{2S_n}\cdot \theta_n(1-\theta_n).\end{equation*}
To conclude, fix $p_0\in(p_c,1)$ and observe that for $p\le p_0$, $1-\theta_n(p)\ge1-\theta_1(p_0)>0$. Then, apply Lemma~\ref{lem:technical} to $f_n=\tfrac2{(1-\theta_1(p_0))}\theta_n$.

Other models can be treated using the OSSS inequality (to mention only two, Voronoi percolation \cite{DumRaoTas17a} and Boolean percolation \cite{DumRaoTas17b}) but the study of the random-cluster model requires a generalization of the OSSS inequality, which we present below.
\bigbreak
Let us make a small detour, analyze what we did in the previous proof, and discuss the study of averages of boolean functions. We proved an inequality of the form \begin{equation}\label{eq:pou}\theta_n'\ge C_n\theta_n\end{equation} for a constant $C_n$ that was large as soon as $\theta_n$ was small. In particular, when $\theta_n$ was decaying polynomially fast, $C_n$ was polynomially large, a statement which allowed us to prove that $\theta_n$ was decaying stretched exponentially fast and then exponentially fast for smaller values of $p$ (see the proof of {\bf P1} of Lemma~\ref{lem:technical}).

Historically, differential inequalities like \eqref{eq:pou} were obtained using abstract sharp threshold theorems. The general theory of sharp thresholds for discrete product spaces was initiated by Kahn, Kalai and Linial in \cite{KahKalLin88} in the case of the uniform measure on $\{0,1\}^n$, i.e.~in the case of $\bbP_p$ with  $p=1/2$. There, Kahn, Kalai and Linial used the Bonami-Beckner inequality \cite{Bek75,Bon70} to deduce inequalities between the variance of a boolean function and so-called influences of this function. Bourgain, Kahn, Kalai,
Katznelson and Linial \cite{BouKahKal92} extended these inequalities to product spaces $[0,1]^n$ and to $\bbP_p$ with arbitrary $p\in[0,1]$. 
For completeness, let us state a version of this result due to Talagrand \cite{Tal94}:
there exists a constant $c>0$ such that for any $p\in[0,1]$ and any increasing event $A$,
$$\bbP_p[A](1-\bbP_p[A])\le c\log\tfrac1{p(1-p)}\sum_{e\in E} \frac{{\rm Cov}_p[\mathbbm 1_A,\omega_e]}{\log (1/{\rm Cov}_p[\mathbbm 1_A,\omega_e])}.$$
Notice that as soon as all covariances are small, the sum of covariances is large. This result can seem counter-intuitive at first but it is definitely very efficient to prove differential inequalities like \eqref{eq:pou}. In particular, ${\rm Cov}_p[\mathbbm 1_A,\omega_e]\le \bbP_p[A]$ so that applying the previous displayed equation to $A=\{0\leftrightarrow\partial\Lambda_n\}$ gives
$$\theta_n(1-\theta_n)\le \frac {c_p}{\log(1/\theta_n)}\theta_n'.$$
In order to compare this inequality to what we got with the OSSS inequality, let us look at the case where $\theta_n$ is decaying polynomially fast. In this case, the value of $C_n$ is of order  $\log n$. This is not a priori sufficient to prove that $\theta_n$ decays exponentially fast for smaller values of $p$ since it only improves the decay of $\theta_n$ by small polynomials. From this point of view, the logarithm in the expression $\log(1/\theta_n)$ is catastrophic.

Mathematicians succeeded to go around this difficulty by considering crossing events (see Section~\ref{sec:5} for more details). A beautiful example of the application of sharp threshold results to percolation theory is the result of Bollob\'as and Riordan about critical points of planar percolation models \cite{BolRio06c,BolRio06}. 

Recently, Graham and Grimmett \cite{GraGri06} succeeded to extend the BKKKL/Talagrand result to random-cluster models. Combined with ideas from \cite{BolRio06}, this led to a computation of the critical point of the random-cluster model (see below). Nonetheless, these proofs involving crossing probabilities are pretty specific to planar models and, to the best of our knowledge, fail to apply in higher dimension. In particular, it seems necessary to use a generalization of the OSSS inequality rather than a generalization of the BKKKL/Talagrand result, which is what we propose to do in the next section.
\bexo
\begin{exercise}[A $k$-dependent percolation model] Consider a family of iid Bernoulli random variables $(\eta_x)_{x\in \bbZ^d}$ of parameter $1-p$ and say that an edge $e\in\bbZ^d$ is open if both endpoints are at distance less than or equal to $R$ from any $x\in\bbZ^d$ with $\eta_x=1$ (it corresponds to taking the vacant set of balls of radius $R$ centered around the vertices $x\in\bbZ^d$ with $\eta_x=1$). Adapt the previous proof to show that the model undergoes a sharp phase transition, and that {\rm (EXP$_p$)} holds for any $p<p_c$.
\end{exercise}

\eexo

\subsection{Sharpness for random-cluster models}\label{sec:2.4}

We now turn to the proof of the following generalization of Theorem~\ref{thm:perco}.
\begin{theorem}[DC, Raoufi, Tassion \cite{DumRaoTas17}]\label{thm:RCM2} Consider the random-cluster model on $\bbZ^d$ with $q\ge1$. 
\begin{enumerate}[noitemsep,nolistsep]
\item\label{item:2} There exists $c>0$ such that for $p>p_c$, $\phi^1_{p,q}[0\leftrightarrow \infty]\ge
 c (p-p_c)$.
\item\label{item:1}For $p<p_c$,  there exists $c_p>0$ such that for all $n\ge1$,
$$\phi^1_{\Lambda_n,p,q}[0\longleftrightarrow \partial\Lambda_n]\le \exp(-c_pn).$$

\end{enumerate}
\end{theorem}

The result extends to any infinite locally-finite quasi-transitive graph $\bbG$. 
The proof will be based on the following improvement of the OSSS inequality \eqref{eq:OSSS}. Below, $\mathrm{Var}_{G,p,q}$ and $\mathrm{Cov}_{G,p,q}$ are respectively the variance and the covariance for $\phi_{G,p,q}^1$.

\begin{theorem}
\label{thm:OSSSr}
Consider $q\ge1$, $p\in[0,1]$, and a finite graph $G$. Fix an increasing function $f:\{0,1\}^E\longrightarrow [0,1]$ and an algorithm $T$. We have
\begin{equation}
    \label{eq:OSSS}
 \mathrm{Var}_{G,p,q}(f)~\le~C_{G,p,q}\sum_{e\in E}  \delta_e(f,T) \, \mathrm{Cov}_{G,p,q} [f,\omega_e] ,
  \end{equation}
  where 
$
\delta_e(f,T):=\bbP_p\big[\exists t\le \tau(\omega)\::\:e_t=e\big]
$
is the revealment (of $f$) for the decision tree $T$, and $c_{G,p,q}$ is defined by
$$C_{G,p,q}:=\frac1{\inf_{e\in E}\mathrm{Var}_{G,p,q}(\omega_e)}.$$
\end{theorem}

Before proving this statement, let us remark that it implies the theorem in the same way as in Bernoulli percolation. 

\begin{proof}[Theorem~\ref{thm:RCM2}]Set $\theta_n(p)=\phi^1_{\Lambda_{2n},p,q}[0\leftrightarrow\partial\Lambda_n]$ and $S_n=\sum_{k=0}^{n-1}\theta_k$. Following the same reasoning as in Lemma~\ref{cor:OSSS}, we find
$$\sum_{e\in E_{2n}}{\rm Cov}_{\Lambda_{2n},p,q}(\mathbbm 1_{0\leftrightarrow \partial\Lambda_n},\omega_e) ~\ge~\mathrm{Var}_{\Lambda_{2n},p,q}(\omega_e)\frac{n\,\theta_n(1-\theta_n)}{\displaystyle4\max_{x\in \Lambda_n}\sum_{k=0}^{n-1}\phi^1_{\Lambda_{2n},p,q}[x\leftrightarrow \partial\Lambda_k(x)]},$$
where $\Lambda_k(x)$ is the box of size $k$ around $x$. Since $\Lambda_{2k}(x)\subset\Lambda_{2n}$ for any $x\in\Lambda_n$ and $2k\le n$, we deduce 
$$\sum_{k=0}^{n-1}\phi^1_{\Lambda_{2n},p,q}[x\leftrightarrow \partial\Lambda_k(x)]\le 2\sum_{k=0}^{(n-1)/2}\phi^1_{\Lambda_{2n},p,q}[x\leftrightarrow \partial\Lambda_k(x)]\stackrel{\eqref{eq:comparison}}\le 2\sum_{k=0}^{(n-1)/2}\theta_k(p)\le 2S_n(p).$$Overall, we find
$$\sum_{e\in E_{2n}}{\rm Cov}_{\Lambda_{2n},p,q}(\mathbbm 1_{0\leftrightarrow \partial\Lambda_n},\omega_e) ~\ge~\mathrm{Var}_{\Lambda_{2n},p,q}(\omega_e)\frac{n}{8S_n}\cdot \theta_n(1-\theta_n).$$
Now, \eqref{eq:derivative formula} trivially extends to random-cluster models with $q>0$ so that 
$$\frac{\rm d}{{\rm d}p}\phi_{\Lambda_{2n},p,q}^1[0\leftrightarrow\partial\Lambda_n]=\tfrac1{p(1-p)}\sum_{e\in E_{2n}}{\rm Cov}_{\Lambda_{2n},p,q}(\mathbbm 1_{0\leftrightarrow \partial\Lambda_n},\omega_e).$$
We deduce that for $p\in[p_0,p_1]$, 
$$\theta_n'\ge c\frac{n}{S_n}\theta_n,$$
where 
$$c:=\tfrac{1}{2}\phi_{\Lambda_{2n},p_0,q}^1(\omega_e)(1-\phi_{\Lambda_{2n},p_1,p}^1[\omega_e])(1-\theta_1(p_1))>0.$$

To conclude, observe that measurability and the comparison between boundary conditions imply that 
$$\liminf\theta_n\ge \liminf \phi^1_{p,q}[0\longleftrightarrow\partial\Lambda_n]=\phi^1_{p,q}[0\longleftrightarrow\infty]$$
and that for any $k\ge1$,
$$\limsup\theta_n\le \limsup \phi^1_{\Lambda_{2n},p,q}[0\longleftrightarrow\partial\Lambda_k]=\phi^1_{p,q}[0\longleftrightarrow\partial\Lambda_k].$$
Letting $k$ tend to infinity implies that $\theta_n$ tends to $\phi^1_{p,q}[0\leftrightarrow\infty]$. We are therefore in position to apply Lemma~\ref{lem:technical}, which implies the first item of Theorem~\ref{thm:RCM2} and the fact that for $p<p_c$, there exists $c_p>0$ such that for any $n\ge0$,
$\theta_n(p)\le \exp(-c_pn).$ It remains to observe that 
$$\phi_{\Lambda_{2n},p,q}^1[0\longleftrightarrow\partial\Lambda_{2n}]\le \phi_{\Lambda_{2n},p,q}^1[0\longleftrightarrow\partial\Lambda_{n}]=\theta_n(p)$$
to obtain the second item of the theorem\footnote{Formally, we only obtained the result for $n$ even, but the result for $n$ odd can be obtained similarly.}.
\end{proof}

We now turn to the proof of Theorem~\ref{thm:OSSSr}. The strategy is a combination of the original proof of the OSSS inequality for product measures (which is a Efron-Stein type reasoning), together with an encoding of random-cluster measures in terms of iid random variables.

  We start by a useful lemma explaining how to construct $\omega$ with a certain law $\mu$ on $\{0,1\}^E$ from iid uniform random variables. Recall the notation $\vec E$ and $e_{[t]}$.
For $u\in[0,1]^n$ and $e\in \vec E$, define $F_e^\mu(u)=x$ inductively for $1\le t\le n$ by
   \begin{equation}\label{eq:ccc}x_{e_t}:=\begin{cases} 1&\text{ if }u_{t}\ge\mu[\omega_{e_{t}}=0\,|\,\omega_{e_{[t-1]}}=x_{e_{[t-1]}}],\\
   0&\text{ otherwise}.\end{cases}\end{equation}

\begin{lemma}\label{lem:01}
Let $\U$ be a iid sequence of uniform $[0,1]$ random variables, and $\e$ a random variable taking values in $\vec E$. Assume that for every $1\le t\le n$, $\U_t$ is independent of $(\e_1,\ldots, \e_t)$, then $\X=F_{\e}^\mu(\U)$ has law $\mu$.
\end{lemma}
  
\begin{proof}
Let  $x\in \{0,1\}^E$ and $e\in \vec E$ such that $\mathbb P[\X=x,\e=e]>0$. The probability $\mathbb P[\X=x,\e=e]$ can be written as
$$\prod_{t=1}^n\mathbb P[\X_{e_t}=x_{e_t}\,|\,\e_{[t]}=e_{[t]},\X_{e_{[t-1]}}=x_{e_{[t-1]}}]\times \prod_{t=1}^n\mathbb P[\e_t=e_t\,|\,\e_{[t-1]}=e_{[t-1]},\X_{e_{[t-1]}}=x_{e_{[t-1]}}].$$
(All the conditionings are well defined, since we assumed $\mathbb P[\X=x,\e=e]>0$.)
Since $\U_t$ is independent of $\e_{[t]}$ and $\U_{[t-1]}$ (and thus $\X_{e_{[t-1]}}$), the definition \eqref{eq:ccc} gives 
$$\mathbb P[\X_{e_t}=x_{e_t}\,|\,\e_{[t]}=e_{[t]},\X_{e_{[t-1]}}=x_{e_{[t-1]}}]=\mu[\omega_{e_t}=x_{e_t}\,|\,\omega_{e_{[t-1]}}=x_{e_{[t-1]}}]$$
so that the first product is equal to $\mu[\omega=x]$ independently of $e$. Fixing $x\in \{0,1\}^E$, and summing on $e\in \vec E$ satisfying  $\mathbb P[\X=x,\e=e]>0$ gives
\begin{align*}
\mathbb P[\X=x]&=\sum_{e}\mathbb P[\X=x,\e=e]\\
&=\mu[\omega=x]\sum_{e}\prod_{t=1}^n\mathbb P[\e_t=e_t|\e_{[t-1]}=e_{[t-1]},\X_{e_{[t-1]}}=x_{e_{[t-1]}}]=\mu[\omega=x].
\end{align*}

\end{proof}

\begin{proof}[Theorem~\ref{thm:OSSS}]
Consider two independent sequences of iid uniform $[0,1]$ random variables $\U$ and $\V$. Write $\mathbb P$ for the coupling between these variables (and $\mathbb E$ for its expectation). 
Construct $(\e,\X,\tau)$ inductively as follows: set $\e_1=e_1$, and for $t\ge1$,
  
\begin{align*}
 \X_{\e_t}&=\begin{cases} 1&\text{ if }\U_{t}\ge\phi_{G,p,q}^1[\omega_{\e_{t}}=0\,|\,\omega_{\e_{[t-1]}}=\X_{\e_{[t-1]}}]\\
 0&\text{ otherwise}\end{cases}\quad\text{ and }\quad \e_{t+1}:=
   \psi_{t+1}(\e_{[t]},\X_{\e_{[t]}}),  \end{align*}
and $\tau:=\min\big\{t\ge1:\forall x\in\{0,1\}^E,x_{\e_{[t]}}=\X_{\e_{[t]}}\Rightarrow f(x)=f(\X)\big\}$. Finally, for $0\le t\le n$, define $\Y^t:=F_\e(\mathbf W^t),$ where
$$\mathbf W^t:= \mathbf W^t (\U,\V) =(\V_1,\dots,\V_t,\U_{t+1},\dots,\U_\tau,\V_{\tau+1},\dots,\V_n)$$
(in particular $\mathbf{W} ^t$ is equal to $\V$ if $t\ge \tau$). 

Lemma~\ref{lem:01} applied to $(\U,\e)$ gives that $\X$ has law $\mu$ and is $\U$-measurable. Lemma~\ref{lem:01} applied to $(\V,\e)$  implies that $\Y^{n}$ has law $\mu$ and is independent of $\U$. Therefore,
\begin{equation*}\label{eq:a}
\phi^1_{G,p,q}[|f-\phi^1_{G,p,q}(f)|]\le \bbE\big[\big|\,\bbE[f(\X)|\U]-\bbE[f(\Y^{n})|\U]\,\big|\big]\le\mathbb E\big[|f(\X)-f(\Y^{n})|\big].
\end{equation*}
%
%
Exactly as for iid random variables, $f(\X)=f(\Y^0)$. Following the same lines as in the iid case, we obtain (recall that $f$ takes values in $[0,1]$)
\begin{align*}  {\rm Var}_{G,p,q}(f)&\le\sum_{e\in E} \sum_{t=1}^{n}\mathbb E\Big[\mathbb E\big[|f(\Y^t)-f(\Y^{t-1})| ~\big|~\U_{[t-1]}\big]\,\mathbbm1_{t\le \tau,\e_t=e}\Big]
\end{align*}
so that
the proof of the theorem follows from the fact that on $\{t\le \tau,\e_t=e\}$,
\begin{equation}\mathbb E\big[\,|f(\Y^t)-f(\Y^{t-1})| ~\big|~\U_{[t-1]}\big]~\le~ \tfrac1{\mathrm{Var}_{G,p,q} (\omega_e)}\,\mathrm{Cov}_{G,p,q} (f, \omega_{e} ).\label{eq:g}\end{equation}
Note that $F_e^\mu(u)$ is both increasing in $u$ and in $\mu$ (for stochastic domination).
We deduce that both $\Y^{t-1}$ and $\Y^t$ are sandwiched between 
$$\mathbf Z:=F_{\e}^{\phi^1_{G,p,q}[\cdot|\omega_e=0]}(\mathbf W^{t-1})=F_{\e}^{\phi^1_{G,p,q}[\cdot|\omega_e=0]}(\mathbf W^{t})$$ and 
$$\mathbf Z':=F_{\e}^{\phi^1_{G,p,q}[\cdot|\omega_e=1]}(\mathbf W^{t-1})=F_{\e}^{\phi^1_{G,p,q}[\cdot|\omega_e=1]}(\mathbf W^{t}).$$
Since $\mathbf W^t$ is independent of $\U_{[t-1]}$, Lemma~\ref{lem:01} and the fact that $f$ is increasing give us 
\begin{align*}
\mathbb E\big[\,|f(\Y^t)-f(\Y^{t-1})| ~\big|~\U_{[t-1]}\big]&\le \mathbb E[f(\mathbf Z')]-\mathbb E[f(\mathbf Z)]\\
&=\phi_{G,p,q}^1[f(\omega)|\omega_e=1]-\phi_{G,p,q}^1[f(\omega)|\omega_e=0]\\
&= \frac{\mathrm{Cov}_{G,p,q} (f, \omega_{e} )}{\phi_{G,p,q}^1[\omega_e](1-\phi_{G,p,q}^1[\omega_e])} .
\end{align*}

\end{proof}

\subsection{Computation of the critical point for random-cluster models on $\bbZ^2$}\label{sec:2.3}

The goal of this section is to explain how one can compute the critical point of the random-cluster model on $\bbZ^2$ using Theorem~\ref{thm:RCM2}. As mentioned in the end of Section~\ref{sec:2.2}, the following theorem was first proved using sharp threshold theorem, and we refer to \cite{BefDum12,DumMan16,DumRaoTas16} for alternative proofs.

\begin{theorem}[Beffara, DC \cite{BefDum12}]\label{thm:p_c FK}
For the random-cluster model on $\bbZ^2$ with cluster-weight $q\!\ge\!1$,
$$p_c=\frac{\sqrt q}{1+\sqrt q}.$$ Also, for $p< p_c$, there exists $c_p>0$ such that
$\phi_{p,q}^1[0\leftrightarrow \partial\Lambda_n]\le \exp(-c_p n)$ for all $n\ge0$.
\end{theorem}

This theorem has the following corollary. 
\begin{corollary}[Beffara, DC \cite{BefDum12}]\label{thm:p_c FK}
The critical inverse-temperature of the Potts model on $\bbZ^2$ satisfies
$$\beta_c(q):=\tfrac{q-1}q\log(1+\sqrt q).$$
\end{corollary}

We start by discussing duality for random-cluster models. The boundary conditions on a finite subgraph $G=(V,E)$ of $\bbZ^2$ are called {\em planar} if they are induced by some configuration $\xi\in\{0,1\}^{\bbE\setminus E}$. For any planar boundary conditions $\xi$, one can associate a dual boundary conditions $\xi^*$ on $G^*$ induced by the configuration $\xi^*_{e^*}=1-\xi_e$ for any $e\notin E$.

As an example, the free boundary conditions correspond to $\xi_e=0$ for all $e\in\bbE\setminus E$. Similarly, when $G$ is connected and has connected complement, the wired boundary conditions correspond to $\xi_e=1$ for all $e\in \bbE\setminus E$. (This explains the notation 0 and 1 for the free and wired boundary conditions.) In this case, the dual of wired boundary conditions is the free ones, and vice-versa. 

A typical example of non-planar boundary conditions is given by ``periodic'' boundary conditions on $\Lambda_n$, for which $(k,n)$ and $(k,-n)$ (resp.~$(n,k)$ and $(-n,k)$) are paired together for every $k\in\llbracket-n,n\rrbracket$. Another (slightly less interesting) example is given by the wired boundary conditions when $G$ has non-connected complement in $\bbZ^2$.

\begin{proposition}[Duality]
Consider a finite graph $G$ and planar boundary conditions $\xi$. If $\omega$ has law $\phi^\xi_{G,p,q}$, then $\omega^*$ has law $\phi^{\xi^*}_{G^*,p^*,q}$, where $p^*$ is the solution of 
$$\frac{pp^*}{(1-p)(1-p^*)}=q.$$
\end{proposition}
There is a specific value of $p$ for which $p=p^*$. This value will be denoted $p_{\rm sd}$, and satisfies
$$p_{\rm sd}(q)=\frac{\sqrt q}{1+\sqrt q}.$$
\begin{proof} Let us start with $G$ connected with connected complement, and free boundary conditions. 
Let $v$, $e$, $f$ and $c$ be the number of vertices, edges, faces and clusters of the graph $(\omega^*)^1$ embedded in the plane\footnote{Recall that $(\omega^*)^1$ is the graph $\omega^*$ where all vertices of $\partial G^*$ are identified together. This graph can clearly be embedded in the plane by ``moving'' the vertices of $\partial G^*$ to a single point chosen in the exterior face of $\omega$, and drawing the edges incident to $\partial G^*$ by ``extending'' the corresponding edges of $\omega^*$ by continuous curves not intersecting each others or edges of $\omega$, and going to this chosen point. }.
We wish to interpret Euler's formula in terms of $k(\omega)$, $o(\omega^*)$ and $k((\omega^*)^1)$. First, $v$ is a constant not depending on $\omega$ and $e$ is equal to $o(\omega^*)$. Also, the bounded faces of the graph are in direct correspondence with the clusters of $\omega$, and therefore $f=k(\omega)+1$ (note that there is exactly one unbounded face). Overall, Euler's formula ($f=c+e-v+1$) gives
$$k(\omega)=k((\omega^*)^1)+o(\omega^*)-v.$$
Set $Z:=Z^0_{G,p,q}$ and recall that $\frac{q(1-p)}p=\frac{p^*}{1-p^*}$. Since $c(\omega)=o(\omega^*)$, we get
\begin{align*}\label{eq:0999}
\phi^0_{G,p,q}[\omega]&=\tfrac{p^{|E|}}Z\big(\tfrac {1-p}p\big)^{c(\omega)}q^{k(\omega)}\\
&=\tfrac{p^{|E|}q^{-v}}Z\big(\tfrac {1-p}p\big)^{o(\omega^*)}q^{k((\omega^*)^1)+o(\omega^*)}\\
&=\tfrac{p^{|E|}q^{-v}}Z\big(\tfrac {p^*}{1-p^*}\big)^{o(\omega^*)}q^{k((\omega^*)^1)}\\
&=\phi^1_{G^*,p^*,q}[\omega^*].
\end{align*}
(Note that we also proved that $Z^1_{G^*,p^*,q}=Z^0_{G,p,q}q^vp^{-|E|}(1-p^*)^{|E|}$.) 

For arbitrary planar boundary conditions, the proof follows from the domain Markov property. Indeed, pick $n$ large enough so that there exists $\psi\in \{0,1\}^{E_n\setminus E}$ inducing the boundary conditions $\xi$ (such an $n$ always exists), and introduce
$$\omega^\psi_e=\begin{cases} \omega_e&\text{ if }e\in E,\\ \psi_e&\text{ if }e\in E_n\setminus E,\end{cases}$$
Since the boundary conditions induced by $\psi^*$ on $E_n^*\setminus E^*$ with the boundary of $\Lambda_n$ wired are exactly $\xi^*$, we deduce that
$$\phi^\xi_{G,p,q}[\omega]\stackrel{\eqref{eq:domain Markov}}=c\,\phi_{\Lambda_n,p,q}^0[\omega^\psi]
=c\,\phi_{\Lambda_n^*,p^*,q}^1[(\omega^\psi)^*]
\stackrel{\eqref{eq:domain Markov}}=\phi^{\xi^*}_{G^*,p^*,q}[\omega^*],$$
where $c$ is a constant not depending on $\omega$. This concludes the proof.
\end{proof}

\bexo

\begin{exercise}[Duality for the random-cluster model on the torus]\label{exo:torus}
Let $\mathbb{T}_n=[0,n]^2$ and consider the boundary conditions where $(0,k)$ and $(n,k)$ are identified for any $0\le k\le n$, and $(j,0)$ and $(j,n)$ are identified for any $0\le j\le n$. We write the measure $\phi_{\bbT_n,p,q}$.
\medbreak\noindent
1. A configuration $\omega$ is said to have a {\em net} if $\omega^*$ does not contain any non-retractible loop. 
Let $s(\omega) $ be the number of nets in $\omega$ (it is equal to 0 or 1). Prove that     \[
    |V| + f(\omega) + 2s(\omega) = k(\omega) + o(\omega) + 1 \, , 
    \]
   where $f(\omega)$ is the number of faces in the configuration.
   \medbreak\noindent
2. Show that $\widetilde{\phi}_{\mathbb{T}_n,p,q}(\omega) = \widetilde{\phi}_{\mathbb{T}_n^{*},p^{*},q}(\omega^{*})$, where 
\[
\widetilde{\phi}_{\mathbb{T}_n,p,q}(\omega) = \sqrt{q}^{2s(\omega)} \cdot \frac{p^{o(\omega)}(1-p)^{c(\omega)} q^{k(\omega)}}{\tilde{Z}^{per}_{\mathbb{T}_n,p,q}} \, .
\]
3. Deduce that the probability of $\calH_n$ is exactly $1/2$ for the measure $\widetilde{\phi}_{\mathbb T_n,p_{\rm sd},q}$.\end{exercise}
\eexo
We are now in a position to prove Theorem~\ref{thm:p_c FK}.
\begin{proof}[Theorem~\ref{thm:p_c FK}]The previous duality relation enables us to generalize the duality argument for crossing events. Indeed, considering the limit (as $G\nearrow \bbZ^2$) of the duality relation between wired and free boundary conditions, we get that the dual measure of $\phi^1_{p,q}$ is $\phi^0_{p^*,q}$.
Recall that $\calH_n$ is the event that the rectangle of size $n+1$ times $n$ is crossed horizontally. For $q\ge1$, the self-duality at $p_{\rm sd}$ implies that 
\begin{equation*}\phi_{p_{\rm sd},q}^1[\calH_n]+\phi_{p_{\rm sd},q}^0[\calH_n]=1.\end{equation*}
The comparison between boundary conditions thus implies
\begin{equation}\label{eq:RSW square q}\phi_{p_{\rm sd},q}^1[\calH_n]\ge \tfrac12\ge \phi_{p_{\rm sd},q}^0[\calH_n].\end{equation}
Note that the $\phi_{p_{\rm sd},q}^1[\calH_n]$ is no longer equal to $1/2$. Indeed, the complement event is still a rotated version of $\calH_n$, but the law of $\omega^*$ is not the same as the one of $\omega$, since the boundary conditions are free instead of wired. 

We are ready to conclude. The fact that $\phi_{p_{\rm sd},q}^1[\calH_n]\ge 1/2$ implies that 
$\phi_{p_{\rm sd},q}^1[0\leftrightarrow \partial\Lambda_n]\ge 1/(2n)$ (exactly as for Bernoulli percolation). Since this quantity is not decaying exponentially fast, Theorem~\ref{thm:RCM2} gives that $p_c\le p_{\rm sd}$.

Also, if $\phi^0_{p,q}[0\leftrightarrow \infty]>0$, then $\displaystyle\lim\phi^0_{p,q}[\calH_n]=1.$ 
Indeed,
the measure is ergodic (Lemma~\ref{ergodicity}) and satisfies the almost sure uniqueness of the infinite cluster (Theorem~\ref{thm:uniqueness}). Since it also satisfies the FKG inequality, the proof of Proposition~\ref{prop:lower} works the same for random-cluster models with $q\ge1$. Together with \eqref{eq:RSW square q}, this implies that $\phi^0_{p_{sd},q}[0\leftrightarrow\infty]=0$ and therefore that $p_{\rm sd}\le p_c$.\end{proof}
\begin{remark}\label{cor:lower bound q}
Note that we just proved that $\phi^0_{p_c,q}[0\leftrightarrow\infty]=0$.
\end{remark}
\bexo
\begin{exercise}[Critical points of the triangular and hexagonal lattices] Define $p$ such that $p^3+1=3p$ and set $p_c$ for the critical parameter of the triangular lattice.
\medbreak\noindent
1. Consider a graph $G$ and add a vertex $x$ inside the triangle $u,v,w$. Modify the graph $F$ by removing edges  $uv$, $vw$ and $wu$, and adding $xu$, $xv$ and $xw$. The new graph is denoted $G'$. Show that the Bernoulli percolation of parameter $p$ on $G$ can be coupled to the Bernoulli percolation of parameter $p$ on $G'$ in such a way that connections between different vertices of $G$ are the same. 
\bigbreak
\begin{center}
\includegraphics[width=0.40\textwidth]{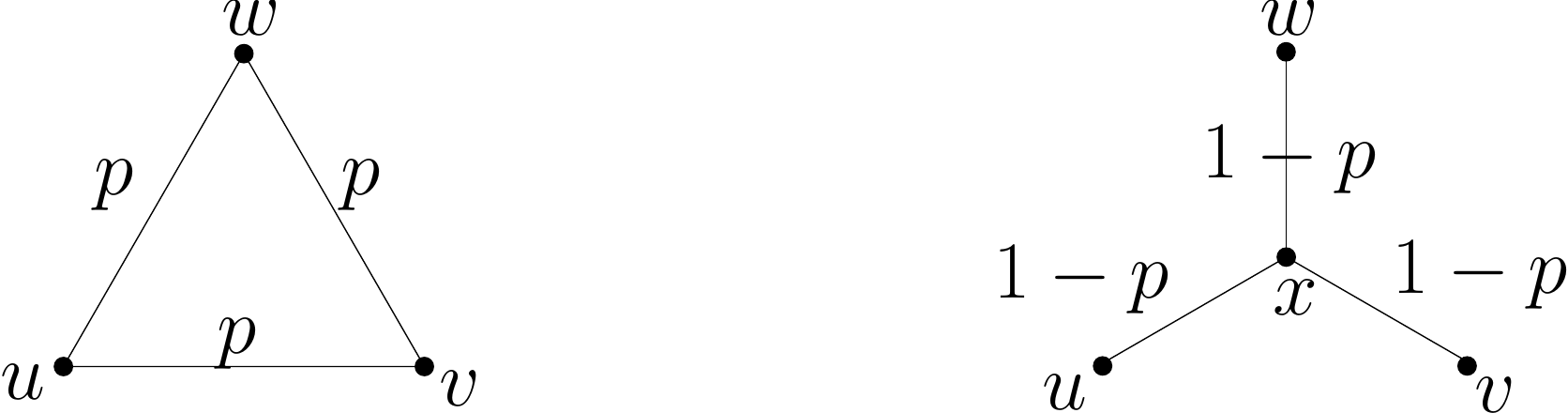}
\end{center}
\medbreak\noindent
2.  Using exponential decay in subcritical for the triangular lattice, show that if $p<p_c$, the percolation of parameter $1-p$ on the hexagonal lattice contains an infinite cluster almost surely. Using the transformation above, reach a contradiction.
\medbreak\noindent
3. Prove similarly that $p\le p_c(\mathbb T)$.
\medbreak\noindent
4. Find a degree three polynomial equation for the critical parameter of the hexagonal lattice.
\medbreak\noindent
5. What happens for the random-cluster model?
\end{exercise}
\eexo

\section{Where are we standing? and a nice conjecture...}

Up to now, we proved that the critical inverse-temperature of the Potts model exists, and that it corresponds to the point where long-range ordering emerges. We also proved monotonicity of correlations. In the specific case of $\bbZ^2$, we computed the critical point exactly. Last but not least, we proved that correlations decay exponentially fast when $\beta<\beta_c$. Overall, we gathered a pretty good understanding of the off-critical phase, but we have little information on the critical one. In particular, we would like to determine whether the phase transition of Potts models is continuous or not. In terms of random-cluster model, it corresponds to deciding whether $\phi^1_{p_c,q}[0\leftrightarrow\infty]$ is equal to 0 or not.

We proved in the previous section that for critical Bernoulli percolation on $\bbZ^2$, there was no infinite cluster almost surely. For $q>1$, we only managed to prove this result for the free boundary conditions. This is therefore not sufficient to discriminate between a continuous and a discontinuous phase transition for planar Potts models.

Before focusing on this question in the next sections, let us briefly mention that even for Bernoulli percolation, knowing whether there exists an infinite cluster at criticality is a very difficult question in general. For $\bbZ^d$ with $d\ge3$,  the absence of infinite cluster at criticality was proved using lace expansion for $d\ge 19$ \cite{HarSla90} (it was recently improved to $d\ge11$ \cite{FitHof15}). The technique involved in the proof is expected to work until $d\ge6$. For $d\in\{3,4,5\}$, the strategy will not work and the following conjecture remains one of the major open questions in our field.
\begin{conjecture}
For any $d\ge2$, $\bbP_{p_c}[0\longleftrightarrow \infty]=0$.
\end{conjecture}
Some partial results were obtained in $\bbZ^3$ in the past decades. For instance, it is known that the probability, at $p_c$ of an infinite cluster in $\bbN\times\bbZ^2$ is zero \cite{BarGriNew91}. Let us also mention that $\bbP_{p_c(\bbZ^2\times G)}[0\longleftrightarrow\infty]$ was proved to be equal to 0 on graphs of the form $\bbZ^2\times G$, where $G$ is finite; see \cite{DumSidTas14}, and on graphs with exponential growth in \cite{BenLyoPer99} and \cite{Hut16} (see also the following exercise).

\bexo
This exercise presents the beautiful proof due to Tom Hutchcroft of absence of percolation at criticality for amenable locally-finite transitive graphs with exponential growth. 
We say that $\bbG$ has exponential growth if there exists $c_{\rm vg}>0$ such that $|\Lambda_n|\ge \exp(c_{\rm vg} n)$.

\begin{exercise}[{$\mathbb P_{p_c}[0\leftrightarrow\infty]=0$} for amenable Cayley graphs with exponential growth] Let $\bbG$ be an amenable infinite locally-finite transitive graphs with exponential growth. 
\medbreak\noindent
1. Use amenability to prove that $\mathbb P_{p_c}[0\leftrightarrow\infty]>0\Longrightarrow\inf\{\bbP_{p_c}[x\leftrightarrow y],x,y\in \bbG\}>0$. {\em Hint:} use Exercise~\ref{exo:uniqueness}.
\medbreak\noindent
2. Use the FKG inequality to prove that $u_n(p)=\inf\{\bbP_{p_c}[x\leftrightarrow 0],x\in\partial\Lambda_n\}$ satisfies that for every $n$ and $m$,
$$u_{n+m}(p)\ge u_n(p)u_m(p).$$
3. Adapt Step 1 of the proof of Theorem~\ref{thm:perco} (see also Question 4 of Exercise~\ref{exo:long range}) to get that for any $p<p_c$,
$\displaystyle\sum_{x\in\bbG}\bbP_p[0\longleftrightarrow x]<\infty.$
\medbreak\noindent
4. Use the two previous questions to deduce that for any $p<p_c$,
$u_n(p)\le \exp(-c_{\rm vg}n)$ for every $n\ge1$.
\medbreak\noindent
5. Conclude.
\end{exercise}
\eexo

\section{Continuity of the phase transition for the Ising model}
Many aspects of the Ising model are simpler to treat than in other models (including Bernoulli percolation). We therefore focus on this model first. We will prove that the phase transition of the model is always continuous for the nearest neighbor model on $\bbZ^d$ with $d\ge2$. Before proceeding further, let us mention that the Ising model does not always undergo a continuous phase transition: the long-range model on $\bbZ$ with coupling constants $J_{x,y}=1/|x-y|^2$ undergoes a discontinuous phase transition (we refer to \cite{AizChaCha88} for details). 

The section is organized as follows. We start by providing a simple argument proving than in two dimensions, the phase transition is continuous. We then introduce a new object, called the random current representation, and study its basic properties. Finally, we use the properties of this model to prove that the phase transition is continuous in dimension $d\ge3$.

\subsection{An elementary argument in dimension $d=2$}

We present a very elegant argument, due to Wendelin Werner, of the following.
\begin{proposition}\label{uniqueness critical}
On $\bbZ^2$, $\mu_{\beta_c}^+[\sigma_0]=0$.\end{proposition}

\begin{proof}The crucial observation is the following: the measure $\mu^{\rm f}_{\beta_c}$ is mixing, and therefore ergodic. Indeed, recall that $\sigma\sim\mu^{\rm f}_{\beta_c}$ can be obtained from a percolation configuration $\omega\sim\phi^0_{p_c,2}$ by coloring independently the different clusters. The absence of infinite cluster for $\phi^0_{p_c,2}$ (Remark~\ref{cor:lower bound q}) enables us to deduce the mixing property of $\mu^{\rm f}_{\beta_c}$ from the one of $\phi^0_{p_c,2}$ (see Exercise~\ref{exo:mixing Ising}).

The Burton-Keane argument implies that when existing, the infinite cluster of minuses is unique. Consider the event $\widetilde \calH_n$ that there exists a path of minuses in $[0,n]^2$ crossing from left to right. The complement of this event contains the event that there exists a path of pluses in $[0,n]^2$ crossing from top to bottom. We deduce that $\mu^{\rm f}_{\beta_c}[\widetilde \calH_n]\le \tfrac12$ for every $n\ge1$. 
The proof of Proposition~\ref{prop:lower} works the same here and we deduce that the probability that there is an infinite cluster of minuses is zero, since otherwise the probability of $\widetilde\calH_n$ would tend to 1. 

We now prove that $\mu^+_{\beta_c}[\sigma_0]$ is smaller than or equal to 0, which immediately implies that it is equal to zero since we already know that it is larger than or equal to 0. Consider the set $\mathsf C$ of $x\in \Lambda_n$ which are not connected to $\partial\Lambda_n$ by a path of minuses. Conditionally on $\{\mathsf C=C\}$, the law of the configuration in $C$ is equal to $\mu_{C,\beta_c}^+$  since $\{\mathsf C=C\}$ is measurable in terms of spins outside $C$ or on $\partial C$, and that spins on $\partial C$ are all pluses (we use the Gibbs property for lattice models, which is obtained similarly to the domain Markov property for random-cluster models). Also note that
$$\mu_{C,\beta_c}^+[\sigma_0]=\phi^1_{C,p_c,2}[0\longleftrightarrow \partial C]\ge \phi^1_{p_c,2}[0\longleftrightarrow\infty]=\mu_{\beta_c}^+[\sigma_0].$$
Note that if $0\notin\mathsf C$, then $\sigma_0=-1$. We deduce that
\begin{align*}
0=\mu_{\beta_c}^{\rm f}[\sigma_0]&=\mu_{\beta_c}^{\rm f}[\sigma_0\mathbbm1_{0\notin \mathsf C}]+\sum_{0\in C\subset \Lambda_n}\mu_{C,\beta_c}^+[\sigma_0]\mu_{\beta_c}^{\rm f}[\mathsf C=C]\\
&\ge-\mu_{\beta_c}^{\rm f}[0\notin \mathsf C]+\mu_{\beta_c}^+[\sigma_0]\sum_{0\in C\subset \Lambda_n}\mu_{\beta_c}^{\rm f}[\mathsf C=C]\\
&=-\mu_{\beta_c}^{\rm f}[0\notin \mathsf C]+\mu_{\beta_c}^+[\sigma_0]\mu_{\beta_c}^{\rm f}[0\in \mathsf C].
\end{align*}
Letting $n$ tend to infinity and using that $\mu_{\beta_c}^{\rm f}[0\notin \mathsf C_n]$ tends to zero (since there is no infinite cluster of minuses) gives the result.
\end{proof}
\bexo
\begin{exercise}\label{exo:mixing Ising}
Prove the mixing property of $\mu_{\beta_c}^{\rm f}$.
\end{exercise}

\begin{exercise}\label{exo:FKG Ising}
Prove that the Ising model satisfies \eqref{eq:comparison} and \eqref{eq:FKG} for the natural order on  $\{\pm1\}^V$. \end{exercise}
\eexo
\subsection{High-temperature expansion, random current representation and percolation interpretation of truncated correlations}


For many reasons, the Ising model is special among Potts models. One of these reasons is the $+/-$ gauge symmetry: flipping all the spins leaves the measure invariant (for free boundary conditions). We will harvest this special feature in the following.


The \emph{high temperature expansion} of the Ising model is a graphical representation introduced by van der Waerden~\cite{Wae41}. It
relies on the following identity based on the fact that $\sigma_x\sigma_y\in\{-1,+1\}$:
\begin{eqnarray}\label{eq:crucial high temperature}
e^{\beta\sigma_x\sigma_y}~=~\cosh (\beta)+\sigma_x\sigma_y \sinh (\beta) =\cosh (\beta)\left[1+\tanh (\beta) \sigma_x\sigma_y\right].
\end{eqnarray}

For a finite graph $G$, the notation $\eta$ will always refer to a percolation configuration in $\{0,1\}^E$ (we will still use the notation $o(\eta)$ for the number of edges in $\eta$). We prefer the notation $\eta$ instead of $\omega$ to highlight the fact that $\eta$ will have {\em source constraints}, i.e.~that the parity of its degree at every vertex will be fixed. More precisely,  write $\partial\eta$ for the set of vertices of $\eta$ with odd degree. Note that $\partial\eta=\emptyset$ is equivalent to saying that $\eta$ is an even subgraph of $G$, i.e.~that the degree at each vertex is even. 

For $A\subset V$, set $$\sigma_A:=\prod_{x\in A}\sigma_x.$$
\begin{proposition}\label{high temperature}
Let $G$ be a finite graph, $\beta>0$, and $A\subset V$. We find 
\begin{align}\label{eq:1106}
\sum_{\sigma\in\{\pm1\}^V}\sigma_A\exp[-\beta H_G^{\rm f}(\sigma)]&=2^{|V|} \cosh (\beta)^{|E|}\sum_{\partial\eta=A}\tanh(\beta)^{o(\eta)}.\end{align}
\end{proposition}
\begin{proof}
Using \eqref{eq:crucial high temperature} for every $xy\in E$ gives
\begin{eqnarray*}
\sum_{\sigma\in\{\pm1\}^V}\sigma_A\exp[-\beta H_G^{\rm f}(\sigma)]&=&\sum_{\sigma\in\{\pm1\}^V}\sigma_A\prod_{xy\in E}e^{\beta\sigma_x\sigma_y}\\
&=&\cosh(\beta)^{|E|}\sum_{\sigma\in\{\pm1\}^V}\sigma_A\prod_{xy\in E}\left[1+\tanh (\beta) \sigma_x\sigma_y\right]\\
&=&\cosh(\beta)^{|E|}\sum_{\sigma\in\{\pm1\}^V}\sum_{\eta\in\{0,1\}^E} \tanh (\beta)^{o(\eta)}\sigma_A\prod_{\substack{xy\in E\\\eta_{xy}=1}}\sigma_x\sigma_y\\
&=&\cosh(\beta)^{|E|}\sum_{\eta\in\{0,1\}^E} \tanh (\beta)^{o(\eta)}\sum_{\sigma\in\{\pm1\}^V}\sigma_A\prod_{\substack{xy\in E\\\eta_{xy}=1}}\sigma_x\sigma_y.\end{eqnarray*}
Using the involution on $\{0,1\}^V$ sending $\sigma$ to the configuration coinciding with $\sigma$ except at $x$ where the spin is flipped, one sees that if any of the terms $\sigma_x$ appears with an odd power in the previous sum over $\sigma$, then the sum equals 0. Since the power corresponds to the degree of $x$ in $\eta$ if $x\notin A$, and is equal to the degree minus 1 if $x\in A$, we deduce that
$$\sum_{\sigma\in\{\pm1\}^V}\sigma_A\prod_{\substack{xy\in E\\\eta_{xy}=1}}\sigma_x\sigma_y=\begin{cases}2^{|V|}&\text{ if }\partial \eta=A,\\ 0&\text{ otherwise}\end{cases}$$ and the formula therefore follows.
\end{proof}
\bexo
\begin{exercise}[Kramers-Wannier duality]\label{exo:KW}
1. Show that there exists a correspondence between even subgraphs of $G$ and spin configurations for the Ising model on $G^*$, with $+$ boundary condition on the exterior face.
\medbreak\noindent
2. Express the partition function of the Ising model at inverse-temperature $\beta^*$ on $G^*$ with $+$ boundary conditions in terms of even subgraphs of $G$.
\medbreak\noindent
3. For which value of $\beta^*$ do we obtain the same expression (up to a multiplicative constant) as  \eqref{eq:1106}.
\end{exercise}
\eexo
The previous expansion of the partition function is called the {\em high-temperature expansion}. We deduce \begin{equation}\label{eq:org}
\mu_{G,\beta}^{\rm f}[\sigma_A]=\frac{\displaystyle\sum_{\partial\eta=A}\tanh(\beta)^{o(\eta)}}{\displaystyle\sum_{\partial\eta=\emptyset}\tanh(\beta)^{o(\eta)}}\ge0.
\end{equation}
(The inequality is called Griffiths' first inequality).
Notice two things about the high-temperature expansion of spin-spin correlations:
\begin{itemize}
\item the sums in the numerator and denominator of \eqref{eq:org} are running on different types of graphs (the source constraints are not the same), which illustrates a failure of this representation: we cannot a priori rewrite this quantity as a probability.
\item when squaring this expression, we end up considering, in the numerator and denominator, two sums over pairs of configurations $\eta_1$ and $\eta_2$ with $\partial \eta_1=\partial\eta_2$. This means that $\eta_1+\eta_2$ has even degree at each vertex, both in the numerator and denominator. \end{itemize}
In order to harvest this second observation, we introduce a system of currents. This introduction is only a small detour, since we will quickly get back to percolation configurations.

A {\em current} $\n$ on $G$ is a function from $E$ to $\bbN:=\{0,1,2,...\}$ (the notation $\n$ will be reserved to currents).    A {\em source} of $\n=(\n_{xy}:xy\in E)$ is a vertex $x$ for which $\sum_{y\sim x}{\n}_{xy}$ is odd. The set of sources of $\n$ is denoted by $\partial\n$. Also set
$$w_\beta(\n)=\prod_{xy\in E}\frac{\displaystyle\beta^{\n_{xy}}}{\n_{xy}!}.$$
One may follow the proof of Proposition~\ref{high temperature} with the Taylor expansion
$$\exp(\beta\sigma_x\sigma_y)=\sum_{\n_{xy}=0}^\infty \frac{(\beta \sigma_x\sigma_y)^{\n_{xy}}}{\n_{xy}!}$$ replacing \eqref{eq:crucial high temperature} to get
\begin{equation}\label{eq:8}
\sum_{\sigma\in\{\pm1\}^V}\sigma_A\exp[-\beta H_G^{\rm f}(\sigma)]=2^{|V|}\sum_{\partial\n=A}w_\beta(\n),
\end{equation}
(this expression is called the {\em random current expansion} of the partition function) from which we deduce an expression for correlations which is very close to \eqref{eq:org} 
\begin{equation}\label{eq:erg}
\mu_{G,\beta}^{\rm f}[\sigma_A]=\frac{\displaystyle\sum_{\partial\n=A}w_\beta(\n)}{\displaystyle\sum_{\partial\n=\emptyset}w_\beta(\n)}.
\end{equation}

The random current perspective on the Ising model's phase transition is driven by the hope that the onset of long range order coincides with a percolation transition in a system of duplicated currents (this point of  view
was used first in \cite{Aiz82,GriHurShe70}, see also \cite{Dum16} and references therein for a recent account). While we managed to rewrite the spin-spin correlations of the Ising model in terms of the random-cluster model or the high-temperature expansion, the representations fail to apply to truncated correlations\footnote{Truncated correlations is a vague term referring to differences of correlations functions (for instance
$\mu^+_\beta[\sigma_x\sigma_y]-\mu^+_\beta[\sigma_x]\mu^+_\beta[\sigma_y]$ or $\mu^+_\beta[\sigma_x\sigma_y]-\mu^{\rm f}_\beta[\sigma_x\sigma_y]$ or $U_4(x_1,x_2,x_3,x_4)$ defined later in this section).}. 
From this point of view, the expression \eqref{eq:erg} is slightly better than \eqref{eq:org} when considering the product of two spin-spin correlations since weighted sums over two ``independent'' currents $\n_1$ and $\n_2$ can be rewritten in terms of the sum over a single current $\m$ (see below). This seemingly tiny difference enables to switch the sources from one current to another one and to recover a probabilistic interpretation in terms of a percolation model.

More precisely, recall that $\calF_A$ is the event that every cluster of the percolation configuration is intersecting $A$ an even number of times\footnote{When $A=\{x,y\}$, the event $\calF_A$ is simply the event that $x$ and $y$ are connected to each others.}. We will prove below that for any $A\subset V$,
\begin{align}\label{eq:aag}\mu_{G,\beta}^{\rm f}[\sigma_A]^2&={\bf P}^{\emptyset}_{G,\beta}[\calF_A],
\end{align}
 where ${\bf P}^{\emptyset}_{G,\beta}$ is a percolation model defined as follows (we define a slightly more general percolation model which will be used later). For $B\subset V$, 
\begin{equation}\label{eq:probability}{\bf P}^B_{G,\beta}[\omega]=\frac{\displaystyle\sum_{\substack{\partial\n_1=B\\\partial\n_2=\emptyset}}w_{\beta}(\n_1)w_\beta(\n_2)\mathbbm1_{\widehat {\n_1+\n_2}=\omega}}{\displaystyle\sum_{\substack{\partial\n_1=B\\ \partial\n_2=\emptyset}}w_\beta(\n_1)w_\beta(\n_2)}\end{equation}
for any $\omega\in\{0,1\}^E$, where to each current $\n$, we associate a percolation configuration $\widehat\n$ on $E$ by setting $\widehat\n_{xy}=1$ if $\n_{xy}>0$, and 0 otherwise. 

This yields an alternative graphical representation for spin-spin correlations, which can be compared to the expression $\mu_{G,\beta}^{\rm f}[\sigma_A]=\phi_{G,p,q}^0[\calF_A]$ obtained using the random-cluster model. It involves the same increasing event $\calF_A$ (see Exercise~\ref{exo:Griffiths 1}), but for a different percolation model, and for the square of spin-spin correlations this time.

Let us now prove the following lemma, which leads immediately to \eqref{eq:aag}.
\begin{lemma}[Switching lemma \cite{GriHurShe70,Aiz82}]\label{lem:switching} For any $A,B\subset V$ and any $F:\bbN^E\rightarrow\bbR$, 
\begin{equation}\label{eq:switching}\sum_{\substack{
\partial\n_1=A\\
\partial\n_2=B}}F(\n_1+\n_2)w_\beta(\n_1)w_\beta(\n_2)=\sum_{\substack{
\partial\n_1=A\Delta B\\
\partial\n_2=\emptyset}}F(\n_1+\n_2)w_\beta(\n_1)w_\beta(\n_2)\mathbbm1_{\widehat {\n_1+\n_2}\in \calF_B},\tag{switch}\end{equation}
where $A\Delta B:=(A\setminus B)\cup(B\setminus A)$ is the symmetric difference between the sets $A$ and $B$.
\end{lemma}  
Before proving this lemma, let us mention a few implications. First,  \eqref{eq:aag} follows directly from this lemma since
\begin{align*}\mu_{G,\beta}^{\rm f}[\sigma_A]^2&=\frac{\displaystyle\sum_{\substack{
\partial\n_1=A\\
\partial\n_2=A}}w_\beta(\n_1)w_\beta(\n_2)}{\displaystyle\sum_{
\substack{
\partial\n_1=\emptyset\\
\partial\n_2=\emptyset}}w_\beta(\n_1)w_\beta(\n_2)}\stackrel{\eqref{eq:switching}}={\bf P}^{\emptyset}_{G,\beta}[\calF_A],
\end{align*}
We can go further and try to rewrite more complicated expressions. For instance,
\begin{align*}\label{eq:orga}\mu_{G,\beta}^{\rm f}[\sigma_A]\mu_{G,\beta}^{\rm f}[\sigma_B]&=\frac{\displaystyle\sum_{\substack{
\partial\n_1=A\\
\partial\n_2=B}}w_\beta(\n_1)w_\beta(\n_2)}{\displaystyle\sum_{
\substack{
\partial\n_1=\emptyset\\
\partial\n_2=\emptyset}}w_\beta(\n_1)w_\beta(\n_2)}\stackrel{\eqref{eq:switching}}=\mu_{G,\beta}^{\rm f}[\sigma_A\sigma_B]\cdot{\bf P}^{A\Delta B}_{G,\beta}[\calF_B].\end{align*}
In particular, the fact that the probability on the right is smaller or equal to 1 gives the second Griffiths inequality
\begin{equation}\label{eq:gri2}\mu_{G,\beta}^{\rm f}[\sigma_A\sigma_B]\ge \mu_{G,\beta}^{\rm f}[\sigma_A]\mu_{G,\beta}^{\rm f}[\sigma_B].\tag{G2}\end{equation} 
Note that here we have an explicit formula for the difference between the average of $\sigma_A\sigma_B$ and the product of the averages, which was not the case in the proof presented in Exercise~\ref{exo:Griffiths 2} (which was using the FKG inequality and the random-cluster model). This will be the main advantage of the previous representation: it will enable us to rewrite truncated correlations in terms of connectivity properties of this new percolation model.

Let us conclude this section by mentioning that we saw two representations of the Ising model partition function in this section: the high-temperature expansion in terms of even subgraphs, and the random current expansion. The coupling between the random-cluster model wth $q=2$ and the Ising model provides us with a third expansion (we leave it to the reader to write it properly). There exist several other representations: the low-temperature expansion, the representation in terms of dimers, the Kac-Ward expansion (see e.g.~\cite{CheCimKas15,CimDum13,Lis13,Lis16}).
 
\bexo

\begin{exercise}[Lebowitz's inequality]\label{exo:U4}
Set $\mu:=\mu_{G,\beta}^{\rm f}$ and $\sigma_i$ for the spin at a vertex $x_i$. Define$$U_4(x_1,x_2,x_3,x_4)=\mu[\sigma_1\sigma_2\sigma_3\sigma_4]-\mu[\sigma_1\sigma_2]\mu[\sigma_3\sigma_4]-\mu[\sigma_1\sigma_3]\mu[\sigma_2\sigma_4]-\mu[\sigma_1\sigma_4]\mu[\sigma_2\sigma_3]$$
for $x_1,x_2,x_3,x_4\in G$.
Using the switching lemma, show that
$$U_4(x_1,x_2,x_3,x_4)=-2\mu[\sigma_1\sigma_2\sigma_3\sigma_4]{\bf P}_{G,\beta}^{\{x_1,x_2,x_3,x_4\}}[x_1,x_2,x_3,x_4\mathrm{\ all\ connected}].$$
Note that in particular $U_4(x_1,x_2,x_3,x_4)\le 0$, which is known as Lebowitz's inequality. 
\end{exercise}
\eexo

\begin{proof}[the switching lemma]
We make the change of variables $\m=\n_1+\n_2$ and $\n=\n_1$. Since
$$w_\beta(\n)w_\beta(\m-\n)=\prod_{xy\in E} \frac{\beta^{(\m-\n)_{xy}}}{(\m-\n)_{xy}!}\frac{\beta^{\n_{xy}}}{\n_{xy}!}=w_\beta(\m) \binom{\m}{\n},$$
where $\binom{\m}{\n}:=\prod_{xy\in E}\binom{\m_{xy}}{\n_{xy}}$, we deduce that 
\begin{equation}\label{eq:swi}\sum_{\substack{
\partial\n_1=A\\
\partial\n_2=B}}F(\n_1+\n_2)w_\beta(\n_1)w_\beta(\n_2)=\sum_{\partial\m=A\Delta B}F(\m)w_{\beta}(\m)\sum_{\n\le\m,\partial \n=B}\binom{\m}{\n}.\end{equation}
Now, consider the multigraph $\calM$ obtained from $\m$ as follows: the vertex set is $V$ and $x$ and $y$ in $V$ are connected by $\m_{xy}$ edges. Then, $\binom{\m}{\n}$ can be interpreted as the number of subgraphs of $\calM$ with exactly $\n_{xy}$ edges between $x$ and $y$. As a consequence, 
$$\sum_{\n\le\m,\partial \n=B}\binom{\m}{\n}=|\{\calN\subset\calM:\partial\calN=B\}|,$$ where  $\partial\calN=B$ means that $\calN$ has odd degree on vertices of $B$, and even degree everywhere else. 

Note that this number is 0 is $\widehat\m\notin\calF_B$. Indeed, any subgraph $\calN$ with $\partial\calN=B$ contains disjoint paths pairing the vertices of $B$ together. In particular, any cluster of $\calM$ intersecting an element $x$ of $B$ must also intersect the element of $B$ paired to $x$ by $\calN$. 

On the other hand, if $\widehat\m\in\calF_B$, then any cluster of $\calM$ intersects an even number of vertices in $B$. We claim that in this case there exists $\calK\subset\calM$ with $\partial\calK=B$. The fact that $\widehat\m\in\calF_B$ clearly implies the existence of a collection of paths in $\calM$ pairing the vertices of $B$ together\footnote{Meaning that these paths start and end in $B$, and each element in $B$ appears exactly once in the set of beginning and ends of these paths.}. A priori, these paths  may self-intersect or intersect each others. We now prove that this is not the case if the collection has minimal total length among all the possible choices for such collections of paths. Assume for instance that there exist an edge $e=xy$ and two paths $\gamma=\gamma_1\circ xy\circ \gamma_2$ and $\gamma'=\gamma'_1\circ yx\circ \gamma'_2$, where we use the intuitive notation that $\gamma$ is the concatenation of a path $\gamma_1$ going to $x$, then using the edge $e$, and a path $\gamma_2$ from $y$ to the end, and similarly for $\gamma'$ (note that we may reverse $\gamma'$, so that we can assume it first goes through $y$ and then through $x$). But in this case the paths $\gamma_1\circ\gamma'_2$ and $\gamma'_1\circ\gamma_2$ also pair the same vertices, and have shorter length. The same argument shows that the paths must be self-avoiding. To conclude, simply set $\calK$ to be the graph with edge set composed of edges in the paths constructed above.

The map $\calN\mapsto\calN\Delta\calK$ is a bijection (in fact an involution) mapping subgraphs of $\calM$ with $\partial \calN=B$ to subgraphs of $\calM$ with $\partial\calN=\emptyset$. As a consequence, in this case
$$|\{\calN\subset\calM:\partial\calN=B\}|=|\{\calN\subset\calM:\partial\calN=\emptyset\}|.$$
Overall,
$$\sum_{\n\le\m,\partial \n=B}\binom{\m}{\n}=\mathbbm 1_{\widehat\m\in\calF_B}\sum_{\n\le\m,\partial \n=\emptyset}\binom{\m}{\n}.$$
Inserting this in \eqref{eq:swi} and making back the change of variables $\n_1=\n$ and $\n_2=\m-\n_1$ gives the result.\end{proof}
\bexo
\begin{exercise}\label{exo:coupling constants}How should currents be defined in order to rewrite correlations for the Ising model with Hamiltonian
$H_G^{\rm f}(\sigma)=-\sum_{x,y\in V} J_{x,y}\sigma_x\sigma_y,$
where $J_{x,y}$ are coupling constants? What is $w_\beta(\n)$ in this context? Is the switching lemma still true?\end{exercise}
\eexo

To conclude this section, note that the currents enter in the definition of ${\bf P}_{G,\beta}^B$ only through their sources and their traces (i.e.~whether they are positive or 0), so that we could have replaced currents taking values in $\bbN$ by objects taking values in $\{0,1,2\}$ with 0 if the current is 0, 1 if it is odd and 2 if it is positive and even. But also note that they do not rely only on the degree of the percolation configuration at every vertex (or equivalently on sources), so that the high-temperature expansion would not have been enough to define ${\bf P}_{G,\beta}^A$. The random current representation is crucial to express truncated correlations functions. This is the end of the detour and we will now try to use the percolation representation coming from random currents to prove continuity of correlations.

\bexo
\begin{exercise}In this exercise, we consider three measures on $G$:
\begin{itemize}
\item The first one, denoted ${\rm P}_{G,\beta}^\emptyset$, is attributing a weight to configurations $\eta\in\{0,1\}^E$ proportional to $\tanh(\beta)^{o(\eta)}\mathbbm 1_{\partial\eta=\emptyset}$ (this measure is sometimes known as the loop $O(1)$ model).
\item The second one, denoted by $\mathsf{P}_{G,\beta}^\emptyset$, is attributing a weight to configuration $n\in\{0,1\}^E$ proportional to $\sum_{\n\in\bbN^E}w_\beta(\n)\mathbbm 1_{\widehat\n=n}$.
\item The last one is given by $\phi_{G,\beta}^0$, where $\beta=-\tfrac12\log(1-p)$. 
\end{itemize}
1. Prove that $n\sim \mathsf{P}_{G,\beta}^\emptyset$ is obtained from $\eta\sim{\rm P}_{G,\beta}^\emptyset$ by opening independently additional edges with parameter $p_1=1-\tfrac1{\cosh\beta}$.
\medbreak\noindent
2. Consider a graph $\omega$. How many even subgraphs does it contain? Deduce from this formula that if one picks uniformly at random an even subgraph $\eta$ from $\omega\sim \phi_{G,\beta}^{\rm f}$, one obtains a random even subgraph of law ${\rm P}_{G,\beta}^\emptyset$.
\medbreak\noindent
3. Prove that by opening independently additional edges from $\eta\sim{\rm P}_{G,\beta}^\emptyset$ with probability 
$p_2=\tanh(\beta)$, one recovers $\omega\sim\phi_{G,\beta}^0$. 
\medbreak\noindent
4. What is the procedure to go from $n\sim \mathsf{P}_{G,\beta}^\emptyset$ to $\omega\sim\phi_{G,\beta}^0$?
\medbreak\noindent
5. Use Kramers-Wannier duality (Exercise~\ref{exo:KW}) to prove that ${\rm P}_{G,\beta}^\emptyset$ is the law of the interfaces between pluses and minuses in an Ising model with $+$ boundary conditions on $G^*$ with inverse-temperature $\beta^*$.
\end{exercise}
\eexo


\subsection{Continuity of the phase transition for Ising models on $\bbZ^d$ for $d\ge3$}\label{sec:continuity Ising}

In this section, we prove that the phase transition of the Ising model is continuous for any $d\ge3$. 
Let us start by saying that, like in the case of $\bbZ^2$, the critical Ising model with free boundary conditions on $\bbZ^d$ does not have long-range ordering. This can easily be seen from a classical result, called the {\em infrared bound}: for any $\beta<\beta_c$,
\begin{equation}\mu_\beta^{\rm f}[\sigma_x\sigma_y]\le \tfrac{C}{\beta} G(x,y),\tag{IR}\label{eq:infrared bound}\end{equation}
where $G(x,y)$ is the Green function of simple random walk, or equivalently the spin-spin correlations for the discrete GFF. The proof of this inequality is based on the so-called reflection-positivity (RP) technique introduced by Fr\"ohlich, Simon and Spencer \cite{FroSimSpe76}; see e.g.~\cite{Bis09} for a review. This technique has many applications in different fields of mathematical physics. We added the constant $C>0$ factor compared to the standard statement (where $C=1/2$) since the infrared bound is proved in Fourier space, and involves an averaging over $y$. One may then use the Messager-Miracle inequality (see Exercise~\ref{exo:messager-miracle} or the original reference \cite{MesMir77}) to get a bound for any fixed $x$ and $y$.  

By letting $\beta\nearrow\beta_c$, \eqref{eq:infrared bound} implies that \begin{equation}\mu_{\beta_c}^{\rm f}[\sigma_x\sigma_y]\le \tfrac C{\beta_c} G(x,y).\label{eq:infrared bound beta_c}\end{equation}
Here, it is important to understand what we did: we took the limit as $\beta\nearrow\beta_c$ of $\mu_{\beta_c}^{\rm f}[\sigma_x\sigma_y]$. This left-continuity is not true for $\mu_{\beta_c}^+[\sigma_x\sigma_y]$ since we use here the following exchange of two supremums:
\begin{align*}\mu_{\beta_c}^{\rm f}[\sigma_x\sigma_y]&=\sup_n\mu_{\Lambda_n,\beta_c}^{\rm f}[\sigma_x\sigma_y]\\
&=\sup_n\sup_{\beta<\beta_c}\mu_{\Lambda_n,\beta}^{\rm f}[\sigma_x\sigma_y]\\
&=\sup_{\beta<\beta_c}\sup_n\mu_{\Lambda_n,\beta}^{\rm f}[\sigma_x\sigma_y]=\sup_{\beta<\beta_c}\mu_{\beta}^{\rm f}[\sigma_x\sigma_y].
\end{align*}
For $\mu_\beta^+[\sigma_x\sigma_y]$,  one of the supremums would be an infimum and the previous argument would not hold. In fact, $\beta\mapsto \mu_{\beta}^+[\sigma_x\sigma_y]$ is right-continuous (since then it involves only infimums). Similarly, $p\mapsto\phi^0_{p,q}[A]$ and $p\mapsto\phi^1_{p,q}[A]$ are respectively left and right continuous for increasing events depending on finitely many edges (and therefore for any event depending on finitely many edges). Also, by taking an increasing sequence of increasing events $A_n$ with limit $A$, then $p\mapsto\phi^0_{p,q}[A]$ is left-continuous. Note that this is not true for any measurable event, an archetypical example being $A:=\{0\leftrightarrow\infty\}$, which is the limit of a decreasing sequence of increasing events.

Let us go back to the consequence of \eqref{eq:infrared bound beta_c}. Since the simple random walk is transient on $\bbZ^d$ for $d\ge3$, the right hand side tends to 0 as $\|x-y\|$ tends to infinity. This claim implies that $({\rm LRO}_{\beta_c})$ does not hold.
\bexo
\begin{exercise}\label{exo:messager-miracle}
We wish to prove that 
\begin{equation}
\mu_\beta^{\rm f}[\sigma_0\sigma_x]\le \mu_\beta^{\rm f}[\sigma_0\sigma_y]\tag{Mes-Mir}\label{eq:messager-miracle}
\end{equation}
if $x=(x_1,\dots,x_d)$ and $y=(y_1,\dots,y_d)$ satisfy either of the following two conditions
\begin{itemize}[nolistsep,noitemsep]
\item[{\bf C1}] $0\le x_1\le y_1$ and $x_i=y_i$ for every $i\ge 2$;
\item[{\bf C2}] $x_1+x_2=y_1+y_2$ and $0\le x_1-x_2\le y_1-y_2$ and $x_i=y_i$ for every $i\ge 3$.
\end{itemize}\medbreak\noindent
1. Define the graph $\bbL$ obtained from $\bbZ^d$ by adding another edge between $u$ and $u+(1,0,\dots,0)$, for each $u$ with first coordinate $u_1$ equal to $x_1$ (let $E$ be the set of new edges). How should we set coupling constants $J_{x,y}$ (in the sense of Exercise~\ref{exo:coupling constants}) on edges of $\bbE\cup E$ to have a model which is equivalent to the original model on $\bbZ^d$?
\medbreak\noindent
2. Define
$
E_1:=\{uv\subset \bbZ^d:u_1,v_1\le x_1\}\cup E$ and $E_2:=\bbE\setminus E_1$. Show that 
\begin{align*}\sum_{\partial\n=\{0,y\}}w(\n)&=\sum_{\substack{\n_1\in\bbN^{E_1},\n_2\in\bbN^{E_2}\\\partial(\n_1+\n_2)=\{0,y\}}}w(\n_1)w(\n_2)\quad\text{and}\quad\sum_{\partial\n=\{0,x\}}w(\n)&=\sum_{\substack{\n_1\in\bbN^{E_1},\n_2\in\bbN^{E_2}\\\partial(\n_1+\n_2)=\{0,x\}}}w(\n_1)w(\n_2).\end{align*}
2. Consider $y=x+(1,0,\dots,0)$. Using the current $\n'_2\in \bbN^{E_1}$ obtained from $\n_2$ by taking the reflection with respect to $\{z\in\bbR^d:z_1=x_1+1/2\}$ (with the convention that an edge of $E$ is sent to the corresponding edge of $\bbZ^d$ with the same endpoints), show that 
$$\sum_{\partial\n=\{0,y\}}w(\n)\le \sum_{\partial\n=\{0,x\}}w(\n).$$
{\em Hint.} Use the multi-graph $\calM$ obtained from $\n_1+\n_2'$ and observe that for terms involved in the left-hand side, $\calM$ necessarily contains a path from $x$ to $\{z\in\bbR^d:z_1=y_1\}$, and that in such case one may use a ``switching lemma''.
\medbreak\noindent 3. Prove \eqref{eq:messager-miracle} under condition {\em {\bf C1}}. Adapt the proof to show \eqref{eq:messager-miracle}  under condition {\em {\bf C2}}.\medbreak\noindent
4. Show that \eqref{eq:messager-miracle} implies that if $e_1=(1,0,\dots,0)$, then for every $x\in\partial\Lambda_n$,
\begin{equation}\label{eq:MM}\mu_\beta^{\rm f}[\sigma_0\sigma_{ne_1}]\ge \mu_\beta^{\rm f}[\sigma_0\sigma_x]\ge\mu_\beta^{\rm f}[\sigma_0\sigma_{dne_1}].\end{equation}

\end{exercise}
\eexo

It is unclear whether the absence of long-range ordering for $\mu_{\beta_c}^{\rm f}$ implies that $\mu_{\beta_c}^+[\sigma_0]=0$ since $\mu_{\beta_c}^{\rm f}$ and $\mu_{\beta_c}^+$ are a priori different. We will see that $\mu_{\beta_c}^{\rm b}$ is different from $\mu_{\beta_c}^{\rm f}$ for $q$-state Potts model with $q\ge4$ in two dimensions. Nonetheless, the following result tells us that this is never the case for the Ising model on $\bbZ^d$.

\begin{theorem}[Aizenman, DC, Sidoravicius \cite{AizDumSid15}]\label{thm:continuous}
For the Ising model at inverse temperature $\beta$ on $\bbZ^d$, $({\rm MAG}_\beta)$ implies $({\rm LRO}_\beta)$. As a consequence, the spontaneous magnetization is equal to $0$ at criticality.\end{theorem}

The proof is based on the percolation representation obtained using random currents. We will prove that the difference between spin-spin correlations with plus and free boundary conditions (which can be understood as truncated correlations) can be expressed in terms of a percolation model. We will then study the ergodic properties of this model to derive that $({\rm MAG}_\beta)$ implies $({\rm LRO}_\beta)$. The strategy is somewhat similar to Section~\ref{sec:uniqueness}.
\renewcommand{\g}{{\mathfrak g}}
\paragraph{Step 1: expressing truncated correlations using a percolation model based on random currents.} Let us start by expressing spin-spin correlations with $+$ boundary conditions in terms of currents. Let ${\rm int}(V):= V\setminus\partial G$ be the set of interior vertices. In order to do so, perform the same expansion (with Taylor series) as for the free boundary conditions to obtain
\begin{equation}\label{eq:88}
\sum_{\sigma\in\{\pm1\}^V}\sigma_A\exp[-\beta H_G^{\rm f}(\sigma)]\mathbbm1_{\sigma_{|\partial G}=+}=2^{|{\rm int}(V)|}\sum_{\partial\n\cap {\rm int}(V)= A}w_\beta(\n),
\end{equation}
where we use that 
$$\sum_{\substack{\sigma\in\{\pm1\}^V\\ \sigma_{|\partial G}=+}}\sigma_A\prod_{xy\in E}(\sigma_x\sigma_y)^{{\bf n}_{xy}}=\begin{cases}2^{|{\rm int}(V)|}&\text{ if }\partial\n\cap{\rm int}(V)=A\\ 0&\text{ otherwise}.\end{cases}$$

Introduce, as in the case of free boundary conditions, the measure 
$${\bf P}^{+}_{G,\beta}[\omega]=\frac{\displaystyle\sum_{\substack{\partial\n_1\cap {\rm int}(V)=\emptyset\\
\partial\n_2=\emptyset}}w_\beta(\n_1)w_\beta(\n_2)\mathbbm{1}_{\widehat{\n_1+\n_2}=\omega}}{\displaystyle\sum_{\substack{\partial\n_1\cap{\rm int}(V)=\emptyset\\
\partial\n_2=\emptyset}}w_\beta(\n_1)w_\beta(\n_2)}.$$
The key observation in the proof below is the following lemma.
\begin{lemma}For $G$ finite and $\beta>0$, we have that 
\begin{align}
\mu_{G,\beta}^+[\sigma_x\sigma_y]\mu_{G,\beta}^{\rm f}[\sigma_x\sigma_y]&={\bf P}^{+}_{G,\beta}[x\longleftrightarrow y]\qquad\qquad\quad&\forall x,y\in V\setminus\partial G,\label{eq:o1}\\
\mu_{G,\beta}^+[\sigma_x\sigma_y]-\mu_{G,\beta}^{\rm f}[\sigma_x\sigma_y]&\le\tfrac1{\sinh(\beta)}{\bf P}^{+}_{G,\beta}[\calA(xy)] \qquad\,&\forall xy\in E,\label{eq:o2}
\end{align}
where $\calA(xy)$ is the event that $\omega_{xy}=1$ and that in $\omega\setminus\{xy\}$, $x$ and $y$ are connected to $\partial G$ but not to each others (see Fig.~\ref{fig:eventA}).\end{lemma}
\begin{proof}
Equation~\ref{eq:o1} follows easily from \eqref{eq:88} and the switching lemma:
\begin{align*}\mu_{G,\beta}^+[\sigma_x\sigma_y]\mu_{G,\beta}^{\rm f}[\sigma_x\sigma_y]&=\frac{\displaystyle\sum_{A\subset \partial G}\ \sum_{\substack{\partial\n_1\cap{\rm int}(V)=A\Delta\{x,y\}\\
\partial\n_2=\{x,y\}}}w_\beta(\n_1)w_\beta(\n_2)}{\displaystyle\sum_{\substack{\partial\n_1\cap{\rm int}(V)=\emptyset\\
\partial\n_2=\emptyset}}w_\beta(\n_1)w_\beta(\n_2)}\\
&=\frac{\displaystyle\sum_{A\subset \partial G}\ \sum_{\substack{\partial\n_1\cap{\rm int}(V)=A\\
\partial\n_2=\emptyset}}w_\beta(\n_1)w_\beta(\n_2)\mathbbm{1}_{x\stackrel{\widehat{\n_1+\n_2}}{\longleftrightarrow}y}}{\displaystyle\sum_{\substack{\partial\n_1\cap{\rm int}(V)=\emptyset\\
\partial\n_2=\emptyset}}w_\beta(\n_1)w_\beta(\n_2)}={\bf P}_{G,\beta}^+[x\longrightarrow y].
\end{align*}
\begin{figure}\begin{center}
\includegraphics[width=0.50\textwidth]{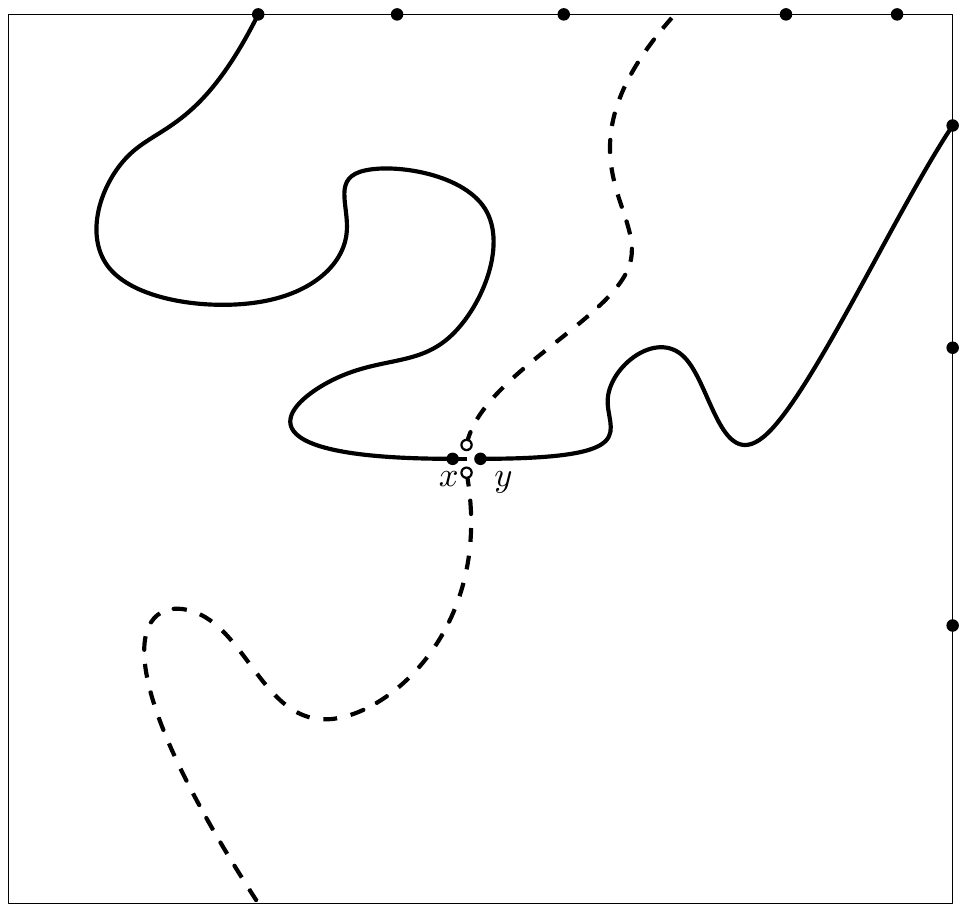}
\caption{\label{fig:eventA} Event $\calA(xy)$. The bullets in the boundary are representing the sources of $\n_1$. }\end{center}
\end{figure}
 Let us now turn to \eqref{eq:o2}, which is slightly more subtle. We may use the switching lemma (in a similar fashion to what we did above) to get
\begin{align}\mu_{G,\beta}^+[\sigma_x\sigma_y]-\mu_{G,\beta}^{\rm f}[\sigma_x\sigma_y]
&=\frac{\displaystyle\sum_{\partial\n_1\cap{\rm int}(V)=\{x,y\}}w_\beta(\n_1)}{\displaystyle\sum_{\partial\n_1\cap{\rm int}(V)=\emptyset}w_\beta(\n_1)}-\frac{\displaystyle\sum_{\partial \n_2=\{x,y\}}w_\beta(\n_2)}{\displaystyle\sum_{\partial \n_2=\emptyset}w_\beta(\n_2)} \notag \\
&=\frac{\displaystyle\sum_{\substack{\partial\n_1\cap{\rm int}(V)=\{x,y\}\\ \partial \n_2=\emptyset}}w_\beta(\n_1)w_\beta(\n_2)\mathbbm{1}_{(\n_1,\n_2)\in \calB(xy)}}{\displaystyle\sum_{\substack{\partial\n_1\cap{\rm int}(V)=\emptyset\\ \partial \n_2=\emptyset}}w_\beta(\n_1)w_\beta(\n_2)}\,,
\end{align}
where $\calB(xy)$ is the event that $x$ and $y$ are not connected in $\widehat{\n_1+\n_2}$. Note that on $\calB(xy)$, the source constraints in $\n_1$ are forcing $x$ and $y$ to be connected to sources of $\n_1$, i.e. to $\partial G$. Also note that $\n_1$ is equal to 0 on $xy$.

Consider $(\n_1,\n_2)\in \calB(xy)$ and define the set of currents $C(\n_1)$ coinciding with $\n_1$ except at $xy$, where they have an odd value. Note that currents in $C(\n_1)$ have no sources in ${\rm int}(V)$ (since $\n_1$ was equal to 0 at $xy$, and that now the current is odd) and that the trace of $\tilde\n_1+\n_2$ is in $\calA(xy)$ for any $\tilde\n_1\in C(\n_1)$. 
Finally, note that the sets $C(\n_1)$ are disjoint for different $\n_1$ and that the sum of the weights of currents in $C(\n_1)$ is equal to $\sinh(\beta)w_\beta(\n_1)$. 
Overall, changing $\n_1$ to the sum over $\tilde\n_1$ in $C(\n_1)$ in the previous sum gives the result.
 \end{proof}

%
\paragraph{Step 2: ergodic properties of the infinite volume version of ${\bf P}_{G,\beta}^+$.} We now prove that one may define an infinite-volume version of ${\bf P}_{G,\beta}^+$, denoted ${\bf P}_\beta^+$, which is invariant under translations and ergodic. We further prove that ${\bf P}_\beta^+$ contains at most one infinite cluster almost surely.

\begin{lemma}\label{def:current infinite}
The sequence of measures ${\bf P}_{\Lambda_n,\beta}^+$ converges to a measure ${\bf P}_{\beta}^+$ on $\{0,1\}^\bbE$ which is invariant under translations and ergodic.\end{lemma}

\begin{proof} Note that $\omega$ with law ${\bf P}_{G,\beta}^+$ is obtained as the union of two independent configurations obtained by taking the trace of only one current, either sourceless or with sources located on $\partial G$. Furthermore, conditioned on the parity of the currents (which we recall is simply the high-temperature expansion $\eta$), whether a current is positive or not is decided independently for each edge (if the current is odd, it must be positive, otherwise it is positive with probability $1-\cosh(\beta)^{-1}$). Therefore, convergence, invariance under translation and ergodicity of the limit of the laws of the parity of each current is implying the claim. Since the proof is the same in both cases, we focus on the current defined on $G$.

Let ${\rm P}^{\rm f}_{G,\beta}$ be the measure defined for  $\eta\in\{0,1\}^E$ by
$${\rm P}^{\rm f}_{G,\beta}[\eta]=\frac{\displaystyle\tanh(\beta)^{o(\eta)}}{\displaystyle\sum_{\partial\eta=\emptyset}\tanh(\beta)^{o(\eta)}}.$$
For a set of edges $F$, define the event $\calC_F$ that $\eta_e=0$ for each $e\in F$. We deduce from Lemma~\ref{high temperature} that
\begin{equation}  \label{eq:ProbE}
{\rm P}^{\rm f}_{G,\beta}[\calC_F] \ 
= \      \mu_{G,\beta}^{\rm f}\big[e^{  -\beta K_F   }\big] \cosh (\beta)^{|F|}\end{equation} 
with 
$ 
K_F(\sigma) := \sum_{xy\in F}    \sigma_x \sigma_y \, . $
The convergence of the above expression follows now directly from the convergence of Ising measures as $G\nearrow\bbZ^d$.  Since the events $\calC_F$ with $F$ finite generate the $\sigma$-algebra, we obtain the convergence of measures. The invariance under translations of $\mu_\beta^{\rm f}$ implies immediately the invariance under translations of the limiting measure ${\rm P}^{\rm f}_{\beta}$.

To prove ergodicity, we prove that the measure is mixing, which we only need to prove for events of the form $\calC_F$.
Fix two finite sets $F$ and $F'$ of edges. 
Using the  expression \eqref{eq:ProbE} for $x$ large enough so that  $F\cap (x+F') =\emptyset$, we find
\begin{equation}   \label{eq:ProbE2}
\frac {  {\rm P}^{\rm f}_{\beta}[\calC_{F\cup(x+F')}]  } 
       {{\rm P}^{\rm f}_\beta[\calC_F] {\rm P}^{\rm f}_\beta[\calC_F'] }   
=     
  \frac {\mu_\beta^{\rm f}\big[e^{   -\beta K_F} e^{   -\beta K_{x+F'}} \big] } {  \mu_\beta^{\rm f}\big[ e^{-\beta K_F } \big]    \mu_\beta^{\rm f}\big[ e^{-\beta K_{F'} } \big]  } .
\end{equation}
Ergodicity can therefore be presented as an implication of the  statement that this ratio tends to 1.   
To see this convergence, observe that $\mu_\beta^{\rm f}$ is itself mixing for functions of the spin-space that are even. Indeed, consider $A$ and $B$ two sets of even cardinality and $x$ large enough that $A\cap(x+B)=\emptyset$. The coupling with the random-cluster model gives
$$\frac{\mu_\beta^{\rm f}[\sigma_A\sigma_{x+B}]}{\mu_\beta^{\rm f}[\sigma_A]\mu_\beta^{\rm f}[\sigma_B]}=\frac{\phi_{p,2}^0[\calF_{A\cup(x+B)}]}{\phi_{p,2}^0[\calF_A]\phi_{p,2}^0[\calF_B]}\longrightarrow 1,
$$
with $p=1-e^{-2\beta}$, where the last convergence is due to the mixing property of $\phi_{p,2}^{\rm f}$ and the fact that 
$$\phi_{p,2}^0[\calF_{A\cup(x+B)}\setminus \calF_{A}\cap\calF_{x+B}]\longrightarrow 0.$$
(This last fact is due to the fact that this event is included in the event that there are two disjoint clusters going from $A$ to distance $\|x\|/2$, which by uniqueness of the infinite cluster, has a probability going to zero as $x$ goes to infinity.) Since the random-variables $\sigma_A$ for $|A|$ even generate the $\sigma$-algebra of even functions of the spin-space, the result follows.
\end{proof} 

The ergodicity of ${\bf P}_\beta^+$ implies that an infinite cluster exists with probability either 0 or 1. We now prove that when it exists, it is unique almost surely.
\begin{lemma}\label{thm:percolation0}
For any $\beta>0$, there exists at most one infinite cluster ${\bf P}_\beta$-almost surely.
\end{lemma}

\begin{proof} The proof of this theorem follows from the Burton-Keane argument presented in the proof of Theorem~\ref{thm:uniqueness}. The only difference lies in the fact that we do not have the finite energy property \eqref{eq:finite energy} anymore. Nevertheless, we have the following insertion tolerance claim, which we leave as an exercise (see Exercise~\ref{exo:insertion tolerance}): for any event $\calA$ depending on edges different from $xy$,
\begin{equation}\label{eq:insertion tolerance}\tag{IT}
{\bf P}_\beta^+[\omega_{xy}=1|\calA]\ge c_{\rm IT}.
\end{equation}
Note that the probability on the left may a priori be arbitrarily close to 1, but the previous lower bound is sufficient for our purpose. 

Recall the notation from the proof of Theorem~\ref{thm:uniqueness}. First, one may check that the proof that ${\bf P}_\beta^+[\calE_{<\infty}\setminus\calE_{\le 1}]=0$ is the same. To exclude the possibility of an infinite number of infinite clusters, one cannot really work with trifurcations anymore since constructing them would require the finite energy rather than the insertion tolerance. Nevertheless, one can work with a notion of coarse trifurcation, where $0$ is a {\em coarse trifurcation} if edges in $E_N$ are open, and $\omega\setminus E_N$ contains at least three infinite clusters intersecting $\Lambda_N$. Similarly, one defines the fact that $x$ is a coarse trifurcation. Note that if there is an infinite number of infinite clusters, then the construction of Theorem~\ref{thm:uniqueness} and \eqref{eq:insertion tolerance} imply that coarse trifurcations occur with positive probability. The end of the proof works the same, except that up to $|\Lambda_N|$ coarse trifurcations can intersect a fixed coarse trifurcation, so that the deterministic bound for the number of trifurcations is now $|\Lambda_N|\cdot|\partial\Lambda_n|$.\end{proof} 
 
\paragraph{Step 3: conclusion of the proof.} The proof now follows readily. We work by contraposition and assume that (LRO$_{\beta}$) does not hold. Then, taking the limit of \eqref{eq:o1} (as $G\nearrow\bbZ^d$) together with ergodicity and the uniqueness of the infinite cluster (when it exists) implies that there is no infinite cluster for ${\bf P}_\beta^+$ almost surely. Now, the limit as $G$ tends to $\bbZ^d$ of $\calA(xy)$ is included in the event that there exists an infinite cluster. Therefore, by taking the limit as $G$ tends to $\bbZ^d$ in \eqref{eq:o2}, we obtain that $\mu^{\rm f}_\beta[\sigma_x\sigma_y]=\mu^+_\beta[\sigma_x\sigma_y]$ for any $xy\in\bbE$. 

To conclude, note that 
this translates into the fact that 
$\phi^0_{p,2}[x\leftrightarrow y]=\phi^1_{p,2}[x\leftrightarrow y]$. Yet, a simple computation\footnote{Or more elegantly the use of Exercise~\ref{eq:reverse ES}, which states that $\phi^0_{p,2}[\omega_{xy}]=p\cdot\mu^{\rm f}_{\beta}[\sigma_x=\sigma_y]$, which combined with $\mu^{\rm f}_{\beta}[\sigma_x=\sigma_y]=2\mu^{\rm f}_{\beta}[\sigma_x\sigma_y]-1=2\phi^0_{p,2}[x\leftrightarrow y]-1$ gives the requested equality.} gives that 
$$\phi^0_{p,2}[\omega_{xy}]=\tfrac{p}2(1+\phi^0_{p,2}[x\longleftrightarrow y]).$$
and similarly for $\phi^1_{p,2}$ so that $\phi^0_{p,2}[\omega_{xy}]=\phi^1_{p,2}[\omega_{xy}]$. We already argued in the proof of Theorem~\ref{uniqueness Dq} that this implies $\phi^0_{p,2}=\phi^1_{p,2}$. In particular, for any $x,y\in\bbZ^d$, 
$$\mu^+_\beta[\sigma_x\sigma_y]=\phi^1_{p,2}[x\longleftrightarrow y]=\phi^0_{p,2}[x\longleftrightarrow y]=\mu^{\rm f}_\beta[\sigma_x\sigma_y]$$
tends to 0 as $\|x-y\|$ tends to infinity (by the infrared bound). By uniqueness of the infinite connected component, we deduce that $\phi^1_{p,2}[0\longleftrightarrow\infty]=0$, which is the claim.
\bexo

\begin{exercise}\label{exo:insertion tolerance}
Prove \eqref{eq:insertion tolerance} using an argument similar to the proof of \eqref{eq:o2}.
\end{exercise}

\begin{exercise}
Consider the Ising model with a magnetic field $h$, i.e.~the model with Hamiltonian
$$H_{G,h}(\sigma)=H_G(\sigma)-\sum_{x\in V} h\sigma_x.$$
We denote the infinite-volume measure by $\mu^{\rm f}_{\beta,h}$.
\medbreak\noindent
1. Interpret the correlations of the model in terms of random currents on the graph $G^h$ with vertex-set $V\cup\{\mathfrak g\}$ and edge-set given by $E\cup\{x\mathfrak g:x\in V\}$. What is the weight $w_{\beta,h}(\n)$ of a current?
\medbreak\noindent
2. Following a reasoning similar to the proof of \eqref{eq:o2}, show that $\mu_{\beta,h}^{\rm f}[\sigma_x\sigma_0]-\mu_{\beta,h}^{\rm f}[\sigma_x]\mu_{\beta,h}^{\rm f}[\sigma_0]$ can be reinterpreted as the probability under a system of duplicated currents that $x$ is connected to $0$ but not to $\mathfrak g$.
\medbreak\noindent
3. Using insertion tolerance, prove that this probability decays exponentially fast in $\|x\|$.
\end{exercise}
\eexo

\subsection{Polynomial decay at criticality for $d\ge3$}

It is natural to ask how fast the spin-spin correlations decay at criticality. We will prove the following

\begin{theorem}\label{thm:pol decay}
For $d\ge3$, there exists $c,C\in(0,\infty)$ such that for every $x\in\bbZ^d$,
\begin{equation}
\frac{c}{\|x\|^{d-1}}\le \mu_{\beta_c}^{\rm f}[\sigma_0\sigma_x]\le \frac{C}{\|x\|^{d-2}}.
\end{equation}
\end{theorem}

\begin{proof}
The upper bound is provided by the infrared bound \eqref{eq:infrared bound}. For the lower bound, we invoke Simon's inequality (see Exercise~\ref{exo:simon}), stating that for every $n\ge1$, 
$$\mu_{\beta}^{\rm f}[\sigma_0\sigma_x]\le \sum_{y\in \partial\Lambda_n}\mu_{\beta}^{\rm f}[\sigma_0\sigma_y]\,\mu_{\beta}^{\rm f}[\sigma_y\sigma_x].$$
Assume that 
$\varphi_\beta(\Lambda_n):= \sum_{y\in \partial\Lambda_n}\mu_{\beta}^{\rm f}[\sigma_0\sigma_y]<1$ for some $n$, and observe that by a reasoning similar to the first step of Section~\ref{sec:2.2}, we find that if $x\in \Lambda_{kn}$,
$$\mu_{\beta}^{\rm f}[\sigma_0\sigma_x]\le \varphi_\beta(\Lambda_n)^{k}.$$
Now, $\beta\mapsto\mu_{\beta}^{\rm +}[\sigma_0\sigma_x]$ is continuous from the right since it is the infimum of the continuous increasing functions $\beta\mapsto \mu_{G,\beta}^{\rm +}[\sigma_0\sigma_x]$. We deduce that 
$$\lim_{\beta\searrow\beta_c} \varphi_\beta(\Lambda_n)\le \lim_{\beta\searrow\beta_c} \sum_{y\in \partial\Lambda_n}\mu_{\beta}^+[\sigma_0\sigma_y]=\sum_{y\in \partial\Lambda_n}\mu_{\beta_c}^+[\sigma_0\sigma_y]=\varphi_{\beta_c}(\Lambda_n),$$
where in the last equality we used that $\mu_{\beta_c}^{\rm f}=\mu_{\beta_c}^{\rm +}$. 

Therefore, if $\varphi_{\beta_c}(\Lambda_n)<1$, then $\varphi_\beta(\Lambda_n)<1$ for some $\beta>\beta_c$. By the reasoning above, this would imply that correlations decay exponentially fast for $\beta>\beta_c$, which is absurd. In conclusion, 
$\varphi_{\beta_c}(\Lambda_n)\ge 1$ for every $n\ge1$.

The Messager-Miracle inequality \eqref{eq:messager-miracle} used twice (more precisely \eqref{eq:MM}) implies that for any $y\in\partial\Lambda_n$,
\begin{equation}\label{eq:papa}\mu_{\beta_c}^{\rm f}[\sigma_0\sigma_{ne_1}]\ge \mu_{\beta_c}^{\rm f}[\sigma_0\sigma_{y}]\ge \mu_{\beta_c}^{\rm f}[\sigma_0\sigma_{dne_1}]\end{equation}
where $e_1=(1,0,\dots,0)$. 
The left inequality together with $\varphi_{\beta_c}(\Lambda_n)\ge 1$ imply that 
$$\mu_{\beta_c}^{\rm f}[\sigma_0\sigma_{ne_1}]\ge \frac1{|\Lambda_n|}$$ 
for every $n$. The proof follows readily from the right inequality of \eqref{eq:papa}.
\end{proof}
\bexo
\begin{exercise}[Simon's inequality] \label{exo:simon}Using the switching lemma, prove Simon's inequality: for any set $S$ disconnecting $x$ from $y$ (in the sense that any path from $x$ to $y$ intersects $S$), 
\begin{equation}\label{eq:Simon}
\mu_{G,\beta}^{\rm f}[\sigma_x\sigma_z]\le \sum_{y\in S}\mu_{G,\beta}^{\rm f}[\sigma_x\sigma_y]\,\mu_{G,\beta}^{\rm f}[\sigma_y\sigma_z].\tag{Simon}
\end{equation}
\end{exercise}
A slightly stronger inequality, called Lieb's inequality, can also be obtained using random currents (the proof is more difficult). The improvement lies in the fact that $\mu_{G,\beta}^{\rm f}[\sigma_x\sigma_y]$ can be replaced by $\mu_{S,\beta}^{\rm f}[\sigma_x\sigma_y]$:
\begin{equation}\label{eq:Lieb}
\mu_{G,\beta}^{\rm f}[\sigma_x\sigma_z]\le \sum_{y\in S}\mu_{S,\beta}^{\rm f}[\sigma_x\sigma_y]\,\mu_{G,\beta}^{\rm f}[\sigma_y\sigma_z].\tag{Lieb}
\end{equation}
\eexo
In fact, one can prove much more in dimension $d\ge5$, and therefore the previous theorem is mostly interesting in three dimension. 
\begin{theorem}[Aizenman, Fernandez \cite{AizFer88}]\label{thm:pol decay exact}
For any $d\ge5$, there exist constants $c_1,c_2\in(0,\infty)$ such that for any $x\in\bbZ^d$,
$$\frac{c_1}{\|x\|^{d-2}}\le \mu_{\beta_c}^{\rm f}[\sigma_0\sigma_x]\le\frac{c_2}{\|x\|^{d-2}}.$$
\end{theorem}

\section{Continuity/Discontinuity of the phase transition for the planar random-cluster model}\label{sec:5}

We now turn to the case of the random-cluster model in two dimensions. We will discuss the following result.
\begin{theorem}\label{thm:continuous RCM}
Consider the random-cluster model with cluster-weight $q\ge1$ on $\bbZ^2$. Then $\phi^1_{p_c,q}[0\leftrightarrow\infty]=0$ if and only if $q\le 4$.
\end{theorem}
As an immediate corollary, we obtain the following result.
\begin{corollary}
The phase transition of the Potts model is continuous for $q\in\{2,3,4\}$ and discontinuous for $q\ge5$.
\end{corollary}
\medbreak
The section is organized as follows. We first study crossing probabilities for planar random-cluster models by building a Russo-Seymour-Welsh type theory for these models. This part enables us to discriminate between two types of behavior: 
\begin{itemize}
\item the continuous one in which crossing probabilities do not go to zero, even when boundary conditions are free (which correspond to the worse ones for increasing events). In this case, the infinite-volume measures with free and wired boundary conditions are equal and correlations decay polynomially fast.
\item the discontinuous one  in which crossing probabilities with free boundary conditions go to zero exponentially fast. In this case, the infinite-volume measure with free boundary conditions looks subcritical in the sense that the probability that 0 is connected to distance $n$ is decaying exponentially fast, while the infinite-volume measure with wired boundary conditions contains an infinite cluster almost surely. 
\end{itemize}
We then prove that for $q\le 4$, the probability of being connected to distance $n$ for the free boundary conditions goes to zero at most polynomially fast, thus proving that we are in the continuous case. In order to do that, we introduce parafermionic observables.
Finally, we discuss the $q>4$ case, in which we sketch the proof that the probability of being connected to distance $n$ decays exponentially fast, thus proving that we are in the discontinuous phase.

\subsection{Crossing probabilities in planar random-cluster models}

We saw that the probability of crossing squares was equal to 1/2 for Bernoulli percolation, and that it was either bounded from above or below by $1/2$ for random-cluster models depending on the boundary conditions. This raises the question of probabilities of crossing more complicated shapes, such as rectangle with aspect ratio $\rho\ne 1$. While this could look like a technical question, we will see that studying crossing probabilities is instrumental in the study of critical random cluster models.

We begin with some general notation. For a rectangle $R:=[a,b]\times[c,d]$ (when $a$, $b$, $c$ or $d$ are not integers, an implicit rounding operation is performed), introduce the event
$\calH(R)$ that $R$ is crossed horizontally, i.e.~that the left side $\{a\}\times[c,d]$ is connected by a path in $\omega\cap R$ to the right side $\{b\}\times[c,d]$. Similarly, define $\calV(R)$ be the event that $R$ is crossed vertically, i.e.~that the bottom side $[a,b]\times\{c\}$ is connected by a path in $\omega\cap R$ to the top side $[a,b]\times\{d\}$.  When $R=[0,n]\times[0,k]$, we rather write $\calV(n,k)$ and $\calH(n,k)$.

\bexo
\begin{exercise}
Consider Bernoulli percolation (of parameter $p$) on a  planar transitive locally finite infinite graph with $\pi/2$ symmetry.
\medbreak\noindent1. Using the rectangles  
  $R_1=[0,n]\times[0,2n]$, $R_2=[0,n]\times[n,3n]$,
  $R_3=[0,n]\times[2n,4n]$,
  $R_4=[0,2n]\times[n,2n]$ and $R_5=[0,2n]\times[2n,3n]$,
show that
\[
\mathbb{P}_p[\calH(n,4n)] \leq 5 \mathbb{P}[\calH(n,2n)] \, .
\]
\medbreak\noindent2.  Deduce that $u_{2n}\le 25u_n^2$ where $u_n=\mathbb{P}_p[\calH(n,2n)]$. Show that $(u_n)$ decays exponentially fast as soon as there exists $n$ such that $u_n<\tfrac1{25}$.
\medbreak\noindent3.  Deduce that $u_n\ge\tfrac1{25}$ for every $n$ or {\rm (EXP$_p$)}. What did we prove at $p_c$?
\end{exercise}

\eexo

\subsubsection{The RSW theory for infinite-volume measures}

Recall from \eqref{eq:RSW square q} that we know that 
$$\phi^1_{p_c,q}[\calH(n,n)]\ge \phi^1_{p_c,q}[\calH(n+1,n)]\ge\tfrac12.$$
It is natural to wish to improve this result by studying crossing probabilities for wired boundary conditions for rectangles of fixed aspect ratio remain bounded away from 0 when $n$ tends to infinity. This is the object of the following theorem.
\begin{theorem}[Beffara, DC \cite{BefDum12}]\label{thm:box}
Let $\rho>0$, there exists $c=c(\rho)>0$ such that for every $n\ge1$,
$$\phi^1_{p_c,q}[\calH(\rho n,n)]\ge c.$$
\end{theorem}

For Bernoulli percolation, a uniform upper bound follows easily from the uniform lower bound and duality since the complement of the event that a rectangle is crossed vertically is the event that the dual rectangle is crossed horizontally in the dual configuration. This is not the case for general random-cluster models since the dual measure is the measure with free boundary conditions. In fact, we will see in the next sections that a uniform upper bound is not necessarily true: crossing probabilities could go to 1 for wired boundary conditions, and to 0 for free ones. It was therefore crucial to state this theorem for ``favorable'' boundary conditions at infinity.

Also, as soon as a uniform lower bound (in $n$) for $\rho=2$ is proved, then one can easily combine crossings in different rectangles to obtain a uniform lower bound for any $\rho>1$. Indeed, define (for integers $i\ge0$) the rectangles $R_i:=[in,(i+2)n]\times[0,n]$ and the squares $S_i:=R_i\cap R_{i+1}$. Then, 
\begin{align*}
\phi_{p_c,q}^1[\calH(\rho n,n)]\ge  \phi_{p_c,q}^1\Big[\bigcap_{i\le \rho}(\calH(R_i)\cap\calV(S_i))\Big]\stackrel{\rm (FKG)}\ge c(2)^{2\lfloor \rho\rfloor}.
\end{align*}
One may even prove lower bounds for crossing probabilities in arbitrary topological rectangles (see Exercise~\ref{exo:top} below).
\bexo
\begin{exercise}\label{exo:top}
Consider a simply connected domain with a smooth boundary $\Omega$ with four distinct points $a$, $b$, $c$ and $d$ on the boundary. Let $(\Omega_\delta,a_\delta,b_\delta,c_\delta,d_\delta)$ be the finite graph with four marked points on the boundary defined as follows: $\Omega_\delta$ is equal to $\Omega\cap\delta\bbZ^2$ (we assume here that it is connected and of connected complement, so that the boundary is a simple path) and $a_\delta$, $b_\delta$, $c_\delta$, $d_\delta$ be the four points of $\partial\Omega_\delta$ closest to $a$, $b$, $c$ and $d$.

Prove that there exists $c=c(q,\Omega,a,b,c,d)>0$ such that for any $\delta>0$, 
$$\phi^1_{p_c,q}[(a_\delta b_\delta) \stackrel{\Omega_\delta}\longleftrightarrow(c_\delta d_\delta)]\ge c,$$
where $(a_\delta b_\delta)$ and $(c_\delta d_\delta)$ are the portions of $\partial\Omega_\delta$ from $a_\delta$ to  $b_\delta$, and from $c_\delta$ to $d_\delta$, when going counterclockwise around $\partial\Omega_\delta$.\end{exercise}

\eexo
On the other hand, it is a priori not completely clear how to obtain a lower bound for $\rho=2$ (or any $\rho>1$) from a lower bound for $\rho=1$. In fact, the main difficulty of Theorem~\ref{thm:box} lies in passing from crossing squares with probabilities bounded uniformly from below to crossing rectangles in the hard direction with probabilities bounded uniformly from below. In other words, the main step is the following proposition.

\begin{proposition}\label{prop:crossing}
For every $n\ge1$,
$\phi_{p_c,q}^1[\calH(2n,n)]\ge \frac{1}{16(1+q^2)}\phi_{p_c,q}^1[\calH(n,n)]^6.$
\end{proposition}

 A statement claiming that crossing a rectangle in the hard direction can be expressed in terms of the probability of crossing squares is called a Russo-Seymour-Welsh (RSW) type theorem. 
For Bernoulli percolation on $\bbZ^2$, this RSW result was first proved in \cite{Rus78,SeyWel78}.
Since then, many proofs have been produced (for Bernoulli), among which \cite{BolRio06c,BolRio06,BolRio10,Tas14b,Tas14}. We refer to \cite{DumTas15c} for a review of recent progress in this field. Here, we provide a proof for random-cluster models.

\begin{proof}
We treat the case of $n$ even, the case $n$ odd can be done similarly. Let us introduce the two rectangles $$R:=[-2n,2n]\times[-n,n]\quad\quad S:=[0,2n]\times[-n,n]\quad\quad S':=[-2n,0]\times[-n,n].$$
Also introduce the notation $\alpha:=\phi_{p_c,q}^1[\calH(S)]$. Also, define the sets 
\begin{eqnarray*}A^+:=\{-2n\}\times[0,n]\ \ \quad& B^+:=\{0\}\times[0,n]\ \quad& C^+:=\{2n\}\times[0,n]\\
A^-:=\{-2n\}\times[-n,0] \quad& B^-:=\{0\}\times[-n,0] \quad& C^-:=\{2n\}\times[-n,0].\end{eqnarray*}
By symmetry with respect to the $x$-axis, the probability that there is a path in $\omega\cap S$ from $B$ to $C^+$ is larger than or equal to $\alpha/2$. Similarly, the probability that there is a path in $\omega\cap S$ from $B^-$ to $C$ is larger than or equal to $\alpha/2$. Since the probability of $\calV(S)$ is also $\alpha$. The combination of these three events implies the event $\calE$ that there exists a path in $\omega\cap S$ from $B^-$ to $C^+$. Thus, the FKG inequality gives
\begin{equation*}\phi_{p_c,q}^1[\calE]\ge \tfrac{\alpha^3}{4}.\end{equation*}
Let $\calE'$ be the event that there exists a path in $\omega\cap S'$ from $A^-$ to $B^+$. 
By symmetry with respect to the origin, we have 
\begin{equation*}\phi_{p_c,q}^1[\calE']\ge \tfrac{\alpha^3}{4}.\end{equation*}
On the event $\calE\cap \calE'$, consider the paths of edges $\Gamma$ and $\Gamma'$ defined by:
\begin{itemize}[noitemsep,nolistsep]
\item $\Gamma$ is the bottom-most open crossing  of $S$ from $B^-$ to $C^+$, 
\item $\Gamma'$ is the top-most open crossing  of $S'$ from $A^-$ to $B^+$,
\end{itemize}
Construct the graph $G=G(\Gamma,\Gamma')$ with edge-set composed of edges with at least one endpoint in the cluster of the origin in $\bbR^2\setminus(\Gamma\cup\Gamma'\cup\sigma\Gamma\cup\sigma\Gamma')$ (here the paths are considered as subsets of $\bbR^2$), where $\sigma\Gamma$ and $\sigma\Gamma'$ are the reflections of $\Gamma$ and $\Gamma'$ with respect to the $y$-axis; see Fig.~\ref{fig:RSW}. 

Let us assume for a moment that we have the following bound: for any two possible realizations $\gamma$ and $\gamma'$ of $\Gamma$ and $\Gamma'$, 
\begin{equation}\label{eq:hahahaha}
\phi_{G,p_c,q}^{\rm mix}[\gamma\stackrel{G}\longleftrightarrow\gamma']\ge \tfrac{1}{1+q},
\end{equation}
where the mix boundary conditions correspond to wired on $\gamma$ and $\gamma'$, and free elsewhere (i.e.~the partition is given by $P_1=\gamma$, $P_2=\gamma'$ and singletons).
Then, 
\begin{align}
\label{eq:hohohoho}\phi_{p_c,q}^1[\calH(4n,2n)]&\stackrel{\phantom{\eqref{eq:hahahaha}}}=\phi_{p_c,q}^1[\calH(R)]\\
\nonumber&\stackrel{\phantom{\eqref{eq:hahahaha}}}\ge\phi_{p_c,q}^1[\{\Gamma\stackrel{G}\longleftrightarrow\Gamma'\}\cap \calE\cap \calE']\\
\nonumber &\stackrel{\phantom{\eqref{eq:hahahaha}}}=\sum_{\gamma,\gamma'}\phi_{p_c,q}^1[\{\gamma\stackrel{G}\longleftrightarrow\gamma'\}\cap \{\Gamma=\gamma,\Gamma'=\gamma'\}]\\
\nonumber&\stackrel{\phantom{\eqref{eq:hahahaha}}}\ge\sum_{\gamma,\gamma'}\phi_{G,p_c,q}^{\rm mix}[\gamma\stackrel{G}\longleftrightarrow\gamma']\cdot\phi_{p_c,q}^1[\Gamma=\gamma,\Gamma'=\gamma']\\
\nonumber&\stackrel{\eqref{eq:hahahaha}}\ge\tfrac1{1+q}\sum_{\gamma,\gamma'}\phi_{p_c,q}^1[\Gamma=\gamma,\Gamma'=\gamma']\\
\nonumber&\stackrel{\phantom{\eqref{eq:hahahaha}}}=\tfrac1{1+q}\,\phi_{p_c,q}^1[\calE\cap \calE']\stackrel{\rm (FKG)}\ge \tfrac{\alpha^6}{16(1+q)},
\end{align}
where in the fourth line we used the fact that $\Gamma=\gamma$ and $\Gamma'=\gamma'$ are measurable events of edges not in $G$, and that the boundary conditions induced on $\partial G$ always dominate the mixed boundary conditions. In the last line, we used that the events $\{\Gamma=\gamma,\Gamma'=\gamma'\}$ partition $\calE\cap \calE'$, and the lower bounds on the probability of the events $\calE$ and $\calE'$ proved above.

We now turn to the proof of \eqref{eq:hahahaha}. We wish to use a symmetry argument (similar to the proof that crossing a square has probability larger or equal to 1/2). We believe the argument to be more transparent on Fig.~\ref{fig:RSW} and we refer to its caption.

Fix $G=(V,E)$. Since the mix boundary conditions are planar boundary conditions, it will be simpler to consider a configuration $\xi\in \{0,1\}^{\bbE \setminus E}$ inducing them. We choose the following one: $\xi_e=1$ for all edges $e\in\gamma\cup\gamma'$
and $\xi_e=0$ for all other edges. Set $\omega^\xi$ to be the configuration coinciding with $\omega$ on $E$, and with $\xi$ on $\bbE\setminus E$.
%

Consider $\omega'$ to be the translation by $(1/2,1/2)$ and then reflection with respect to the $y$ axis of $(\omega^\xi)^*$.  By duality, the law of $\omega'$ on $G$ is dominated by the ${\rm mix'}$ boundary conditions defined to be wired on $\gamma\cup\gamma'$, and free elsewhere (i.e.~$P_1=\gamma\cup\gamma'$ and then singletons\footnote{Note that they are not equal to the mix boundary conditions since $\gamma$ and $\gamma'$ are wired together.}).
The absence of path in $\omega$ from $\gamma$ to $\gamma'$  is included in the event that there is a path in $\omega'_{|E}$ from $\gamma$ to $\gamma'$, so that
$$1-\phi_{G,p_c,q}^{\rm mix}[\gamma\stackrel{G}\longleftrightarrow\gamma']\le \phi_{G,p_c,q}^{\rm mix'}[\gamma\stackrel{G}\longleftrightarrow\gamma']\le q\,\phi_{G,p_c,q}^{\rm mix}[\gamma\stackrel{G}\longleftrightarrow\gamma'],$$
where in the second inequality we used that the Radon-Nikodym derivative is smaller or equal to $q$ since $k_{\rm mix}(\omega)-k_{\rm mix'}(\omega)\in\{0,1\}$. The inequality \eqref{eq:hahahaha} follows readily. This concludes the proof.
\end{proof}

\begin{figure}
\begin{center}
\includegraphics[width=1.00\textwidth]{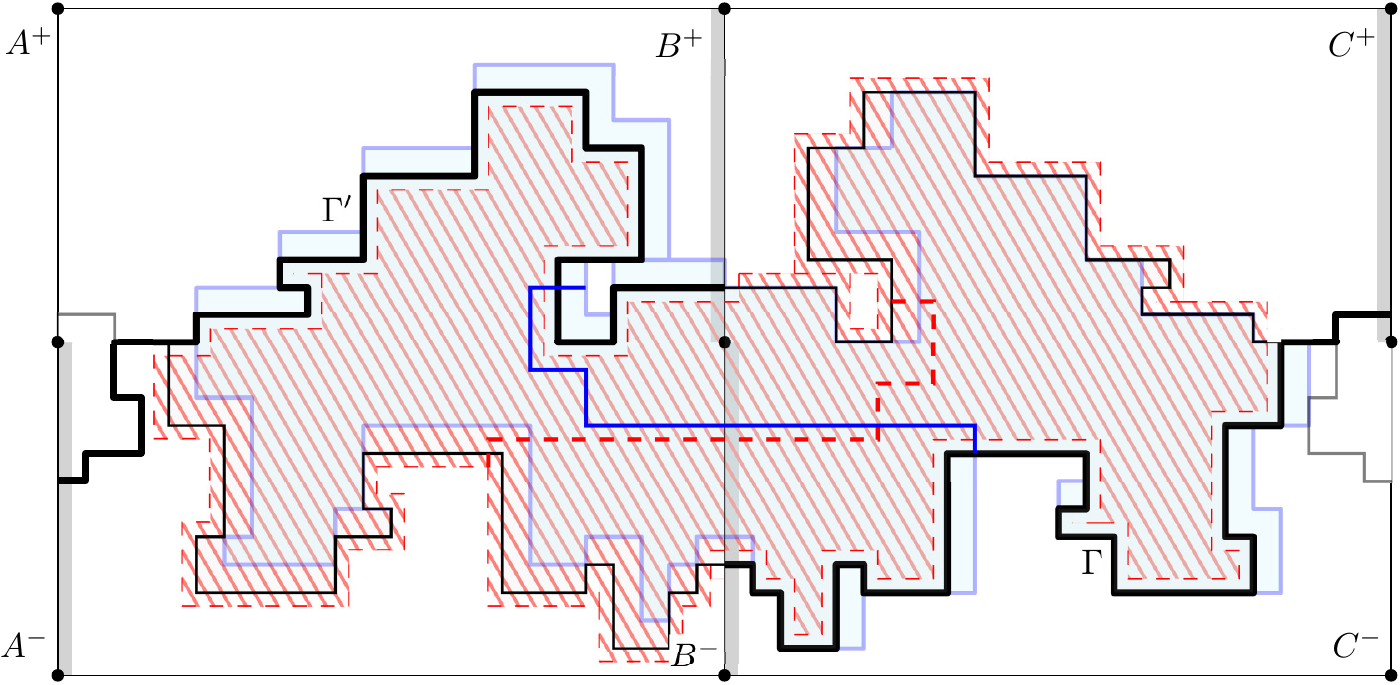}
\end{center}
\caption{The sets $A^\pm$, $B^\pm$ and $C^\pm$. We depicted $\Gamma$, $\Gamma'$ and their symmetric with respect to the $y$-axis. In gray, the set $G$. Hatched in red, the dual graph of $G$ together with a path in $\omega^*$ preventing the existence of a path from $\Gamma$ to $\Gamma'$ in $G$. In blue, the translation by $(1/2,1/2)$ of the symmetric of $G$ with respect to the $y$-axis, as well as the image by the same transformation of the dashed path. This path of $\omega'$ crossed $G$ from $\Gamma$ to $\Gamma'$. }\label{fig:RSW}
\end{figure}

Let us conclude this section by recalling that crossing probabilities in rectangles are expected to converge to explicit functions of $\rho$ as $n$ tends to infinity. More generally, crossing probabilities in topological rectangles should be conformally invariant; see  \cite{Smi01} for the case of site percolation (see also \cite{BefDum13,Wer09} for reviews) and \cite{CheSmi12,BenDumHon14,Izy15} for the case of the Ising model. 
\bexo
Here, we present a beautiful argument due to Vincent Tassion proving some weak form of crossing property for general FKG measures (with sufficient symmetry). We refer to \cite{Tas14b} for more details.
\begin{exercise}[Weak RSW for FKG measures]
Consider a measure $\mu$ on $\{0,1\}^\bbE$ which is invariant under the graph isomorphisms of $\bbZ^2$ onto itself. We further assume that $\mu$ satisfies the FKG inequality. We assume that 
$
\inf_n \mu[\calH(n,n)] > 0 \, .
$
The goal of this exercise is prove that
\begin{equation}\label{weakRSW}
\limsup_{n} \mu[\calH(3n,n)] > 0 \, .
\end{equation}
\noindent
1. Let $\calE_n$ be the event that the left side of $[-n,n]^2$ is connected to the top-right corner $(n,n)$. Use the FKG inequality to prove that $\limsup_n\mu[\calE_n] > 0$ implies \eqref{weakRSW}.
\medbreak\noindent
2. Assume the limit superior above is zero. Now, for any $-n \leq \alpha < \beta \leq n$, define the event $\mathcal{F}_n(\alpha,\beta)$ to be the existence of a crossing from the left side of $[-n,n]^2$ to the segment $\{n \} \times [\alpha,\beta]$. We consider the function
    \[
    h_n(\alpha) = \mu[\mathcal{F}_n(0,\alpha)] - \mu[\mathcal{F}_n(\alpha,n)] \, .
    \]
    Show that $h_n$ is an increasing function, and that there exists $c_0>0$ such that $h_n(n) > c_0$ for all $n$.
\medbreak\noindent
3. Assume that $h_n(n/2) < c_0/2$. Use {\rm (FKG)} to prove that \eqref{weakRSW}.

\medbreak\noindent
4. Assume that $h_n(n/2) > c_0/2$, and let $\alpha_n = \inf\{ \alpha : h(\alpha) > c_0/2\}$. Define the event $\mathcal{X}_n(\alpha)$ by the existence of a cluster in $[-n,n]^2$ connecting the four segments $\{-n\} \times [-n, -\alpha]$, $\{-n\} \times [\alpha, n]$, $\{n\} \times [-n, -\alpha]$, and $\{n\} \times [\alpha, n]$. Prove that there exists a constant $c_1>0$ independent of $n$ such that
$
    \mu[\mathcal{X}_n(\alpha)] \geq c_1 \, .
$
\medbreak\noindent
5. Prove that, for infinitely many $n$'s, $\alpha_n < 2 \alpha_{2n/3}$.
\medbreak\noindent
6. Prove that, whenever $\alpha_n < 2 \alpha_{2n/3}$, there exists a constant $c_2$ such that
$
\mu[\mathcal{H}(8/3n, 2n)] > c_2.$ 
 Conclude.

\end{exercise}

\eexo

\subsubsection{A dichotomy for random-cluster models}

Physicists work with several definitions of continuous phase transitions. For instance, a continuous phase transition may refer to the divergence of the correlation length, the continuity of the order parameter (here the spontaneous magnetization or the density of the infinite cluster), the uniqueness of the Gibbs states at criticality, the divergence of the susceptibility, the scale invariance at criticality, etc. From a mathematical point of view, these properties are not clearly equivalent (there are examples of models for which they are not), and they therefore refer to a priori different notions of continuous phase transition. 

In the following result, we use the study of crossing probabilities to prove that all these properties are equivalent for the planar random-cluster model.

 \begin{theorem}[DC, Sidoravicius, Tassion \cite{DumSidTas13}]
\label{thm:main}
Let $q\ge 1$, the following assertions are equivalent at criticality:
 \begin{enumerate}
 \item[{\rm \bf P1}] (Absence of infinite cluster) $\phi_{p_c,q}^1[0\longleftrightarrow \infty]=0$.
 \item[{\rm \bf P2}] (Uniqueness of the infinite-volume measure) $\phi_{p_c,q}^0=\phi_{p_c,q}^1$.
     \item[{\rm \bf P3}] (Infinite susceptibility) $\displaystyle \sum_{x\in \mathbb Z^2} \phi_{p_c,q}^0[0\longleftrightarrow x] =\infty.$
      \item[{\rm \bf P4a}] (Slow decay with free boundary conditions) $\displaystyle \lim_{n\to \infty} \tfrac 1 {n^{1/3}} \log \phi_{p_c,q}^0[0\longleftrightarrow \partial\Lambda_n] =0.$
\item[{\rm \bf P4b}] (Sub-exponential decay for free boundary conditions) 
 $\displaystyle \lim_{n\to \infty} \tfrac 1 n \log \phi_{p_c,q}^0[0\longleftrightarrow \partial\Lambda_n] =0.$
  \item[{\rm \bf P5}] (Uniform crossing probabilities) There exists $c=c(\rho)>0$ such that for all $n\geq 1$ and all boundary conditions $\xi$, if $R$ denotes the rectangle $[-n,(\rho+1)n]\times[-n,2n]$, then
   \begin{equation}c\le \phi_{R,p_c,q}^\xi\left[\calH(\rho n,n)\right]\le 1-c.\label{eq:oko}\end{equation}
  \end{enumerate}
\end{theorem}

The previous theorem does not show that these conditions are all satisfied, only that they are equivalent. In fact, whether the conditions are satisfied or not depend on the value of $q$, as we will see in the next two sections. 

While Properties {\bf P1--P4b} are quite straightforward to interpret, {\bf P5} is maybe more mysterious. One may wonder why having bounds that are uniform in boundary conditions is so relevant. The answer will become clear in the next sections: uniformity in boundary conditions is crucial to handle quantitatively dependencies between events in different parts of the graph. Note that the lower bound in {\bf P5} is a priori much stronger than the result of Theorem~\ref{thm:box} since the study of the previous section provided no information for free boundary conditions, even for crossing squares. 

Let us conclude this discussion by noticing that property {\bf P5} is not equivalent to the stronger statement {\bf P5'} where boundary conditions are put on the boundary of the rectangle $R':=[0,\rho n]\times[0,n]$ instead of $R$. In fact, the probability of crossing $R'$ with free boundary conditions on $\partial R'$ tends to 0 for the random-cluster model with cluster-weight $q=4$, while {\bf P5} is still true there. One may show that {\bf P5'} is true for $q<4$, but the proof is more complicated (see \cite{DumHonNol11} and \cite{DumSidTas13} for proofs for $q=2$ and $q\in[1,4)$ respectively). 

Last but not least, observe that the upper bound in \eqref{eq:oko} follows from the lower bound by duality. 

The proof of Theorem~\ref{thm:main} can be divided in several steps. First, one can see that several implications are essentially trivial. 
 \begin{proposition}
We have that {\bf P5}$\Rightarrow${\bf P1}$\Rightarrow${\bf P2}$\Rightarrow${\bf P3}$\Rightarrow${\bf P4a}$\Rightarrow${\bf P4b}.
\end{proposition}
The last implication {\bf P4b}$\Rightarrow${\bf P5} is the most difficult and is postponed to the next section. In fact we will only prove {\bf P4a}$\Rightarrow${\bf P5} since this will be sufficient for the applications we have in mind. We refer to \cite{DumSidTas13} for the proof of {\bf P4b}$\Rightarrow${\bf P5}.

\begin{proof}
The implications {\bf P3}$\Rightarrow${\bf P4a}$\Rightarrow${\bf P4b} are completely obvious, and {\bf P1}$\Rightarrow${\bf P2} is the object of Exercise~\ref{exo:12}. 
For {\bf P5}$\Rightarrow${\bf P1}, introduce the event $\calA:=\calV([-3n,3n]\times[2n,3n])$. If $\partial\Lambda_n$ is connected to $\partial\Lambda_{4n}$, then one of the four rotated versions of the event $\calA$ must also occur (where the angles of the rotation are $\tfrac\pi2k$ with $0\le k\le 3$). Therefore,$$\phi_{\Lambda_{4n}\setminus\Lambda_n,p_c,q}^1\left[\partial\Lambda_n\longleftrightarrow \partial\Lambda_{4n}\right]\stackrel{\eqref{eq:FKG}}\le 1-\phi_{\Lambda_{4n}\setminus\Lambda_n,p_c,q}^1[\calA^c]^4\stackrel{{\bf P5}}\le1- c,$$
where $c:=c(6)^4$ (we also used the comparison between boundary conditions in the second inequality). 
By successive applications of the domain Markov property and the comparison between boundary conditions (Exercise~\ref{exo:0}), we deduce the existence of $\alpha>0$ such that 
\begin{equation}\label{eq:sup pol}\phi_{\Lambda_{n},p_c,q}^1\left[0\longleftrightarrow\partial\Lambda_n\right]\le \prod_{4^k\le n}\phi_{\Lambda_{4^k}\setminus\Lambda_{4^{k-1}},p_c,q}^1\left[\partial\Lambda_{4^{k-1}}\longleftrightarrow \partial\Lambda_{4^k}\right]\le (1-c)^{\lfloor\log_4 n\rfloor}\le n^{-\alpha}\end{equation}
which gives {\bf P1} by passing to the limit.
For
{\bf P2}$\Rightarrow${\bf P3}, recall the definition of $\calH_n$ so that
\begin{equation}\label{eq:sub pol}n\,\phi_{p_c,q}^0[0\longleftrightarrow \partial\Lambda_n]\ge \phi^0_{p_c,q}[\calH_n]\stackrel{\bf P2}=\phi^1_{p_c,q}[\calH_n]\stackrel{\eqref{eq:RSW square q}}=1/2.\end{equation}
We deduce that 
$$\sum_{x\in\bbZ^2}\phi_{p_c,q}^0[0\longleftrightarrow x]=\sum_{n=0}^\infty \sum_{x\in\partial\Lambda_n}\phi_{p_c,q}^0[0\longleftrightarrow x]\ge \sum_{n\ge1}\phi_{p_c,q}^0[0\longleftrightarrow \partial\Lambda_n]\ge \sum_{n\ge1}\tfrac1{2n}=+\infty.$$
\end{proof}

Note that \eqref{eq:sup pol} and \eqref{eq:sub pol} show that under {\bf P5}, for all $n\ge1$,\begin{equation}\label{eq:polynomial}\tag{PD}
\frac1{2n}\le \phi^0_{p_c,q}[0\longleftrightarrow \partial\Lambda_n]\le \frac{1}{n^{\alpha}}.
\end{equation}
This is one among a long list of properties implied by {\bf P5}. Let us mention a few others: mixing properties (Exercise~\ref{exo:pol mixing}, the existence of sub-sequential scaling limits for interfaces, the value for certain critical exponents called {\em universal} critical exponents (it has nothing to do with the universality for the model itself), the fractal nature of large clusters (with some explicit bounds on the Hausdorff dimension). It is also an important step towards the understanding of conformal invariance of the model,
 scaling relations between several critical exponents, etc.

\bexo
In the next two exercises, we assume {\bf P5}.

\begin{exercise}\label{exo:quasi}
1. Prove that there exists $c>0$ such that 
$\displaystyle\phi_{p_c,q}^0[0\longleftrightarrow\partial\Lambda_n]\le c\phi_{p_c,q}^0[0\longleftrightarrow\partial\Lambda_{2n}].$\medbreak\noindent
2. Prove that there exist $c_1,c_2>0$ such that for any $x\in\partial\Lambda_n$,
$$c_1\phi_{p_c,q}^0[0\longleftrightarrow\partial\Lambda_n]^2\le\phi_{p_c,q}^0[0\longleftrightarrow x]\le c_2\phi_{p_c,q}^0[0\longleftrightarrow\partial\Lambda_n]^2.$$
\end{exercise}

\begin{exercise}[Polynomial mixing]\label{exo:pol mixing}
1. Show that there exists a constant $c>0$ such that for any $n\ge 2k$ and any event $\calA$ depending on edges in $\Lambda_k$ only,
$\phi^\xi_{\Lambda_{k},p_c,q}[\Lambda_k\not\longleftrightarrow\partial\Lambda_n|\calA]\ge 1-(\tfrac kn)^c.$
\medbreak\noindent
2. Construct a coupling between $\omega\sim \phi^\xi_{\Lambda_{2n},p_c,q}$ and $\tilde\omega\sim \phi^1_{\Lambda_{2n},p_c,q}$ in such a way that $\omega$ and $\tilde\omega$ coincide on $\Lambda_k$ when $\Lambda_k$ is not connected to $\partial\Lambda_n$ in $\tilde\omega$. {\em Hint.} Construct the coupling step by step using  an exploration of the cluster connected to the boundary. Deduce that 
$$\phi^\xi_{\Lambda_{k},p_c,q}[\calA]\ge \big(1-(\tfrac kn)^c\big)\phi^1_{\Lambda_{k},p_c,q}[\calA].$$
\medbreak\noindent
3. Construct a coupling between $\omega\sim \phi^\xi_{\Lambda_{2n},p_c,q}$ and $\tilde\omega\sim \phi^1_{\Lambda_{2n},p_c,q}$ in such a way that $\omega$ and $\tilde\omega$ coincide on $\Lambda_k$ when there exists an open circuit in $\omega$ surrounding $\Lambda_k$. Deduce that 
$$\phi^1_{\Lambda_{k},p_c,q}[\calA]\ge \big(1-(\tfrac kn)^c\big)\phi^\xi_{\Lambda_{k},p_c,q}[\calA].$$
\medbreak\noindent
4.  Deduce that for any event $\calB$ depending on edges outside $\Lambda_n$ only,  
 $$\Big|\phi_{p_c,q}^1[\calA\cap \calB]-\phi_{p_c,q}^1[\calA]\phi_{p_c,q}^1[\calB]\Big|\le 2\left(\tfrac kn\right)^c \phi_{p_c,q}^1[\calA]\phi_{p_c,q}^1[\calB].$$
\end{exercise}
\eexo

\subsubsection{Proof of {\bf P4a}$\Rightarrow${\bf P5} of Theorem~\ref{thm:main}}

We drop the dependency on $q\ge1$ and $p_c$ in the subscripts of the measures.  In the next proofs, we omit certain details of reasonings concerning comparison with respect to boundary conditions. We already encountered such arguments several times (for instance in Exercise~\ref{exo:unique Gibbs} and in the proof of Proposition~\ref{prop:crossing}). We encourage the reader to try to fill up the details of each one of these omissions (Exercise~\ref{exo:11}). 

In order to prove {\bf P4a}$\Rightarrow${\bf P5}, we developed a geometric renormalization for crossing probabilities: crossing probabilities at scale $2n$ are expressed in terms of crossing probabilities at scale $n$. The renormalization scheme is built in such a way that as soon as the crossing probability passes below a certain threshold, they start decaying stretched exponentially fast. As a consequence, either crossing probabilities remain bounded away from 0, or they decay to 0 stretched exponentially fast. 
Let $\calA_n$ be the event that there exists a circuit\footnote{i.e.~a path of edges starting and ending at the same point.} of open edges in $\Lambda_{2n}\setminus \Lambda_n$ surrounding the origin and set 
$$u_n:=\phi_{\Lambda_{8n}}^0[\mathcal A _n].$$ 
The proof articulates around Proposition~\ref{lem:induction1} below, which relates $u_n$ with $u_{7n}$.

\begin{proposition}\label{lem:induction1}
  There exists a
  constant $C<\infty$ such that  
  $ u_{7n}\leq C\, u_n^2$ for all $n\geq 1$.
\end{proposition}

This statement allows to prove recursively that 
$Cu_{7^kn}\le (Cu_n)^{2^k}$. In particular, if there exists $n$ such that $Cu_n<1$, then $u_{7^kn}$ decays stretched exponentially fast. 
Therefore, the proof of {\bf P4a}$\Rightarrow${\bf P5} follows trivially from the previous proposition and the following fairly elementary facts:
\begin{itemize} 
\item $(u_n)$ bounded away from 0 implies {\bf P5} (Exercise~\ref{exo:RSW}).
\item Stretched exponential decay of $u_{7^kn}$ implies stretched exponential decay of $\phi^0[0\leftrightarrow \partial\Lambda_n]$ (Exercise~\ref{exo:from An to point}). Note that this last fact is very intuitive since in order to have a circuit in $\Lambda_{2n}\setminus \Lambda_n$ surrounding the origin, one must have fairly big clusters.
\end{itemize}
\bexo
\begin{exercise}\label{exo:RSW}
1. Fix $\ep>0$ and $\rho>0$. Combine circuits in annuli to prove the existence of $c=c(\rho,\ep)>0$ such that for all $n$, if $R=[-\ep n,(\rho+\ep)n]\times[-\ep n,(1+\ep)n]$ then
$
\phi^0_{R}[\calH(\rho n,n)]\ge c.
$\medbreak\noindent
2. Deduce that $
\phi^\xi_{R}[\calH(\rho n,n)]\ge c.
$ for every boundary conditions $\xi$.
\medbreak\noindent
3. 
Use duality to prove that for any $\xi$ and $n$, $\phi^\xi_{R}[\calH(\rho n,n)] \le 1-c'$ for some $c'=c'(\rho,\ep)>0$.
\end{exercise}

\begin{exercise}\label{exo:from An to point} In this exercise, we assume that $\limsup\tfrac1{n^\alpha}\log \phi^0[0\longleftrightarrow\partial\Lambda_n]=0$ for some constant $\alpha>0$.\medbreak\noindent
1. Prove that 
$\displaystyle\phi^0[0\longleftrightarrow\partial\Lambda_n]\le \sum_{k\ge n}\sum_{x\notin \Lambda_k}\phi^0_{\Lambda_k} [0\longleftrightarrow x]$. {\em Hint.} Use the farthest point on the cluster of $0$ and an argument similar to Exercise~\ref{exo:12}.
\medbreak\noindent
2. Deduce that $\displaystyle\limsup_{x\rightarrow\infty}\tfrac1{k^\alpha}\log \phi^0_{\Lambda_k} [0\longleftrightarrow x]=0$, where $k$ is defined in such a way that $x\in\partial\Lambda_k$.
\medbreak\noindent
3. Prove that $\displaystyle\limsup_{x\rightarrow\infty}\tfrac1{k^\alpha}\log \phi^0_{\Lambda_{3k}} [0\longleftrightarrow 2ke_1]=0$, where $e_1=(1,0,\dots,0)$.
\medbreak\noindent
4. Prove that for every $n$, $\displaystyle\limsup_{k\rightarrow\infty} \tfrac1{k^\alpha}\log u_{7^kn}=0$.
\end{exercise}
\eexo
To prove Proposition~\ref{lem:induction1}, first consider 
the strip $\bbS=\bbZ\times[-n,2n]$, and the  random-cluster measure $\phi^{1/0}_{\bbS}$ with free boundary conditions on $\bbZ\times
\{-n\}$ and wired everywhere else. We refer to Exercise~\ref{exo:crossing in strip} for the (slightly technical) proof of this lemma.
\begin{lemma}\label{lem:stripRSW}
  For all $\rho>0$, there exists a constant
  $c>0$ such that for all $n\geq 1$,
  \begin{equation}
    \label{eq:1}
    \phi^{1/0}_{\bbS}[\calH(\rho n,n)] \ge c.  
\end{equation}
\end{lemma}

Even though we do not provide a proof of this statement, the intuition is fairly convincing: the boundary conditions are still somehow ``balanced'' between primal and dual configurations, and it is therefore not so surprising that crossing probabilities are bounded away from above.

In the next lemma, we consider horizontal crossings in rectangular shaped domains with free boundary conditions on the bottom and wired elsewhere.

\begin{lemma}\label{lem:push} 
  For all $\rho>0$ and $\ell\ge 2$, there exists $c=c(\rho,\ell)>0$ such that for all $n>0$,
  \begin{equation}
    \label{eq:7} 
    \phi_D^{1/0}[\calH\left(\rho n,n\right)]\ge c
  \end{equation}
with $D=[0,\rho n]\times [-n,\ell n]$, and $\phi_D^{1/0}$ is the random-cluster measure with free boundary conditions on the bottom side, and wired on the three other sides.
\end{lemma}

\begin{proof}
For $\ell=2$, Lemma~\ref{lem:stripRSW} and the comparison between boundary conditions (used on the sides) implies the result readily.
 Now, assume that the result holds for $\ell$ and let us prove it for $\ell+1$. The comparison between boundary conditions in $[0,\rho n]\times[0,(\ell+1)n]$ implies that 
 $$
\phi_D^{1/0}[\calH(R)]\ge c(\rho,\ell),
$$
where $R=[0,\rho n]\times[n,2n]$. 
The comparison between boundary conditions implies that conditioned on $\calH(R)$, the measure restricted to edges in $R':=[0,\rho n]\times[0,n]$ dominates the restriction (to $R'$) of the measure on $D':=\bbS\cap D$ with free boundary conditions on the bottom of $D'$, and wired on the other sides. We deduce that 
$$  \phi_D^{1/0}[\calH\left(\rho n,n\right)\cap\calH(R)]\ge c(\rho,2)\phi_D^{1/0}[\calH(R)]\ge c(\rho,2)c(\rho,\ell).$$ 
\end{proof}

\begin{proof}[Proposition~\ref{lem:induction1}]
Fix $n\ge1$ and set $N:=56n$. Below, the constants $c_i$ are independent of $n$. 
 Define $\calA_n^\pm$ to be the translates of the event $\calA_n$ by $z_\pm:=(\pm 5n,0)$. 
  
Conditioned on $\mathcal A_{7n}$, the restriction of the measure to $\Lambda_{7n}$ dominates the restriction of the measure with wired boundary conditions at infinity. Using this in the second inequality, we find
  \begin{equation}
    \label{eq:10}
    \phi^0_{\Lambda_{N}}[\mathcal A_n^+ \cap  \mathcal A_n^-] \geq \phi^0_{\Lambda_{N}}[\mathcal A_n^+ \cap  \mathcal A_n^-\cap \calA_{7n}]\ge \phi^1[\mathcal A_n^+ \cap  \mathcal A_n^-]\,u_{7n} \ge c_1 u_{7n},
  \end{equation}
  where in the last inequality we combined crossings in rectangles of aspect ratio 4 to create circuits, and then used Theorem~\ref{thm:box} to bound the probability from below (which is justified since the boundary conditions are wired at infinity).
  
  Let $\calB_n$ be the event that $R:=[-N,N]\times[-2n,2n]$ is not connected to $R':=[-N,N]\times[-3n,3n]$.
Under $\phi^0_{N}[\,\cdot\,|\mathcal A_n^+
  \cap \mathcal A_n^-]$, the boundary conditions outside of $R$ are 
  dominated by wired boundary conditions on $R$ and free
  boundary conditions on the boundary of $\Lambda_N$. As a consequence,
  Lemma~\ref{lem:push} applied to the dual measure in the two rectangles $[-N,N]\times[2n,N]$ and $[-N,N]\times[-N,-2n]$ implies that 
    \begin{equation}
    \label{eq:12}
    \phi_{\Lambda_{N}}^0[\calB_n| \mathcal
      A_n^+ \cap  \mathcal A_n^-]\ge c_2.
  \end{equation}
Altogether, \eqref{eq:10} and \eqref{eq:12} lead to the
estimate
  \begin{equation}
    \label{eq:13}
    \phi_{\Lambda_{N}}^0[ \mathcal
      A_n^+ \cap  \mathcal A_n^-\cap\mathcal B_n]\geq c_3\,u_{7n}.  \end{equation}
  \begin{figure}[htbp]
    \centering
    \includegraphics[width=1.00\textwidth]{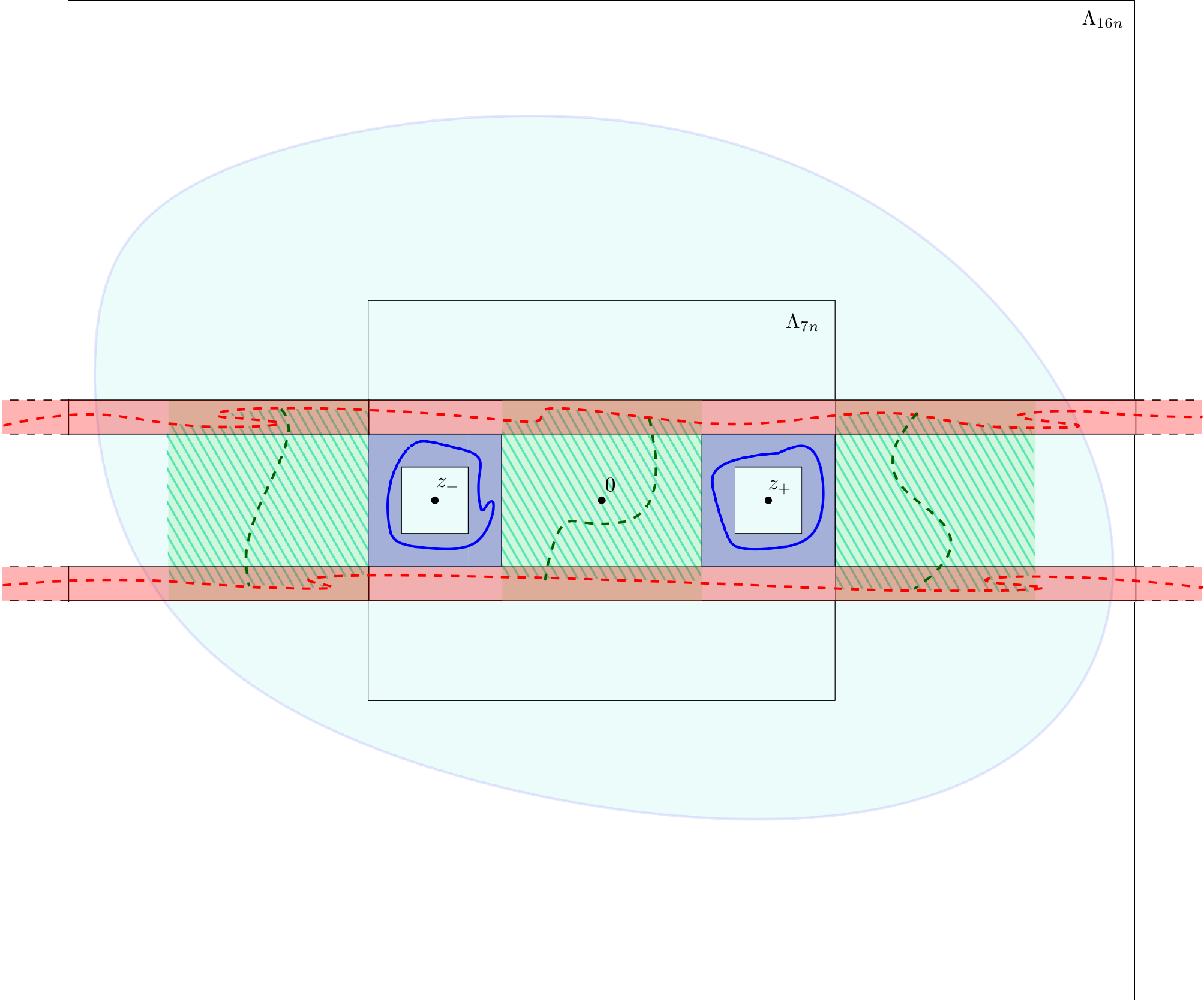} 

    \caption{The different events involved in the construction. In light blue, a circuit in $\omega$ implying the occurrence of the event $\calA_{7n}$. Inside the circuit, the measure dominates a random-cluster measure with wired boundary conditions at infinity. Therefore, conditionally on $\calA_{7n}$, one can construct the two blue circuits (corresponding to $\calA^{\pm}_n$) with positive probability. For the rest of the construction, the event $\calA_{7n}$ is not taken into account anymore. 
The dashed paths in the red areas correspond to two paths in $\omega^*$, which imply the occurrence of the event $\calB_n$. In the top red rectangle (which goes further left and right, but could not be drawn on the picture), conditionally on $\calA^+_n\cap\calA_n^-$, the boundary conditions are dominated by wired boundary conditions on the bottom and free on the boundary of $\partial\Lambda_N$, hence one can apply Lemma~\ref{lem:push}. The same reasoning is valid for the bottom red rectangle.
The dashed green paths correspond to events in $\omega^*$ implying the occurrence of $\calC_n$, $\calD_n$ and $\calE_n$. The dashed green areas correspond to the intersection of $\mathsf C$ with $[-3n,3n]^2$, $[-13n,-7n]\times[-3n,3n]$ and $[7n,13n]\times[-3n,3n]$. In these areas, the boundary conditions are dominated by the wired boundary conditions on left and right, and free on top and bottom.
}
    \label{fig:eventsBH}
  \end{figure}
Define $\mathsf C$ to be the set of points in $R'$ which are not connected to the top or the bottom sides of $R'$. Let $\calC_n$ be the event that the left and right sides of $S:=[-3n,3n]^2$ are not connected together in $\mathsf C$. Conditionally on $A_n^+ \cap  \mathcal A_n^-\cap\mathcal B_n\cap\{\mathsf C=C\}$, the boundary conditions in $C$ are dominated by the boundary conditions of the restriction (to $C$) of  the measure in $S$ with free on the top and bottom sides of $S$, and wired on the left and right. A duality argument in $S$ implies that the probability to have a top to bottom dual crossing is bounded from below by $\tfrac 1{1+q}$. This implies that in $C$, the probability of a dual crossing from top to bottom is a fortiori bounded from below by $\tfrac1{1+q}$. Averaging on the possible values of $C$, this implies  
\begin{equation}\label{eq:aye}
    \phi_{\Lambda_{N}}^0[\calC_n\:\big|\:\calA_n^+\cap\calA_n^-\cap\mathcal
      B_n]\geq \tfrac1{1+q}.
  \end{equation}
A similar reasoning gives that if $\calD_n$ and $\calE_n$ denote respectively the events that the left and right sides of $[-13n,7n]\times[-3n,3n]$ and $[7n,13n]\times[-3n,3n]$ are not connected in $\mathsf S$, we have       \begin{equation*}
    \phi_{\Lambda_{N}}^0[\calC_n\cap\calD_n\cap\calE_n\:\big|\:\calA_n^+\cap\calA_n^-\cap\mathcal
      B_n]\geq \tfrac1{(1+q)^3}
  \end{equation*}
   which, together with \eqref{eq:13}, leads to 
  \begin{equation}
    \label{eq:15}
    \phi_{\Lambda_{N}}^0[\calA_n^+\cap\calA_n^-\cap\mathcal B_n \cap  \calC_n\cap\calD_n\cap\calE_n]\geq c_4  \, u_{7n}.
  \end{equation}
 
  Now, on $\calA_n^-\cap\mathcal B_n \cap  \calC_n\cap\calD_n\cap\calE_n$, there is a dual circuit in the box $\Lambda$ of size $8n$ around $z_+$ surrounding the box $\Lambda'$ of size $2n$ around $z_+$ and therefore the comparison between boundary conditions implies that conditioned on this event, the boundary conditions in $\Lambda'$ are dominated by the free boundary conditions on $\partial\Lambda$. As a consequence, $$ \phi_{\Lambda_{N}}^0[\calA_n^+|\calA_n^-\cap\mathcal B_n \cap  \calC_n\cap\calD_n\cap\calE_n]\leq   \phi_{\Lambda_{8n}}^0[\mathcal A _{n}]=u_n.$$
 Similarly, 
 $  \phi_{\Lambda_{N}}^0[\calA_n^-|\mathcal B_n \cap  \calC_n\cap\calD_n\cap\calE_n]\leq u_n$. Plugging these two estimates in  \eqref{eq:15} gives 
 $u_n^2\ge c_4  u_{7n},$
 which concludes the proof.\end{proof}
 
 \bexo
 \begin{exercise}\label{exo:11}
 Fill up all the details of the different comparison between boundary conditions used in the last two proofs.
 \end{exercise}
 
\begin{exercise}\label{exo:crossing in strip} Below, we use the notation $A\stackrel{S}\longleftrightarrow B$ to denote the existence of a path from $A$ to $B$ staying in $S$. We assume that $\phi^{1/0}_{\bbS}[\calH(\rho n,n)] \le \tfrac12$. Set 
$$\bbS':=\bbZ\times[0,n]\quad R:=[0,9\lambda]\times[0,n]\quad R':=[4\lambda,9\lambda]\times[0,n].$$ Set $\lambda=n/11$ and $\ell_i=[i\lambda,(i+1)\lambda]\times\{0\}$.\medbreak\noindent
1. Show that if $\phi^{1/0}_{\bbS}[\ell_i\stackrel{\bbS'}\longleftrightarrow\ell_{i+2}]\ge c$, then $\phi^{1/0}_{\bbS}[\calH(\rho n,n)] \ge c^{11\rho}$.
\medbreak\noindent
2. Show that $\phi^{1/0}_{\bbS}[\calV(\rho n,n)]\ge \tfrac12$.
\medbreak\noindent
3. Deduce that one of the following two conditions occur:
\begin{itemize}[noitemsep,nolistsep]
\item[C1] $\phi^{1/0}_\bbS[\ell_4\stackrel{R}\longleftrightarrow \bbZ\times\{n\}]\ge \tfrac{1}{44\rho}$.
\item[C2] $\phi^{1/0}_\bbS[\ell_4\stackrel{R}\longleftrightarrow \{0\}\times\bbZ]\ge \tfrac{1}{88\rho}$.
\end{itemize}
\medbreak\noindent
4. Assume that C1 holds true. Show that 
$\phi^{1/0}_\bbS[\ell_2\stackrel{\bbS'}\longleftrightarrow \ell_4]\ge \tfrac1{1+q}(\tfrac{1}{36\rho})^2$. {\em Hint.} Use the same reasoning as for the proof of \eqref{eq:aye}. \medbreak\noindent
5. Assume that C2 holds true. Show that 
$$\phi^{1/0}_\bbS[\{\ell_4\stackrel{R}\longleftrightarrow \{7\lambda\}\times\bbZ\}\cap\{\ell_6\stackrel{R'}\longleftrightarrow\{4\lambda\}\times\bbZ\}]\ge (\tfrac{1}{88\rho})^2.$$
* Construct a symmetric domain to prove that 
$$\phi^{1/0}_\bbS[\ell_4\stackrel{\bbS'}\longleftrightarrow \ell_6]\ge \tfrac1{1+q}(\tfrac{1}{88\rho})^2.$$
6. Conclude.
\end{exercise}
\eexo

\subsection{Proving continuity for $q\le 4$: the parafermionic observables}
In this section, we prove that for  $q\in[1,4]$, {\bf P1--5} are satisfied by proving that {\bf P4a} is satisfied. In order to do so, we introduce the so-called parafermionic observables. The next section is intended to offer an elementary application of the parafermionic observable by studying a slightly different problem, namely the question of computing the connective constant of the hexagonal lattice. We will then go back to the random-cluster model later on.

\subsubsection{Computing the connective constant of the hexagonal lattice}

Let ${\mathbb H}=(\bbV,\bbE)$ be the hexagonal lattice (for now, we assume that 0 is a vertex of $\bbH$ and we assume that the edge on the right of 0 is horizontal). Points in the plane are considered as complex numbers. A {\em walk} $\gamma$ {\em of length $n$} is a path $\gamma:\{0,\dots,n\}\mapsto\bbV$ such that $\gamma_0=0$ and $\gamma_i\gamma_{i+1}\in \bbE$ for any $i<n$. The walk is {\em self-avoiding} if $\gamma_i=\gamma_j$ implies $i=j$.  Let $c_n$ be the number of self-avoiding walks of length $n$. 

A self-avoiding walk of length $n+m$ can be uniquely cut into a self-avoiding walk of length $n$ and a translation of a self-avoiding walk of length $m$. Hence,
\begin{equation*}
	c_{n+m}\leq c_nc_m,
\end{equation*}
from which it follows (by Fekete's lemma on sub-multiplicative sequences of real numbers) that there exists $\mu_c\in[1,+\infty)$, called the {\em connective constant}, such that
$$
	\mu_c:=\lim_{n\rightarrow \infty}c_n^{\,1/n}.
$$
On the hexagonal lattice, Nienhuis  \cite{Nie82,Nie84} used the Coulomb gas formalism to conjecture non-rigorously what $\mu_c$ should be. In this section, we present a mathematical proof of this prediction.

\begin{theorem}[DC, Smirnov \cite{DumSmi12}]\label{theorem}
We have $\mu_c=\sqrt{2+\sqrt{2}}.$
\end{theorem}

Before diving into the argument, let us recall the following classical fact. We choose to leave the proof of this statement as an exercise (Exercise~\ref{exo:bridges}) since the argument is instructive. A {\em self-avoiding bridge} is a self-avoiding walk $\gamma:\{0,\dots,n\}\mapsto  \bbH$ satisfying that $0<{\rm Re}(\gamma_i)\le {\rm Re}(\gamma_n)$ for every $1\le i\le n$. Let $b_n$ be the number of bridges of length $n$.

\begin{proposition}[Hammersley-Welsh \cite{HamWel62}]\label{prop:HW}
We have that 
$\displaystyle\lim_{n\rightarrow\infty}b_n^{1/n}=\mu_c.$
\end{proposition}

\begin{mdframed}[backgroundcolor=lightgray!20]\scriptsize
\begin{exercise}\label{exo:bridges}
\medbreak\noindent
1. Prove that $b_n^{1/n}$ converges to a value $\mu$ and that $b_n\le \mu^n$ for all $n$.
\medbreak\noindent
2. Let $h_n$ be the number of (half-space) self-avoiding walks with ${\rm Re}(\gamma_i)>0$ for all $i\ge1$. Prove that 
$$c_n\le \sum_{k=0}^n h_{k+1}h_{n+1-k}.$$
{\em Hint.} Cut the walk at a point of maximal first coordinate and add horizontal edges. Deduce that $\displaystyle\lim_{n\rightarrow\infty}h_n^{1/n}=\mu_c$.
\medbreak\noindent
3. By decomposing with respect to the last point with maximal first coordinate, show that 
$$h_n\le \sum_{k=0}^n b_k h_{n-k}.$$
4. Let $p_n$ be the number of partitions of $n$ into integers, i.e.~the number of $h_1\ge h_2\ge \dots\ge h_\ell$ such that $h_1+\dots+h_\ell=n$. Let $P_n=\sum_{k=0}^np_k$. By iterating the decomposition above, and observing that the width of the different half-space walks is decreasing, deduce that 
$$h_n\le P_n\mu^n.$$
5. Prove that the generation function $P$of the number $p_n$ of partitions of an integer satisfies
$$P(t)=\sum_{n=0}^\infty P_nt^n=\prod_{n=1}^\infty \frac1{1-t^n}.$$\medbreak\noindent
6. Deduce that $\mu=\mu_c$.
{\rm Remark:} One may also invoke a result of Hardy-Ramanujan stating that $p_n\le \exp(O(\sqrt n))$ to make the previous result quantitative. 
\end{exercise}

\eexo

\begin{figure}[ht]
    \begin{center}
      \includegraphics[width=0.30\hsize]{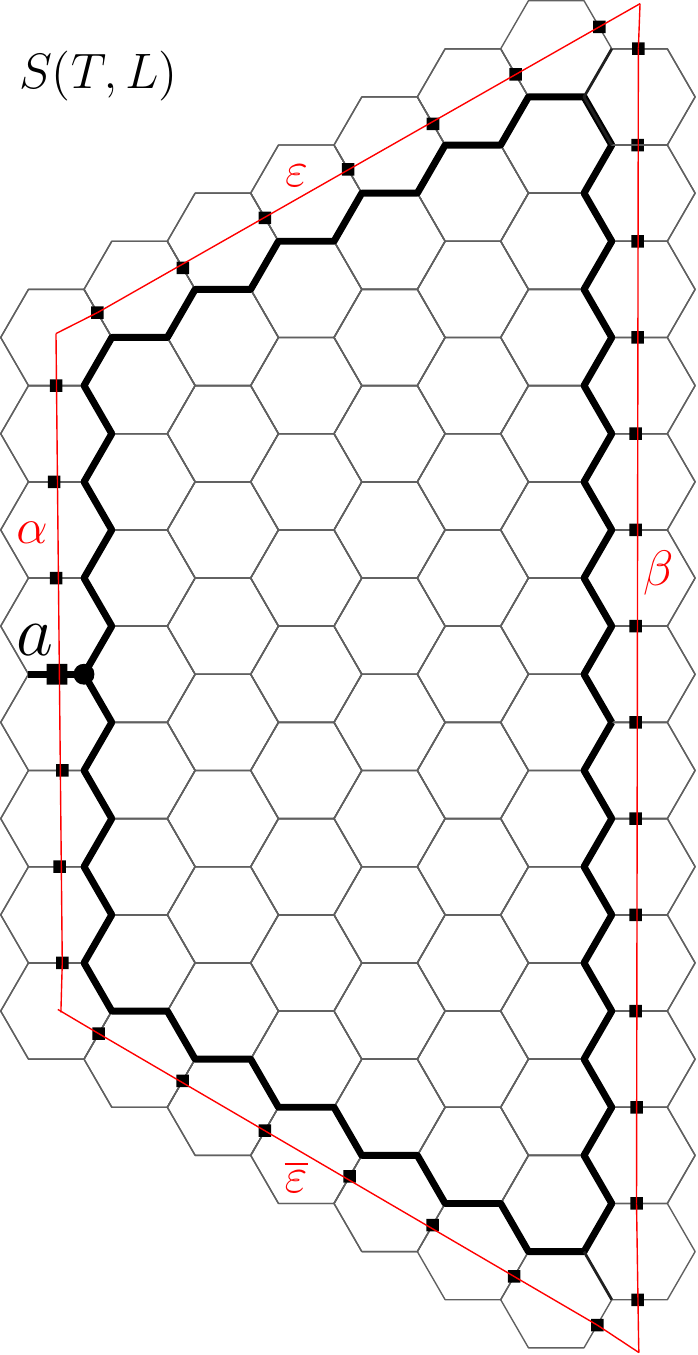}
     \end{center}
    \caption{The graph $S(T,L)$ and its boundary parts $\alpha$, $\beta$, $\ep$ and $\bar{\ep}$.}
    \label{fig:domain} 
  \end{figure}

Assume that the lattice has mesh-size 1 and is shifted by $(-\tfrac12,0)$ so that the origin is now a {\em mid-edge}, i.e.~the middle of an edge, which we call $a$. We also assume that this edge is horizontal (as in Fig.~\ref{fig:domain}). We now consider that self-avoiding walks are in fact starting at $a$ and ending at mid-edges. Their length, denoted by $|\gamma|$, is still the number of vertices on it. We consider a truncated vertical strip $S(T,L)$ of width $T$ cut at height $L$ at an angle of $\pi/3$ (see Fig.~\ref{fig:domain}), i.e.
\begin{align*}
S(T,L):=\{z\in\bbC:0\leq {\rm Re}(z)\leq \tfrac{3}2T \text{ and }\sqrt 3 |{\rm Im}(z)|\le 3L+{\rm Re}(z)\}.
\end{align*}
Denote by $\alpha$ the left boundary of $S(T,L)$ and by $\beta$ the right one. Symbols $\ep$ and $\bar{\ep}$ denote the top and bottom boundaries of $S(T,L)$. For $x>0$, introduce the following quantities:
\begin{align*}
	A_{T,L}&:=\sum_{\substack{\gamma\subset S(T,L)\\
	\gamma_n\in \alpha\setminus \{a\}}}x^{|\gamma|}\qquad\qquad B_{T,L}:=\sum_{\substack{\gamma\subset S(T,L)\\\gamma_n\in\beta}}x^{|\gamma|}\qquad\qquad
	E_{T,L}:=\sum_{\substack{\gamma\subset S(T,L)\\\gamma_n\in \ep\cup\bar{\ep}}}x^{|\gamma|}.
\end{align*}
We will prove the following lemma.
\begin{lemma}\label{lem:fundamental}
	If $x:=1/\sqrt{2+\sqrt 2}$, then for any $T,L\ge0$,
	\begin{equation}\label{equation box}
		1=\cos \left(\tfrac {3\pi}{8}\right)A_{T,L}+B_{T,L}+\cos \left(\tfrac{\pi}{4}\right)E_{T,L}.
	\end{equation}
\end{lemma}
Before proving this statement, let us show how it implies the claim. 
Observe that sequences $(A_{T,L})_{L>0}$ and $(B_{T,L})_{L>0}$ are increasing in $L$ and are bounded. They therefore converge. We immediately deduce that $(E_{T,L})_{L>0}$ also does. 
Let $A_T$, $B_T$ and $E_T$ be the corresponding limits.
  \paragraph{Upper bound on the connective constant} Observe that $B_T\le 1$ for any $T$ (since $B_{T,L}\le 1$) so that for any $y<x$,
$$\sum_{n=0}^{\infty}b_ny^n\le \sum_{T\ge0}B_T(\tfrac{y}{x})^T<\infty$$
(we use that a bridge of width $T$ has length at least $T$). Proposition~\ref{prop:HW} thus implies
\begin{equation}\label{eq:connective upper}
\mu_c=\lim_{n\rightarrow\infty}b_n^{1/n}\le \sqrt{2+\sqrt 2}.
\end{equation}
  \paragraph{Lower bound on the connective constant} Assume first that $E_T>0$ for some $T$. Then,\begin{equation*}
	\sum_{n=0}^nc_nx^n\geq \sum_{L=0}^\infty E_{T,L}=+\infty,
\end{equation*}
which implies $\mu_c\ge \sqrt{2+\sqrt 2}$.
 Assume on the contrary that $E_T=0$ for all $T$. Taking the limit in \eqref{equation box} implies
\begin{equation}\label{infinite strip}
	1=\cos \left(\tfrac {3\pi}{8}\right)A_T+B_T.
\end{equation}
Observe that self-avoiding walks entering into account for $A_{T}$ and not for $A_{T-1}$ have to visit a vertex $x\in \bbV$ on the right of the strip of width $T$, i.e.~satisfying ${\rm Re}(x)=\tfrac32T-\tfrac12$. Cutting such a walk at the first such point (and adding half-edges to the two halves), we obtain two bridges. We conclude that
\begin{equation}\label{rec relation}
	A_{T}-A_{T-1}\leq \tfrac1xB_T^{\ 2}.
\end{equation}
Combining \eqref{infinite strip} for $T-1$ and $T$ with \eqref{rec relation} gives
\begin{align*}
0=1-1
&=\cos \left(\tfrac {3\pi}{8}\right)(A_{T}-A_{T-1})+B_{T}-B_{T-1}\leq \cos \left(\tfrac {3\pi}{8}\right)\tfrac1xB_{T}^{\ 2}+B_{T}-B_{T-1},
\end{align*}
so
$$\cos \left(\tfrac {3\pi}{8}\right)\tfrac1xB_{T}^{\ 2}+B_{T}\geq B_{T-1}.$$
By induction, it is easy to check that 
$$B_T\geq \frac {\min [B_1,x/\cos \left(\tfrac {3\pi}{8}\right)]}{T}$$
for every $T\geq1$. This implies that $\mu_c\ge\sqrt{2+\sqrt 2}$ in this case as well since
$$\sum_{n=0}^\infty b_nx^n=\sum_{T=0}^\infty B_T=+\infty.$$

At the light of the previous discussion, we shall now prove Lemma~\ref{lem:fundamental}. Fix $T$ and $L$. Introduce the \emph{parafermionic observable}\footnote{Let us mention that there are other instances of parafermionic observables for the self-avoiding walk, see \cite{BeaBouDum14,Gla13}. We do not discuss this further here since our goal is to quickly move back to the random-cluster model.} defined as follows: for a mid-edge $z$ in $S(T,L)$, set
\begin{equation*}
	F(z):=\sum_{\substack{\gamma\subset S(T,L)\\ \gamma\text{ ends at }z}}{\rm e}^{-{\rm i}\sigma W_\gamma(a,z)} x^{|\gamma|},
\end{equation*}
where $\sigma:=\tfrac58$ and $W_\gamma(u,v)$ is equal to $\tfrac\pi3$ times the number of left turns minus the number of right turns  made by the walk $\gamma$ when going from $u$ to $v$.

\begin{lemma}\label{lem:relation}
For any $v\in\bbV\cap S(T,L)$, 
\begin{equation}\label{relation around vertex}
		(p-v)F(p)+(q-v)F(q)+(r-v)F(r)=0,
	\end{equation}
	where $p,q,r$ are the mid-edges of the three edges incident to $v$.
\end{lemma}

\begin{proof}In this proof, we further assume that the mid-edges $p$, $q$ and $r$ are oriented counterclockwise around $v$. Note that $(p-v)F(p)+(q-v)F(q)+(r-v)F(r)$ is a sum of ``contributions''
$$c(\gamma)=(z-v){\rm e}^{-{\rm i}\sigma W_{\gamma}(a,z)} x^{|\gamma|}$$
 over all possible walks $\gamma$ finishing at $z\in\{p,q,r\}$. The set of such walks can be partitioned into pairs and triplets of walks in the following way, see Fig \ref{fig:pairs}:
\medbreak\noindent
{\em Walks visiting the three mid-edges $p$, $q$ and $r$ can be grouped in pairs:}  If a walk $\gamma_1$ visits all three mid-edges, it means that the edges belonging to $\gamma_1$ form a self-avoiding path up to $v$ plus (up to a half-edge) a self-avoiding loop from $v$ to $v$. One can associate to $\gamma_1$ the walk passing through the same edges, but exploring the loop from $v$ to $v$ in the other direction. 
\medbreak\noindent
{\em Walks not visiting the three mid-edges $p$, $q$ and $r$ can be grouped in triplets:}
 If a walk $\gamma_1$ visits only one mid-edge, it can be grouped with two walks $\gamma_2$ and $\gamma_3$ that visit exactly two mid-edges by prolonging the walk one step further (there are two possible choices). The reverse is true: a walk visiting exactly two mid-edges belongs to the group of a walk visiting only one mid-edge (this walk is obtained by erasing the last step). 
 \medbreak
If the sum of contributions for each pair and each triplet described above vanishes, then the total sum is zero. 
We now intend to show that this is the case. 

Let $\gamma_1$ and $\gamma_2$ be two walks that are grouped as in the first case. Without loss of generality, we assume that $\gamma_1$ ends at $q$ and $\gamma_2$ ends at $r$. Since $\gamma_1$ and $\gamma_2$ coincide up to the mid-edge $p$ (they are matched together), we deduce that 
$|\gamma_1|=|\gamma_2|$
and 
\begin{align*}W_{\gamma_1}(a,q)&=W_{\gamma_1}(a,p)+W_{\gamma_1}(p,q)=W_{\gamma_1}(a,p)-\tfrac{4\pi}{3},\\
	W_{\gamma_2}(a,r)&=W_{\gamma_2}(a,p)+W_{\gamma_2}(p,r)=W_{\gamma_1}(a,p)+\tfrac{4\pi}{3}.\end{align*}
In order to evaluate the winding of $\gamma_1$ between $p$ and $q$, we used the fact that $a$ is on the boundary of $S(T,L)$ so that the walk does necessarily four more turns on the right than turns on the left between $p$ and $q$. Altogether,
\begin{align*}c(\gamma_1)+c(\gamma_2)&=(q-v){\rm e}^{-{\rm i}\sigma W_{\gamma_1}(a,q)} x^{|\gamma_1|}+(r-v){\rm e}^{-{\rm i}\sigma W_{\gamma_2}(a,r)} x^{|\gamma_2|}\\
&=(p-v){\rm e}^{-{\rm i}\sigma W_{\gamma_1}(a,p)}x^{|\gamma_1|}\left(j\bar{\lambda}^4+\bar{j}\lambda^4\right)=0
\end{align*}
where $j={\rm e}^{{\rm i}2\pi/3}$ and $\lambda=\exp (-{\rm i}5\pi/24)$ (here we use the crucial choice of $\sigma=\tfrac58$). 

Let $\gamma_1,\gamma_2,\gamma_3$ be three walks matched as in the second case. Without loss of generality, we assume that $\gamma_1$ ends at $p$ and that $\gamma_2$ and $\gamma_3$ extend $\gamma_1$ to $q$ and $r$ respectively. As before, we easily find that
$|\gamma_2|=|\gamma_3|=|\gamma_1|+1$
and 
\begin{align*}W_{\gamma_2}(a,q)&=W_{\gamma_2}(a,p)+W_{\gamma_2}(p,q)=W_{\gamma_1}(a,p)-\tfrac{\pi}{3},\\
	W_{\gamma_3}(a,r)&=W_{\gamma_3}(a,p)+W_{\gamma_3}(p,r)=W_{\gamma_1}(a,p)+\tfrac{\pi}{3}.\end{align*}Following the same steps as above, we obtain
\begin{align*}c(\gamma_1)+c(\gamma_2)+c(\gamma_3)&=(p-v)e^{-{\rm i}\sigma W_{\gamma_1}(a,p)}x^{|\gamma_1|}\left(1+xj\bar{\lambda}+x\bar{j}\lambda\right)=0.
\end{align*}
Here is the \emph{only} place where we use the crucial fact that $x^{-1}=\sqrt{2+\sqrt2}=2\cos \frac{\pi}{8}$. 
The claim follows readily by summing over all pairs and triplets.
\end{proof}

\bexo
\begin{exercise}[Parafermionic observable for the loop $O(n)$-model]
Consider the loop $O(n)$ model defined as follows. Let $E(\Omega)$ be the set of even subgraphs of $\Omega\subset\bbH$ (equivalently, these are the families of non-intersecting loops). Also, let $E_{a,z}(\Omega)$ be the family of loops, plus one self-avoiding walk $\gamma(\omega)$ going from $a$ to $z$ not intersecting any of the loops.
Define the {\em parafermionic observable} $$ F(z)=\sum_{\omega\in E_{a,z}(\Omega)}{\rm e}^{-{\rm i}\sigma W_\gamma(a,z)}x^{|\omega|}n^{\ell(\omega)},$$
where $|\omega|$ is the total length of the loops and the self-avoiding walk, and $\ell(\omega)$ is the number of loops. Note that this model generalizes both the self-avoiding walk ($n=0$) and the Ising model on the hexagonal lattice ($n=1$) via the high-temperature expansion.
\medbreak\noindent
Show that for $n\in[0,2]$, there exist two values of $\sigma$, and for each one a single value of $x$ such that $F$ satisfies \eqref{relation around vertex}. {\em The smallest of the two values of $x$ is conjectured by Nienhuis to be the critical point of the model.} 
\end{exercise}

\eexo
\begin{figure}[ht]
    \begin{center}
      \includegraphics[width=0.90\hsize]{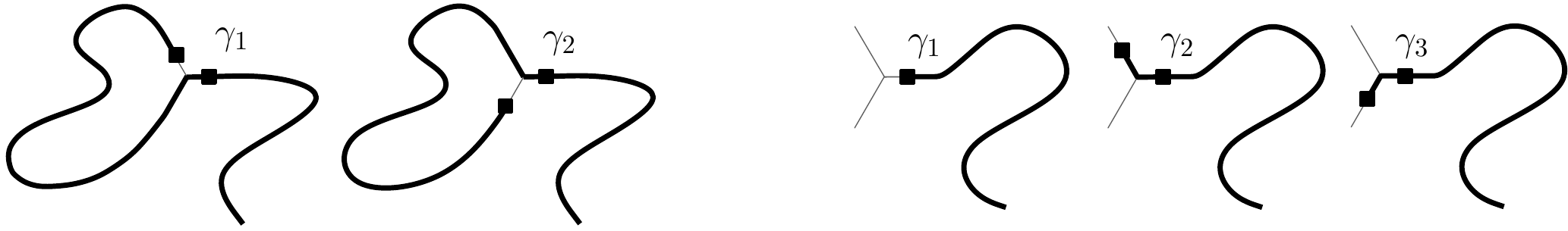}
     \end{center}
    \caption{\textbf{Left: }a pair of walks visiting the three mid-edges and matched together. \textbf{Right:} a triplet of walks, one visiting one mid-edge, the other two visiting two mid-edges, which are matched together.}
    \label{fig:pairs} 
  \end{figure} 

\begin{proof}[Lemma~\ref{lem:fundamental}]
Sum the relation \eqref{relation around vertex} over all $v\in \bbV\cap S(T,L)$. Values at interior mid-edges cancel and we end up with
\begin{equation}\label{sum}
	0=-\sum_{z\in \alpha}F(z)+\sum_{z\in \beta}F(z)+j\sum_{z\in \ep}F(z)+\bar{j}\sum_{z\in \bar{\ep}}F(z),
\end{equation}
where $j={\rm e}^{2{\rm i}\pi/3}$.
Using the symmetry of the domain with respect to the $x$ axis, we deduce that $F(\bar{z})=\bar{F}(z)$. Observe that the winding of any self-avoiding walk from $a$ to the bottom part of $\alpha$ is $-\pi$ while the winding to the top part is $\pi$. We conclude
\begin{align*}
	\sum_{z\in \alpha}F(z)&=F(a)+\sum_{z\in \alpha\setminus \{a\}}F(z)=1+\frac{{\rm e}^{-{\rm i}5\pi/8}+{\rm e}^{{\rm i}5\pi/8}}{2}A_{T,L}=1-\cos \left(\tfrac{3\pi}{8}\right)A_{T,L}.
\end{align*}
Above, we have used the fact that the only walk from $a$ to $a$ is of length $0$. Similarly, the winding from $a$ to any half-edge in $\beta$ (resp. $\ep$ and $\bar{\ep}$) is 0 (resp. $\frac {2\pi}3$ and $-\frac{2\pi}3$), therefore 
$$\sum_{z\in \beta}F(z)=B_{T,L}\quad \text{and}\quad j\sum_{z\in \ep}F(z)+\bar{j}\sum_{z\in \bar{\ep}}F(z)=\cos  \left(\tfrac\pi 4\right)E_{T,L}.$$
The lemma follows readily by plugging these three formul\ae\ in \eqref{sum}.
\end{proof}

The proof of Lemma~\ref{lem:fundamental} can be understood in the following way. Coefficients in \eqref{relation around vertex} are three cubic roots of unity multiplied by $p-v$, so that the left-hand side can be seen as a discrete integral along an elementary contour on the dual lattice in the following sense. For a closed path $c=(z_i)_{i\le n}$ of vertices in the triangular lattice $\bbT$ dual to $\bbH$, define the discrete integral of a function $F$ on mid-edges by
	\begin{equation}\oint_{c}F(z)dz:=\sum_{i=0}^{n-1}F\left(\tfrac{z_i+z_{i+1}}{2}\right)(z_{i+1}-z_i).\end{equation}
	Equation~\eqref{relation around vertex} at $v\in\bbV$ implies that the discrete contour integral going around the face of $\bbT$ corresponding to $v$ is zero. Decomposing a closed path into a sum of elementary triangles gives  that the discrete integral along any closed path vanishes.
		
	The fact that the integral of the parafermionic observable along closed path vanishes is a glimpse of conformal invariance of the model in the sense that the observable satisfies a weak notion of discrete holomorphicity. Nevertheless, these relations do not uniquely determine $F$.
	 Indeed, the number of mid-edges (and therefore of unknown variables) exceeds the number of linear relations \eqref{relation around vertex} (which corresponds to the number of vertices).
	 	 Nonetheless, one can combine the fact that the discrete integral along the exterior boundary of $S(T,L)$ vanishes with the fact that the winding of self-avoiding walks ending at boundary mid-edges is deterministic and explicit. This extra information is sufficient to derive some non-trivial information on the model. In the next section, we will use a similar idea in the case of random-cluster models.

\begin{figure}
\begin{center}
\includegraphics[width=0.80\textwidth]{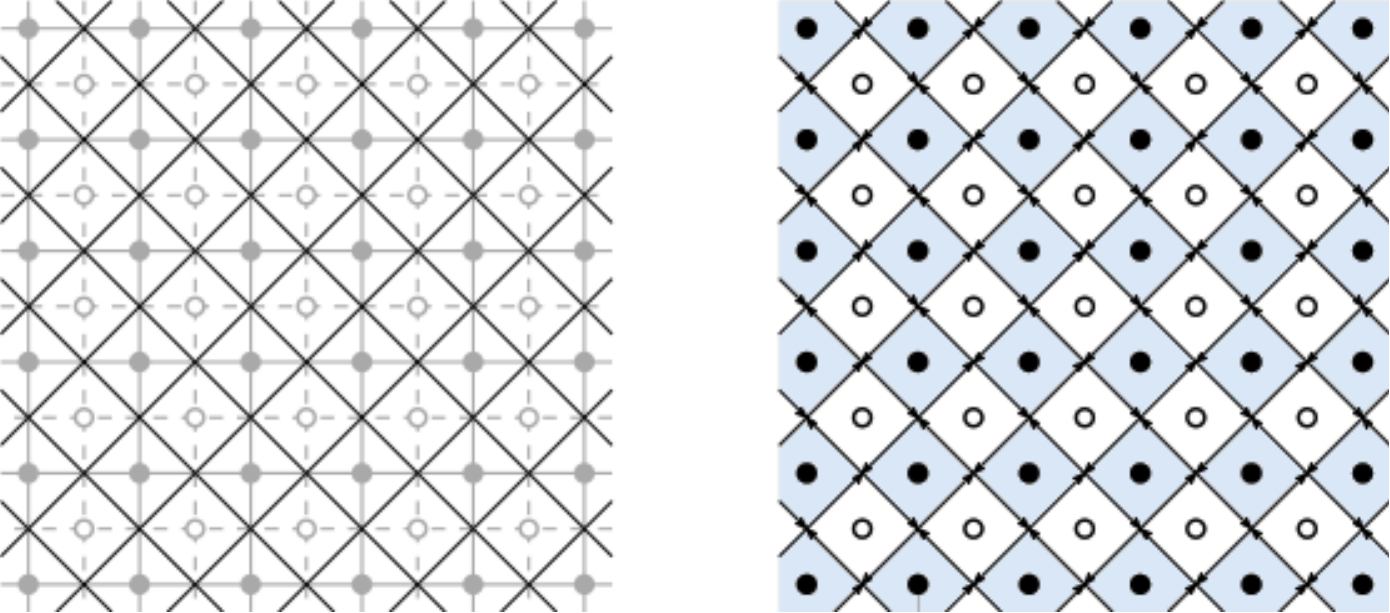}
\caption{On the left, the lattice $\bbZ^2$, its dual lattice $(\bbZ^2)^*$ and medial lattice $(\bbZ^2)^\diamond$. On the right, a natural orientation on the medial lattice.}\label{fig:lattices}\end{center}
\end{figure}

\begin{figure}
\begin{center}
\includegraphics[width=0.9\textwidth]{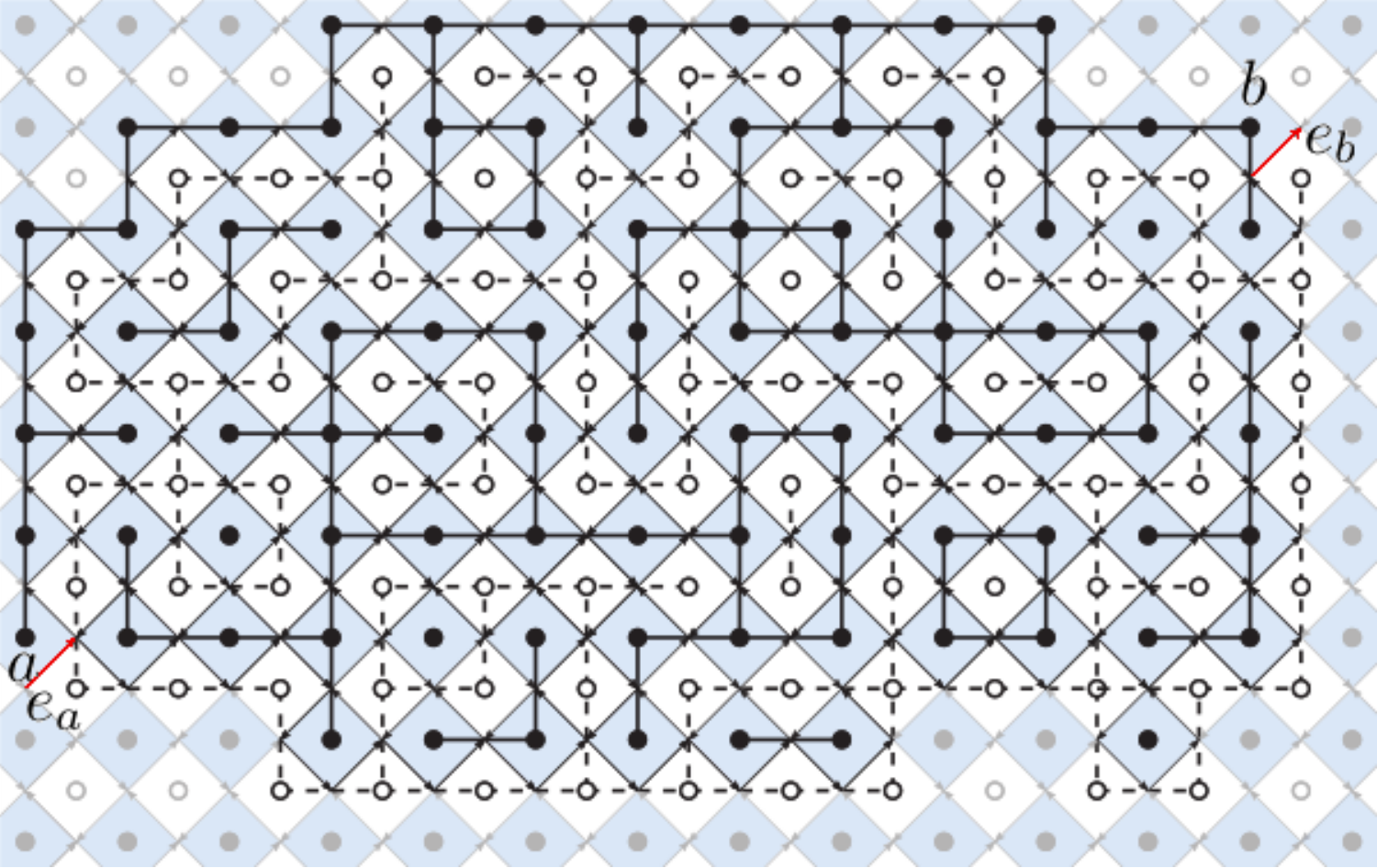}
\caption{\label{fig:primal domain} The configuration $\omega$ (in bold lines) with its dual configuration $\omega^*$ (in dashed lines). Notice that the edges of $\omega$ are open on $(ba)$, and that those of $\omega^*$ are open on $(ab)^*$.}
\end{center}
\begin{center}
\includegraphics[width=0.9\textwidth]{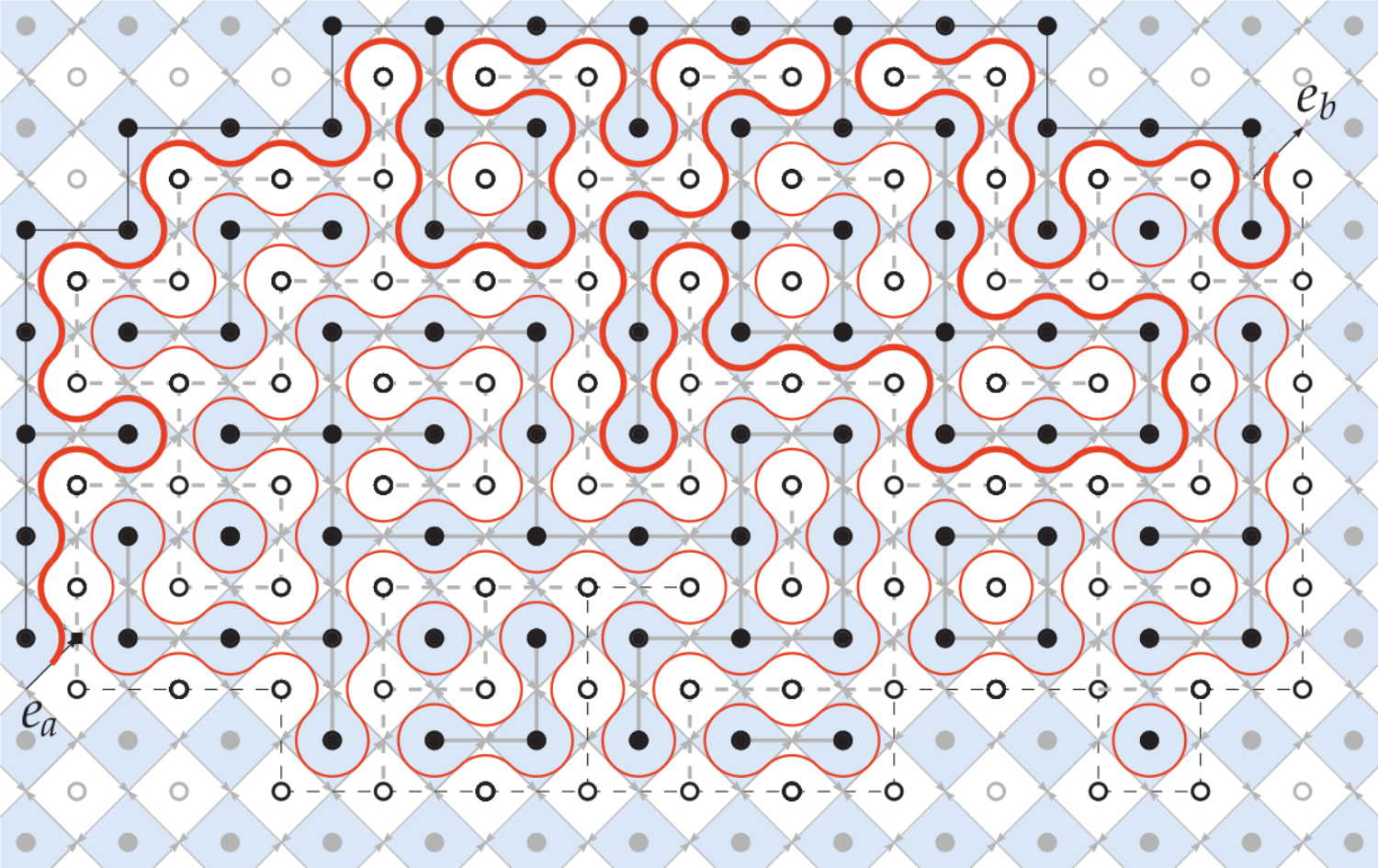} 

\caption{\label{fig:3} The loop configuration $\overline\omega$ associated to the primal and dual configurations $\omega$ and $\omega^*$ in the previous picture. The exploration path is drawn in bold. It starts at $e_a$ and finishes at $e_b$. }
\end{center}
\end{figure}

\subsubsection{The loop representation and the parafermionic observable}

In order to define parafermionic observables for random-cluster models, we first discuss the loop representation of the  model.
\medbreak
{\em In the definitions below, we recommend to look at Figures~\ref{fig:lattices}, \ref{fig:primal domain} and \ref{fig:3}.}
\medbreak

Let $\Omega$ be a connected graph with connected complement in $\bbZ^2$, and $a$ and $b$ two vertices on its boundary. The triplet $(\Omega,a,b)$ is called a Dobrushin domain. The set $\partial\Omega$ is divided into two boundary arcs denoted by $(ab)$ and $(ba)$: the first one goes from $a$ to $b$ when going counterclockwise around $\partial\Omega$, while the second goes from $b$ to $a$. The {\em Dobrushin boundary conditions} are
defined to be free on $(ab)$ and wired on $(ba)$. In other words, the partition is composed of $(ba)$ together with singletons. Note that the state of edges on $(ba)$ is now irrelevant since the vertices of $(ba)$ are wired together anyway. We will therefore consider that edges on $(ba)$ are not in $\Omega$ (this will be relevant when defining $\Omega^*$). Also, the Dobrushin boundary conditions are planar, and it is therefore convenient to choose a configuration $\xi$ inducing them. We set $\xi_e=0$ for all $e\in \bbE\setminus E$ except for edges on $(ba)$, for which $\xi_e=1$. 
Below, the measure on $(\Omega,a,b)$ with Dobrushin boundary conditions is denoted by $\phi^{a,b}_{\Omega,p,q}$.  

Let $\Omega^*$ be the dual of the graph $\Omega$ (recall that edges in $(ba)$ are not part of $\Omega$ anymore).  We draw the dual configuration $\omega^*$ with the additional condition that edges between vertices of $\partial\Omega^*$ that are bordering $(ab)$ are open in $\omega^*$ (we call the set of such edges $(ab)^*$). This is coherent with the duality relation since the dual boundary conditions of the Dobrushin ones are induced by the configuration $\xi^*$ equal to 1 on $(ab)^*$, and 0 elsewhere. Keep in mind that from this point of view, primal and dual models play symmetric roles with respect to Dobrushin boundary conditions.

 We now explain how to construct the loop configuration, which is defined on another graph, called the {\em medial graph}. This graph is defined as follows. Let $(\bbZ^2)^\diamond$ be the {\em medial} lattice defined as follows. The set of vertices is given by the midpoints of edges of $\bbZ^2$. The edges are the pairs of nearest vertices (i.e.~vertices at distance $\sqrt 2/2$ of each others). It is a rotated and rescaled version of $\bbZ^2$, see Fig.~\ref{fig:lattices}.  For future reference, note that the edges of the medial lattice can be oriented in a counter-clockwise way around faces that are centered on a vertex of $\bbZ^2$ (the dark faces on Fig.~\ref{fig:lattices}). Let $\Omega^\diamond$ be the subgraph of $(\bbZ^2)^\diamond$ made of vertices corresponding to an edge of $\Omega$ or $\Omega^*$. Let $e_a$ and $e_b$ be the two medial edges entering and exiting $\Omega^\diamond$ between the arc $(ba)$ and $(ab)^*$ (see Fig.~\ref{fig:primal domain}).

Draw self-avoiding loops on $\Omega^\diamond$ as follows: a loop arriving at a vertex of the medial lattice 
always makes a $\pm \pi/2$ turn at vertices so as not to cross the
edges of $\omega$ or $\omega^*$, see Fig.~\ref{fig:3}. The loop configuration is defined in an unequivocal way since:
\begin{itemize}
\item there is either an edge of $\omega$ or an edge of $\omega^*$ crossing non-boundary vertices in $\Omega^\diamond$, and therefore there is exactly one coherent way for the loop to turn at non-boundary vertices.
\item the edges of $\omega$ in $(ba)$ and the edges of $\omega^*$ in $(ab)^*$ are such that the loops at boundary vertices turn in order to remain in $\Omega^\diamond$.\end{itemize}
From now on, the loop configuration associated to $\omega$ is denoted by $\overline \omega$. 
Beware that the denomination is slightly misleading: $\overline\omega$ is made of loops together with a self-avoiding path going from $e_a$ to $e_b$, see Figures~\ref{fig:3}. This curve is called the 
\emph{exploration path} and is denoted by $\gamma=\gamma(\omega)$.  

We allow ourselves a slight abuse of notation: below, $\phi_{\Omega,p,q}^{a,b}$ denotes the measure on percolation configurations as well as its push-forward by the map $\omega\mapsto \overline\omega$. Therefore, the measure $\phi^{a,b}_{\Omega,p,q}$ will sometimes refer to a measure on loop configurations.

\begin{proposition}\label{prop:loop}Let $\Omega$ be a connected finite subgraph of $\bbZ^2$  connected complement in $\bbZ^2$. Let $p\in[0,1]$ and $q>0$. For any configuration $\omega$,
$$\phi_{\Omega,p,q}^{a,b}[\overline\omega]~=~\frac{x^{o(\omega)}\sqrt q^{\ell(\overline\omega)}}{\overline Z_{\Omega,p,q}},
$$
where $x:=\frac{p}{\sqrt q (1-p)}$, $\ell(\overline\omega)$ is the number of loops\footnote{The exploration path $\gamma$ is considered as a loop and counts as 1 in $\ell(\overline\omega)$.} in $\overline\omega$ and $\overline Z_{\Omega,p,q}$ is a normalizing constant.\end{proposition}

In particular, $x=1$ when $p=p_c(q)$ and the probability of a loop configuration is expressed in terms of the number of loops only. 

\begin{proof}
Let $v$ be the number of vertices of the graph $\Omega$ where $(ba)$ has been contracted to a point. An induction on the number of open edges shows that
\begin{equation}\label{eq:ab}\ell(\overline\omega)=2k(\omega)+o(\omega)-v.\end{equation}
Indeed, if there is no open edge, then $\ell(\overline \omega)=k(\omega)=v$ since there is a loop around each one of the vertices of $\Omega\setminus(ba)$, and one exploration path. Now, adding an edge can either:
\begin{itemize}[noitemsep,nolistsep]
\item join two clusters of $\omega$, thus decreasing both the numbers of loops and clusters by 1,
\item close a cycle in $\omega$, thus increasing the number of loops by 1 and not changing the number of clusters.
\end{itemize}
Equation \eqref{eq:ab} implies that
\begin{align*}p^{o(\omega)}(1-p)^{c(\omega)}q^{k(\omega)}&=p^{o(\omega)}(1-p)^{|E|-o(\omega)}q^{k(\omega)}\\
&=(1-p)^{|E|}\sqrt q^{v}\big(\tfrac{p}{(1-p)\sqrt q}\big)^{o(\omega)}\sqrt q^{2k(\omega)+o(\omega)-v}\\
&=(1-p)^{|E|}\sqrt q^{v}x^{o(\omega)}\sqrt q^{\ell(\overline\omega)}.
\end{align*}
The proof follows readily.
\end{proof}

We are now ready to define the parafermionic observable.
Recall that $\gamma=\gamma(\omega)$ is the exploration path in the loop configuration $\overline\omega$. The \emph{winding} 
$\text{W}_{\gamma}(e,e')$ of the exploration path $\gamma$ between two medial-edges $e$ and 
$e'$ of the medial graph is equal to $\pi/2$ times the number of left turns minus the number of right turns  done by the curve between $e$ and $e'$. When $e$ or $e'$ are not on $\gamma$, we set the winding to be equal to 0. 
\begin{definition}\label{def:parafermionic observable}
Consider a Dobrushin domain $(\Omega,a,b)$. 
 The {\em parafermionic observable} $F=F(\Omega,p,q,a,b)$  is defined for any (medial) edge $e$ of $\Omega^\diamond$ by
\begin{equation*}
  F(e) ~:=~\phi^{a,b}_{\Omega,p,q}[{\rm e}^{{\rm i}\sigma 
  \text{W}_{\gamma}(e,e_b)} \mathbbm 1_{e\in \gamma}],\end{equation*}
where $\sigma$ is a solution of the equation
\begin{equation}\label{eq:hahaha}\displaystyle  \sin (\sigma \pi/2) = \sqrt{q}/2.\end{equation}
\end{definition}

Note that $\sigma$ belongs to $\bbR$ for $q\le 4$ and to $1+i\bbR$ for $q>4$. This suggests that the critical behavior of random-cluster model is different for $q>4$ and $q\le 4$. For $q\in[0,4]$, $\sigma$ has the physical interpretation of a spin, which is fractional in general, hence the name parafermionic\footnote{Fermions have half-integer spins while bosons have integer spins, there are no particles with fractional spin, but the use of such fractional spins at a theoretical level has been very fruitful in physics.}. For $q>4$, $\sigma$ is not real anymore and does not have any physical interpretation. 

These observables first appeared in the context of the Ising model (there they are called order-disorder operators) and dimer models. They were later on extended to the random-cluster model and the loop $O(n)$-model by Smirnov \cite{Smi06} (see \cite{DumSmi12a} for more details). Since then, these observables have been at the heart of the study of these models. They also appeared in a slightly different form in several physics papers going back to the early eighties \cite{FraKad80,BerLeC91}. They have been the focus of much attention in recent years: physicists exhibited such observables in a large class of models of two-dimensional statistical physics \cite{IkhCar09,RajCar07,RivCar06,Car09,IkhWesWhe13}.

\subsubsection{Contour integrals of the parafermionic observable}\label{sec:contours}

The parafermionic observable satisfies a very special property at criticality. 

\begin{theorem}[Vanishing contour integrals]\label{thm:contours}
Fix  $q>0$ and $p=p_c$, For any Dobrushin domain $(\Omega,a,b)$ and any vertex of $\Omega^\diamond$ with four incident edges in $\Omega^\diamond$,
\begin{equation}F(e_1)-F(e_3)={\rm i}F(e_2)-{\rm i}F(e_4),\label{rel_vertex}\end{equation}
where $e_1$, $e_2$, $e_3$ and $e_4$ are the four edges incident to this vertex, indexed in counterclockwise order.\end{theorem}
As in the case of the self-avoiding walk, interpret \eqref{rel_vertex} as follows: the integral of $F$ along a small square around a face is equal to 0. One may also sum this relation on every vertex to obtain that discrete contour integrals vanish.

 \begin{figure}[ht]
    \begin{center}
      \includegraphics[width=0.70\textwidth]{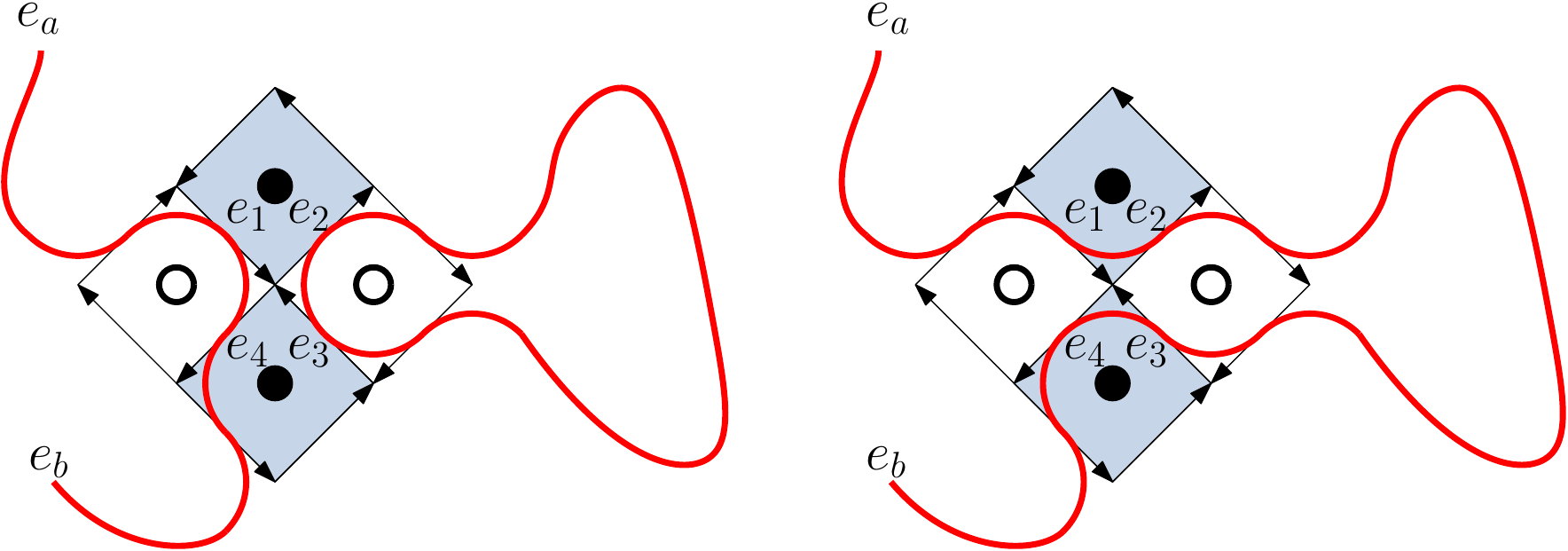}
    \end{center}
    \caption{    \label{fig:configuration} {\bf Left.} The neighborhood of $v$ for two associated configurations $\omega$ and $\omega'$.}

  \end{figure}

\begin{proof}We follow a strategy close to the proof of Lemma~\ref{lem:fundamental} and pair configurations in such a way that sums of contributions cancel.

  Let $e$ be an edge of $\Omega^\diamond$ and let
  \begin{eqnarray*}\mathsf X_e(\omega) &:=& 
 {\rm e}^{{\rm 
  i}\sigma {\rm W}_{\gamma(\omega)}(e,e_b)} \mathbbm 1_{e\in \gamma(\omega)}  \phi_{\Omega,p_c,q}^{a,b}[\omega] 
  \end{eqnarray*}
 be the contribution 
  of the configuration $\omega$ to $F(e)$.  Let 
  $\omega'$ be the configuration obtained from $\omega$ by switching the state  open or 
  closed of the edge in $\omega$ passing through $v$. 
Since $\omega\mapsto\omega'$ is an involution, the following 
  relation holds: $$F(e)~=~\sum_{\omega} \mathsf X_e(\omega)~=~\tfrac{1}{2} 
  \sum_{\omega} \left[ \mathsf X_e(\omega)+\mathsf X_e(\omega') \right]\!.$$ To 
  prove \eqref{rel_vertex}, it is thus sufficient to show that for any configuration $\omega$, \begin{equation}
    \label{cd}
    \mathsf X_{e_1}(\omega) + \mathsf X_{e_1}(\omega') - \mathsf X_{e_3}(\omega) - \mathsf X_{e_3}(\omega') ~=~{\rm i} [
   \mathsf X_{e_2}(\omega)+ \mathsf X_{e_2}(\omega')- 
    \mathsf X_{e_4}(\omega) -
   \mathsf X_{e_4}(\omega')].
  \end{equation}
   
  \noindent There are three possible cases:
\paragraph{Case 1.}No edge incident to $v$ belongs to $\gamma(\omega)$. Then, none of these edges is incident to $\gamma(\omega')$ either. For any $e$ incident to $v$, the contribution to \eqref{cd} is equal to 0 so that \eqref{cd} trivially holds.  
  
\paragraph{Case 2.} Two edges incident to $v$ belong to $\gamma(\omega)$, see Fig.~\ref{fig:configuration}. Since $\gamma(\omega)$ and the medial lattice possess a natural orientation, $\gamma(\omega)$ enters through either $e_1$ or $e_3$ and leaves 
  through $e_2$ or $e_4$. Assume that $\gamma(\omega)$ 
  enters through the edge $e_1$ and leaves through the edge 
  $e_4$. It is then
  possible to compute the contributions for $\omega$ and $\omega'$ of all the edges incident to $v$ in terms of $\mathsf X=\mathsf X_{e_1}(\omega)$. Indeed,
since $\omega'$ has one less loop, we find
    \begin{align*}\phi_{\Omega,p_c,q}^{a,b}[\omega']&=\tfrac{1}{\sqrt q}\phi_{\Omega,p_c,q}^{a,b}[\omega].\end{align*}
 Furthermore, windings of $\gamma(\omega)$ and $\gamma(\omega')$ at $e_2$, $e_3$ and $e_4$ can be expressed using the winding at 
      $e_1$ (for instance, $W_{\gamma(\omega)}(e_2,e_b)=W_{\gamma(\omega)}(e_1,e_b)-\pi/2$
  -- the other cases are treated similarly). 
  The contributions are given in the following table.
  \begin{center}\begin{tabular}{|c|c|c|c|c|}
    \hline
    configuration  &  $e_1$ &  $e_2$ &  $e_3$ & $e_4$\\
     \hline
    $\omega$ & $\mathsf X$ & 0 & 0 & ${\rm e}^{{\rm 
    i}\sigma\pi/2}\mathsf X$ \\
    \hline
    $\omega'$ & $\frac{\mathsf X}{\sqrt q}$ & ${\rm e}^{{\rm i}\sigma\pi}\frac{\mathsf X}{\sqrt q}$ & ${\rm 
    e}^{-{\rm i}\sigma\pi/2}\frac{\mathsf X}{\sqrt q}$ & ${\rm e}^{{\rm i}\sigma\pi/2}\frac{\mathsf X}{\sqrt q}$\\
    \hline
  \end{tabular}\end{center}
  Using the identity ${\rm e}^{{\rm i}\sigma\pi/2}-{\rm e}^{-{\rm 
  i}\sigma\pi/2}={\rm i}\sqrt{q}$, we deduce \eqref{cd} by summing (with the right weight) the contributions 
  of all the edges incident to $v$.
  
\paragraph{Case 3.} The four edges incident to $v$ belong to $\gamma(\omega)$. Then only two of these edges belong to $\gamma(\omega')$ and the computation is similar to Case 2 by exchanging the weights of $\omega'$ and $\omega$.
\medbreak
 In conclusion, \eqref{cd} is always satisfied and the claim is proved. 
\end{proof}

\subsubsection{Continuous phase transition for random-cluster models with $q\in[1,4]$}

This section is devoted to the proof of the following result.
\begin{theorem}[DC  \cite{Dum12}]\label{thm:decide}
For $q\in[1,4]$, the property {\bf P4a} is satisfied.
\end{theorem}
As a consequence, we deduce from Theorem~\ref{thm:main} that the properties {\bf P1}--{\bf P5} also are\footnote{We did not prove that {\bf P4b} implies {\bf P5}, but since {\bf P4a} implies {\bf P4b} and {\bf P5}, this follows readily.}. This gives ``one half'' of Theorem~\ref{thm:continuous RCM}. 

We first focus on the case $q\le 2$. The proof follows an argument similar to the computation for self-avoiding walks: we will use that the discrete contour integral along the boundary of a domain vanishes together with the fact that windings are deterministic on the boundary.

\begin{proof}[Theorem~\ref{thm:decide} in the case $q\in{[1,2]}$] In this proof, the first and second coordinates of a vertex $x\in\bbZ^2$ are denoted by $x_1$ and $x_2$. Also, define $\widetilde \Lambda_n:=\{x\in\bbZ^2:|x_1|+|x_2|\le n\}$. 

Fix $n$ odd. Consider a degenerated case of Dobrushin domain in which $$\Omega:=\{x\in \widetilde\Lambda_n\text{ such that }x_1+x_2\le 0\}$$
and $(ba)=\{0\}$ as well as $(ab)=\partial\Omega$. In this case, the parafermionic observable $F$ still makes sense: $e_a$ and $e_b$ are the edges of $\Omega^\diamond$ north-west and south-east of $0$, and $\gamma(\omega)$ is the loop going around $0$ (and therefore through $e_a$ and $e_b$). Note that, by definition, the Dobrushin boundary conditions are coinciding with the free boundary conditions in this context since the arc $(ba)$ is restricted to a point. 

Summing \eqref{rel_vertex} on every vertex $v\in\Omega^\diamond$, we obtain that 
\begin{equation*}
\sum_{e\in\alpha}F(e)=\sum_{e\in\beta}F(e)+{\rm i}\sum_{e\in \ep}F(e)-{\rm i}\sum_{e\in \bar{\ep}}F(e),
\end{equation*}
where $\alpha$, $\ep$, $\beta$ and $\overline{\ep}$ are respectively the sets of medial edges intersecting the north-east, north-west, south-west and south-east boundaries of $\Omega^\diamond$. This immediately leads to
\begin{equation}\label{eq:popo}
\Big|\sum_{e\in\alpha}F(e)\big|\le \sum_{e\notin \alpha}|F(e)|,
\end{equation}
where the sum on the right is on edges of $\Omega^\diamond$ intersecting the boundary only. Any such edge $e$ is bordering a vertex $x\in\partial\Omega$. Also, $\gamma(\omega)$ goes through $e$ if and only if $x$ and 0 are connected by a path of edges in $\omega$. We deduce that 
\begin{equation}\label{eq:uu}|F(e)|=\phi^0_{\Omega,p_c,q}[0\longleftrightarrow x]\stackrel{{\rm (CBC)}}\le \phi^0_{p_c,q}[0\longleftrightarrow x].
\end{equation}
Since there are exactly two medial edges bordering a prescribed vertex, and that each such vertex $x$ is in $\partial\widetilde\Lambda_n$, \eqref{eq:popo} becomes
\begin{equation}\label{eq:ii}
\Big|\sum_{e\in\alpha}F(e)\big|\le 2\sum_{x\in\partial\widetilde\Lambda_n}\phi^0_{p_c,q}[0\longleftrightarrow x].
\end{equation}
Let us now focus on the term on the left. First, note that since $\gamma(\omega)$ deterministically goes through $e_a$ and $e_b$, we get
\begin{equation}\label{eq:okij}F(e_a)+F(e_b)=1+{\rm e}^{{\rm i}\pi \sigma}=2\cos(\tfrac\pi2\sigma){\rm e}^{{\rm i}\sigma\pi/2}.\end{equation}
Second, pick an edge $e\in \alpha\setminus\{e_a,e_b\}$. Since the winding of the loop is deterministic, we may improve the equality in \eqref{eq:uu} into
\begin{equation}F(e)={\rm e}^{{\rm i}\sigma W(e)}\phi^0_{\Omega,p_c,q}[0\longleftrightarrow x],\label{eq:hu}\end{equation}
where $x$ is the vertex of $\partial\Omega$ bordered by $e$, and $W(e)\in\{-\pi,0,\pi,2\pi\}$ depending on which side of 0 the edge $e$ is, and whether it is pointing inside or outside of $\Omega^\diamond$.

Define $$S:=\{x\in\partial\Omega\setminus \partial\widetilde\Lambda_n:x_1>0\}.$$ By gathering the contributions of edges bordering a vertex $x\in S$ and its symmetric $-x$, and using the symmetry of $\Omega$ with respect to the line $x_1=x_2$, we deduce from \eqref{eq:okij} and the previous displayed equation that
\begin{align*}
\sum_{e\in\alpha\setminus\{e_a,e_b\}}F(e)
&=\sum_{x\in S}({\rm e}^{2{\rm i}\pi\sigma}+{\rm e}^{{\rm i}\pi\sigma}+1+{\rm e}^{-{\rm i}\pi\sigma})\phi_{\Omega,p_c,q}^0[0\longleftrightarrow x]\\
&=\tfrac{\sin(\sigma2\pi)}{\sin(\sigma\pi/2)}{\rm e}^{{\rm i}\sigma\pi/2}\sum_{x\in S}\phi_{\Omega,p_c,q}^0[0\longleftrightarrow x].
\end{align*}
For $q\in[1,2]$, $\cos(\sigma\pi/2)>0$ and $\frac{\sin(2\pi\sigma)}{\sin(\tfrac\pi2\sigma)}\ge0$. We deduce that
\begin{equation*}
\Big|\sum_{e\in\alpha}F(e)\big|\ge 2\cos(\tfrac\pi2\sigma)>0.
\end{equation*}
Plugging this lower bound in \eqref{eq:ii} and then summing over odd $n$ gives
$$\sum_{x\in\bbZ^2}\phi^0_{p_c,q}[0\longleftrightarrow x]=\infty,$$
which is {\bf P3}. Since {\bf P3} implies {\bf P4a}, the proof follows.\end{proof}

Observe that for $q>2$, the value of $\sigma$ is such that $\sin(2\pi\sigma)$ becomes negative so that we may not conclude directly anymore. One may wonder whether this is just a technical problem, or whether something deeper is hidden behind this. 
It is natural to predict that the following quantity decays like a power law: 
$$\phi^0_{\Omega,p_c,q}[0\longleftrightarrow \partial \Lambda_{n/2}]=n^{-\alpha(q,\pi)+o(1)},$$
where $\alpha(q,\pi)$ is a constant depending on $q$ only ($\pi$ refers to the ``angle of the opening'' of $\Omega$ at 0), and $o(1)$ denotes a quantity tending to 0 as $n$ tends to infinity.
Moreover, one may argue using {\bf P5} (which we believe is true) that the event that $x\longleftrightarrow 0$ in $\Omega$ has a probability close to the probability that 0 and $x$ are connected to distance $n/2$ in $\Omega$ (see also Exercise~\ref{exo:quasi}). For $x$ not too close to the corners, the boundary of $\Omega$ looks like a straight line and it is therefore natural to predict that
$$\phi^0_{\Omega,p_c,q}[0\longleftrightarrow x]=n^{-2\alpha(q,\pi)+o(1)}.$$
Summing over all $x$ (the vertices near the corner do not contribute substantially) we should find 
\begin{equation}\label{eq:convergence}\sum_{\|x\|_1=n}\phi^0_{\Omega,p_c,q}[0\longleftrightarrow x]=  n^{1-2\alpha(q,\pi)+o(1)}.\end{equation}
Now, it is conjectured in physics that 
$$\alpha(q,\pi)=1-2\frac{\arccos(\sqrt q/2)}{\pi}.$$ Therefore, for $q\in(2,4]$, the quantity on the left-hand side of \eqref{eq:convergence} is converging to 0 as $n\rightarrow\infty$ and the strategy consisting in proving that it remains bounded away from 0 is hopeless for $q>2$.

Nevertheless, we did not have to consider a flat boundary near 0 in the first place. For instance, one may consider $\Omega'$ obtained by taking the set of $x=(x_1,x_2)$ with $\|x\|_1\le n$ and $(x_1,x_2)\ne(n,0)$ with $n\ge0$. Then, one expects that
$$\phi^0_{\Omega',p_c,q}[0\longleftrightarrow \partial\Lambda_{n/2}]=n^{-\alpha(q,2\pi)+o(1)},$$
where $\alpha(q,2\pi)$ is a value which is a priori smaller than $\alpha(q,\pi)$ since $\bbS$ is larger ($2\pi$ refers this time to the ``opening angle'' of $\Omega'$ at 0). Therefore, if one applies the same reasoning as above, we may prove that
$$\sum_{\|x\|_1=n}\phi^0_{\Omega',p_c,q}[0\longleftrightarrow x]= n^{1-\alpha(q,\pi)-\alpha(q,2\pi)+o(1)}.$$
In fact, we know how to predict $\alpha(q,2\pi)$: the map $z\mapsto z^2$ maps $\bbR_+^*\times\bbR$ to $\bbR^2\setminus -\bbR_+$, conformal invariance (see Section~\ref{sec:6} for more details) predicts that $\alpha(q,2\pi)=\alpha(q,\pi)/2$. As a consequence, 
$$\sum_{\|x\|_1=n}\phi^0_{\Omega',p_c,q}[0\longleftrightarrow x]= n^{1-\frac32\alpha(q,\pi)+o(1)},$$
so that this quantity can indeed be larger or equal to 1 provided that $q\le 3$. 

The previous discussion remained at the level of predictions. It relies on conformal invariance, which is extremely hard to get, and definitely much more advanced that what we are seeking for. A very good news is that the strategy of the previous proof can indeed be applied to $\Omega'$ instead of $\Omega$ to give that for $q\le 3$, there exists $c=c(q)>0$ such that for any $n\ge1$,
\begin{equation*}\sum_{\|x\|_1=n}\phi^0_{\Omega',p_c,q}[0\longleftrightarrow x]\ge c.\end{equation*}
Since $\Omega'$ is a subset of $\bbZ^2$, the comparison between boundary conditions implies that for any $q\le 3$.
\begin{equation*}\sum_{x\in\bbZ^2}\phi^0_{p_c,q}[0\longleftrightarrow x]=\infty,\end{equation*}
thus extending the result to every $q\le3$. We leave the details to Exercise~\ref{exo:qle3}.
\bexo
\begin{exercise}\label{exo:qle3}
Fill up the details of the $q\le 3$ case by considering $\Omega'$ instead of $\Omega$.
\end{exercise}
\eexo
This reasoning does not directly extend to $q>3$ since $\frac32\alpha(q,\pi)>1$ in this case. Nevertheless, one could consider a graph generalizing $\Omega$ and $\Omega'$ with a ``larger opening than $2\pi$'' at 0. In fact, one may even consider a graph
with ``infinite opening'' at 0 by considering subgraphs of the 
universal cover $\bbU$ of the plane minus a face of $\bbZ^2$, see Fig.~\ref{fig:U}. This is what was done in \cite{Dum12}. The drawback of taking this set $\bbU$ is that it is not a subset of $\bbZ^2$ anymore. Thus, one has to translate the information obtained for the random-cluster model on $\bbU$ into information for the random-cluster model on $\bbZ^2$, which is a priori difficult since there is no easy comparison between the two graphs (for instance the comparison between boundary conditions is not sufficient). This is the reason why in general one obtain {\bf P4a} instead of {\bf P3}.

\begin{figure}
\begin{center}
\includegraphics[width=1\textwidth]{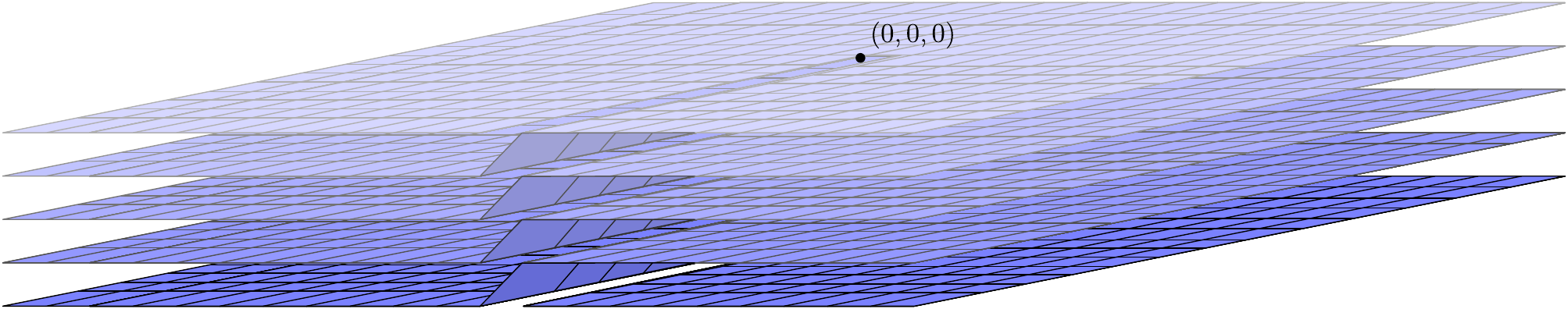}
\caption{\label{fig:U}The graph $\mathbb U$.}
\end{center}
\end{figure}

\subsubsection{Discontinuous phase transition for the random-cluster model with $q>4$}

The goal of this section is to briefly discuss the following theorem. This completes the results of the previous sections and determines the continuous/discontinuous nature of the phase transition for every $q\ge1$. Below, we keep the notation $\widetilde\Lambda_n$ for the box of size $n$ for the graph distance.
\begin{theorem}[DC, Gagnebin, Harel, Manolescu, Tassion \cite{DumGanHar16}]\label{thm:RCM}
	For $q>4$, the properties {\bf P1--5} are not satisfied. In particular
	\begin{equation}\label{eq:aaf}
	\lim_{n\rightarrow\infty} -\tfrac1n\log\phi_{p_c,q}^0[0\longleftrightarrow\partial\widetilde\Lambda_n]=
	\lambda + 2 \sum_{k=1 }^\infty \tfrac{(-1)^k}k \tanh(k\lambda)>0,
	\end{equation}
	where $\lambda>0$ satisfies $\cosh(\lambda)=\frac{\sqrt q}2$.
\end{theorem}
Note that in particular, one may get the asymptotic in \eqref{eq:aaf} as $q\searrow 4$: it behaves asymptotically as $8\exp\left(-\pi^2/\sqrt{q-4}\right)$. Physically, that means that the correlation length of the models explodes very quickly (much faster than any polynomial) as $q$ approaches 4.

Before sketching the ideas involved in the proof of this statement, let us make a small detour and prove that {\bf P1--5} cannot be satisfied for $q\gg1$ (see \cite{Dum17} for details).

  \begin{figure}
  \begin{center}
 \includegraphics[width=0.65\textwidth]{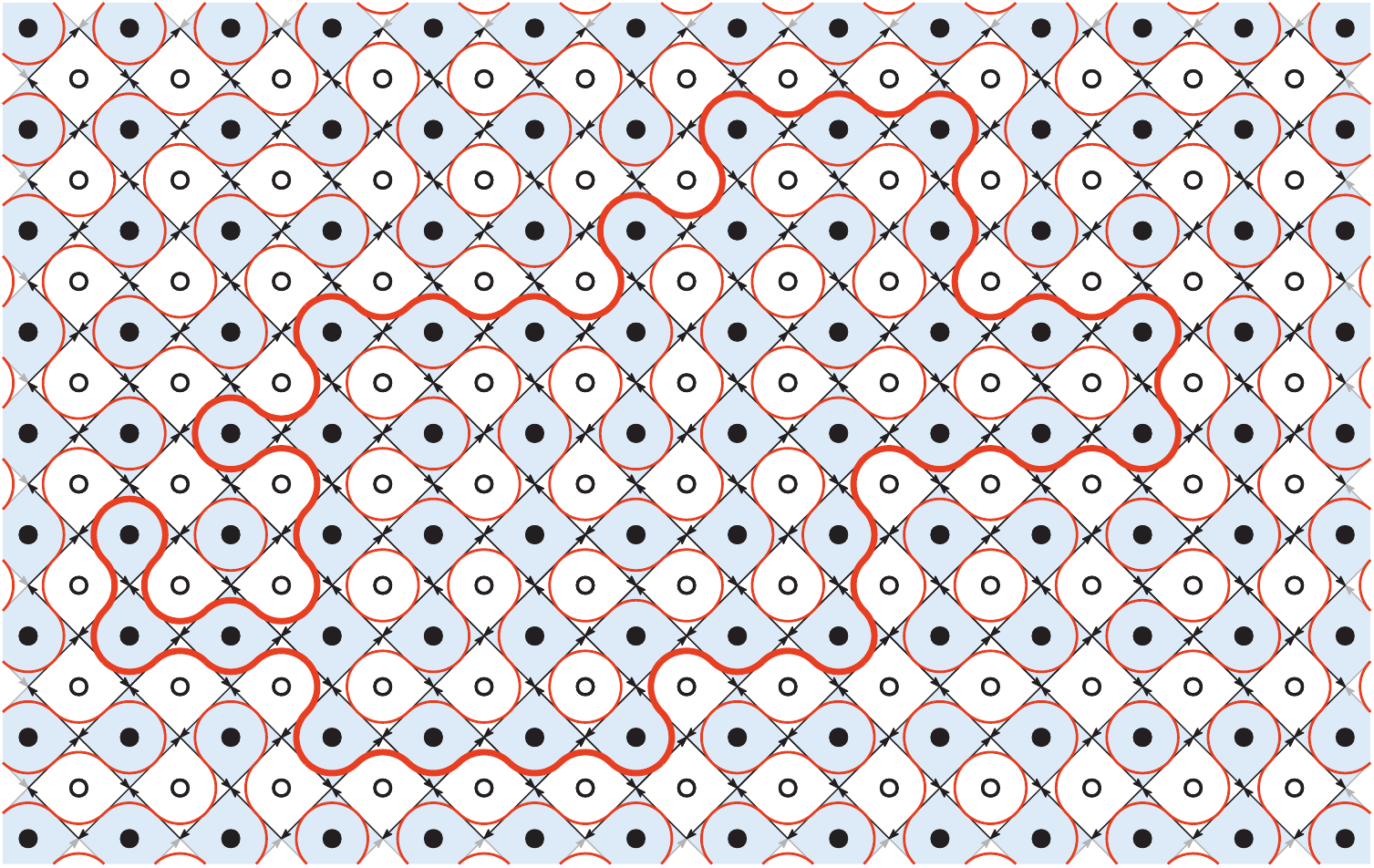}
 \caption{\label{fig:4}Consider a loop configuration $\overline\omega$ containing the loop $L$ (in bold).}
 \end{center}
 \end{figure}
 \begin{figure}
\center{  \includegraphics[width=0.64\textwidth]{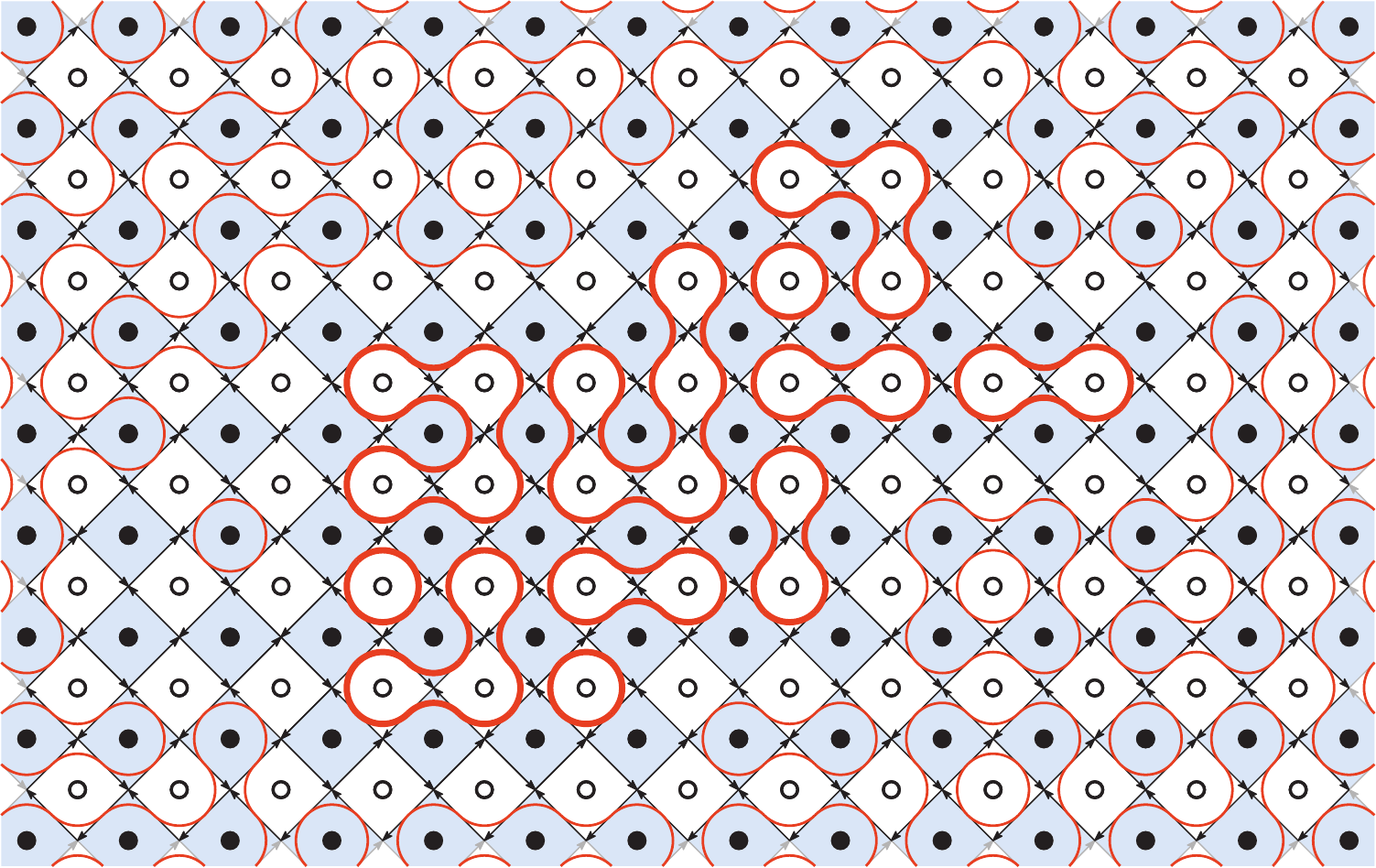}}
 \caption{(Step 1) Remove the loop $L$ from $\overline\omega$. The loops inside $L$ are depicted in bold.}
   \center{
  \includegraphics[width=0.64\textwidth]{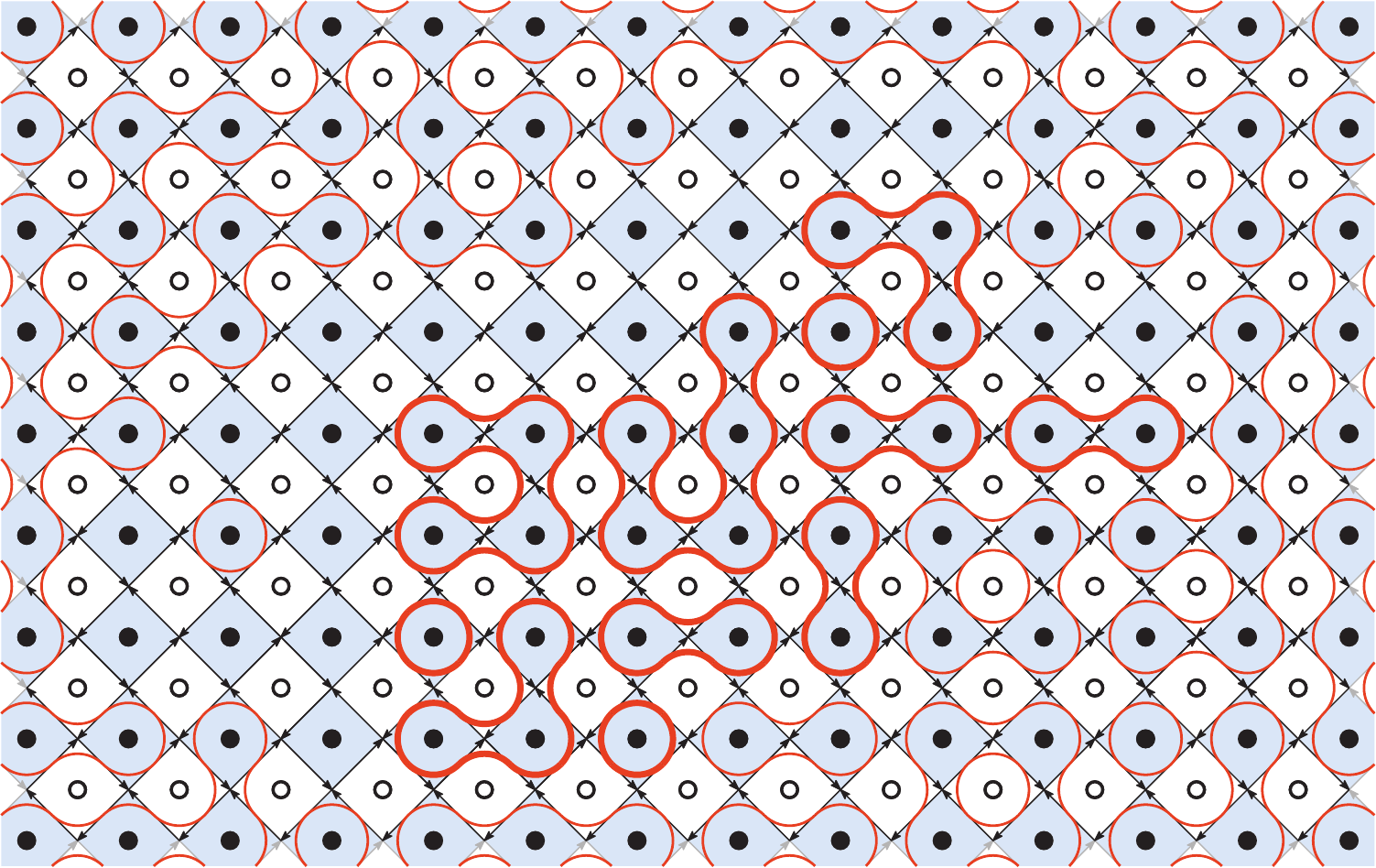}}
 \caption{(Step 2) Translate the loops inside $L$ in the south-east direction.}

   \begin{center}
   \includegraphics[width=0.64\textwidth]{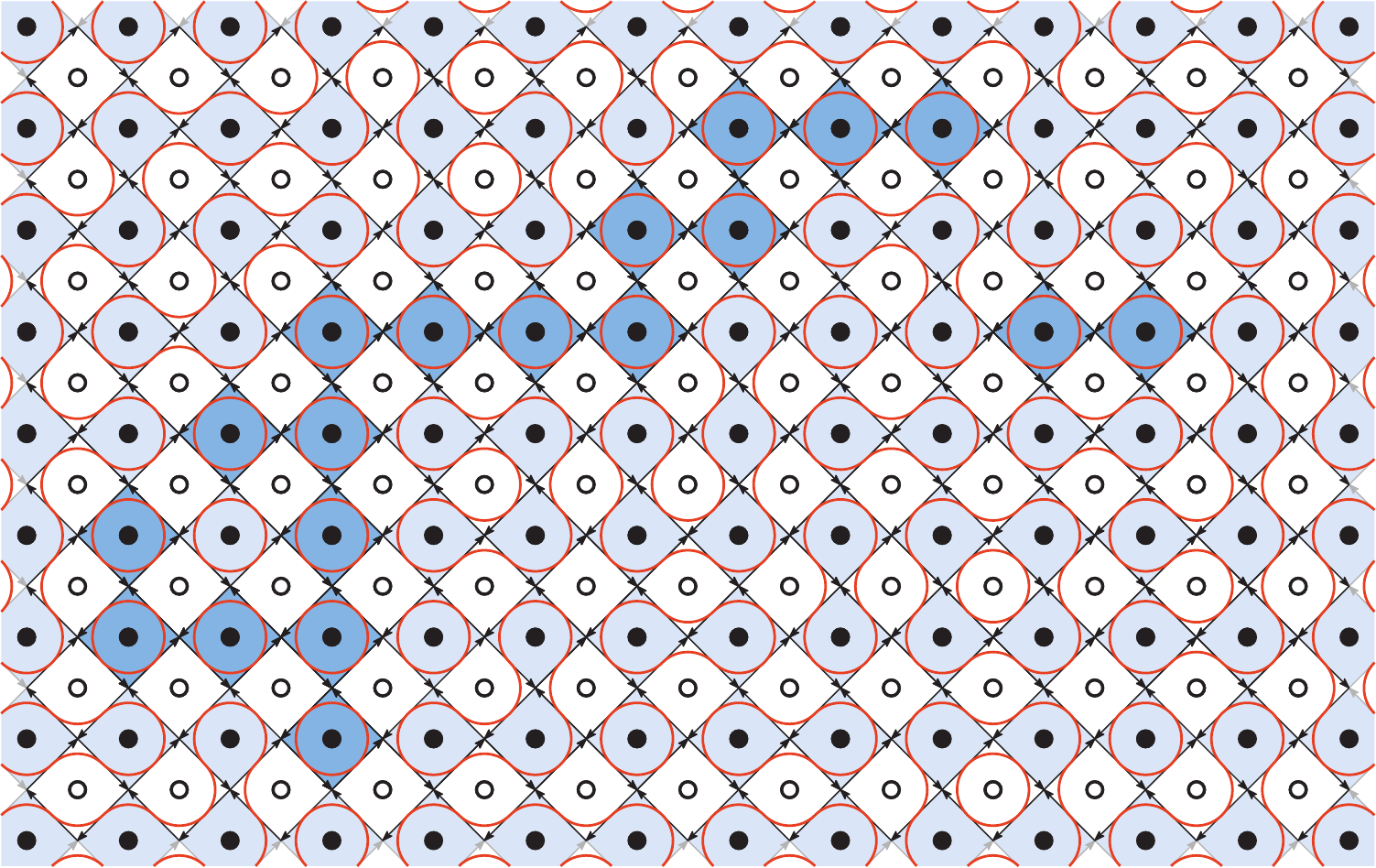}
 \caption{\label{fig:6}(Step 3) Fill the ``holes'' (depicted in darker gray) with loops of length four.}
 \end{center} \end{figure}

\begin{proof}[discontinuity for $q>256$.]
Consider a loop $L$ of the medial lattice $(\bbZ^2)^\diamond$ surrounding the origin. We assume that $L$ is oriented counter-clockwise. Let $n$ be the number of edges of $(\bbZ^2)^\diamond$ on $L$ and consider a graph $\Omega$ containing the full loop. Let $\calE_L$ be the event that the loop $L$ is a loop of the configuration $\overline\omega$. 

Our goal is to bound $\phi^{0}_{\Omega,p_c,q}[\calE_L]$. In order to do so, we construct a one-to-one ``repair map'' $f_L$ from $\calE_L$ to the set of loop configurations on $\Omega$ such that the image $f_L(\overline\omega)$ has much larger probability than the probability of $\overline\omega$. This will imply a bound on the probability of $\calE_L$ (see below).

Let $\overline \omega$ be a loop configuration in $\calE_L$. A loop of $\overline \omega$ is said to be {\em inside} (resp. {\em outside}) $L$ if it is included in the bounded connected component of 0 in $\bbR^2\setminus L$. Perform the following three successive modifications on $\overline\omega$  (See Figure~\ref{fig:6} for an illustration.)  to obtain a configuration $f_L(\overline\omega)$:\medbreak\noindent
{\bf Step 1.} Remove the loop $L$ from $\overline \omega$. 
\\{\bf Step 2.} Translate the loops of $\overline\omega$ which are inside $L$ by the vector $\frac{1-i}2$. 
\\{\bf Step 3.} Complete the configuration thus obtained by putting loops of length four around black faces of $\Omega^\diamond$ bordered by an edge which is not covered by any loop after Step 2.  
 \bigbreak\noindent
The configuration $f_L(\overline\omega)$ is a loop configuration on $\Omega^\diamond$ (Exercise~\ref{exo:loop}). 
Furthermore, Step 1 of the construction removes a loop from $\overline \omega$, but Step 3 adds one loop per edge of $L$ pointing south-west. Since the number of edges added in the last step is four times this number, and that the final configuration has as many edges as the first one, we deduce that this number is equal to $n/4$. Thus, we have
 \begin{align*}\phi^0_{\Omega,p_c,q}[\overline\omega]=\sqrt q^{\ell(\overline\omega)-\ell(f_L(\overline\omega))}\phi^0_{\Omega,p_c,q}[f_L(\overline\omega)]&=\sqrt q^{1-n/4}\phi^0_{\Omega,p_c,q}[f_L(\overline\omega)].\end{align*}
Using the previous equality in the second line and the fact that $f_L$ is one-to-one in the third (this uses the fact that $L$ is fixed at the beginning of the proof), we deduce that
\begin{align*}\phi^0_{\Omega,p_c,q}[\calE_L]&=\sum_{\overline\omega\in \calE_L}\phi^0_{\Omega,p_c,q}[\overline\omega]\\
&= q^{1/2-n/8}\,\sum_{\overline\omega\in \calE_L}\phi^0_{\Omega,p_c,q}[f_L(\overline\omega)]\\
&= q^{1/2-n/8}\,\phi^0_{\Omega,p_c,q}[f_L(\calE_L)]\le q^{1/2-n/8}.\end{align*}

Let us now prove that connectivity properties decay exponentially fast provided that $q>256$. Consider two vertices 0 and $x$ and a graph $\Omega$ containing both $0$ and $x$. If $0$ and $x$ are connected to each others in $\omega$, then there must exist a loop in $\overline \omega$ surrounding $0$ and $x$ which is oriented counter-clockwise (simply take the exterior-most such loop). Since any such loop contains at least $\|x\|$ edges, we deduce that 
\begin{align*}\phi^0_{\Omega,p_c,q}[0\longleftrightarrow x]&\le\sum_{\substack{L\text{ surrounding } \\ 
0\text{ and }x}}\phi^0_{\Omega,p_c,q}[\calE_L]\\
&\le\sum_{n\ge \|x\|}\sum_{\substack{L\text{ of length $n$}\\ 
\text{surrounding }0}}q^{1/2-n/8}\\
&\le \sum_{n\ge \|x\|} n2^n\cdot q^{1/2-n/8}.\end{align*}
In the last line we used that the number of loops surrounding 0 with $n$ edges on $\Omega^\diamond$ is smaller than $n2^n$. Letting $\Omega$ tend to the full lattice $\bbZ^2$, we deduce that
\begin{align*}\phi^{0}_{p_c,q}[0\longleftrightarrow x]&\le \sum_{n\ge \|x\|} n2^n\cdot q^{1/2-n/8}\le \exp(-c\|x\|).\end{align*}The existence of $c>0$ follows from the assumption $2q^{-1/8}<1$.\end{proof}

\bexo

\begin{exercise}\label{exo:loop}
Prove that the repair map $f_L$  actually yields a loop configuration.
\end{exercise}

\eexo

\paragraph{Mapping to the six-vertex model and sketch of the proof for $q>4$} We do not discuss the exact computation of the correlation length. The proof is based on a relation between the random-cluster model on a graph $\Omega$ and the six-vertex model on its medial graph $\Omega^\diamond$.

The six-vertex model was initially proposed by Pauling in 1931 for the study of the thermodynamic properties of ice.
While we are mainly interested in it for its connection to the random-cluster model, 
the six-vertex model is a major object of study on its own right.
We do not attempt to give here an overview of the model and 
we rather refer to \cite{Resh10} and Chapter~8 of \cite{Bax89} (and references therein) for a bibliography on the subject.

The mapping between the random-cluster model and the six-vertex model being very sensitive to boundary conditions, we will work on a torus. As in the previous section, the first and second coordinates of $x\in\bbZ^2$ are denoted by $x_1$ and $x_2$. For $M$ and $N$, consider the subgraph $\bbT=\bbT(M,N)$ of the square lattice induced by the set of vertices
$$\{x\in\bbZ^2: 0\le x_1+x_2\le M\text{ and } |x_1-x_2|\le N\}.$$ 
Introduce the periodic boundary conditions {\rm per} in which  $x$ and $y$ on $\partial\bbT$ are identified together iff 
$x_1+x_2=y_1+y_2$ or $x_1-x_2=y_1-y_2$. Together with these boundary conditions, $\bbT$ may be seen as a torus.

An {\em arrow configuration} $\vec{\omega}$ on $\bbT^\diamond$ (the medial graph is defined in an obvious fashion here) is a map attributing to each edge $xy\in E$ one of the two oriented edges $(x,y)$ and $(y,x)$. We say that an arrow configuration satisfies the {\em ice rule} if each vertex of $\mathbb T^\diamond$ is incident to two edges pointing towards it (and therefore to two edges pointing outwards from it). The ice rule leaves six possible configurations at each vertex, depicted in Fig.~\ref{fig:60}, whence the name of the model. 
Each arrow configuration $\vec\omega$ receives a weight
\begin{align}\label{eq:w_6V}
	w_{\rm 6V}(\vec{\omega}) := 
	\begin{cases}
	a^{n_1 + n_2} \cdot b^{n_3 + n_4} \cdot c^{n_5 + n_6}&\text{ if }\vec{\omega}\text{ satisfies the ice rule,}
	\\ \qquad\qquad0&\text{ otherwise},
	\end{cases}
\end{align}
where $a,b,c$ are three positive numbers, and $n_i$ denotes the number of vertices with configuration $i\in\{1,\dots,6\}$ in $\vec\omega$. In what follows, we focus on the case $a=b=1$ and $c>2$, and will therefore only consider such weights from now on. 
\bigbreak
\begin{figure}[htb]
	\begin{center}
		\includegraphics[width=0.7\textwidth]{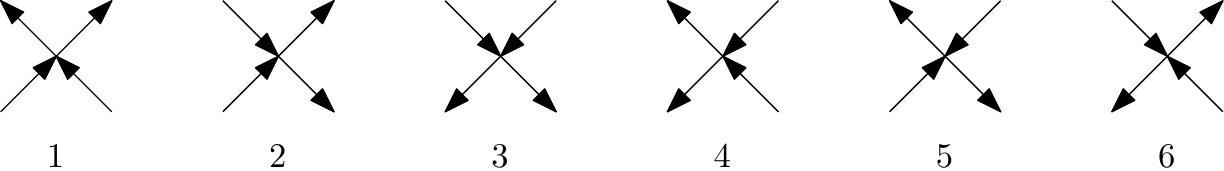}
	\end{center}
	
	\caption{\label{fig:60}The $6$ possibilities for vertices in the six-vertex model.
	Each possibility comes with a weight $a$, $b$ or $c$.}
\end{figure}

In our context, the interest of the six-vertex model stems from its solvability using the transfer-matrix formalism. More precisely, the partition function of a toroidal six-vertex model may be expressed as the trace of the $M$-th power of a matrix $V$ called the {\em transfer matrix}, whose leading eigenvalues can be computed using the so-called {\em Bethe-Ansatz}. This part does not invoke probability at all, and relies heavily on exact computations. For more details on the subject, we refer the curious reader to \cite{DumGanHar16b,DumGanHar16}. Here, we will only use the following consequence of the study.

For a six-vertex configuration $\vec\omega$ on $\mathbb T^\diamond$, write $|\vec\omega|$ for the number of north-east arrows intersecting the line $x_1+x_2=0$
(this number is the same for all lines $x_1+x_2=k$ with $-M\le k\le M$). The total number of arrows in each line is $2N$. It can be shown that typical configurations have $N$ such arrows. In fact, one may prove a more refined statement.
Set 
\begin{align*}
	Z_{6V}(N,M) = \sum_{\vec\omega} w_{6V}(\vec\omega)\qquad\text{ and }\qquad\widetilde Z_{6V}(N,M) = \sum_{\vec\omega:\, |\vec\omega| = N-1} w_{6V}(\vec\omega).
\end{align*}

\begin{theorem}\label{thm:6V}
	For $c>2$ and $r > 0$ integer, fix $\lambda>0$ satisfying $e^\lambda+e^{-\lambda}=c^2$. Then,
	\begin{align}
	\lim_{N\rightarrow\infty}\lim_{M\rightarrow\infty}-\tfrac1{M}\log\Big(\frac{\widetilde Z_{6V}(N,M)}{Z_{6V}(N,M)}\Big)&=\lambda + 2 \sum_{k=1 }^\infty \tfrac{(-1)^k}k \tanh (k\lambda)>0\label{eq:aggg}.
	\end{align}
\end{theorem}

Our goal now is to explain how one deduces discontinuity of the phase transition for random-cluster models from this theorem. In order to do so, we relate the random-cluster model to the six-vertex model.
We denote the random-cluster measure on $\bbT$ by $\phi_{\bbT,p_c,q}^{\rm per}$ (there is no boundary conditions since $\bbT$ has no boundary). Let $k_{\rm nc}(\omega)$ be the number of non-retractible clusters of $\omega$, and $\calA$ the event that both $\omega$ and $\omega^*$ contain exactly one cluster winding around the torus in the south-west  north-east direction.

\begin{proposition}\label{prop:6V}
	Let $q > 4$ and set $c = \sqrt{2 + \sqrt q}$.  For $N,M$ even, 
		\begin{align*}
		\phi_{\bbT,p_c,q}^{\rm per}[\calA] &= q\,\frac{\widetilde Z_{6V}(N,M)}{Z_{6V}(N,M)}\,\phi_{\bbT,p_c,q}^{\rm per}\Big[\big(\tfrac4{q} \big)^{k_{\rm nc}(\omega)}\Big] .	\end{align*}
\end{proposition}

\begin{figure}
	\begin{center}
	\hspace*{-1cm}
		\includegraphics[width=0.58\textwidth,angle=90]{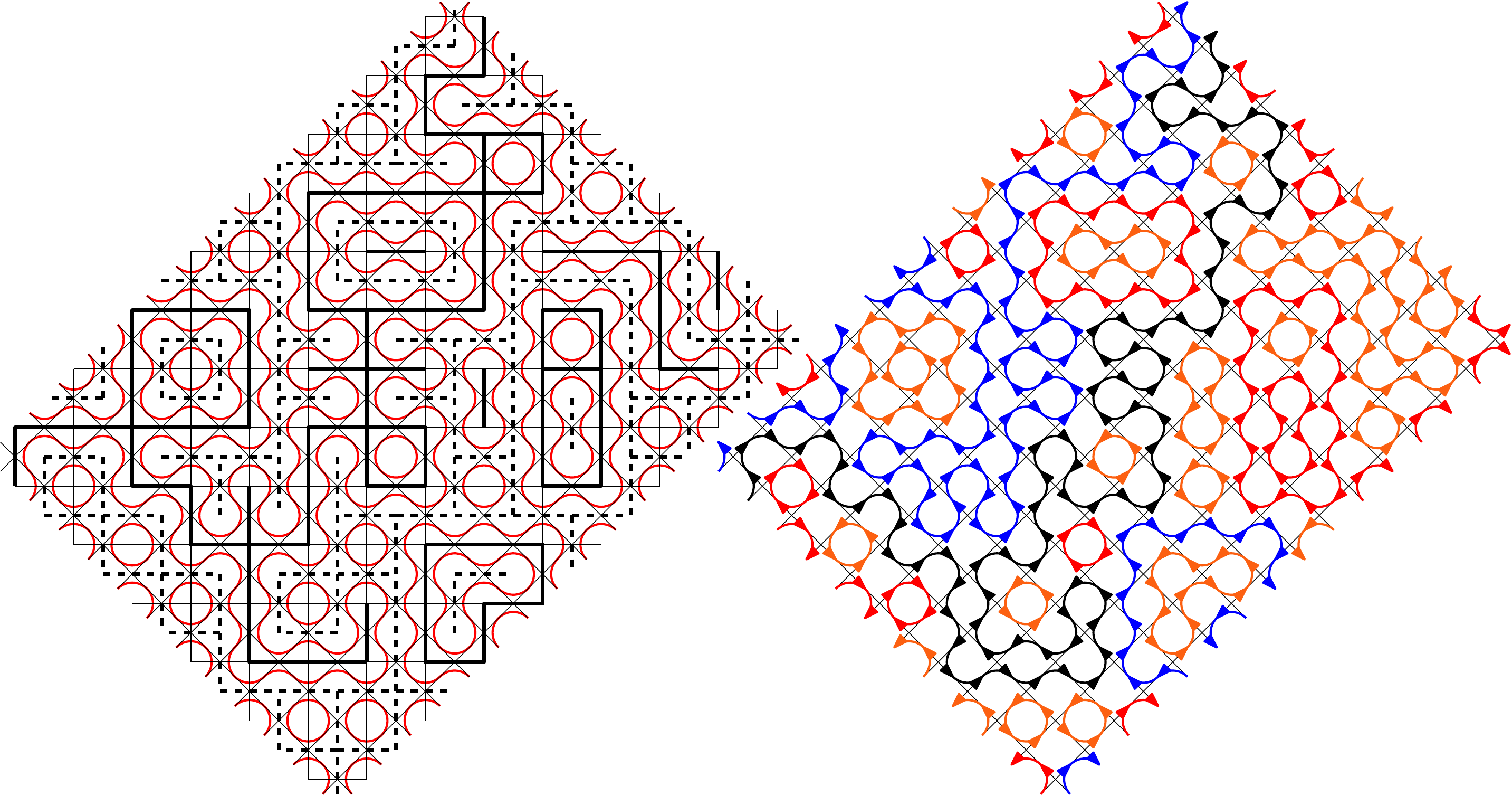}
		\hspace*{-1cm}
		\includegraphics[width=0.58\textwidth,angle=90]{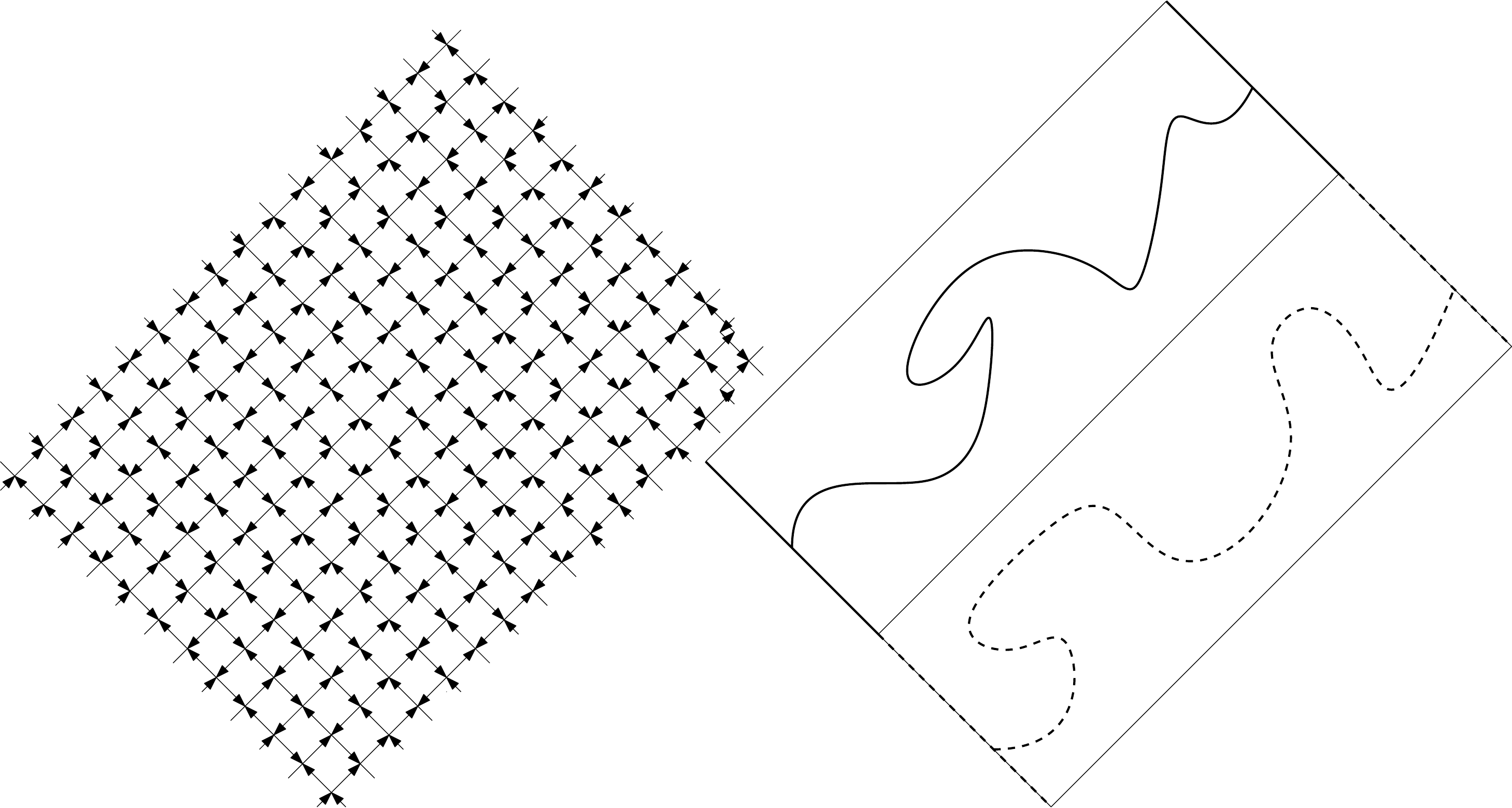}

	\end{center}
	\caption{The different steps in the correspondence between the random-cluster model and the six-vertex model on a torus. 
	{\bf Top-left.} A random-cluster configuration and its dual, as well as the corresponding loop configuration. {\bf Top-right.}
	An orientation of the loop configuration (retractible loops oriented counter-clockwise in red, clockwise in orange, in blue and black, the two non-retractible loops). {\bf Bottom-left.} The resulting six-vertex configuration. 
	Note that in the first picture, there exist both a primal and dual component winding vertically around the torus; 
	this leads to two loops that wind vertically (see second picture); 
	if these loops are oriented in the same direction (as in the third picture) 
	then the number of up arrows on every row of the six-vertex configuration is equal to $N \pm 1$. {\bf Bottom-right.} The intersection of the events $\calE$, $\calE'$, $\calF$ and $\calF'$ implies the event $\calA$.}
	\label{fig:correspondence}
\end{figure}

\begin{proof}
Define $w_{\rm RC}(\omega)=p^{o(\omega)}(1-p)^{c(\omega)}q^{k(\omega)}$. As in Proposition~\ref{prop:loop}, we may use Euler's formula (see Exercise~\ref{exo:torus}) on the torus to show that 
\begin{equation}\label{eq:air}\sqrt{q}^{ \ell(\overline\omega)  + 2s(\omega)}=c_0\,w_{\rm RC}(\omega),\end{equation}
where $s(\omega)$ is the indicator function of the event that all clusters of $\omega^*$ are retractible, and $c_0>0$ is independent of the configuration. 
	
Write $\omega^{\ol}$ for oriented loop configurations, i.e.~configurations of loops to which we associated an orientation. Let
$\ell_-(\omega^\ol)$ and $\ell_+(\omega^\ol)$ for the number of retractible loops of $\omega^\ol$ 
which are oriented clockwise and counter-clockwise, respectively.
Introduce $e^\mu+e^{-\mu}=\sqrt q$ and write, for an oriented loop configuration $\omega^\ol$, 
$$ w_{\ol} (\omega^{\ol}) = e^{\mu\ell_+(\omega^\ol)} \, e^{-\mu\ell_-(\omega^\ol)}.$$
	Fix $\omega$ a random-cluster configuration and consider its associated loop configuration $\overline\omega$. 
	In summing the $2^{\ell(\omega)}$ oriented loop configurations $\omega^\ol$ obtained from $\overline\omega$ by orienting loops, we find
		\begin{align}\label{eq:airair}
		\sum_{\omega^\ol} w_{\ol} (\omega^{\ol})  
		= \big(1+1\big)^{\ell_0(\overline\omega)}\big(e^\mu + e^{-\mu}\big)^{\ell(\overline\omega)-\ell_0(\overline\omega)}
		= c_0\big(\tfrac4q\big)^{k_{\rm nc}(\omega)}q^{-s(\omega)} w_{\rm RC}(\omega),
	\end{align}
where $\ell_0(\overline\omega)$ is the number of non-retractive loops of $\overline\omega$. In the last equality, we used \eqref{eq:air} and the fact when $s(\omega)=0$, any non-retractible cluster corresponds to two non-retractible loops. We also used that when $s(\omega)=0$, there is no non-retractible loop.

Notice now that an oriented loop configuration gives rise to $8$ different configurations at each vertex. 
These are depicted in Fig.~\ref{fig:oriented_loop_vertices}. 
For an oriented loop configuration $\omega^\ol$, write $n_{i}(\omega^\ol)$ for the number of vertices of type $i$ in $\omega^\ol$, 
with $i = 1,\, 2,\, 3,\, 4,\,  5A,\, 5B,\, 6A,\, 6B$.

\begin{figure}[htb]
	\begin{center}
		\includegraphics[width=0.9\textwidth]{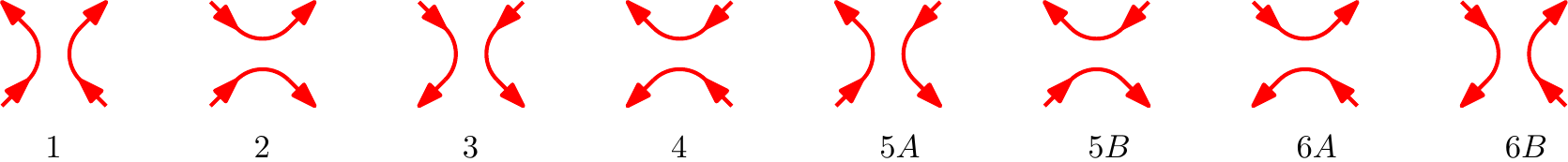}
	\end{center}
	\caption{The $8$ different types of vertices encountered in an oriented loop configuration.}
	\label{fig:oriented_loop_vertices}
\end{figure}

	The retractible loops of $\omega^\ol$ which are oriented clockwise have total winding $-2\pi$, 
	while those oriented counter-clockwise have winding $2\pi$. 
	Loops which are not retractible have total winding $0$. 
	Write $W(\ell)$ for the winding of a loop $\ell \in \omega^\ol$. Then
	\begin{align*}
		w_{\ol}(\omega^\ol) =  \exp\Big(\frac{\mu}{2\pi}\sum_{\ell \in \omega^\ol}W(\ell)\Big),
	\end{align*}
	where the sum is over all loops $\ell$ of $\omega^\ol$.
	The winding of each loop may be computed by summing the winding of every turn along the loop. 
	The compounded winding of the two pieces of paths appearing in the different configurations in Fig.~\ref{fig:oriented_loop_vertices} are 
	\begin{itemize}[noitemsep,nolistsep]
		\item vertices of type $1,\dots,4$: total winding $0$;
		\item vertices of type $5A$ and $6A$: total winding $\pi$;
		\item vertices of type $5B$ and $6B$: total winding $-\pi$.
	\end{itemize} 
	The total winding of all loops may therefore be expressed as 
	\begin{align*}
		\sum_{\ell \in \omega^\ol}W(\ell) =\pi\ \big[ n_{5A}(\omega^\ol) + n_{6A}(\omega^\ol) - n_{5B}(\omega^\ol) - n_{6B}(\omega^\ol)\big].
	\end{align*}
We therefore deduce that for any oriented loop configuration $\omega^\ol$, 
	\begin{align}\label{eq:airairair}
		w_{\ol}(\omega^\ol) = e^{\frac\mu2[n_{5A}(\omega^\ol) + n_{6A}(\omega^\ol)]}\ e^{-\frac\mu2[n_{5B}(\omega^\ol) + n_{6B}(\omega^\ol)]}.
	\end{align}

For the final step of the correspondence, notice that each diagram in Fig.~\ref{fig:oriented_loop_vertices} 
corresponds to a six-vertex local configuration (as those depicted in Fig.~\ref{fig:60}). 
Indeed, configurations $5A$ and $5B$ correspond to configuration $5$ in Fig.~\ref{fig:60} 
and configurations $6A$ and $6B$ correspond to configuration $6$ in Fig.~\ref{fig:60}.
The first four configurations of Fig.~\ref{fig:oriented_loop_vertices} 
correspond to the first four in Fig.~\ref{fig:60}, respectively.

Thus, to each oriented loop configuration $\omega^\ol$ is associated a six vertex configuration $\vec{\omega}$. 
Note that the map associating $\vec{\omega}$ to $\omega^\ol$ is not injective since
there are $2^{n_5(\vec\omega) + n_6(\vec\omega)}$ oriented loop configurations corresponding to each $\vec{\omega}$. In fact, for a six-vertex configuration $\vec{\omega}$, if $N_{5,6}(\vec{\omega})$ is the set of vertices of type $5$ and $6$ in $\vec\omega$, then the choice of $c=\sqrt{2 + \sqrt{q}}=e^{\frac\mu2} + e^{-\frac\mu2}$ gives that
	\begin{align}\label{eq:rew}
		w_{\rm 6V}(\vec{\omega}) 
		=\prod_{u \in N_{5,6}(\vec\omega)} \big(e^{\frac\mu2} + e^{-\frac\mu2}\big)
		=\sum_{\varepsilon \in \{\pm 1\}^{N_{5,6}(\vec\omega)}}\, \prod_{u \in N_{5,6}(\vec\omega)} e^{\frac\mu2\varepsilon(u)}
		\stackrel{\eqref{eq:airairair}}= \sum_{\omega^\ol} w_{\ol}(\omega^\ol).
	\end{align}
We are now in a position to prove the statement of the proposition. First,	\begin{align*}
		c_0\sum_{\omega} \big(\tfrac4q\big)^{k_{\rm nc}(\omega)}q^{-s(\omega)} w_{\rm RC}(\omega)\stackrel{\eqref{eq:airair}}=  \sum_{\omega^\ol} w_{\ol}(\omega^\ol)
		\stackrel{\eqref{eq:rew}}= \sum_{\vec\omega} w_{6V}(\vec\omega) = Z_{6V}(N,M).
	\end{align*}
Second, using that $s(\omega)=0$ and $k_{\rm nc}(\omega)=1$ on the event $\calA$,  we find	\begin{align*}
		c_0\sum_{\omega\in\calA} w_{\rm RC}(\omega)		&= c_0\frac {q}4\sum_{\omega\in\calA} w_{\rm RC}(\omega)\big(\tfrac4q\big)^{k_{\rm nc}(\omega)}q^{-s(\omega)}
\\
		&= q \sum_{|\vec\omega| =N - 1}w_{\rm 6V} (\vec\omega) 
		= q \ \widetilde Z_{\rm 6V}(N,M).
	\end{align*}
In the second step, we used that there are four ways of orienting the two loops bordering the unique non-retractible cluster of $\omega\in\calA$, and that one of them leads to $|\omega^\ol| =N - 1$. Dividing by the partition function of the random-cluster model and then taking the ratio of the two last displayed equations leads to the result.\end{proof}
 
Theorem~\ref{thm:RCM} now follows pretty easily. Indeed, one may show that for all $\delta>0$, there exists $N$ large enough that for all $M$,
\begin{equation}\label{eq:kjh}\phi_{\bbT,p_c,q}^{\rm per}\Big[\big(\tfrac4{q} \big)^{k_{\rm nc}(\omega)}\Big] \le \exp(\delta M).\end{equation}
 This corresponds to proving that there is not a density of non-retractible clusters winding around the torus. This fact follows easily from the fact that $\phi^0_{p_c,q}$ does not contain any infinite cluster (Exercise~\ref{exo:popo}). 

Thus, \eqref{eq:kjh}, Proposition~\ref{prop:6V} and Theorem~\ref{thm:6V} give the existence of $c_0>0$ such that for all fixed $N$ large enough and $M\ge M_0(N)$,
\begin{equation}\phi_{\bbT,p_c,q}^{\rm per}[\calA] \le \exp(-c_0M).\label{eq:uhgih}\end{equation}

Now, consider the ``rotated rectangles'' $R=\{x\in\bbT:x_1\le x_2\}$ and $R'=\{x\in\bbT:x_1>x_2\}$. Assume that {\bf P5} is  satisfied, one obtains easily by combining crossings that 
\begin{equation}\label{eq:plo}\phi_{R,p_c,q}^0[\calF]\ge c_0^{M/N}\qquad\text{and}\qquad\phi_{R',p_c,q}^1[\calF']\ge c_0^{M/N},\end{equation}
where $\calF$ is the event that there exists a path in $\omega\cap R$ going from the line $x_1+x_2=0$ to the line $x_1+x_2=M$, and $\calF'$ is the event that there exists a path in $\omega^*\cap (R')^*$ from the line $x_1+x_2=0$ to the line $x_1+x_2=M$.
Now, let $\calE$  be the event that all the edges in $R$ with one endpoint in $x_1+x_2=0$ are open, and $\calE'$ be the event that all the edges in $R'$ with one endpoint in $x_1+x_2=0$ are closed. Note that on $\calE\cap\calF\cap\calE'\cap\calF'$, there exists exactly one cluster in $\omega$ and one cluster in $\omega^*$ winding around the torus; see Fig.~\ref{fig:correspondence}.

The comparison between boundary conditions implies that conditionally on $\calE\cap\calV(R)$, the boundary conditions in $R'$ are dominated by wired boundary conditions. We obtain\begin{align*}\phi_{\bbT,p_c,q}^{\rm per}[\calA]&\stackrel{\phantom{\rm (FKG)}}\ge \phi_{\bbT,p_c,q}^{\rm per}[\calE\cap\calF\cap\calE'\cap\calF']\\
&\stackrel{\phantom{\rm (FKG)}}\ge \phi_{R,p_c,q}^0[\calE\cap\calF] \phi_{R',p_c,q}^1[\calE'\cap\calF']\\
&\stackrel{\rm (FKG)}\ge \phi_{R,p_c,q}^0[\calE]\phi_{R,p_c,q}^0[\calF] \phi_{R',p_c,q}^1[\calE']\phi_{R',p_c,q}^1[\calF']\\
&\stackrel{\phantom{\rm (FKG)}}\ge c_{\rm FE}^N\,c_0^{2M/N},
\end{align*}
where in the last line, we used \eqref{eq:finite energy} and \eqref{eq:plo}. By picking $N$ large enough and then letting $M$ go to infinity, we obtain a contradiction with \eqref{eq:uhgih}, so that {\bf P5} cannot be satisfied and the phase transition is discontinuous.

\begin{remark}In fact, one may even prove directly that {\bf P4b} does not hold (this is of value for these lectures since we did not formally prove that {\bf P4b} was equivalent to {\bf P5}). We refer to Exercise~\ref{exo:expo} for details.
\end{remark}

\bexo
\begin{exercise}\label{exo:popo}We wish to prove that for all $\delta>0$, for $N$ and $M$ large enough,
$$\phi_{\bbT,p_c,q}^{\rm per}\Big[\big(\tfrac4{q} \big)^{k_{\rm nc}(\omega)}\Big] \le \exp(\delta M).$$
1. 
Show that there exists $c_0>0$ depending on $q$ only such that for all $M$ and $N$, if $n=\delta M-N$, then\begin{equation}\label{eq:kkl}
\phi_{\mathbb T,p_c,q}(k_{\rm nc}(\omega)\ge \delta M)\le 
c_0^{M+N}\phi_{\bbT,p_c,q}^0[\exists n\text{ disjoint clusters crossing $\mathbb T$ from north-west to south-east}].
\end{equation}
\medbreak\noindent 2. Consider the event $\calE(x_1,\dots,x_n)$ that the points $x_1,\dots,x_n$ on the north-west side of $\bbT$ are connected to the bottom-east side by open paths, and $x_1,\dots,x_n$ are all in different clusters. 
Conditioning inductively on clusters crossing $\bbT$ from north-west to south-east, show that 
$$\phi_{\bbT,p_c,q}^0[\calE(x_1,\dots,x_n)]\le \phi^0_{p_c,q}[0\longleftrightarrow\partial\Lambda_N]^n.$$
3. Conclude. 
\end{exercise}

\begin{exercise}\label{exo:expo}
We wish to prove that {\bf P4b} cannot hold if \eqref{eq:uhg} is true.
\medbreak\noindent
1. Show that if {\bf P4b} does not hold, then for every $\delta>0$ there exists an infinite number of $n$ such that 
$$\phi_{\bbT,p_c,q}^0[(0,0)\longleftrightarrow (n,n)]\ge  \exp(-\delta n).$$
{\em Hint.} One may follow the same strategy as in Exercise~\ref{exo:from An to point}.
\medbreak\noindent
2. Deduce that for $N$ large enough, $\phi_{\bbT,p_c,q}^0[\calF]\ge c \exp(-\delta M)$ for some constant $c>0$ depending on $N$ only.
\medbreak\noindent
3. Conclude as in the proof that {\bf P5} does not hold.
\end{exercise}

\eexo

\section{Conformal invariance of the Ising model on $\bbZ^2$}\label{sec:6}

 \begin{mdframed}[backgroundcolor=green!00]
We will also adopt an important convention in this section. We now focus on  the random-cluster model with cluster-weight $q=2$. Also, we define $\bbL$ to be the rotation by $\pi/4$ of the graph $\sqrt 2\bbZ^2$. Generically, $(\Omega,a,b)$ will be a Dobrushin subdomain of $\bbL$ with the additional assumption that $e_b\in\bbR_+$ (where $e_b$ is seen as a complex number). Note that in this case $e_b$ is simply equal to 1.
\end{mdframed}

 For a discrete Dobrushin domain $(\Omega,a,b)$, denote $e\ni v$ if $v$ is one of the endpoints of $e$, and set $\partial\Omega^\diamond$ for the set of vertices of $\Omega^\diamond$ incident to exactly two edges of $\Omega^\diamond$. Define the {\em vertex fermionic observable} on vertices of $\Omega^\diamond$ by the formula
$$f(v):=\begin{cases}\medskip\displaystyle\tfrac12\sum_{e\ni v}F(e)&\text{ if }v\in\Omega^\diamond\setminus\partial\Omega^\diamond,\\
\displaystyle\tfrac2{2+\sqrt2}\sum_{e\ni v}F(e)&\text{ if }v\in\partial\Omega^\diamond,\end{cases}$$
where $F$ is the (edge) fermionic observable on $(\Omega,a,b)$ defined in Definition~\ref{def:parafermionic observable}. 
\medbreak
We are interested in the geometry at large scale of the critical Ising model on $\bbL$ (in particular the asymptotics of the vertex fermionic observable). 
A Dobrushin domain $(\Omega_\delta,a_\delta,b_\delta)$ will be a Dobrushin domain defined as a subgraph of the $\delta\bbL$, still 
with the convention that seen as a complex number, $e_b\in\bbR_+$. In particular, the length of the edges of $\Omega^\diamond$ is $\delta$.
 We extend the notions of Dobrushin domain, edge and vertex fermionic observables to this context.

We will focus on discrete Dobrushin domains $(\Omega_\delta,a_\delta,b_\delta)$ approximating in a better and better way a simply connected domain ${\bf \Omega}\subset \bbC$ with two points ${\bf a}$ and ${\bf b}$ on the boundary. We choose the notion of Carath\'eodory convergence for these approximations, i.e.~that $\psi_\delta \longrightarrow \psi$ on any compact subset $K\subset \bbR\times(0,\infty)$,
where $\psi$ is the unique conformal map from the upper half-plane $\bbR\times(0,\infty)$ to ${\bf \Omega}$ sending $0$ to ${\bf a}$, $\infty$ to ${\bf b}$, and with derivative at infinity equal to 1, and $\psi_\delta$ is the unique conformal map from $\bbH$ to $\Omega_\delta^\diamond$  sending $0$ to $a_\delta^\diamond$, $\infty$ to $b_\delta^\diamond$ and with derivative at infinity equal to 1. Here, we consider $\Omega_\delta^\diamond$ as a simply connected domain of $\bbC$ by taking the union of its faces\footnote{If it has ``pinched'' points, we add a tiny ball of size $\ep\ll\delta$. The very precise definition is not relevant here since the  definition is a complicated way of phrasing an intuitive notion of convergence.}.
\bigbreak
The first result of this section deals with the limit of the parafermionic observable (which we call {\em fermionic} observable in this case).

\begin{theorem}[Smirnov \cite{Smi10}]\label{thm:observable}
Fix $q=2$ and $p=p_c$. Let $(\Omega_\delta,a_\delta,b_\delta)$ be Dobrushin domains approximating a simply connected domain ${\bf \Omega}$ with two marked points ${\bf a}$ and ${\bf b}$ on its boundary. If $f_\delta$ denotes the vertex fermionic observable on $(\Omega_\delta,a_\delta,b_\delta)$, then
$$
\lim_{\delta\rightarrow0}\tfrac1{\sqrt{2\delta}}f_{\delta}=\sqrt{\phi'},$$ where $\phi$ is a conformal map from ${\bf \Omega}$ to the strip $\mathbb R\times(0,1)$ mapping ${\bf a}$ to $-\infty$ and ${\bf b}$ to $\infty$. 
\end{theorem}

Above, the convergence of functions is the uniform convergence on every compact subset of ${\bf \Omega}$. Since functions $f_\delta$ are defined on the graph $\Omega_\delta^\diamond$ only, we perform an implicit extension of the function to the whole graph, for instance by setting $f_\delta(y)=f_\delta(x)$ for the whole face above $x\in\Omega^\diamond$.
Note that the constraint that $e_b=\delta$ is not really relevant. We could relax this constraint by simply renormalizing $f_\delta$ by $1/\sqrt{2e_b}$ where $e_b$ is seen as a complex number. One word of caution here, $\delta$ is not the meshsize of the original lattice on which the random-cluster model is defined, but the meshsize of the medial lattice.
Also notice that the map $\phi$ is not unique a priori  since one could add any real constant to $\phi$, but this modification does not change its derivative.

The second result we will prove deals with the limit of the exploration path (we postpone the discussion to Section~\ref{sec:6.3}).

\begin{theorem}[Chelkak, DC, Hongler, Kemppainen, Smirnov \cite{CheDumHon14}]\label{thm:interface} 
Fix $q=2$ and $p=p_c$. Let $(\Omega_\delta,a_\delta,b_\delta)$ be Dobrushin domains approximating a simply connected domain ${\bf \Omega}$ with two marked points ${\bf a}$ and ${\bf b}$ on its boundary.  
The exploration path $\gamma_{(\Omega_\delta,a_\delta,b_\delta)}$ in $(\Omega_\delta,a_\delta,b_\delta)$  converges weakly to the Schramm-Loewner Evolution with parameter $\kappa=16/3$ as $\delta$ tends to $0$.
\end{theorem}

Above, the topology of the weak convergence is given by the metric $d$ on the set $X$ of continuous parametrized curves defined for $\gamma_1:I\rightarrow \mathbb C$ and $\gamma_2:J\rightarrow \mathbb C$ by
$$d(\gamma_1,\gamma_2)=\min_{\substack{\varphi_1:[0,1]\rightarrow I\\ \varphi_2:[0,1]\rightarrow J}}\ \sup_{t\in [0,1]}\ |\gamma_1(\varphi_1(t))-\gamma_2(\varphi_2(t))|,$$
where the minimization is over {\em increasing bijective functions} $\varphi_1$ and $\varphi_2$.
\bigbreak
A fermionic observable for the Ising model itself (and not of its random-cluster representation) was proved to be conformally invariant in \cite{CheSmi12}. Since then, many other quantities of the model were proved to be conformally invariant\footnote{Let us mention crossing probabilities \cite{BenDumHon14,Izy15}, interfaces with different boundary conditions \cite{CheDumHon14,HonKyt13}, full family of interfaces \cite{BenHon16,KemSmi16}, the energy fields \cite{HonSmi13,Hon10}.
The observable has also been used off criticality, see \cite{BefDum12a,DumGarPet14}.
}. Let us focus on one important case, namely the spin-spin correlations. 

\begin{theorem}[Chelkak, Hongler, Izyurov \cite{CheHonIzy15}]\label{thm:spin}
Let $\Omega_\delta$ be domains approximating a simply connected domain ${\bf \Omega}$. Consider also $a^1_\delta,\dots,a^k_\delta$ in $\Omega_\delta$ converging to points ${\bf a^1},\dots,{\bf a^k}$ in ${\bf \Omega}$.  Then,
$$\lim_{\delta\rightarrow0}\delta^{-n/8}\mu_{\Omega_\delta,\beta_c}^{\rm f}\big[\sigma_{a_\delta^1}\cdots\,\sigma_{a_\delta^k}\big]=\langle\sigma_{{\bf a^1}}\cdots\,\sigma_{{\bf a^k}}\rangle_{\bf \Omega},$$
where  $ \langle\sigma_{{\bf a^1}}\cdots\,\sigma_{{\bf a^k}}\rangle_{\bf \Omega}$ satisfies
$$\langle\sigma_{{\bf a^1}}\cdots\,\sigma_{{\bf a^k}}\rangle_{\bf \Omega}=|\phi'({\bf a^1})|^{1/8}\cdots\,|\phi'({\bf a^k})|^{1/8}\langle\sigma_{\phi({\bf a^1})}\cdots\,\sigma_{\phi({\bf a^k})}\rangle_{\phi({\bf \Omega})}$$
for any conformal map $\phi$ on ${\bf \Omega}$.
\end{theorem}

Note that this theorem shows that the critical exponent of the spin-spin correlations is $1/8$, i.e.~that 
\begin{equation}\mu_{\beta_c}[\sigma_0\sigma_x]=\|x\|^{-1/4+o(1)}.\label{eq:spin-spin}\end{equation}
In fact, this result is simpler to obtain and goes back to the middle of the 20th century (see \cite{McCWu83} and references therein).

The general form of $\langle--\rangle_{\bf \Omega}$ was predicted by means of Conformal Field Theory in \cite{BurGui87}. The method of \cite{CheHonIzy15} gives another formula (which is slightly less explicit). 
The proof relies on similar ideas as the proof of Theorem~\ref{thm:observable} (namely $s$-holomorphicity), but is substantially harder. We do not include it here and refer to \cite{CheHonIzy15} for details. 

In the next two sections, we prove Theorems~\ref{thm:observable} and \ref{thm:interface}. %

\subsection{Conformal invariance of the fermionic observable}

In this section, we prove Theorem~\ref{thm:observable}. We do so in two steps. We first prove that the vertex fermionic observable satisfies a certain boundary value problem on $\Omega^\diamond$. Then, we show that this boundary value problem has a unique solution converging to $\sqrt{\phi'}$ when taking Dobrushin domains $(\Omega_\delta,a_\delta,b_\delta)$ converging in the Carath\'eodory sense to $({\bf \Omega},{\bf a},{\bf b})$. 

\subsubsection{$s$-holomorphic functions and connection to a boundary value problem}

We will use a very specific property of $q=2$, which is that $\sigma=\tfrac12$ in this case. This special value of $\sigma$ enables us to prove the following:

\begin{lemma}\label{lem:argument}
Fix a Dobrushin domain $(\Omega,a,b)$. For any edge $e$ of $\Omega^\diamond$, the edge fermionic observable $F(e)$ belongs to $\sqrt{\overline e}\,\bbR$.\end{lemma}
Note that the definition of the square root is irrelevant since we are only interested in its value up to a $\pm1$ multiplicative factor.
\begin{proof}
The winding $W_{\gamma(\omega)}(e,e_b)$ at an edge $e$ can only take its value in the set 
  $W+2\pi\mathbb{Z}$ where $W$ is the winding at $e$ of an arbitrary 
 oriented path going from $e$ to $e_b$. Therefore, the winding 
  weight involved in the definition of $F(e)$ is always equal to 
  ${\rm e}^{{\rm i}W/2}$ or $-{\rm e}^{{\rm i}W/2}$, ergo $F(e)\in{\rm e}^{{\rm i}W/2}\bbR$, which is the claim by the definition of the square root and the fact that $e_b=1$.   \end{proof}
  
Together with the relations \eqref{rel_vertex}, the previous lemma has an important implication: while there were half the number of relations necessary to determine $F$ in the general $q>0$ case, we now know sufficiently many additional relations to hope to be able to compute $F$. We will harvest this new fact by introducing the notion of 
{\em $s$-holomorphic} functions, which was developed in \cite{CheSmi11,CheSmi12,Smi10}.  
For any edge $e$ (recall that $e$ is oriented and can therefore be seen as a complex number), define $${\rm P}_{e}[x]=\tfrac12(x+\overline{e}\,\overline {x}),$$
which is nothing but the projection of $x$ on the line $\sqrt {\overline e}\,\bbR$. 

\begin{definition}[Smirnov]
A function $f:\Omega^\diamond\rightarrow \mathbb C$ is {\em $s$-holomorphic} if for any edge $e=uv$ of $\Omega^\diamond$, we have
$${\rm P}_{e}[f(u)]={\rm P}_{e}[f(v)].$$
\end{definition}

The notion of $s$-holomorphicity is related to the classical notion of discrete holomorphic functions. On $\Omega^\diamond$, $f$ is discrete holomorphic if if satisfies the {\em discrete Cauchy-Riemann equations}
\begin{equation}f\left(v_1\right)-{\rm i}f(v_2)-f\left(v_3\right)+{\rm i}f\left(v_4\right)=0\label{eq:kru}\end{equation}
for every $x\in\Omega\cup\Omega^*$, where the $v_i$ are the four vertices around $x$ indexed in counterclockwise order. Discrete holomorphic functions $f$ distinctively appeared for the first time in the papers \cite{Isa41,Isa52} of Isaacs. Note that a $s$-holomorphic function is discrete holomorphic, since 
the definition of $s$-holomorphicity gives that for every $e=uv$,
\begin{equation}e[f(u)-f(v)]=\overline{f(v)}-\overline{f(u)},\label{eq:kkru}\end{equation}
and that summing this relation for the four edges around $x$ gives \eqref{eq:kru}.

The reason why $s$-holomorphic functions are easier to handle that discrete holomorphic function will become clear in the next section. 
In this section, we stick to the proof that the vertex fermionic observable is $s$-holomorphic, and that it satisfies some specific boundary conditions.

For a Dobrushin domain $(\Omega,a,b)$, let  $b^\diamond$ be the vertex of $\Omega^\diamond$ at the beginning of the oriented edge $e_b$. Also, let  $\nu_v=e+e'$ with $e$ and $e'$ the two edges of $\Omega^\diamond$ incident to $v$. The vector $\nu_v$ can be interpreted as a discrete version of the tangent vector along the boundary, when going from $a$ to $b$. 
 \begin{theorem} Let $(\Omega,a,b)$ be a Dobrushin domain.  The vertex fermionic observable  $f$ is $s$-holomorphic and satisfies ${\rm P}_{e_b}[f(b^\diamond)]=1$ and $\nu_vf(v)^2\in \bbR_+$
 for any $v\in \partial\Omega^\diamond$. 
 \end{theorem}
  \begin{proof} 
 The key to the proof is the following claim: for any $e\ni v$,
\begin{equation}\label{eq:kkkk}{\rm P}_{e}[f(v)]=F(e).\end{equation}
To prove this claim, consider $v$ with four medial edges $e_1$, $e_2$, $e_3$ and $e_4$ incident to it (we index them in counterclockwise order). Note that \eqref{rel_vertex} reads
$$e_1F(e_1)+e_3F(e_3)=e_2F(e_2)+e_4F(e_4).$$
Furthermore,  Lemma~\ref{lem:argument} gives that
\begin{equation}\label{eq:uhg}\overline{F(e)}=eF(e).\end{equation} Plugging this in the previous equality and using the conjugation, we find
$$F(e_1)+F(e_2)=F(e_3)+F(e_4)~\Big(=\tfrac 12\sum_{e\ni v} F(e)\Big).$$
The term under parentheses is nothing else but $f(v)$. Using Lemma~\ref{lem:argument} again, we see that $F(e_1)$ and $F(e_3)$ are two orthogonal vectors belonging to $\sqrt{\overline {e}_1}\bbR$ and $\sqrt{\overline {e}_3}\bbR$ respectively whose sum is $f(v)$, so that the claim follows readily for $e_1$ and $e_3$. One proves the claim for $e_2$ and $e_4$ in a similar way.  
\medbreak
Let us now treat the case of $v\in\partial\Omega^\diamond$  (the normalization $2/(2+\sqrt2)$ will play a role here). Let $e$ and $e'$ be the two edges of $\Omega^\diamond$ incident to $v$. Recalling that the winding on the boundary is deterministic, and that $e\in \gamma$ if and only if $e'\in\gamma$, gives\begin{equation}\label{eq:equation boundary}\sqrt {e'}F(e')=\phi_{\Omega,p_c,2}^{a,b}[e'\in\gamma]=\phi_{\Omega,p_c,2}^{a,b}[e\in\gamma]=\sqrt eF(e).\end{equation}
(Here, we choose the square root so that $\sqrt{e'}=\pm{\rm e}^{{\rm i}\pi/4}\sqrt e$.) This gives
\begin{equation}\tfrac{2+\sqrt2}2f(v)=F(e)+F(e')=(\sqrt{\overline e}+\sqrt{\overline e'})\,\phi_{\Omega,p_c,2}^{a,b}[e\in\gamma]
.\label{eq:ohohoh}\end{equation}
We deduce that $f(v)\in\sqrt{\overline{e}+\overline{e}'}\bbR$. Since $e=\pm{\rm i}e'$, a quick study of the complex arguments of $f(v)$, $F(e)$ and $F(e')$ immediately gives that ${\rm P}_{e}[f(v)]=F(e)$ and ${\rm P}_{e'}[f(v)]=F(e')$.
%
\bigbreak
Now that \eqref{eq:kkkk} is proved, we can conclude. First, observe that the $s$-holomorphicity is trivial, since for any edge $e=uv$, the claim shows that
${\rm P}_e[f(u)]=F(e)={\rm P}_e[f(v)].$
Second, 
${\rm P}_{e_b}[f(b)]=F(e_b)=1$. The last property follows from $f(v)\in\sqrt{\overline{e}+\overline{e}'}\bbR$. 
\end{proof}

Theorem~\ref{thm:observable} therefore follows from the following result, which is a general statement on $s$-holomorphic functions.

 \begin{theorem}\label{thm:BVP} For a family of Dobrushin domains $(\Omega_\delta,a_\delta,b_\delta)$ approximating a simply connected domain ${\bf \Omega}$ with two points ${\bf a}$ and ${\bf b}$ on its boundary, let $f_\delta$ be a $s$-holomorphic function satisfying ${\rm P}_{e_b}[f_\delta(b)]=1$ and $\nu_vf_\delta(v)^2\in\bbR_+$
 for any $v\in \partial\Omega^\diamond_\delta$. Then,
 $$\lim_{\delta\rightarrow0}\tfrac1{\sqrt {2\delta}}f_\delta=\sqrt{\phi'},$$
 where $\phi$ is a conformal map from ${\bf \Omega}$ to the strip $\mathbb R\times(0,1)$ mapping ${\bf a}$ to $-\infty$ and ${\bf b}$ to $\infty$. \end{theorem}

 We now turn to the proof of this statement, which will not involve the random-cluster anymore.
 
 \begin{remark}
Let us discuss the general $q\ne 2$ case. Equation~\eqref{eq:kru} looks similar to \eqref{rel_vertex}. Therefore, one may think of the (edge) parafermionic observable as a function defined on vertices of the medial graph $\Omega^{\diamond\diamond}$ of $\Omega^\diamond$ satisfying half of the discrete Cauchy-Riemann equations -- namely those around faces of $\Omega^{\diamond\diamond}$ corresponding to vertices of $(\Omega^\diamond)^*$ (for the other faces, we do not know how to get the corresponding relations, which probably are not even true at the discrete level for $q\ne 2$). Such an interpretation is nonetheless slightly misleading, since the edge parafermionic observable does not really converge to a function in the scaling limit. Indeed, in the case of the fermionic observable ($q=2$), the edge fermionic observable is the projection of the vertex fermionic observable, and therefore converges to different limits depending on the orientation of the edge of $\Omega^\diamond$ associated to the corresponding vertex of $\Omega^{\diamond\diamond}$.  \end{remark}

\subsubsection{Proof of Theorem~\ref{thm:BVP}}


The idea of the proof of Theorem~\ref{thm:BVP} will be to prove that solutions of this discrete Boundary value problem (with Riemann-Hilbert type boundary conditions on the boundary, i.e.~conditions on the function being parallel to a certain power of the tangent vector) must converge to the solution of their analog in the continuum.  Unfortunately, treating this discrete boundary value problem directly is a mess, and we prefer to transport our problem as follows. 
The function $\Im(\phi)$ is the unique harmonic function in ${\bf \Omega}$ equal to 1 on the arc $({\bf ab})$, and 0 on the arc $({\bf ba})$. Therefore, one may try to prove that a  discrete version $H_\delta$ of the imaginary part of the primitive of $\tfrac1{2\delta}f_\delta^2$ satisfies some approximate Dirichlet boundary value problem in the discrete, and that therefore this function must converge to $\Im(\phi)$ as $\delta$ tends to 0. This has much more chances to work, since Dirichlet boundary value problems are easier to handle.
\bigbreak
For now, let us start by studying $s$-holomorphic functions on a domain $\Omega^\diamond$ with $e_b=1$. For any such $s$-holomorphic function $f$, we associate the function $F=F_f$  defined on edges $e=uv$ of $\Omega^\diamond$ by 
\begin{equation}\label{eq:ihihu}F(e):={\rm P}_e[f(v)]={\rm P}_e[f(u)].\end{equation}  
We also introduce the (unique) function $H=H_f:\Omega\cup \Omega^*\rightarrow \mathbb C$ such that 
 $H(b)=1$ 
and
\begin{equation}\label{eq:au}H(x)-H(y)=|F(e)|^2\end{equation}
for every $x\in\Omega$ and $y\in \Omega^*$, where $e$ is the medial edge bordering both $x$ and $y$.
To justify the existence of such a function, construct $H(x)$ by summing increments along an arbitrary path from $b$ to $x$. The fact that this function satisfies \eqref{eq:au} for all neighboring $x$ and $y$ comes from the fact that the definition does not depend on the choice of the path. 

This last fact can be justified as follows: the domain is the union of all the faces of the medial lattice within it. As a consequence, the property that the definition does not depend on the choice of the path is equivalent to the property that for any vertex $v\in\Omega^\diamond\setminus\partial\Omega^\diamond$, if $e_1,\dots,e_4$ denote the four medial edges with end-point $v$ indexed in counter-clockwise order, then the paths going through $e_1$ and $e_2$, and the one going through $e_4$ and $e_3$ contribute the same (see Fig.~\ref{fig:configuration3}), i.e.
\begin{equation*}\label{square}|F(e_1)|^2 -|F(e_2)|^2=|F(e_4)|^2-|F(e_3)|^2,\end{equation*}
Since $F(e_1)$ and $F(e_3)$ are orthogonal (idem for $F(e_2)$ and $F(e_4)$), the previous equality follows from 
\begin{equation}\label{square2}|F(e_1)|^2 +|F(e_3)|^2=|f(v)|^2=|F(e_2)|^2+|F(e_4)|^2.\end{equation}

\begin{figure}
\begin{center}
\includegraphics[width=0.50\textwidth]{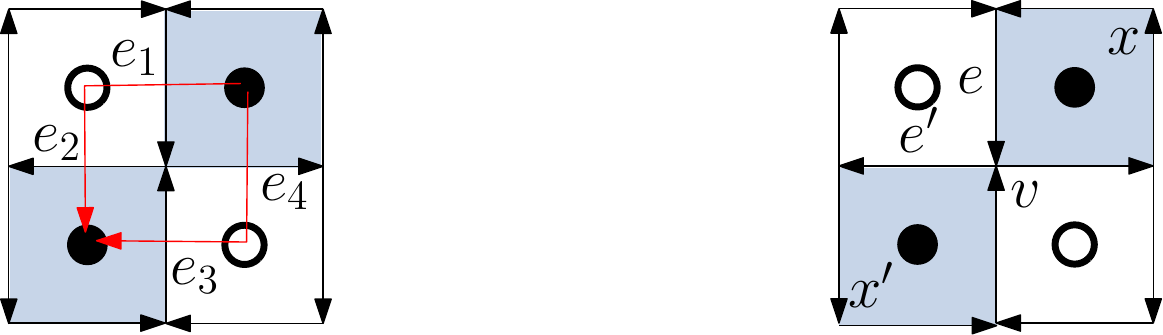}
\caption{\label{fig:configuration3} On the left the two paths going through $e_1$ and $e_2$, and $e_4$ and $e_3$. On the right, the notation for the proof of \eqref{eq:image}.}\end{center}
\end{figure}
The existence of $H$ is the main reason why it is more convenient to work with $s$-holomorphic functions rather than the less constraining notion of discrete holomorphicity. Also, we hope that the brief discussion on boundary value problems above provides sufficient motivation for the introduction of $H$: as shown in the following theorem, the function $H$ should be interpreted as the discrete analogue of  $\Im \left(\int^z \tfrac12f^2\right)$, which satisfies some nice property of sub and super harmonicity. 

Below, the discrete Laplacian of $H$ is defined by the formula
$$\Delta H(x):=\sum_{y}[H(y)-H(x)],$$ where the sum is over neighbors of $x$ in $\Omega$ (or $\Omega^*$ if $x\in\Omega^*$).
\begin{theorem}\label{definition H}
If $x,x'\in\Omega\cup\Omega^*$ correspond to two opposite faces of $\Omega^\diamond$ bordered by $v\in\Omega^\diamond$, 
\begin{equation}\label{eq:image}H(x)-H(x')~=~\tfrac12{\rm Im} \left[f(v)^2\cdot(x-x')\right].
\end{equation}
Furthermore, $\Delta H(x)\ge0$ for every $x\in\Omega\setminus\partial\Omega$ and $\Delta H(y)\le 0$ for every $y\in\Omega^*\setminus\partial\Omega^*$.
\end{theorem}

\begin{proof}
[\eqref{eq:image}.] Assume that $x$ and $x'$ belong to $\Omega$ (the case of $x$ and $x'$ belonging to $\Omega^*$ is the same). Let $e$ and $e'$ two edges of $\Omega^\diamond$ incident to $v$ bordering the same white face. We further assume that $e$ and $e'$ are respectively bordering the faces of $x$ and $x'$; see Fig.~\ref{fig:configuration3}. 
The $s$-holomorphicity implies that 
\begin{align*}|F(e)|^2=\tfrac{1}{4}[ef(v)^2+\overline{ef(v)}^2+2|f(v)|^2].
\end{align*}
Using a similar relation for $|F(e')|^2$, we obtain
\begin{align*}H(x)-H(x')&=|F(e)|^2-|F(e')|^2\\
&=\tfrac{1}{4}[(e-e')f(v)^2+\overline{(e-e')}\overline{f(v)}^2]=\tfrac12\Re[f(v)^2(e-e')].\end{align*}
The proof follows by observing that $e-e'={\rm i}(x-x')$.
\paragraph{Proof of sub-harmonicity.} Fix $x\in\Omega$. Let $A$, $B$, $C$ and $D$ be the values of $f$ on the vertices of $\Omega^\diamond$ north-east, north-west, south-west and south-east of $x$. Recall that

(a) $A-B=\overline A-\overline B$ \qquad by $s$-holomorphicity at the medial edge north of $x$ (equal to ${\rm i}$),

(b) $C-D=\overline D-\overline C$ \quad by $s$-holomorphicity at the medial edge south of $x$ (equal to $1$),

(c) $A-C={\rm i}(D-B)$\quad  by discrete holomorphicity \eqref{eq:kru} around $x$.\\
Then,
\begin{align}A^2+{\rm i}B^2-C^2-{\rm i}D^2&\stackrel{\phantom{(a,b)}}= (A-C)(A+C)+{\rm i}(B-D)(B+D)\nonumber\\
&\stackrel{\phantom{,}(c)\phantom{b}}=(A-C)(A+C-B-D)\nonumber\\
&\stackrel{(a,b)}=(A-C)(\overline A-\overline B+\overline D-\overline C)\nonumber\\
&\stackrel{\phantom{,}(c)\phantom{b}}=(1+{\rm i})|A-C|^2.\label{eq:ppp}\end{align}
Taking the imaginary part of the quantity obtained by multiplying the previous expression by $\frac{1+{\rm i}}2$ (which is equal to $\tfrac12(x'-x)$ , where $x'$ is the vertex of $\Omega$ north-east of $x$), \eqref{eq:image} gives $$\Delta H(x)=|A-C|^2\ge0.$$ Similarly, one may check that $\Delta H(x)=-|A-C|^2\le 0$ for $x\in\Omega^*$.
\end{proof}
Until now, we treated general $s$-holomorphic functions, but from this point we focus on the implications of boundary conditions. Let us start by the following easy lemma.
\begin{lemma}\label{lem:boundary conditions}Consider a $s$-holomorphic function $f$ satisfying 
$F(e_b)=1$ and $\nu_vf(v)^2\in\bbR_+$ for all $v\in\partial\Omega^\diamond$. Then, the function $H$ is equal to 1 on $(ba)$ and 0 on $(ab)^*$. \end{lemma}

\begin{proof}
Equation \eqref{eq:image} and the condition $(x-x')f(v)^2=\pm\nu_vf(v)^2\in\bbR$ give that $H$ is constant on $(ba)$ and $(ab)^*$ respectively. The fact that $H=1$ on $(ba)$ thus follows from the definition   $H(b)=1$. The claim that $H=0$ on $(ab)^*$ follows from the fact that for $w\in (ab)^*$ neighboring $b$, 
$$
H(w)\stackrel{\eqref{eq:au}}=H(b)-|F(e_b)|^2=1-1=0.
$$
\end{proof}
On the other part $(ab)$ of the boundary of $\Omega$, we would like to say that $H$ is roughly 0. This is true but not so simple to prove. In order to circumvent this difficulty, we choose another path: we add a ``layer'' or additional vertices, and fix the value of $H$ to be 0  on these new vertices (for simplicity, we consider all these vertices as one single ghost vertex $\mathfrak g$). With this definition, $H$ is not quite super-harmonic on $\Omega\cup\{\mathfrak g\}$ but it almost is: one can define a modified Laplacian on the boundary for which $H$ is super-harmonic. This procedure is explained formally below (we do a similar construction for $\Omega^*$). 

Introduce two additional ghost vertices $\mathfrak g$ and $\mathfrak g^*$ to $\Omega$ and $\Omega^*$ respectively. Define  the continuous-time random walk $X^x$
starting at $x$ and jumping with rate $1$ on edges of $\Omega$ and rate $\tfrac{2}{1+\sqrt 2} N_x$ to $\mathfrak g$, where $N_x$ is the number of vertices of $\partial\Omega^\diamond$ bordering $x$. Note that $X^x$ jumps to $\mathfrak g$ with positive rate only when it is on the boundary of $\Omega$. Also, from now on the Laplacian $\widehat\Delta$ on $\Omega$ denotes the generator of the random walk, which is defined by 
$$\widehat\Delta H(x):=\Delta H(x)+\tfrac2{1+\sqrt 2} N_x[H(\mathfrak g)-H(x)].$$
Similarly, we denote by $X^y$ the continuous-time
random walk starting at $y$ and jumping with rate $1$ on edges of $\Omega^*$ and with rate $ \tfrac2{1+\sqrt 2} N_y$ to $\mathfrak g^*$. We extends $H$ to $\mathfrak g$ and $\mathfrak g^*$ by setting $H(\mathfrak g)=0$ and $H(\mathfrak g^*)=1$.

\begin{lemma}\label{lem:boundary conditions2}Consider a $s$-holomorphic function $f$ satisfying 
$F(e_b)=1$ and $\nu_vf(v)^2\in\bbR_+$ for all $v\in\partial\Omega^\diamond$. Then, $\widehat\Delta H\ge0$ on $\Omega\setminus (ba)$ and $\widehat\Delta H \le0$ on $\Omega^*\setminus (ab)^*$.
\end{lemma}

\begin{proof}
Let us prove that $\widehat\Delta H(x)\ge0$ for $x\in\Omega\setminus (ba)$ (the proof for $x\in\Omega^*\setminus(ab)^*$ follows the same lines).  If $x\notin\partial\Omega$, one has $\widehat\Delta=\Delta$ and the result follows from  Theorem~\ref{definition H}. We therefore focus our attention on $x\in(ab)$. We use the same computation as in \eqref{eq:ppp}, except that for $v\in\partial\Omega^\diamond$, we replace the expression
$${\rm Im} [f(v)^2\cdot(v-x)]=\tfrac12{\rm Im} [f(v)^2\cdot(x'-x)]=H(x')-H(x)$$ given by \eqref{eq:image} by the expression
\begin{equation}\label{eq:yhg}{\rm Im}[f(v)^2(v-x)]=\tfrac2{1+\sqrt 2}[H(\mathfrak g)-H(x)].\end{equation}
In order to prove \eqref{eq:yhg}, 
 use that $v-x=-\frac{{\rm i}}2\nu_v$ (since $x\in(ab)$) and $\nu_vf(v)^2\in\bbR_+$ to get
 $${\rm Im}[f(v)^2(v-x)]=-\tfrac{\sqrt 2}2|f(v)|^2.$$
Using the same reasoning as for \eqref{eq:ohohoh} and the fact that $\nu_v$ has length $\sqrt 2$, we find that 
$$|f(v)|^2=\tfrac{4|1+{\rm e}^{{\rm i}\pi/4}|^2}{(2+\sqrt 2)^2}|F(e)|^2=\tfrac{2\sqrt 2}{1+\sqrt 2}|F(e)|^2.$$
Therefore, \eqref{eq:yhg} follows from the two previous equalities together with $H(\mathfrak g)=0$ and $H(x)=|F(e)|^2$ (which is true since there is $y\in(ab)^*$ neighboring $x$, which satisfies $H(y)=0$). 
\end{proof}

We are now in a position to prove Theorem~\ref{thm:BVP}.
\begin{proof}[Theorem~\ref{thm:BVP}] 
For $f_\delta$, let $H_\delta$ constructed via the relation \eqref{eq:au} and the condition $H_\delta(b_\delta)=1$. Note that all the previous properties of $H$ extend to $H_\delta$ (with trivial modification of the definition of $\Delta$ and $\widetilde\Delta$), except \eqref{eq:image}, which becomes
\begin{equation}
\label{eq:image2} H_\delta(x')-H_\delta(x)=\tfrac1{2\delta}{\rm Im}[f_\delta(v)^2(x'-x)]
\end{equation}
since the edge $x'-x$ does not have length $\sqrt 2$ anymore but $\sqrt 2\delta$ instead.
\bigbreak
 We start by proving that $(H_\delta)$ converges\footnote{Recall that here and below, we consider the convergence on every compact subset of ${\bf \Omega}$.
}. 
We set $H_\delta^\bullet$ and $H_\delta^\circ$ for the restrictions of $H_\delta$ to $\Omega_\delta$ and $\Omega_\delta^*$. 
Define
$${\bf Hm}^\bullet_\delta(x):=\bbP[X^x\text{ hits }(b_\delta a_\delta)\text{ before }\mathfrak g]\text{ and }
{\bf Hm}^\circ_\delta(y):=\bbP[X^y\text{ hits }(a_\delta b_\delta)^*\text{ before }\mathfrak g^*].$$ The function ${\bf Hm}^\bullet_\delta$ is the harmonic solution on $\Omega_\delta$ of the discrete Dirichlet problem with boundary conditions $1$ on $(b_\delta a_\delta)$ and 0 on $\mathfrak g$. Since the random-walk jumps on $\mathfrak g$ only when it is on $(a_\delta b_\delta)$, one may show that it converges to the harmonic solution of the Dirichlet problem with boundary conditions $1$ on $({\bf ba})$ and 0 on $({\bf ab})$ -- i.e.~to $\Im(\phi)$ -- as $\delta$ tends to 0  (see Exercise~\ref{exo:convergence Dirichlet} for details). Since $H^\bullet_\delta$ is sub-harmonic by Lemmata~\ref{lem:boundary conditions} and \ref{lem:boundary conditions2}, one has $H^\bullet_\delta\le {\bf Hm}^\bullet_\delta$ and therefore
$$\limsup_{\delta\rightarrow 0}H^\bullet_\delta\le \Im(\phi).$$
Similarly, ${\bf Hm}^\circ_\delta$ tends to $\Im(\phi)$. Since $H^\circ_\delta$ is super-harmonic, $H_\delta^\circ\ge  {\bf Hm}^\circ_\delta$ and
$$\liminf_{\delta\rightarrow 0}H^\circ_\delta\ge  \Im(\phi).$$
Since $H^\bullet(x)\ge H^\circ(y)$ for $y$ neighboring $x$, we deduce that $H_\delta$ converges to $\Im(\phi)$.
\bigbreak
Let us now prove that $(f_\delta)$ converges. Consider a holomorphic sub-sequential limit $f$ (if it exists) of $f_\delta/\sqrt{2\delta}$. Also set $F$ to be a primitive of $f^2$. By  \eqref{eq:image2}, 
$H_\delta$
is equal to the imaginary part of the primitive of $\tfrac1{2\delta}f_\delta^2$, so that by passing to the limit and using the first part of the proof, ${\rm Im}(F)={\rm Im}(\phi)+C$. Since $f$ is holomorphic, we know that $F$ also is, so that it must be equal to $\phi$ up to an additive (real valued) constant. By differentiating and taking the square root, we deduce that $f=\sqrt{\phi'}$. To conclude, it only remains to prove that $(f_\delta)$ is pre-compact and that any sub-sequential limit is holomorphic, which is done in the next lemma.
\end{proof}
\begin{lemma}
The family of functions $(\tfrac1{\sqrt {2\delta}}f_\delta)$ is pre-compact for the uniform convergence on every compact. Furthermore, any sub-sequential limit is holomorphic on ${\bf \Omega}$.
\end{lemma}
In the next proof, we postpone three facts to exercises. We want to highlight the fact that we do not swift any difficulty under the carpet: these statements are very simple and educating to prove and we therefore prefer to leave them to the reader.
\begin{proof} Since the functions $f_\delta$ is discrete holomorphic,  the statement follows (see Exercise~\ref{exo:compactness harmonic} for details) from the fact that $(\tfrac1{\sqrt {2\delta}}f_\delta)$ is square integrable, i.e.~that for any compact subset ${\bf K}$ of ${\bf \Omega}$, there exists a constant $C=C({\bf K})>0$ such that for all $\delta$,
\begin{equation}\label{eq:lok}\delta\sum_{x\in \delta\bbL\cap {\bf K}} |f_\delta(x)|^2\le C.\end{equation}
In particular, \eqref{eq:image} implies that 
\begin{align}\nonumber \tfrac{\sqrt 2}2|f_\delta(v)|^2&=\tfrac12{\rm Im}[f_\delta(v)^2(x'-x)]+\tfrac12{\rm Re}[f_\delta(v)^2(x'-x)]\\
&=H^\bullet(x')-H^\bullet(x)+H^\circ(y')-H^\circ(y)
\label{eq:lol},\end{align}
where $x,x'\in\Omega_\delta$ and $y,y'\in\Omega^*_\delta$ are the four faces bordering $v$ indexed so that $x'-x={\rm i}(y'-y)$.
Since $H^\bullet_\delta$ is bounded and sub-harmonic, Exercise~\ref{exo:compactness subharmonic} implies that 
\begin{equation}\delta\sum_{x\in \delta\bbL\cap {\bf K}}|H^\bullet_\delta(x)-H^\bullet_\delta(x')|\le C,\label{eq:lop}\end{equation}where the sum is over edges $x'$ with $xx'$ an edge of $\delta\bbL$. Similarly, one obtains the same bound for $H^\circ_\delta$. This, together with \eqref{eq:lol}, implies \eqref{eq:lok}.\end{proof}

\bexo

\begin{exercise}[Dirichlet problem]\label{exo:convergence Dirichlet}
1. Prove that there exists $\alpha>0$ such that for any $0<r<\tfrac12$ and any curve $\gamma$ inside $\mathbb D:=\{z:|z|<1\}$ from $\{z:|z|=1\}$ to $\{z:|z|=r\}$, the probability that a random walk on $\mathbb D\cap\delta\bbL$ starting at 0 exits $\mathbb D\cap\delta\bbL$ without crossing $\gamma$ is smaller than $r^\alpha$ uniformly in $\delta>0$.
\medbreak\noindent
2. Deduce that ${\bf Hm}^\bullet_\delta$ tends to 0 on $(ab)$.
\medbreak\noindent
3. Using the convergence of the simple random-walk to Brownian motion, prove the convergence of ${\bf Hm}^\bullet_\delta$ to the solution of the Dirichlet problem with 0 boundary conditions on $(ab)$, and 1 on $(ba)$. 
\end{exercise}

\begin{exercise}[Regularity of discrete harmonic functions]\label{exo:bounded}
1.  Consider $\Lambda:=[-1,1]^2$. Show that there exists $C>0$ such that, for each $\delta>0$, one may couple two lazy random-walks $X$ and $Y$ starting from $0$ and its neighbor $x$ in $\Lambda\cap\delta\bbL$ in such a way that $\bbP[X_\tau\ne Y_\tau]\le C\delta$, where $\tau$ is the hitting time of the box of the boundary of $\Lambda$. 
\medbreak\noindent 2. Deduce that a bounded harmonic function $h$ on $\Lambda$ satisfies $|h(x)-h(y)|\le C\delta$.\medbreak\noindent
3. Let $H_{\Lambda}(x,y)$ be the probability that the random walk starting from $x$ exits $\Lambda$ by $y$. Show that $H_{\Lambda}(x,y)\le C'\delta$.
\end{exercise}

\begin{exercise}[Limit of discrete holomorphic functions]
Prove that a discrete holomorphic function $f$ on $\delta\bbZ^2$ is discrete harmonic for the leap-frog Laplacian, i.e.~that 
$\Delta f_\delta(x)=0$, where 
$$\Delta f_\delta(x)=\sum_{\ep,\ep'\in\{\pm\delta\}}(f_\delta(x+(\ep,\ep'))-f(x)).$$
Prove that a convergent family of discrete holomorphic functions $f_\delta$ on $\delta\bbZ^2$ converges to a holomorphic function $f$. {\em Hint.} Observe that all the discrete versions of the partial derivatives with respect to $x$ and $y$ converge using Exercise~\ref{exo:bounded}.
\end{exercise}

\begin{exercise}[Precompactness criteria for discrete harmonic functions]\label{exo:compactness harmonic}Below, $\|f\|_\infty=\sup\{|f(x):x\in{\bf\Omega}\cap\delta\bbZ^2\}$ and $\|f\|_2=\delta^2\sum_{x\in{\bf\Omega}\cap\delta\bbZ^2}f(x)$.\medbreak\noindent
1. Show that a family of $\|\cdot\|_\infty$-bounded harmonic functions $(f_\delta)$ on ${\bf\Omega}$ is precompact for the uniform convergence on compact subsets. {\em Hint.} Use the second question of Exercise~\ref{exo:bounded}.
\medbreak\noindent
2. Show that a family of $\|\cdot\|_2$-bounded harmonic functions $(f_\delta)$ on ${\bf\Omega}$ is precompact for the uniform convergence on compact subsets. {\em Hint.} Use the third question of Exercise~\ref{exo:bounded} and the Cauchy-Schwarz inequality. 
 \end{exercise}

\begin{exercise}[Regularity of sub-harmonic functions] \label{exo:compactness subharmonic}
Let $H$ be a sub-harmonic function on $\Omega_\delta:={\bf \Omega}\cap\delta\bbL$, with 0 boundary conditions on $\partial\Omega_\delta$. 
\medbreak\noindent
1. Show that  $H(x)=\sum_{y\in \Omega_\delta} G_{\Omega_\delta}(x,y)\Delta H(y)$, where $G_{\Omega_\delta}(x,y)$ is the expected time a random-walk starting at $x$ spends at $y$ before exiting $\Omega_\delta$.
\medbreak\noindent
2. Prove that $G_{\Omega_\delta}$ is harmonic in $x\ne y$. Deduce that for two neighbors $x$ and $x'$ on $\Omega_\delta$,
$$|G_{\Omega_\delta}(x,y)-G_{\Omega_\delta}(x',y)|\le \frac{C\delta}{|x-y|\wedge d(x,\partial{\bf \Omega})}.$$
\medbreak\noindent
3. Deduce that for any compact subset ${\bf K}$ of ${\bf \Omega}$, there exists $C({\bf K})>0$ such that for any $\delta$,

$$\delta \sum_{x\in {\bf K}\cap\delta\bbL} |H(x)-H(x')|\le C,$$
where $x'$ is an arbitrary choice of neighbor of $x$.
\medbreak\noindent
4. What can we say for bounded boundary conditions?
\medbreak\noindent
5. Deduce \eqref{eq:lop} for $H_\delta^\bullet$.
\end{exercise}
\eexo

\subsection{Conformal invariance of the exploration path}\label{sec:6.3}

 Conformal field theory leads to the prediction that the exploration path $\gamma_{(\Omega_\delta,a_\delta,b_\delta)}$ in the Dobrushin domains $(\Omega_\delta,a_\delta,b_\delta)$ mentioned before converges as
$\delta\rightarrow 0$ to a random, continuous, non-self-crossing curve $\gamma_{({\bf \Omega},{\bf a},{\bf b})}$
from ${\bf a}$ to ${\bf b}$ staying in $\overline {\bf \Omega}$, and which
is expected to be conformally invariant in the following sense.

\begin{definition}
  A family of random non-self-crossing continuous curves
  $\gamma_{({\bf \Omega},{\bf a},{\bf b})}$, going from ${\bf a}$ to ${\bf b}$ and contained in $\overline {\bf \Omega}$,  indexed by simply connected
  domains ${\bf \Omega}$ with two marked points ${\bf a}$ and ${\bf b}$ on the boundary is
  \emph{{conformally invariant}} if for any $({\bf \Omega},{\bf a},{\bf b})$ and any
  conformal map $\psi:{\bf \Omega}\rightarrow \mathbb C$, $$\psi
  (\gamma_{({\bf \Omega},{\bf a},{\bf b})})~\text{has the same law
    as}~\gamma_{(\psi({\bf \Omega}),\psi({\bf a}),\psi({\bf b}))}.$$
\end{definition}

In 1999, Schramm proposed a natural candidate for the possible conformally invariant families of continuous non-self-crossing curves. He noticed that interfaces of discrete models further satisfy the {\em domain Markov property}\index{domain Markov property} which, together with the assumption of conformal invariance, determines a one-parameter family of possible random curves. In \cite{Sch00}, he introduced the Stochastic Loewner
evolution ($\mathsf{SLE}$ for short) which is now known as the
{Schramm--Loewner evolution}. Our goal is not to present in details this well studied model, and we rather refer the reader to the following expositions and references therein \cite{Law05}. Here, we wish to prove Theorem~\ref{thm:interface} and therefore briefly remind the definition of $\mathsf{SLE}$s.

Set $\mathbb H$ to be the upper half-plane $\bbR\times(0,\infty)$. Fix a simply connected subdomain $H$ of $\bbH$  such that $\mathbb H\setminus H$ is compact.
Riemann's mapping theorem guarantees\footnote{The proof of the existence of this map is not
completely obvious and requires Schwarz's reflection principle.} the existence of a {\em unique} conformal
map $g_H$ from $H$ onto $\mathbb H$ such that $$g_H(z)~:=~z + \tfrac{C}{z} + O\left(
\tfrac1{z^2} \right).$$ The constant $C$ is called the {\em $h$-capacity} of $H$.

There is a natural way to parametrize certain continuous non-self-crossing curves
$\Gamma:\mathbb R_+\rightarrow \overline{\mathbb H}$ with
$\Gamma(0)=0$ and with $\Gamma(s)$ going to $\infty$ when $s\rightarrow
\infty$. For every $s$, let $H_s$ be the connected component of
$\mathbb H\setminus\Gamma[0,s]$ containing $\infty$, and denote its $h$-capacity by $C_s$. The
continuity of the curve guarantees that $C_s$ grows continuously, so that it
is possible to parametrize the curve via a time-change $s(t)$ in such a way
that $C_{s(t)}=2t$. This parametrization is called the \emph{$h$-capacity
parametrization}. Below, we will assume that the parametrization is the $h$-capacity, and reflect this by using
the letter $t$ for the time parameter.

Let $(W_t)_{t>0}$ be a continuous real-valued function\footnote{Again, one usually requires a few things about this function, but let us omit these technical conditions here.}. Fix $z\in\bbH$ and consider the map $t\mapsto g_t(z)$ satisfying the following differential equation up to its explosion time:
\begin{equation}\label{fg}\partial_t
g_t(z)~=~\frac{2}{g_t(z)-W_t}.\end{equation}
For every fix $t$, let $H_t$ be the set of $z$ for which the explosion time of the differential equation above is strictly larger than $t$. One may verify that $H_t$ is a simply connected open set and that $\overline\bbH\setminus H_t$ is compact. Furthermore, the map $z\mapsto g_t(z)$  is a conformal map from $H_t$ to $\bbH$.
 If there
exists a parametrized curve $(\Gamma_t)_{t>0}$ such that for any
$t>0$, $H_t$ is the connected component of $\mathbb H\setminus
\Gamma[0,t]$ containing $\infty$, the curve $(\Gamma_t)_{t>0}$ is called {\em (the curve generating) the Loewner chain with driving process} $(W_t)_{t>0}$. 

The Loewner chain in $({\bf \Omega},{\bf a},{\bf b})$ with driving function $(W_t)_{t>0}$ is simply the image of the Loewner chain in $(\bbH,0,\infty)$ by a conformal from $(\bbH,0,\infty)$ to $({\bf \Omega},{\bf a},{\bf b})$.
\begin{definition}
For $\kappa>0$ and $({\bf \Omega},{\bf a},{\bf b})$, $\mathsf{SLE}(\kappa)$ is the random Loewner evolution in
$({\bf \Omega},{\bf a},{\bf b})$ with driving process $\sqrt \kappa B_t$,
where $(B_t)$ is a standard Brownian motion. 
\end{definition}

The strategy of the proof of Theorem~\ref{thm:interface} is the following. The first step consists in proving that the family $(\gamma_{(\Omega_\delta,a_\delta,b_\delta)})$ is tight for the weak convergence and that any sub-sequential limit $\gamma$ is a curve generating a Loewner chain for a  continuous driving process $(W_t)$ satisfying some integrability conditions. The proof of this fact is technical and can be found in \cite{KemSmi12,CheDumHon14,DumSmi12a}. It is based on a Aizenman-Burchard type argument based on crossing estimates obtained in Property {\bf P5} of Theorem~\ref{thm:main} (see also \cite{CheDumHon13,DumHonNol11} for a stronger statement in the case of the Ising model).

The second step of the proof is based on the fermionic observable, which can be seen as a martingale for the exploration process. This fact implies that its limit is a martingale for $\gamma$. This martingale property, together with It\^o's formula, allows to prove that 
$W_t$ and $W_t^2-\kappa t$ are martingales (where $\kappa$ equals 16/3). L\'evy's theorem thus implies that $W_t=\sqrt\kappa B_t$. This identifies $\mathsf{SLE}(\kappa$) as being the only possible sub-sequential limit, which proves that $(\gamma_{(\Omega_\delta,a_\delta,b_\delta)})$ converges to $\mathsf{SLE}(\kappa$). We now provide more details for this second step.

Below, $\Omega\setminus\gamma[0,n]$ is the slit domain obtained from $\Omega$ by removing all the edges crossed by the exploration path up to time $n$. Also, $\gamma(n)$ denotes the vertex of $\Omega$ bordered by the last edge of $\gamma[0,n]$.
\begin{lemma}
Let $\delta>0$. The random variable $M_n(z):=f_{\Omega\setminus\gamma[0,n],\gamma(n),b}(z)$  is a martingale with respect to $(\mathcal F_n)$ where $\mathcal F_n$ is the $\sigma$-algebra generated by $\gamma[0,n]$. 
\end{lemma}
\begin{proof} 
The random variable $M_n(z)$ is a linear combination of the random variables $M_n(e):=F_{\Omega\setminus\gamma[0,n],\gamma(n),b}(e)$ for $e\ni z$ so that we only need to treat the later random variables.  The fact that conditionally on $\gamma[0,n]$, the law in $\Omega\setminus\gamma[0,n]$ is a random-cluster model with Dobrushin boundary conditions  implies that $M_n(e)$ is equal to ${\rm e}^{\frac12 {\rm i}W_{\gamma}(e,e_b)}\mathbbm{1}_{e\in \gamma}$ conditionally on $\mathcal F_n$, therefore it is automatically a closed martingale.
\end{proof}

\begin{proof}[Theorem~\ref{thm:interface}]
We treat the case of the upper half-plane ${\bf \Omega}=\bbH$ with ${\bf a}=0$ and ${\bf b}=\infty$. The general case follows by first applying a conformal map from $({\bf \Omega},{\bf a},{\bf b})$ to $(\bbH,0,\infty)$. Consider $\gamma$ a sub-sequential limit of $\gamma_{(\Omega_\delta,a_\delta,b_\delta)}$ and assume that its driving process is equal to $(W_t)$.  Define $g_t$ as above.
For $z\in \bbH$ and $\delta>0$, define $M_n^\delta(z)$ for $\gamma_{(\Omega_\delta,a_\delta,b_\delta)}$ as above too.

The stopping time theorem implies that $M^\delta_{\tau_t}(z)$ is a martingale with respect to $\mathcal F_{\tau_t}$, where $\tau_t$ is the first time at which $\gamma_{(\Omega_\delta,a_\delta,b_\delta)}$ has a $h$-capacity larger than $t$. Now, if $M^\delta_{\tau_t}(z)$ converges uniformly as $\delta$ tends to 0, then, the limit $M_t(z)$ is a martingale with respect to the $\sigma$-algebra $\mathcal G_t$ generated by the curve $\gamma$ up to the first time its $h$-capacity exceeds $t$. By definition of the parametrization, this time is $t$, and $\mathcal G_t$ is the $\sigma$-algebra generated by $\gamma[0,t]$. 

Since the conformal map from $\mathbb H\setminus \gamma[0,t]$ to $\mathbb R\times(0,1)$, normalized to send $\gamma_t$ to $-\infty$ and $\infty$ to $\infty$ is $\frac1\pi\ln (g_t-W_t)$, Theorem~\ref{thm:observable} gives that $M_t^\delta(z)$ converges to
\begin{eqnarray}\sqrt \pi M_t(z)=\sqrt{[\ln (g_t(z)-W_t)]'}=\Big(\frac {g'_t(z)}{g_t(z)-W_t}\Big)^{1/2},\end{eqnarray} 
which is therefore a martingale for the filtration $(\calG_t)$. Formally, in order to apply Theorem~\ref{thm:observable}, one needs $z$ and $\gamma[0,\tau_t]$ to be well apart. For this reason, we only obtain that $M_{t\wedge\sigma}^z$ is a martingale for $\mathcal G_{t\wedge\sigma}$, where $\sigma$ is the hitting time of the boundary of the ball of size $R<|z|$ by the curve $\gamma$. 

Recall that
$g_t(z)=z+\tfrac{2t}{z}+O\left(\tfrac1{z^2}\right)$ and $g_t'(z)~=~1-\tfrac{2t}{z^2}+O\left( \tfrac1{z^3}\right)$
so that for $t$,
\begin{align*}\sqrt {\pi z}\, M_{t}(z)&=\Big(\frac{1-\tfrac{2t}{z^2}+O\big(\tfrac{1}{z^3}\big)}{1-\tfrac{W_{t}}z+\tfrac{2t}{z^2}+O\big(\tfrac{1}{z^3}\big)}\Big)^{1/2}
=1
+\tfrac1{2z}W_{t}+\tfrac1{8z^2}(3W_{t}^2-16t)+O\left(\tfrac1{z^3}\right).\end{align*}
Taking the conditional expectation against $\calG_{s\wedge\sigma}$ (with $s\le t$) gives
\begin{align*}\sqrt {\pi z}\,\mathbb E [M_{t\wedge\sigma}(z)|\mathcal G_{s\wedge\sigma}]
&=1
+\tfrac1{2z}\mathbb E[W_{t\wedge\sigma}|\mathcal G_{s\wedge\sigma}]+\tfrac1{8z^2}\mathbb E[3W_{t\wedge\sigma}^2-16(t\wedge\sigma)|\mathcal G_{s\wedge\sigma}]+O\left(\tfrac1{z^3}\right).\end{align*}
Since $M_{t\wedge\sigma}(z)$ is a martingale, $\mathbb E [M_{t\wedge\sigma}(z)|\mathcal G_{s\wedge\sigma}]=M_{s\wedge\sigma}(z)$. Therefore, the terms in the previous asymptotic developments (in $1/z$) can be matched together by letting $z$ tend to infinity so that 
$$\mathbb E [W_{t\wedge\sigma}|\mathcal G_{s\wedge\sigma}]=W_{s\wedge\sigma}\quad\text{ and }\quad\mathbb E[W_{t\wedge\sigma}^2-\tfrac {16}3(t\wedge\sigma)|\mathcal G_{s\wedge\sigma}]=W_{s\wedge\sigma}^2-\tfrac {16}3(s\wedge\sigma).$$ 
One can now let $R$ and thus $\sigma$ go to infinity to obtain
$$\mathbb E [W_t|\mathcal G_s]=W_s\quad\text{ and }\quad\mathbb E[W_t^2-\tfrac {16}3t|\mathcal G_s]=W_s^2-\tfrac {16}3s.$$
(Note that some integrability condition on $W_t$ is necessary to justify passing to the limit here.)
The driving process $W_t$ being continuous, L\'evy's theorem implies that $W_t=\sqrt{16/3}B_t$ where $B_t$ is a standard Brownian motion. 
Since we considered an arbitrary sub-sequential limit, this directly proves that $(\gamma_{(\Omega_\delta,a_\delta,b_\delta)})$ converges weakly to $\mathsf{SLE}(16/3)$.
\end{proof}

Note that despite the fact that the fermionic observable may not seem like a very natural choice at first sight, it is in fact corresponding to a discretization of a very natural martingale of $\mathsf{SLE}(16/3)$.

\section{Where are we standing? And more conjectures...}

It is time to conclude these lectures. To summarize, we proved that the Potts model and its random-cluster representation undergo phase transitions between ordered and disordered phases. We also showed that the long-range order and the spontaneous magnetization phases of the Potts model coincide. Then, we proceeded to prove that the phase transition was sharp, meaning that correlations decay exponentially fast below the critical inverse-temperature. 

After this study of the phases $\beta<\beta_c$ and $\beta>\beta_c$, we moved to the study of the $\beta=\beta_c$ phase. We determined that the phase transition of the Potts model is continuous in any dimension if $q=2$ (i.e.~for the Ising model), and that it is continuous if $q\le 4$ and discontinuous for $q>4$ in two dimensions. This gives us the opportunity of mentioning the first major question left open by this manuscript:
\begin{conjecture}\label{conj:1}
Prove that the phase transition of the nearest-neighbor Potts model on $\bbZ^d$ (with $d\ge3$) is discontinuous for any $q\ge3$.
\end{conjecture}
Let us mention that this conjecture is proved in special cases, namely 
\begin{itemize}[noitemsep]
\item if $d$ is fixed and $q\ge q_c(d)\gg1$ \cite{KotShl82}, 
\item if $q\ge3$ is fixed and $d\ge d_c(q)\gg1$ \cite{BisCha03},
\item if $q\ge3$ and $d\ge2$, but the range of the interactions is sufficiently spread-out \cite{BisChaCra06,GobMer07}.
\end{itemize}

When the phase transition is continuous, there should be some conformally invariant scaling limit. In two dimensions, this concerns any $q\le 4$, and not only the $q=2$ case mentioned previously in these lectures.
One may formulate the conformal invariance conjecture for random-cluster models with $q\le 4$ in the following way.
\begin{conjecture}[Schramm]\label{conj:interface} Fix $q\le 4$ and $p=p_c$. Let $(\Omega_\delta,a_\delta,b_\delta)$ be Dobrushin domains approximating a simply connected domain ${\bf \Omega}$ with two marked points ${\bf a}$ and ${\bf b}$ on its boundary.  
The exploration path $\gamma_{(\Omega_\delta,a_\delta,b_\delta)}$ in $(\Omega_\delta,a_\delta,b_\delta)$  converges weakly to $\mathsf{SLE}(\kappa)$  as $\delta$ tends to $0$, where 
$$\kappa=\tfrac{8}{\sigma+1}=\tfrac{4\pi}{\pi-\arccos(\sqrt q/2)}.$$\end{conjecture}

The values of $\kappa$ range from $\kappa=4$ for $q=4$ to $\kappa=8$ for $q=0$. Also note that $\kappa=6$ corresponds to $q=1$, as expected. Following the same strategy as in the previous section, the previous conjecture would follow from the convergence of vertex parafermionic observables (they are defined for general $q$ as the vertex fermionic observable).
\begin{conjecture}[Smirnov]\label{conj:observable}
Fix $q\le 4$ and $p=p_c$. Let $(\Omega_\delta,a_\delta,b_\delta)$ be Dobrushin domains approximating a simply connected domain ${\bf \Omega}$ with two marked points ${\bf a}$ and ${\bf b}$ on its boundary. If $f_\delta$ denotes the vertex parafermionic observable on $(\Omega_\delta,a_\delta,b_\delta)$ defined as the average of the edge fermionic observable on neighboring edges, then
$$
\lim_{\delta\rightarrow0}(2\delta)^{-\sigma}f_{\delta}=(\phi')^\sigma,$$ where $\phi$ is a conformal map from ${\bf \Omega}$ to the strip $\mathbb R\times(0,1)$ mapping ${\bf a}$ to $-\infty$ and ${\bf b}$ to $\infty$. 
\end{conjecture}

For the Ising model in higher dimension (the other Potts models are predicted to have a discontinuous phase transition by Conjecture~\ref{conj:1}), the model still undergoes a continuous phase transition and it therefore makes sense to study the critical phase in more details.

We mentioned in Theorem~\ref{thm:pol decay exact} that the critical exponent of the spin-spin correlations of the Ising model in dimension four and higher is the mean-field one, i.e. that 
$$\mu_{\beta_c}[\sigma_x\sigma_y] \approx \frac{1}{\|x-y\|^{d-2+\delta}}$$
with $\delta=0$. Also, note that in two dimensions this is not the case since by \eqref{eq:spin-spin}, $\delta=1/4$. In three dimensions, the best known result is Theorem~\ref{thm:pol decay}, which gets rephrased as $\delta\in[0,1]$. The following improvement would be of great value.
\begin{conjecture}
Consider the three dimensional Ising model. There exists $\ep>0$ and $c_0,c_1\in(0,\infty)$ such that for all $x,y\in\bbZ^3$,
$$\frac{c_0}{\|x-y\|^{2-\ep}}\le \mu_{\beta_c}[\sigma_x\sigma_y]\le \frac{c_1}{\|x-y\|^{1+\ep}}.$$
\end{conjecture}

Another question of interest is the question of triviality/non-triviality of the scaling limit of the spin-field. In other words, the question is to measure whether the spin-spin correlations factorize like Gaussian field (i.e.~whether they satisfy the Wick's rule or not). One usually defines the {\em renormalized coupling constant}\,\footnote{Since Wick's rule is equivalent to the fact that $U_4(x_1,x_2,x_3,x_4)$ vanishes (see the definition in Exercise~\ref{exo:U4}), this quantity is a measure of how non-Gaussian the field $(\sigma_x:x\in\mathbb Z^d)$ is.} 
\begin{equation}
g(\beta):=\sum_{x_2,x_3,x_4\in\mathbb Z^d}\frac{U_4(0,x_2,x_3,x_4)}{\chi(\beta)^2\xi(\beta)^d},
\end{equation}
where $U_4(x_1,x_2,x_3,x_4)$ was defined in Exercise~\ref{exo:U4} and ($e_1$ is a unit vector in $\mathbb Z^d$)
\begin{equation}
\chi(\beta):=\sum_{x\in\mathbb Z^d}\mu_{\beta}^{\rm f}[\sigma_0\sigma_x]\quad\text{ and }\quad\xi(\beta):=\big(\lim_{n\rightarrow\infty}-\tfrac1n\log \mu_{\beta}^{\rm f}[\sigma_0\sigma_{ne_1}]\big)^{-1}.
\end{equation}

If $g(\beta)$ tends to 0 as $\beta\nearrow\beta_c$, the field is said to be {\em trivial}. Otherwise, it is said to be {\em non-trivial}. 
Aizenman \cite{Aiz82} and Fr\"ohlich \cite{Fro82} proved that the Ising model is trivial for $d\ge5$. In two dimensions, one can use Theorem~\ref{thm:spin} to prove that the Ising model is non-trivial (in fact one can prove this result in a simpler way, but let us avoid discussing this here). Interestingly enough, Aizenman's proof of triviality is one of the first use of the random current at its full power and it is therefore fair to say that proving this result was one of the motivation for the use of such currents. This leaves the following conjecture open.
\begin{conjecture}
Prove that the three-dimensional Ising model is non-trivial, and that the four-dimensional Ising model is. 
\end{conjecture}

Physics predictions go much further. One expects  conformal invariance in any dimension (in fact as soon as the phase transition is continuous). Conformal symmetry brings less information on the model in dimensions greater than 2, but recent developments in conformal bootstrap illustrate that still much can be said using these techniques, see \cite{Ric12}. It therefore motivates the question of proving conformal invariance in dimension three, which looks like a tremendously difficult problem.

\bexo

\begin{exercise}
[Triviality of Ising in dimension $d\ge5$]

Consider a graph $G$ and denote by ${\rm P}^A_G$ the measure on currents (here we mean one current, not two) on $G$ with set of sources equal to $A$. Set $\sigma_i$ for the spin at $x_i$. 
\medbreak\noindent
1. Show that 
$$U_4(x_1,\dots,x_4)=-2\mu_{G,\beta}^{\rm f}[\sigma_1\sigma_2]\mu_{G,\beta}^{\rm f}[\sigma_3\sigma_4] \cdot {\rm P}^{\{x_1,x_2\}}_G\otimes{\rm P}^{\{x_3,x_4\}}_G[x_1,x_2,x_3,x_4\textrm{ all connected}].$$
2. Prove that for any $y\in\bbZ^d$,
$\displaystyle{\rm P}^{\{x_1,x_2\}}_G\otimes{\rm P}^{\emptyset}_G[x_1\stackrel{\widehat{\n_1+\n_2}}\longleftrightarrow y]=\frac{\mu_{G,\beta}^{\rm f}[\sigma_1\sigma_y]\mu_{G,\beta}^{\rm f}[\sigma_y\sigma_2]}{\mu_{G,\beta}^{\rm f}[\sigma_1\sigma_2]}.$\medbreak\noindent
3. Use two new sourceless currents $\n_3$ and $\n_4$ and the union bound to prove that 
$$0\le -U_4(x_1,\dots,x_4)\le 2\sum_{y\in\bbZ^d}\mu_{G,\beta}^{\rm f}[\sigma_y\sigma_1]\mu_{G,\beta}^{\rm f}[\sigma_y\sigma_2]\mu_{G,\beta}^{\rm f}[\sigma_y\sigma_3]\mu_{G,\beta}^{\rm f}[\sigma_y\sigma_4].$$
4. Deduce that 
$0\le -g(\beta)\le \frac{\chi(\beta)^{2}}{\xi(\beta)^d}.$\medbreak\noindent
5. Show that for every $x\in\bbZ^d$,
$\mu_{G,\beta}^{\rm f}[\sigma_0\sigma_x]\le \exp(-\|x\|_\infty/\xi(\beta)).$\medbreak\noindent
6. Using \eqref{eq:infrared bound}, show that $\chi(\beta)\le C\xi(\beta)^2\log \xi(\beta)^2$ and conclude that $g(\beta)$ tends to 0 when $d\ge5$.
\end{exercise}

\eexo

\bibliographystyle{plain}
\bibliography{biblicomplete}

\end{document}